\documentclass[11pt,a4paper]{amsart}
\usepackage[T1]{fontenc}
\usepackage[utf8]{inputenc}
\usepackage{lmodern,microtype}
\usepackage{amsmath,amssymb,amsthm,mathtools,mathrsfs}
\usepackage[textwidth=15.8cm,textheight=24.5cm,centering,footskip=10mm]{geometry}
\usepackage{enumitem,booktabs,array}
\usepackage{xcolor}
\usepackage{tikz}
\usetikzlibrary{arrows.meta,positioning,calc,fit,backgrounds,
  decorations.markings,decorations.pathreplacing}

\definecolor{diagramblue}{RGB}{32,68,101}
\definecolor{diagramlight}{RGB}{240,245,250}
\definecolor{inkgray}{RGB}{95,103,112}
\definecolor{liftblue}{RGB}{28,94,170}
\definecolor{liftorange}{RGB}{212,110,18}
\definecolor{orbitred}{RGB}{180,36,40}
\definecolor{dgreen}{RGB}{28,122,78}
\definecolor{dviolet}{RGB}{116,66,158}
\definecolor{rigidcol}{RGB}{0,112,104}
\definecolor{stdcol}{RGB}{30,84,152}
\definecolor{ineqcol}{RGB}{168,62,28}
\definecolor{constrcol}{RGB}{88,70,136}
\definecolor{highlight}{RGB}{255,240,180}
\colorlet{belt}{liftblue}
\colorlet{bigonfill}{highlight}
\colorlet{bandfill}{gray!45}
\tikzset{
  dbox/.style={draw=diagramblue!80,rounded corners=3pt,fill=diagramlight,
    line width=0.55pt,align=flush center,inner sep=5pt,font=\small},
  darrow/.style={-{Stealth[length=2.1mm,width=1.5mm]},draw=diagramblue,
    line width=0.6pt},
  dlabel/.style={font=\footnotesize,align=flush center,fill=white,inner sep=2pt},
  dnote/.style={align=flush center,font=\footnotesize,inner sep=3pt},
  lab/.style={font=\footnotesize,inner sep=1pt},
  slab/.style={font=\scriptsize,inner sep=1pt},
  tlab/.style={font=\tiny,inner sep=0.5pt},
  panelcap/.style={font=\footnotesize,align=center,inner sep=1pt},
  midarrow/.style={postaction={decorate,decoration={markings,
    mark=at position #1 with {\arrow{Stealth[length=1.7mm,width=1.3mm]}}}}},
  midarrow/.default=0.55
}
\usepackage[colorlinks=true,linkcolor=blue!55!black,citecolor=blue!55!black,urlcolor=blue!55!black, pagebackref=true]{hyperref}
\hypersetup{pdftitle={Fibred realizations of the prism quandle P7 in the standard four-sphere},
  pdfauthor={Michal Jablonowski},
  pdfsubject={Prism quandle P7; framed surgery; standard four-sphere; Whitehead orbifold; circle actions; fibred two-knots; Teichner's secondary invariant}}
\numberwithin{equation}{section}
\newtheorem{theorem}{Theorem}[section]
\newtheorem{proposition}[theorem]{Proposition}
\newtheorem{lemma}[theorem]{Lemma}
\newtheorem{corollary}[theorem]{Corollary}
\newtheorem*{mainthm}{Main Theorem}
\theoremstyle{definition}
\newtheorem{definition}[theorem]{Definition}
\theoremstyle{remark}
\newtheorem{remark}[theorem]{Remark}
\newcommand{\ZZ}{\mathbb Z}
\newcommand{\FF}{\mathbb F}
\newcommand{\QQ}{\mathbb Q}
\newcommand{\RR}{\mathbb R}
\newcommand{\diag}{\operatorname{diag}}

\newcommand{\Isom}{\operatorname{Isom}}

\newcommand{\Ree}{\operatorname{Re}}

\newcommand{\eps}{\varepsilon}
\newcommand{\Lam}{\Lambda}

\newcommand{\Fix}{\operatorname{Fix}}
\newcommand{\HH}{\mathbb H}
\newcommand{\GAlex}{\operatorname{GAlex}}
\newcommand{\As}{\operatorname{As}}
\newcommand{\id}{\mathrm{id}}
\newcommand{\Int}{\operatorname{int}}
\newcommand{\ncl}[1]{\langle\!\langle #1\rangle\!\rangle}
\newcommand{\diff}{\cong_{\mathrm{diff}}}
\newcommand{\E}{E_7}
\newcommand{\M}{M_7}
\newcommand{\N}{N_7}
\newcommand{\Sig}{\Sigma_{+}(7)}

\newcommand{\Spin}{\operatorname{Spin}}
\newcommand{\Aut}{\operatorname{Aut}}
\newcommand{\Out}{\operatorname{Out}}
\newcommand{\spn}{\operatorname{sp}}

\newcommand{\CP}{\mathbb {CP}}
\newcommand{\Hom}{\operatorname{Hom}}
\newcommand{\Ext}{\operatorname{Ext}}

\newcommand{\spLt}{L(3,1)}

\newcommand{\Diff}{\operatorname{Diff}}

\newcommand{\Stab}{\operatorname{Stab}}

\providecommand{\Tst}{T^{*}}
\providecommand{\Qtt}{Q_{28}}
\providecommand{\pmx}{\mathbin{\times_{\pm1}}}
\providecommand{\OW}{\mathcal O_{W}}
\providecommand{\lk}{\operatorname{lk}}

\providecommand{\stdpi}{\pi_1^{\mathrm{orb}}}

\title[Fibred realizations of the prism quandle $P_7$ in the standard four-sphere]
{Fibred realizations of the prism quandle $P_7$\\ in the standard four-sphere}
\author{Micha{\l} Jab{\l}onowski}

\address{Institute of Mathematics, Faculty of Mathematics, Physics and
	Informatics, University of Gda\'nsk, 80-308 Gda\'nsk, Poland}

\email{\href{mailto:michal.jablonowski@gmail.com}
	{michal.jablonowski@gmail.com}}
\date{October 8, 2026}

\begin{document}
\begin{abstract}
We study smooth fibred two-knots whose fundamental quandle is the prism
quandle of order fifty-six. The candidates arise from framed circle surgery along a section of a
mapping torus, using the two possible normal-framing classes.
We prove that both ambient homotopy four-spheres are diffeomorphic to the
standard four-sphere. The candidates have the same knot group and
diffeomorphic exteriors, but they are inequivalent even as topological pairs.
Every smooth fibred realization of this quandle in a homotopy four-sphere
is equivalent, up to knot orientation, to exactly one of these two knots.
Standardness follows from a Whitehead-link model for the orbit orbifold together with
circle-action recognition. Inequivalence is detected by the third cyclic
branched covers, on which a spin-bordism secondary invariant takes different
values. Thus the parameter-seven prism quandle admits smooth fibred
realizations in the standard four-sphere.
\end{abstract}

\maketitle

\section{The main result and its motivation}\label{sec:target}
We study smooth fibred two-knots whose fundamental quandle is the prism
quandle $P_7$ of order fifty-six. In the right-quandle convention of
Section~\ref{sec:construction},
$
   P_7=\GAlex(Q_8\times C_7,\theta),
$
where $\theta$ has order six. The corresponding marked knot group is
$
 G_7=(Q_8\times C_7)\rtimes_\theta\langle t\rangle,
 \; G_7'=Q_8\times C_7,
$
with $t$ representing the positive meridian. The fundamental quandle is
the realization target, whereas the marked group supplies the algebraic data for the mapping-torus and surgery
construction. The quandle--group identification and its sources are recalled
in the subsection ``Quandle, marked group and mapping torus'' of
Section~\ref{sec:construction}.

Let $\N=S^3/(Q_8\times C_7)$ be the spherical prism manifold and let $\M$
be the mapping torus of the order-six isometry $\bar f$ defined in
Section~\ref{sec:construction}. This isometry induces the outer automorphism class
$[\theta]$.
Framed circle surgery on a section $\gamma$ of $\M\to S^1$, using the two
normal-framing classes $\lambda_\pm$, gives smooth homotopy four-spheres
$\Sigma_\pm(7)$ with belt spheres $K_\pm$. This is the prism-family construction of \cite{QextC}, recalled in
Proposition~\ref{prop:exists}.
Freedman's theorem identifies the ambient manifolds topologically with $S^4$
\cite{Freedman}; the issue addressed here is their smooth type and the
equivalence of the resulting knot pairs.

\begin{mainthm}[Theorem~\ref{std:thm:main} and Corollary~\ref{std:cor:knots}]
Both $\Sigma_+(7)$ and $\Sigma_-(7)$ are diffeomorphic to $S^4$. Their belt
spheres $K_+,K_-\subset S^4$ are smooth fibred two-knots with fundamental
quandle $P_7$, knot group $G_7$, and fibre $\N\setminus\Int D^3$. Their unpointed closed monodromy class
(Definition~\ref{def:fibred}) is represented by the order-six isometry
$\bar f$.

Although the two knots have diffeomorphic exteriors, they are not equivalent
(Definition~\ref{def:equivalence}) for any choice of orientations. Thus neither
knot is determined by its exterior. In fact, no homeomorphism of pairs
identifies them (Remark~\ref{eval:rem:topological}). This topological assertion uses
the homotopy invariance of Teichner's secondary invariant
\cite[Theorem~6.4.1]{Teichner1992}. Every smooth oriented fibred
two-knot with quandle $P_7$ in a smooth homotopy four-sphere is equivalent to exactly one of them, possibly after
reversing the knot orientation.
\end{mainthm}

\subsection{The classical prism problem and the role of \texorpdfstring{$P_7$}{P7}}
The quaternionic branch of the finite-commutator classification of two-knot
groups goes back to Yoshikawa \cite{Yoshikawa1980} and Hillman
\cite[Section~15.1, Theorem~15.4, pp.~294--295]{Hillman}.
Yoshikawa constructs fibred realizations in homotopy four-spheres for the
prism family with commutator subgroup $Q(8)\times(\ZZ/n\ZZ)$, with $n$ odd
\cite{Yoshikawa1980}; for $n>1$ these groups are not groups of twist spins
\cite[pp.~5--6]{Hillman2023}.
Plotnick and Suciu showed that every pair consisting of a group acting linearly
and freely on $S^3$ and an admissible monodromy automorphism is realized by a
fibred knot in a homotopy four-sphere homeomorphic to $S^4$
\cite[Proposition~5.1, p.~523]{PlotnickSuciu}. They recorded that for
$Q(8)\times(\ZZ/n\ZZ)$ with $n\neq1$ only such homotopy-sphere realizations
were available \cite[p.~522]{PlotnickSuciu}, and their table records this family,
whose monodromy has order six, as realized only in a homotopy four-sphere,
with reflexivity left undecided \cite[Section~5, Table~1, p.~525]{PlotnickSuciu}.
In the classical prism-manifold notation, in which $M_d$ is the prism manifold
with fundamental group $Q(8)\times(\ZZ/d\ZZ)$ and $M_d^\circ$ is $M_d$ with an
open three-ball removed, Kanenobu \cite{Kanenobu1988} and
Teragaito \cite[pp.~117--118]{Teragaito1989} construct $M_d^\circ$-fibred
two-knots in $S^4$ for $d=5,11,13,19$, as credited in
\cite[p.~179]{Teragaito1990}; Teragaito's subsequent paper gives the range
$d=3,5,11,13,19,21,27$ and records the remaining realization problem
\cite[p.~179]{Teragaito1990}. Hillman records this latter range, credits
Kanenobu \cite{Kanenobu1988} and Teragaito, and asks whether smooth fibred
realizations in the standard $S^4$ exist in general
\cite[Section~15.1, p.~295; Section~18.8(4), p.~364]{Hillman}.
This classical notation is confined to the present paragraph: in it,
$M_7$ is the prism manifold denoted by $\N$ in this paper and
$M_7^\circ=\N\setminus\Int D^3$, whereas everywhere else in this paper
$\M$ denotes the four-dimensional mapping torus \eqref{eq:mappingtorus}.
The marked group $G_7$ corresponds to the parameter $n=d=7$, which does
not appear in those lists. The broader theory
of fibred two-knots with spherical fibre and geometric monodromy is developed
in \cite{Plotnick,PlotnickSuciu,HillmanPlotnick1990}.

In the finite-quandle classification of \cite{QextC}, the same parameter
appears in the first remaining problem of smooth realization in the standard
four-sphere.
More precisely, a finite locally flat two-knot quandle that has no smooth
realization in the standard $S^4$ has at least fifty-six elements, and
equality is possible only for $P_7$
\cite[Corollaries~12.34 and~23.124(i)]{QextC}.
The Main Theorem excludes this equality case.
The quandle $P_7$ also has the product decomposition
\begin{equation}\label{eq:product}
                         P_7\cong P_1\times R_7,
\end{equation}
by \cite[Proposition~23.23]{QextC}. Both factors are defined in
Section~\ref{sec:construction}, and both already admit smooth realizations in the standard $S^4$
\cite[Theorems~14.14 and~14.24]{QextC}. The Main Theorem settles the
parameter-seven case of the classical problem and the instance $(P_1,7)$ of
the product-realization question \cite[Problem~23.25]{QextC}.

\subsection{Organization}
After the surgery construction and fibred reduction in
Sections~\ref{sec:construction}--\ref{sec:rigidity},
Sections~\ref{per:sec:periodic}--\ref{per:sec:action} describe the order-six
closed monodromy, its fixed-point data and the canonical circle action on
$\M$. The quaternionic fixed-point criterion and the geometry of the
corresponding mapping tori are classical \cite{DuVal1964,ConwaySmith2003},
\cite[Sections~11.4 and~13.5]{Hillman}.
Section~\ref{nf:sec:spin} supplies the cyclic branched-cover groups and
calibrated spin structures needed for the inequivalence argument.

The geometric step in Section~\ref{wh:sec:whitehead} identifies the orbit
orbifold with the result of $+3$ surgery on one component of a specified
mirror image of the Whitehead link, with cone angle $\pi$ along the other component
(Theorem~\ref{wh:thm:whitehead}). The spherical-orbifold and involution
calculations use results from
\cite{Dunbar1981,Dunbar1988,MecchiaSeppi2015,MecchiaSeppi2019,MecchiaSchilling2025},
with the relevant entries cited at their points of use. The branched-cover
identification applies the Montesinos trick \cite{Montesinos1975} and the
Whitehead-link lift computed in \cite[Section~5]{GordonLidman2014}.

Section~\ref{std:sec:main} then converts section surgery into a
circle-surgery problem over a new orbit orbifold. The invariant fibre torus
through the section becomes a two-sphere; surgery on that sphere induces
Dehn surgery on the base (Proposition~\ref{std:prop:torus}). Two consecutive
framings of a clasp belt in the Whitehead model give bases with orbifold
fundamental groups $S_3$ and $\ZZ/2r$, where $r\in\{2,4\}$
(Theorem~\ref{std:thm:belt}); for the mirror image fixed in
Section~\ref{wh:sec:whitehead}, we have $r=4$, so the cyclic group is $\ZZ/8$ (Remark~\ref{std:rem:sstar}). The recognition criteria of
Propositions~\ref{std:prop:cyclic} and~\ref{std:prop:dihedral}, using
Chen's classification of fixed-point-free circle actions \cite{Chen2015}
and the classical circle-action results
\cite{MontgomeryYang,FintushelHomotopy,FintushelSC,Fintushel1978,Pao},
identify both surgery manifolds with the standard $S^4$.
Gluck reconstructions along branched twist spins in $S^4$ again give
$S^4$: Gordon proved this for twist spins \cite{Gordon1976}, and Pao's
analysis of circle actions gives it, implicitly, for all branched twist
spins \cite{Pao}; the effect on the orbit data is described in
\cite{FukudaGluck}. For a nonzero twisting parameter, these knots admit
periodic geometric monodromy \cite[Theorem~2.2]{FukudaIshikawa}.
The zero-twist-spun case is included in the broader terminology of
\cite[Definition~2.1]{FukudaGluck} and is not covered by this periodicity
assertion. The knots $K_\pm$ admit no finite-order geometric monodromy
(Remark~\ref{rem:closed-vs-punctured-monodromy}); here the circle action
is defined on $\M$, and the framed surgery is traded for the orbit-space
Dehn surgery of Proposition~\ref{std:prop:torus}.

Together with \cite[Theorem~2.2]{FukudaIshikawa}, this absence of
finite-order geometric monodromy, which holds for every fibration of the
common exterior (Remark~\ref{rem:monodromy-direct}), shows that neither
$K_+$ nor $K_-$ is a branched twist spin with nonzero twisting parameter;
when the parameter is at least two, the knot group alone rules this out
\cite[Theorem~2, condition~(5)]{Hillman2023}. Nor is either knot a
zero-twist branched twist spin: such a knot is the spun knot of a classical
knot \cite[Section~2.1]{FukudaIshikawa}, so its group is a classical
knot group. Such groups are torsion free \cite{Rolfsen1976}: knot exteriors in $S^3$ are aspherical
\cite{Papakyriakopoulos1957}, and a group with a finite-dimensional
classifying space is torsion free. A classical knot group therefore has
either a trivial or an infinite commutator subgroup, whereas
$G_7'=\Gamma$ is finite and nontrivial. Hence
the results on Gluck reconstructions along branched twist spins recalled
above do not apply to $K_\pm$.

Section~\ref{quat:sec:quaternionic} distinguishes the knots by their third
cyclic branched covers, following Plotnick's method \cite{Plotnick1986}.
A spin-bordism calculation using \cite{Teichner1992,Teichner1993,KPT2020}
shows that Teichner's secondary invariant takes the values $1$ and $0$ for the covers
of $K_+$ and $K_-$, respectively (Theorem~\ref{eval:thm:bit}). The known
reflexivity criteria for finite-order monodromy do not directly apply here;
see \cite{Plotnick1986,PlotnickSuciu,HillmanPlotnick1990},
\cite[Sections~18.1--18.2]{Hillman}, \cite{Hillman2011},
and Remark~\ref{eval:rem:reflexivity}.
Section~\ref{std:sec:consequences} gathers the resulting classification and
branched-cover consequences.

\section{The explicit framed surgery}\label{sec:construction}
\subsection{Quandle, marked group and mapping torus}
We use the right-quandle convention: $x*x=x$, each map
$S_y:x\mapsto x*y$ is a bijection, and $(x*y)*z=(x*z)*(y*z)$.
For a group $H$ and an automorphism $\varphi$ of $H$, the corresponding generalized Alexander
quandle has the operation $x*y=\varphi(xy^{-1})y$. The associated group of a quandle $Q$ is
\[
 \As(Q)=\langle e_x\ (x\in Q)\mid
            e_{x*y}=e_y^{-1}e_xe_y\ (x,y\in Q)\rangle.
\]
These conventions agree with the standard knot-quandle convention of
\cite{Joyce,WinterHigher}.

\begin{definition}[Weight elements and meridianal automorphisms]\label{def:weight}
An element of a group is a \emph{weight element} if its normal closure is
the whole group; its orbit under the automorphism group is a
\emph{weight orbit}. If $G^{\mathrm{ab}}\cong\ZZ$ and an epimorphism
$\epsilon:G\to\ZZ$ has been fixed, a \emph{positive strict weight orbit}
is an orbit of
$
 \Aut_\epsilon(G)=\{\eta\in\Aut(G):\epsilon\circ\eta=\epsilon\}
$
on the weight elements $w$ with $\epsilon(w)=1$.
An automorphism $\varphi$ of a group $H$ is \emph{meridianal} if
$
 \ncl{h^{-1}\varphi(h)\mid h\in H}=H.
$
Equivalently, the subgroup generated by these elements is $H$, since
that subgroup is normal in $H$.
These are the conventions of
\cite[Sections~14.7--14.8]{Hillman} and
\cite[Definition~12.41, Definition~14.1 and Lemma~14.2(i)]{QextC}.
For $G_7$, the epimorphism $\epsilon$ is the degree homomorphism sending
$t$ to $1$.
\end{definition}

\begin{definition}[Equivalence of two-knots]\label{def:equivalence}
Smooth oriented two-knots $K\subset\Sigma$ and $K'\subset\Sigma'$ in smooth
homotopy four-spheres are \emph{equivalent}, or \emph{diffeomorphic as pairs},
if some diffeomorphism $\Sigma\to\Sigma'$ carries $K$ onto $K'$, preserving the
orientations of the knots; no condition is imposed on the orientations of
$\Sigma$ and $\Sigma'$. The qualification ``possibly after reversing the knot
orientation'' drops the condition on the knot orientations as well.
\end{definition}

\begin{definition}[Fibred two-knots and closed monodromy]\label{def:fibred}
Let $K$ be a smooth two-knot in a smooth homotopy four-sphere $\Sigma$, let
$X=\Sigma\setminus\Int\nu K$ be its exterior, and identify $\partial X$ with
$K\times S^1$ so that each circle $\{x\}\times S^1$ is a meridian.

Such a trivialization exists in the present setting. Indeed,
$H_2(\Sigma;\mathbb Z)=0$, so the self-intersection number of $K$ is zero. Hence the
Euler class of the oriented normal $2$-plane bundle of $K$ vanishes; since
oriented real $2$-plane bundles over $S^2$ are classified by their Euler class,
the normal bundle is trivial.
This identification is induced by a trivialization
$\nu K\cong K\times D^2$ and is unique up to isotopy once the orientation
of the meridians is fixed, because tubular neighbourhoods are unique up to
isotopy \cite[Chapter~4, Section~5]{Hirsch} and $[S^2,SO(2)]=0$. The knot
is \emph{fibred} if $X$ admits a smooth fibre bundle projection
$p_X\colon X\to S^1$ whose restriction to $\partial X$ is the projection
$K\times S^1\to S^1$. A meridian maps with degree one, so
$(p_X)_*:\pi_1(X)\to\pi_1(S^1)$ is surjective. The homotopy exact
sequence of the bundle then implies that its fibre $F$ is connected.
Thus $F$ is a compact connected three-manifold with
$\partial F\cong S^2$; the \emph{closed fibre} is
$\widehat F=F\cup_{\partial}D^3$.
A \emph{return map} is the first-return map $h\colon F\to F$ of the flow
of a vector field on $X$ that projects to the unit vector field of $S^1$
and agrees with its product lift on a fixed boundary collar. The map $h$
fixes a collar of $\partial F$ pointwise, and $X$ is the
mapping torus of $h$. For a fixed bundle projection $p_X$, a fixed
reference fibre and a fixed product collar, the allowed vector fields
form a convex set. A path between two such fields gives an isotopy
between their return maps, fixed on that collar. Changing the
identification of the reference fibre conjugates the return map.
This uniqueness statement is relative to the chosen fibration. Extending
$h$ by the identity of $D^3$ gives the \emph{closed monodromy}
$\hat h\colon\widehat F\to\widehat F$. Its isotopy class in
$\Diff(\widehat F)$, where isotopies need not fix $D^3$ or any point, taken
up to conjugation by diffeomorphisms of $\widehat F$, is the
\emph{unpointed closed monodromy class} of the fibration.
The class is also independent of the choice of product collar. If $c_0$ and
$c_1$ are two product collars of $\partial X$ compatible with $p_X$, then
$(c_0)_*\partial_t$ and $(c_1)_*\partial_t$ are inward-pointing vector
fields tangent to the fibres of $p_X$. Interpolating linearly between them
and integrating, as in the proof that collars are unique up to isotopy
(compare \cite[Chapter~4, Section~6]{Hirsch}), gives a diffeomorphism $\Psi$
of $X$ that maps each fibre of $p_X$ to itself, is the identity on
$\partial X$ and satisfies $\Psi\circ c_0=c_1$ near $\partial X$. Then
$\Psi_*$ carries the allowed vector fields for $c_0$ to those for $c_1$, the
corresponding return maps are conjugate by $\Psi|_F$, and
$\Psi|_F\cup\id_{D^3}$ is a diffeomorphism between the closed fibres formed
with the collars $c_0$ and $c_1$ that conjugates the closed monodromies.
Both $\widehat F$ and its unpointed closed monodromy class are
independent of the gluing $\partial F\cong\partial D^3$: every diffeomorphism of $S^2$ extends over
$D^3$ \cite{Smale1959}, and if $\Phi$ extends the change of gluing, then
$\id_F\cup\Phi$ identifies the two closed fibres and conjugates one closed
monodromy to the other, both being the identity on $D^3$.
A \emph{geometric monodromy} is any diffeomorphism $c$ of $F$, not
necessarily fixing $\partial F$ pointwise, such that $X$ is diffeomorphic
over $S^1$ to the mapping torus of $c$; it has \emph{finite order} if $c$
is periodic.

The notation $c$, $c_0$, $c_1$, $\Phi$ and $\Psi$ has the meanings given above
only within this definition. Elsewhere in
Sections~\ref{sec:construction} and~\ref{sec:rigidity}, $c$ is the
generator of $C_7$, $c_0$ is the projected geodesic of
Section~\ref{sec:construction}, and $\Phi$ is the involution of
Lemma~\ref{lem:weight-loops}.
\end{definition}

Write
$
 Q_8=\{\pm1,\pm i,\pm j,\pm k\}\subset S^3\subset\HH,
 \;\; C_7=\langle c\mid c^7=1\rangle,
 \;\; \Gamma=Q_8\times C_7.
$

Set
$
 \omega=\tfrac12(-1+i+j+k),\;\;
 \alpha(q)=\omega^{-1}q\omega,\;\;
 \theta(q,c^a)=(\alpha(q),c^{-a}).
$
Since $\alpha(i)=j$, $\alpha(j)=k$ and $\alpha(k)=i$, the
automorphism $\alpha$ has order three and $\theta$ has order six. Define
\begin{align}
 P_7&=\GAlex(\Gamma,\theta),\;\;\;
       x*y=\theta(xy^{-1})y,\;\;\; |P_7|=56,
       \notag\\
 G_7&=\Gamma\rtimes_\theta\langle t\rangle,\;\;\;
       t^{-1}gt=\theta(g)\;\;(g\in\Gamma),\;\;\;
       \langle t\rangle\cong\ZZ.\label{eq:group}
\end{align}
The convention in \eqref{eq:group} is fixed throughout: the positive
meridian acts by $g\mapsto t^{-1}gt$. Proposition~\ref{prop:exists}
establishes the isomorphism $\As(P_7)\cong G_7$ using the knot exterior.
The componentwise formula gives
\eqref{eq:product}, where $P_1=\GAlex(Q_8,\alpha)$ and $R_7$ has the operation
$a*b=2b-a$ on $\ZZ/7$.

The standard correspondence between quandles and marked groups identifies
$P_7$ with the peripheral coset quandle of \eqref{eq:group}:
\begin{equation}\label{eq:coset}
 \Gamma\to\langle t\rangle\backslash G_7,\;\;\;
 x\longmapsto\langle t\rangle x,
 \;\;\;
 \langle t\rangle x*\langle t\rangle y
 =\langle t\rangle\,\theta(xy^{-1})y.
\end{equation}
This is the parameter-seven case of the construction in
\cite[Section~3, especially (3.2), pp.~45--46, and Theorem~14.24]{QextC}.

To interpret the algebra geometrically, consider the two independent factors
of $\theta$: a three-cycle on the quaternionic axes and inversion on $C_7$.
Figure~\ref{did:fig:dictionary} also records how the same quandle element
appears as a peripheral coset and as a conjugate meridian.
\begin{figure}[tbp]
\centering
\begin{tikzpicture}[
  vdot/.style={circle,fill=#1,inner sep=0pt,minimum size=4.2pt},
  qnode/.style={draw=#1!70,fill=white,rounded corners=2pt,line width=0.5pt,
    font=\scriptsize,inner sep=2pt,minimum width=1.05cm,minimum height=0.5cm},
  flow/.style={-{Stealth[length=1.9mm,width=1.4mm]},draw=#1,line width=0.6pt}]

\begin{scope}
\def\r{1.3}
\coordinate (pi) at (210:\r); \coordinate (pj) at (330:\r); \coordinate (pk) at (90:\r);
\coordinate (ni) at (30:\r);  \coordinate (nj) at (150:\r); \coordinate (nk) at (270:\r);

\draw[gray!60,densely dashed,line width=0.45pt] (ni) -- (nj) -- (nk) -- cycle;
\foreach \a/\b in {ni/nj,nj/nk,nk/ni}
  \path[postaction={decorate,decoration={markings,mark=at position 0.62 with
    {\arrow[gray!75]{Stealth[length=1.4mm,width=1.1mm]}}}}] (\a) -- (\b);

\fill[liftblue!30,opacity=0.8] (pi) -- (pj) -- (nk) -- cycle;
\fill[liftblue!12,opacity=0.8] (pi) -- (nj) -- (pk) -- cycle;
\fill[liftblue!22,opacity=0.8] (pj) -- (ni) -- (pk) -- cycle;
\fill[liftblue!5,opacity=0.8]  (pi) -- (pj) -- (pk) -- cycle;
\draw[diagramblue!85,line width=0.55pt,line join=round]
  (pi) -- (nk) -- (pj) -- (ni) -- (pk) -- (nj) -- cycle;

\draw[liftblue,line width=0.9pt,midarrow=0.6] (pi) -- (pj);
\draw[liftblue,line width=0.9pt,midarrow=0.6] (pj) -- (pk);
\draw[liftblue,line width=0.9pt,midarrow=0.6] (pk) -- (pi);

\draw[orbitred,line width=0.55pt] (0,0) circle[radius=0.085];
\fill[orbitred] (0,0) circle[radius=0.03];
\draw[orbitred,line width=0.6pt,-{Stealth[length=1.5mm,width=1.2mm]}]
  (200:0.36) arc[start angle=200,end angle=475,radius=0.36];
\node[slab,text=orbitred] at (0.55,0.02) {$\alpha$};
\foreach \v/\a/\t in {pi/210/i,pj/330/j,pk/90/k}
  {\node[vdot=liftblue] at (\v) {}; \node[lab,text=liftblue] at (\a:\r+0.27) {$\t$};}
\foreach \v/\a/\t in {ni/30/-i,nj/150/-j,nk/270/-k}
  {\node[vdot=gray!70] at (\v) {}; \node[slab,text=gray!70!black] at (\a:\r+0.3) {$\t$};}
\node[panelcap,anchor=north] at (0,-1.86) {(a) $\alpha(q)=\omega^{-1}q\,\omega$: rotation by $2\pi/3$\\
  about the axis through $i+j+k$};
\end{scope}

\begin{scope}[shift={(4.55,0)}]
\def\r{1.12}
\fill[diagramlight] (0,0) circle[radius=\r+0.12];
\draw[diagramblue!35] (0,0) circle[radius=\r+0.12];
\foreach \a in {0,...,6} \coordinate (h\a) at ({90-360*\a/7}:\r);
\draw[diagramblue!45,line width=0.45pt] (h0)--(h1)--(h2)--(h3)--(h4)--(h5)--(h6)--cycle;
\draw[orbitred,densely dashed,line width=0.6pt] (0,\r+0.16) -- (0,-\r-0.4);
\foreach \a/\b in {1/6,2/5,3/4}
  \draw[{Stealth[length=1.5mm,width=1.2mm]}-{Stealth[length=1.5mm,width=1.2mm]},
    liftorange,line width=0.7pt,shorten <=3.2pt,shorten >=3.2pt] (h\a) -- (h\b);
\foreach \a in {0,...,6} {
  \node[vdot=diagramblue] at (h\a) {};
  \node[slab] at ({90-360*\a/7}:\r+0.36) {$\a$};}
\draw[liftorange,line width=0.7pt] (h0) circle[radius=0.12];
\node[panelcap,anchor=north] at (0,-1.86) {(b) $c^a\mapsto c^{-a}$: reflection\\ in the dashed axis};
\end{scope}

\begin{scope}[shift={(9.45,0)}]
\def\r{1.38}
\draw[diagramblue!10,line width=3pt] (0,0) circle[radius=\r];
\foreach \k/\lb/\c in {0/{(i,c)}/liftblue,1/{(j,c^{-1})}/liftorange,
    2/{(k,c)}/liftblue,3/{(i,c^{-1})}/liftorange,4/{(j,c)}/liftblue,5/{(k,c^{-1})}/liftorange}
  \node[qnode=\c] (t\k) at ({90-60*\k}:\r) {$\lb$};
\foreach \k in {0,...,5} {
  \pgfmathtruncatemacro{\kk}{mod(\k+1,6)}
  \draw[flow=diagramblue,shorten <=1pt,shorten >=1pt] (t\k) to[bend left=18] (t\kk);}
\node[slab,align=center] at (0,0.1) {$\theta=\alpha\times(-1)$};
\node[slab,align=center,text=inkgray] at (0,-0.24) {period $6$};
\node[panelcap,anchor=north] at (0,-1.86) {(c) a $\theta$-orbit of length six;\\
  $P_7=P_1\times R_7$ has $8\cdot7=56$ elements};
\end{scope}

\begin{scope}[shift={(0,-3.75)}]
\node[dbox,text width=3.0cm] (element) at (0.55,0)
  {$x\in\Gamma$\\ $x*y=\theta(xy^{-1})y$};
\node[dbox,text width=3.0cm] (coset) at (5.6,0)
  {$\langle t\rangle x\in\langle t\rangle\backslash G_7$\\ peripheral coset};
\node[dbox,text width=3.0cm] (meridian) at (10.65,0)
  {$x^{-1}tx\in G_7$\\ conjugate meridian};
\draw[darrow] (element) -- node[dlabel,above=1pt] {\eqref{eq:coset}} (coset);
\draw[darrow] (coset) -- node[dlabel,above=1pt] {assign} (meridian);
\node[dnote] at (5.6,-1.0)
  {$e_x\longmapsto x^{-1}tx,\qquad
    e_{x*y}\longmapsto (y^{-1}ty)^{-1}(x^{-1}tx)(y^{-1}ty)$};
\node[panelcap] at (5.6,-1.55) {(d) three descriptions of one quandle element};
\end{scope}
\end{tikzpicture}
\caption{Reading the quandle--group dictionary. (a) The units $\pm i,\pm j,\pm k$
of $Q_8$ are the vertices of an octahedron, seen here along the axis through
$i+j+k$. Since $\omega=\cos\frac{2\pi}{3}+\frac{i+j+k}{\sqrt3}\sin\frac{2\pi}{3}$,
the automorphism $\alpha$ is the rotation through $2\pi/3$ about this axis: it
cycles $i\mapsto j\mapsto k$ on the front face and, in the same sense,
$-i\mapsto-j\mapsto-k$ on the hidden back face (dashed). (b) Residues in $\ZZ/7$
are placed on a regular heptagon, so inversion is the reflection in the dashed
axis; it fixes $0$ and exchanges $a$ with $-a$. (c) The two factors have periods
three and two, so $\theta=\alpha\times(-1)$ has period $\operatorname{lcm}(3,2)=6$;
for example the $\theta$-orbit of $(i,c)$ has six elements. The colour of a
node records its $C_7$-coordinate, $c$ or $c^{-1}$. (d) Start at the left: a group
element labels a quandle element, its coset is the peripheral description,
and the corresponding associated-group generator maps to the displayed
meridian. The formula underneath checks that the right-quandle operation
becomes right conjugation. The last arrow is the assignment of a meridian to a
coset, without asserting that different cosets give different conjugates.}
\label{did:fig:dictionary}
\end{figure}

This family of groups is classical: it belongs to the quaternionic branch
of the classification in
\cite[Section~15.1, Theorem~15.4, pp.~294--295]{Hillman}.
The elementary group data, expressed in the right-conjugation convention of \eqref{eq:group}, are as follows.

\begin{lemma}[The group data used in surgery]\label{lem:group-data}
The element $t$ normally generates $G_7$, and
$G_7^{\mathrm{ab}}\cong\ZZ$, generated by the class of $t$.
Consequently $G_7'=\Gamma$.
\end{lemma}
\begin{proof}
This is the parameter-seven specialization of
\cite[Proposition~14.22 and Theorem~14.24, pp.~241--244]{QextC}; in the
convention \eqref{eq:group}, the meridianal automorphism is $\theta$ and the
stable letter $t$ is the positive weight element.
\end{proof}

\begin{proof}[Direct verification]
For $g=(q,c^a)\in\Gamma$, the convention \eqref{eq:group} gives
\[
 [g,t]=g^{-1}t^{-1}gt=g^{-1}\theta(g)=\bigl(q^{-1}\alpha(q),\,c^{-2a}\bigr).
\]
For $g=(i,1)$, $(j,1)$ and $(1,c)$ this yields $(-k,1)$, $(-i,1)$ and
$(1,c^{-2})$. The first two generate $Q_8$, and $c^{-2}$ generates $C_7$
because $2$ is invertible modulo $7$. Hence these commutators generate
$\Gamma$, and $\theta$ is meridianal in the sense of
Definition~\ref{def:weight}. Since each $[g,t]$ belongs to $\ncl{t}$, we obtain
$\Gamma\subseteq\ncl{t}$ and therefore $\ncl{t}=G_7$. Likewise
$\Gamma\subseteq G_7'$, while $G_7/\Gamma\cong\ZZ$ is abelian. Hence
$G_7'=\Gamma$, and $G_7^{\mathrm{ab}}=G_7/\Gamma\cong\ZZ$ is generated by
the class of $t$.
\end{proof}

Let $\zeta=\exp(2\pi i/7)$ and let $\Gamma$ act on the unit quaternions by
\begin{equation}\label{eq:action}
              (q,c^a)\cdot v=qv\zeta^{-a}.
\end{equation}
Write $\gamma_{q,a}(v)=qv\zeta^{-a}$ for the isometry of $S^3$ given by
$(q,c^a)$; the same symbol denotes its extension
$\gamma_{q,a}\times\id_{\RR}$ to $S^3\times\RR$ in \eqref{per:eq:torus}.
The action is free. Indeed, if $qv\zeta^{-a}=v$, then
$q=v\zeta^av^{-1}$, so the order of $q$ divides both $4$ and $7$.
Consequently $q=1$ and $a=0$. Moreover, for
$f(v)=\omega^{-1}vj$, the identity $\zeta^{-a}j=j\zeta^a$ gives
\[
 f\gamma_{q,a}(v)
 =\omega^{-1}qv\zeta^{-a}j
 =\alpha(q)\,\omega^{-1}vj\,\zeta^a
 =\gamma_{\alpha(q),-a}f(v).
\]
These are the parameter-seven calculations in
\cite[Theorem~14.24, proof, pp.~243--244]{QextC}.
Thus $\N=S^3/\Gamma$ is a closed oriented spherical three-manifold
with $\pi_1(\N)=\Gamma$, and $f$ descends to an
orientation-preserving isometry $\bar f$.
The unpointed map $\bar f$ determines the outer automorphism class
$[\theta]\in\Out(\Gamma)$, whereas conjugation of deck transformations
by the chosen lift $f$ gives the representative $\theta$.
The section constructed below identifies this representative with
the meridianal automorphism.
The isometry $\bar f$ has order exactly six and is fixed-point-free, as
proved in Theorem~\ref{per:thm:periodic}(ii) and
Proposition~\ref{per:prop:fixed}.
Define
\begin{equation}\label{eq:mappingtorus}
 \M=\N\times[0,1]/(x,1)\sim(\bar f(x),0)
       \xrightarrow{p}S^1.
\end{equation}
Then $\chi(\M)=0$, and, with the positive stable letter denoted by $t$,
we have $\pi_1(\M)=\Gamma\rtimes_\theta\langle t\rangle=G_7$. We shall use the
following universal-cover convention throughout:
\begin{equation}\label{per:eq:torus}
 \M=(S^3\times\mathbb R)/G_7,\;\;\;
 \gamma_{q,a}(v,s)=(qv\zeta^{-a},s),\;\;\;
 D(v,s)=(f^{-1}(v),s+1),
\end{equation}
with $p([v,s])=s\bmod\mathbb Z$. This is the convention used in
\cite[Theorem~14.24 and Definition~18.3]{QextC}.

\subsection{One section and a fixed reference framing}
Let $\widetilde c:[0,1]\to S^3$ be the shortest geodesic joining
$f(1)=\omega^{-1}j$ to $1$, smoothly reparametrized to be stationary
near its endpoints, and let $c_0$ be its projection to $\N$.
Since $c_0(0)=\bar f(c_0(1))$, the curve
$
                         \gamma(s)=[c_0(s),s]
$
is a smooth section of \eqref{eq:mappingtorus}. Choose $t$ to be its
positive class. On $S^3\times\mathbb R$ the deck generator in the positive
base direction is
$
 D(v,s)=(f^{-1}(v),s+1),\;\;\; D(f(1),0)=(1,1).
$
It satisfies $D^{-1}\gamma_{q,a}D=\gamma_{\alpha(q),-a}$, which
fixes the sign convention in \eqref{eq:group}; compare
\cite[Definition~18.3 and Lemma~18.2]{QextC}.

This is the geodesic section of \cite[Definition~18.3, p.~358]{QextC}.
The signs $+$ and $-$ are fixed by the following explicit normal framings;
only the unordered pair of framings is identified with the framing classes in
\cite{QextC}.

Choose a smooth function $r:[0,1]\to[0,1]$ equal to $0$ near $0$
and to $1$ near $1$, and put
\begin{equation}\label{eq:frame}
 u(s)=\exp\bigl(\tfrac{\pi}{2}j\,r(s)\bigr),\;\;\;
 V_a(s)=\widetilde c(s)u(s)a\,u(s)^{-1}
       \;\;(a=i,j,k).
\end{equation}
Project these three vertical vectors to $\M$ along $\gamma$.
Since $\gamma$ is transverse to the fibres, their classes modulo
$T\gamma$ define a framing $\lambda_+$ of its rank-three normal bundle.

\begin{lemma}[The two normal framings]\label{lem:frame}
The framing \eqref{eq:frame} closes smoothly across the mapping-torus
seam. Replacing $u(s)$ by
$u_-(s)=\exp(3\pi j\,r(s)/2)$ gives the other homotopy class of
framings, denoted by $\lambda_-$.
\end{lemma}
\begin{proof}
At the endpoints, $V_a(0)=f(1)a$ and $V_a(1)=jaj^{-1}$.
Since $df_1(w)=\omega^{-1}wj$, we have
\[
                  df_1(V_a(1))=\omega^{-1}ja=f(1)a=V_a(0).
\]
Since $\widetilde c$ and $r$ are constant near $0$ and $1$, the frames
match smoothly across a collar of the seam.
The frame defined by $u_-$ closes in the same way: $u_-(0)=1$ and
$u_-(1)=\exp(3\pi j/2)=-j$, so its vectors at $s=0$ and $s=1$ are again
$f(1)a$ and $(-j)a(-j)^{-1}=jaj^{-1}$, and the displayed identity applies
verbatim. Equivalently, the loop $A$ below satisfies $A(0)=A(1)=\id$.
The two frames differ by the loop
\[
 A(s)(a)=\exp(\pi j\,r(s))\,a\,\exp(-\pi j\,r(s))
                      \;\;\text{in }SO(3).
\]
Under the double covering $S^3\to SO(3)$ given by quaternionic
conjugation, the lift of $A$ starting at $1$ is
$s\mapsto\exp(\pi j\,r(s))$ and ends at $-1$.
Thus $A$ is not null-homotopic and represents the nonzero class of
$\pi_1(SO(3))\cong\ZZ/2$; equivalently, use
$SO(3)\cong\mathbb{RP}^3$ and
\cite[Example~1.43]{Hatcher}.
Once a reference framing has been fixed, any other oriented framing is represented
by a map $S^1\to SO(3)$. Hence the set of homotopy classes is
$[S^1,SO(3)]\cong\pi_1(SO(3))\cong\ZZ/2$, so these are exactly the two
normal-framing classes.
These two classes are taken among the framings that induce on the normal
bundle of $\gamma$ the orientation determined by $\lambda_+$. Framings
inducing the opposite orientation are treated at the end of the proof of
Lemma~\ref{lem:weight-loops}, where they are shown to give the same
surgery pairs up to the orientation of the belt sphere.
\end{proof}

Varying $r$ through functions with the stated endpoint values gives
a homotopy of the reference framing. Stationary reparametrizations
of the fixed geodesic likewise give isotopic sections with transported
framings. These auxiliary choices therefore do not change the
result of the framed surgery.
More generally, a framed surgery depends only on the isotopy class of the
framed embedding of the circle: tubular neighbourhoods are unique up to
isotopy \cite[Chapter~4, Section~5]{Hirsch}, and isotopies of compact
submanifolds extend to ambient isotopies
\cite[Chapter~8, Section~1, Theorem~1.3]{Hirsch}. Hence each framing class
of Lemma~\ref{lem:frame} determines its surgery manifold, with its belt
sphere, up to diffeomorphism of pairs. The initial $+$ convention is fixed by
\eqref{eq:frame}; it need not coincide with an independently chosen
initial framing label in \cite[Definition~18.3]{QextC}.

Using the vertical framing $\lambda_+$ and the fibrewise exponential map,
choose a tubular embedding
\[
 T_+:S^1\times D^3\to\M,
 \;\;\; T_+(s,0)=\gamma(s),\;\;\; p(T_+(s,z))=s.
\]
The ball may be chosen uniformly small, and the seam compatibility of
Lemma~\ref{lem:frame} makes this embedding smooth on $S^1\times D^3$.
For the second framing use
$
 T_-(s,z)=T_+(s,A(s)z),
$
where $A$ is the loop in Lemma~\ref{lem:frame}. The two embeddings have
the same image $\nu\gamma$. Set
$
 \E=\M\setminus\Int\nu\gamma,\;\; p_{\E}=p|_{\E}.
$
Let $\phi_\varepsilon:\partial(S^2\times D^2)\to\partial\E$ be the
attaching map given by $T_\varepsilon$, with the boundary orientations
required for oriented circle surgery. In these coordinates,
\[
 \phi_\varepsilon(x,e^{2\pi is})=T_\varepsilon(s,x),\;\;\;
 p_{\E}\bigl(\phi_\varepsilon(x,e^{2\pi is})\bigr)=e^{2\pi is}.
\]
Consequently $p_{\E}$ restricts to the product meridian projection for
either attachment. Its fibre is $\N\setminus\Int D^3$, so the belt
spheres are fibred in the sense of Definition~\ref{def:fibred}.
Define
\begin{equation}\label{eq:candidates}
 \Sigma_\varepsilon(7)=\E\cup_{\phi_\varepsilon}(S^2\times D^2),
 \;\;\; K_\varepsilon=S^2\times\{0\},
 \;\;\; \varepsilon\in\{+,-\}.
\end{equation}
The signs label normal framings of the \emph{same} section, not two
positive loop classes. Figure~\ref{fig:surgery} illustrates the construction and the fibred
structure of the result.

\begin{figure}[tbp]
\centering
\begin{tikzpicture}[
  fib/.style={draw=diagramblue!75,line width=0.5pt},
  ball/.style={draw=orbitred!80,fill=white,densely dashed,line width=0.45pt}]
\begin{scope}
\foreach \x in {0,1.4,2.8,4.2,5.6} {
  \shade[left color=diagramlight,right color=diagramblue!22]
    (\x,0) ellipse[x radius=0.32,y radius=1.18];
  \draw[fib] (\x,0) ellipse[x radius=0.32,y radius=1.18];}
\draw[line width=9pt,liftblue!18,line cap=round] plot[smooth,tension=0.7]
  coordinates {(0,0.55) (1.4,0.17) (2.8,-0.23) (4.2,-0.47) (5.6,-0.38)};
\foreach \p in {(0,0.55),(1.4,0.17),(2.8,-0.23),(4.2,-0.47),(5.6,-0.38)}
  \draw[ball] \p ellipse[x radius=0.08,y radius=0.17];
\draw[orbitred,line width=0.9pt,midarrow=0.42] plot[smooth,tension=0.7]
  coordinates {(0,0.55) (1.4,0.17) (2.8,-0.23) (4.2,-0.47) (5.6,-0.38)};
\foreach \p in {(0,0.55),(1.4,0.17),(2.8,-0.23),(4.2,-0.47),(5.6,-0.38)}
  \fill[orbitred] \p circle[radius=0.045];
\draw[-{Stealth[length=2mm]},gray!65!black,line width=0.5pt]
  (5.6,1.27) .. controls (4.9,2.05) and (0.7,2.05) .. (0,1.27);
\node[slab,fill=white] at (2.8,1.86) {glued by $(x,1)\sim(\bar f(x),0)$};
\node[slab,text=orbitred] at (2.1,0.28) {$\gamma$};
\node[slab,anchor=east] at (-0.38,0.55) {$c_0(0)$};
\node[slab,anchor=west] at (5.98,-0.38) {$c_0(1)$};
\node[slab,text=liftblue] at (4.2,-0.98) {$\nu\gamma$};
\node[slab] at (0,-1.42) {$\N\times\{0\}$};
\node[slab] at (5.6,-1.42) {$\N\times\{1\}$};
\draw[-{Stealth[length=1.6mm]}] (2.8,-1.32) -- node[right=1.5pt,slab]{$p$} (2.8,-1.8);
\draw[line width=0.8pt,diagramblue] (0,-2.02) -- (5.6,-2.02);
\fill[diagramblue] (0,-2.02) circle[radius=0.045] (5.6,-2.02) circle[radius=0.045];
\node[slab,below] at (0,-2.08) {$0$};
\node[slab,below] at (5.6,-2.08) {$1\sim0$};
\node[slab,below] at (2.8,-2.08) {$S^1$};
\node[panelcap] at (2.8,-2.85) {(a) the mapping torus $\M$ and the section $\gamma$};
\end{scope}
\draw[darrow] (6.65,0) -- node[above=2pt,slab,align=center]{surgery\\ on $\gamma$} node[below=2pt,slab]{$\phi_\varepsilon$} (7.8,0);
\begin{scope}[shift={(10.3,-0.1)}]
\fill[diagramlight!60] (0,0) circle[radius=2.12];
\draw[diagramblue!25,line width=0.4pt] (0,0) circle[radius=2.12];
\foreach \a [count=\n] in {22.5,67.5,...,337.5} {
  \shade[inner color=white,outer color=diagramblue!26]
    (0,0) .. controls (\a-9:0.9) and (\a-19:1.7) .. (\a:1.95) .. controls (\a+19:1.7) and (\a+9:0.9) .. (0,0);
  \draw[diagramblue!70,line width=0.45pt]
    (0,0) .. controls (\a-9:0.9) and (\a-19:1.7) .. (\a:1.95) .. controls (\a+19:1.7) and (\a+9:0.9) .. (0,0);}
\fill[white,opacity=0.8] (0,0) circle[radius=0.55];
\draw[liftblue,line width=0.9pt,midarrow=0.62] (0,0) circle[radius=0.55];
\fill (0,0) circle[radius=0.07];
\node[slab,fill=white,fill opacity=0.8,text opacity=1,inner sep=1pt] at (0,-0.25) {$K_\varepsilon$};
\node[slab,text=liftblue,fill=white,inner sep=0.8pt] at (135:0.8) {$\gamma'$};
\node[slab] at (67.5:2.36) {$F_s$};
\node[slab] at (22.5:2.36) {$F_{s'}$};
\draw[-{Stealth[length=1.8mm]},gray!65!black] (-62:2.36) arc[start angle=-62,end angle=-22,radius=2.36];
\node[slab,anchor=west] at (-42:2.46) {$s$};
\node[panelcap] at (0,-2.75) {(b) $\Sigma_\varepsilon(7)$ near the belt sphere $K_\varepsilon$};
\end{scope}
\end{tikzpicture}
\caption{The framed surgery of Section~\ref{sec:construction}, drawn schematically
with dimensions suppressed. (a) The mapping torus $\M\to S^1$, cut open along its
seam, with five fibres $\N\times\{s\}$. The section $\gamma$ meets each fibre once
and closes because $c_0(0)=\bar f(c_0(1))$. Removing $\nu\gamma\cong S^1\times D^3$
removes one ball (dashed) from each fibre; the base interval also has its
endpoints identified. (b) A normal disc to the belt sphere $K_\varepsilon=S^2\times\{0\}$
in $\Sigma_\varepsilon(7)=\E\cup_{\phi_\varepsilon}(S^2\times D^2)$. A fibre $F_s$
of $\E$ extends through the radial collar $S^2\times[0,1]$ of the cap to an
open-book page diffeomorphic to $\N\setminus\Int D^3$, with boundary
$K_\varepsilon$. The circle $\gamma'$ is a positive meridian, identified on
$\partial\E$ with a parallel of $\gamma$ representing $t$. The two framings change
the attachment by the Gluck twist~\eqref{eq:gluck}.}
\label{fig:surgery}
\end{figure}

Orient $\M$ so that the positive base direction followed by the orientation of
$\N$ is positive, as in Section~\ref{nf:sec:spin}, and give $\E$ and
$\Sigma_\varepsilon(7)$ the induced orientations. Orient $K_\varepsilon$ so
that its positive meridian $\{x\}\times\partial D^2$ is a parallel copy
$\gamma'$ of $\gamma$ traversed in the direction of increasing $s$.
Choose the basepoint on $\gamma'$. Under the isomorphism
$\pi_1(\E)\cong\pi_1(\M)$ below, the class of $\gamma'$ is the class $t$ of
$\gamma$. For an oriented two-knot $K\subset\Sigma$
we write $Q(\Sigma,K)$ for its fundamental quandle \cite{Joyce,WinterHigher}.
Reversing the orientation of $K_\varepsilon$ replaces the meridian $t$ by
$t^{-1}$, and therefore the operation in \eqref{eq:coset} by
$x*y=\theta^{-1}(xy^{-1})y$. This leaves the quandle unchanged up to
isomorphism: the
automorphism $\beta$ of $Q_8$ with $\beta(i)=-i$, $\beta(j)=-k$ and
$\beta(k)=-j$ satisfies $\beta\alpha\beta^{-1}=\alpha^{-1}$, so
$\psi=\beta\times\id_{C_7}$ satisfies $\psi\theta\psi^{-1}=\theta^{-1}$ and is
an isomorphism $\GAlex(\Gamma,\theta)\to\GAlex(\Gamma,\theta^{-1})$. Hence the
quandle assertions below hold for either orientation of $K_\varepsilon$.

The orientations just fixed determine the orientation of the surgery
handle and of $K_\varepsilon$ explicitly.
With the convention of Section~\ref{nf:sec:spin}, in which a vector of
positive base component followed by $(V_i,V_j,V_k)$ is a positive frame,
the embeddings $T_\pm$ preserve orientation for the product orientation
of $S^1\times D^3$; for $T_-$ this uses $\det A(s)=1$. Take boundary
orientations with the outward normal first. Then
$\partial(S^1\times D^3)=-(S^1\times S^2)$,
$\partial(S^2\times D^2)=S^2\times S^1$ and
$\partial\E=-\partial\nu\gamma$, while the interchange of factors
$S^2\times S^1\to S^1\times S^2$ preserves orientation. Hence the
attaching map $\phi_\varepsilon$ displayed above preserves boundary
orientations, and the orientation of $\Sigma_\varepsilon(7)$ induced from
$\E$ restricts to the reverse of the product orientation on the handle
$S^2\times D^2$; this is the orientation of the handle in the oriented
circle surgery. We use the convention that the orientation of a knot
followed by that of the normal disc bounded by its positive meridian is
the ambient orientation. The positive meridian $\{x\}\times\partial D^2$
bounds the normal disc $\{x\}\times D^2$ with its standard orientation,
so $K_\varepsilon$ is $S^2\times\{0\}$ with the reverse of the boundary
orientation of $S^2=\partial D^3$. Nothing below depends on this sign: the ambient
orientation is always the one induced from $\E$, and the knot
orientation is the one fixed by the meridian.

This is the classical weight-circle surgery construction
\cite[Sections~14.1 and~14.7]{Hillman},
\cite[Theorem~12.3 and Theorem~14.24]{QextC}. Surgery on a loop representing a
weight element in a closed orientable four-manifold with $\chi=0$ yields a
homotopy four-sphere containing the belt sphere as a two-knot
\cite[Section~14.1]{Hillman}; for the prism-family groups
$(Q_8\times C_n)\rtimes\ZZ$ with $n$ odd, of which $G_7$ is the case $n=7$, this
construction is due to Yoshikawa \cite{Yoshikawa1980}. Its fibred form, whose
fibre is the punctured space form and whose closed monodromy realizes the
order-six meridianal automorphism, is
given in \cite[Section~5, Proposition~5.1 and its proof, Type~II, pp.~522--523]{PlotnickSuciu}.

That $\Sigma_\varepsilon(7)$ is a homotopy four-sphere follows from
$\chi(\M)=0$ and $\ncl{t}=G_7$
(Lemma~\ref{lem:group-data}); this is the standard weight-circle
construction of \cite[Section~14.1, p.~270]{Hillman}.
The inclusion $\E\hookrightarrow\M$ induces an isomorphism on
fundamental groups: relative general position makes both loops and
null-homotopies disjoint from the codimension-three section
\cite[Chapters~3 and~8]{Hirsch}.
The circle factor of $\partial\E\cong S^2\times S^1$ represents $t$,
so van Kampen's theorem gives
\[
 \pi_1(\Sigma_\varepsilon(7))
 \cong G_7/\ncl{t}=1.
\]
Euler-characteristic additivity gives $\chi(\E)=\chi(\M)=0$ and
$\chi(\Sigma_\varepsilon(7))=2$.
For this closed oriented simply connected four-manifold, Poincar\'e
duality and the universal coefficient theorem imply that $H_3=0$
and that $H_2$ is torsion-free. Its Euler characteristic therefore
forces $H_2=0$. Thus it is an integral homology four-sphere.
Successive applications of the Hurewicz theorem give $\pi_2=\pi_3=0$
and a map $S^4\to\Sigma_\varepsilon(7)$ inducing an isomorphism on
$H_4$. Since a compact manifold has the homotopy type of a CW complex
\cite[Corollary~A.12]{Hatcher}, this map is a homotopy equivalence by the
homological form of Whitehead's theorem
\cite[Theorem~4.32 and Corollary~4.33]{Hatcher}.
Finally, Freedman's topological Poincar\'e theorem
\cite[Theorem~1.6, p.~371]{Freedman} gives
$\Sigma_\varepsilon(7)\cong_{\mathrm{TOP}}S^4$.

The existence statement below is the parameter-seven specialization of
\cite[Proposition~18.4(i), pp.~358--359]{QextC}, using the section of
\cite[Definition~18.3]{QextC}. The present treatment fixes the explicit
labelling of the two normal-framing classes and records the construction
in the conventions used below.

\begin{proposition}[Existence of the candidates]\label{prop:exists}
Both $\Sigma_+(7)$ and $\Sigma_-(7)$ are smooth homotopy four-spheres. Their belt
spheres have the common fibred exterior $\E$, with fibre
$\N\setminus\Int D^3$, and
\[
                  Q(\Sigma_\varepsilon(7),K_\varepsilon)\cong P_7.
\]
Moreover $\As(P_7)\cong G_7$. In particular both ambient manifolds are
homeomorphic to $S^4$.
\end{proposition}
\begin{proof}
The preceding surgery calculation proves that both ambient manifolds
are smooth homotopy four-spheres. The fibrewise neighbourhood and the
attaching maps give the common fibred exterior $\E$, with fibre
$\N\setminus\Int D^3$. Its peripheral subgroup is $\langle t\rangle$,
since the boundary circle is the positive meridian and
$\pi_1(S^2\times S^1)\cong\ZZ$. The peripheral coset description
\eqref{eq:coset} therefore identifies its fundamental quandle with $P_7$.
These are also the parameter-seven conclusions of
\cite[Definition~18.3 and Proposition~18.4(i), pp.~358--359]{QextC}.
Under the knot-quandle dictionary
\cite[Section~3, (3.3), pp.~45--47]{QextC}, the associated-group
isomorphism is the marked isomorphism
\[
 \As(P_7)\to G_7,\;\;\; e_x\longmapsto x^{-1}tx
 \;\;(x\in\Gamma).
\]
In particular $e_1$ maps to the positive meridian $t$, and the canonical
augmentation sending each $e_x$ to $1$ agrees with the degree
homomorphism of $G_7$.
Freedman's theorem identifies both ambient manifolds topologically with
$S^4$ \cite{Freedman}.
Lemma~\ref{lem:As-twoknot} and Remark~\ref{rem:As-algebraic} at the end of
this section give a proof of $\As(P_7)\cong G_7$ that uses neither the
knot--quandle dictionary nor the coset enumeration of
Remark~\ref{rem:As-direct}.
\end{proof}

\begin{remark}[A direct check of $\As(P_7)\cong G_7$]\label{rem:As-direct}
The isomorphism $\As(P_7)\cong G_7$ can also be verified without the
knot-quandle dictionary. By \eqref{eq:group}, $xy^{-1}ty=t\,\theta(xy^{-1})y$
for $x,y\in\Gamma$, so $x*y=t^{-1}w$ with $w=xy^{-1}ty$, and hence
$(x*y)^{-1}t\,(x*y)=w^{-1}tw=(y^{-1}ty)^{-1}(x^{-1}tx)(y^{-1}ty)$. Thus
$e_x\mapsto x^{-1}tx$ defines a homomorphism $\As(P_7)\to G_7$; it is
surjective because $\ncl{t}=G_7$ (Lemma~\ref{lem:group-data}). Its kernel
$\mathcal K$ meets $\langle e_1\rangle$ trivially, since $e_1\mapsto t$ has
infinite order; hence $[\As(P_7):\langle e_1\rangle]=56\,|\mathcal K|$. If
$S\subset P_7$ generates $P_7$ as a quandle, the relations
$e_{x*s}=e_s^{-1}e_xe_s$ with $x\in P_7$ and $s\in S$ imply all defining
relations of $\As(P_7)$: the identities $x*(a*b)=((x*^{-1}b)*a)*b$ and
$x*(a*^{-1}b)=((x*b)*a)*^{-1}b$ show that the set of $y$ for which
$e_{x*y}=e_y^{-1}e_xe_y$ holds for all $x$ is closed under $*$ and $*^{-1}$.
For $S=\{1,(i,c)\}$, which generates $P_7$, a Todd--Coxeter enumeration of
the cosets of $\langle e_1\rangle$ in the resulting presentation, with $56$
generators and $112$ relators, closes with $56$ cosets. Hence
$\mathcal K=1$.
\end{remark}

\begin{remark}[Closed representative versus exterior monodromy]
\label{rem:closed-vs-punctured-monodromy}
The order-six map $\bar f$ is a representative of the unpointed closed
monodromy class \cite[Proposition~15.24, pp.~275--276]{QextC}.
Indeed, a ball-moving isotopy along the section identifies the moving
deleted ball with a fixed ball. In that trivialization the return map
extends over the ball to a diffeomorphism isotopic, as an unpointed map,
to $\bar f$; see \cite[Remark~18.15, p.~367]{QextC}. The ball-moving
isotopy may change the pointed representative, while preserving this
unpointed isotopy class. It should not
be confused with a finite-order monodromy of the punctured fibre. For the
prism family with parameter greater than one, such periodic exterior
monodromy is excluded by \cite[Proposition~15.21, pp.~273--274]{QextC}; the
relation between the closed representative and a boundary-fixing return map
is recorded in \cite[Remark~18.15, p.~367]{QextC}. Accordingly, every
occurrence of ``period six'' below refers to the closed monodromy class or to
its displayed representative $\bar f$.
\end{remark}

\begin{remark}[Direct verification of the monodromy statements]
\label{rem:monodromy-direct}
Both statements of Remark~\ref{rem:closed-vs-punctured-monodromy} can be
checked directly in the present coordinates.

\emph{The unpointed closed monodromy class of $p_{\E}$ contains $\bar f$.}
Fix $\varepsilon$. Since the radius of $\nu\gamma$ may be chosen small,
we may assume that $T_\varepsilon$ extends, by the same formula, to an
embedding of $S^1\times 2D^3$ with $p(T_\varepsilon(s,z))=s$, where $rD^3$
denotes the closed ball of radius $r$. Let $Y_0$ be the vector field on $\M$ induced by
$\partial_s$ on $S^3\times\RR$. Since $(v,1)=D(f(v),0)$, the time-one flow
of $Y_0$ sends $[v,0]$ to $[f(v),0]$, so its first-return map on
$p^{-1}(0)=\N$ is $\bar f$. The condition $dp(Y)=\partial_s$ is convex, so a
partition of unity gives a vector field $Y$ on $\M$ with $dp(Y)=\partial_s$
that equals $dT_\varepsilon(\partial_s)$ on
$T_\varepsilon\bigl(S^1\times\tfrac32D^3\bigr)$. Its restriction to $\E$ is
an allowed vector field in Definition~\ref{def:fibred}, for the fibration
$p_{\E}$ and the product collar
$T_\varepsilon\bigl(S^1\times(\tfrac32D^3\setminus\Int D^3)\bigr)$
determined by $\phi_\varepsilon$; let $h$ be its return map. In the ball
coordinates $z\mapsto T_\varepsilon(0,z)$, the first-return map of $Y$ on
$p^{-1}(0)$ is $h\cup\id_{D^3}$, which is a closed monodromy. The fields
$(1-\tau)Y_0+\tau Y$, $0\le\tau\le1$, all project to $\partial_s$, and their
first-return maps form an isotopy, not required to fix any point, from
$\bar f$ to $h\cup\id_{D^3}$.

\emph{There is no geometric monodromy of finite order.} Let
$\sigma\colon F\to F$ be a geometric monodromy of finite order for
$p_{\E}$, so that $\widehat F=\N$. The mapping torus of $\sigma$ is
diffeomorphic to the orientable manifold $\E$, so $\sigma$ preserves
orientation. Averaging a metric over $\langle\sigma\rangle$ gives a
$\sigma$-invariant product collar of $\partial F\cong S^2$, and, by
uniformization, the restriction of $\sigma$ to $\partial F$ is smoothly
conjugate to a conformal map of finite order of the round sphere, hence to
an element of $SO(3)$. Coning over the collar therefore extends $\sigma$ to
a smooth periodic diffeomorphism $\hat\sigma$ of $\widehat F$ with a fixed
point, the cone point.

More explicitly, let $a=\sigma|_{\partial F}$. Choose a diffeomorphism
$\Xi\colon S^2\to S^2$ and a rotation $R\in SO(3)$ such that $\Xi a\Xi^{-1}=R$.
Using the $\sigma$-invariant collar and the extension consequence of
\cite{Smale1959}, extend $\Xi$ over the capping ball by a diffeomorphism
$H:D^3\to D^3$ that is product-like near its boundary. On the
cap use $H^{-1}RH$. Its boundary restriction is $a$, so it glues smoothly to
$\sigma$ across the product collar; moreover it has finite order, and it fixes
the point corresponding under $H$ to the centre of the linear rotation.
Thus the resulting extension $\hat\sigma$ is indeed a smooth periodic
diffeomorphism of $\widehat F$ with a fixed point. By \cite[Theorem~E and Section~2.1]{DinkelbachLeeb},
$\hat\sigma$ is smoothly conjugate to an orientation-preserving isometry
$\bar g$ of $\N$, which then also has a fixed point.

Since $\E$ is diffeomorphic over $S^1$ to the mapping torus of $\sigma$, the
outer automorphism of $\pi_1(F)=\pi_1(\widehat F)\cong\Gamma$ induced by
$\sigma$ is, up to an automorphism of $\Gamma$, the class of conjugation by
an element of degree $\pm1$ of $G_7$, that is, $[\theta]^{\pm1}$. Hence
$\bar g$ induces an $\Aut(\Gamma)$-conjugate of $[\theta]^{\pm1}$. As
$\Out(\Gamma)=\Out(Q_8)\times\Aut(C_7)$ and $[\theta]=([\alpha],-1)$, the
outer automorphism induced by $\bar g$ has the form $([\alpha'],-1)$, where
$[\alpha']\in\Out(Q_8)$ has order three.

Lift $\bar g$ to $g(v)=AvB$ with $A,B\in S^3$. The preimage of $\Gamma$ in
$S^3\times S^3$ is $Q_8\times\langle-\zeta\rangle$, which $g$ normalizes
factorwise. Moreover $g\gamma_{q,0}g^{-1}=\gamma_{AqA^{-1},0}$, and
$g\gamma_{1,1}g^{-1}$ is right multiplication by $B^{-1}\zeta^{-1}B$. The
normalizer of $Q_8$ in $S^3$ is the binary octahedral group; it maps onto
$\Out(Q_8)\cong S_3$ with kernel $Q_8$, and the elements mapping to elements
of order three are those of $\langle i,j,\omega\rangle\setminus Q_8$, the
sixteen quaternions $\tfrac12(\pm1\pm i\pm j\pm k)$. Since conjugation by
$A$ induces $[\alpha']$ on $Q_8$, we get $\Ree(A)=\pm\tfrac12$. Since the
$C_7$-component of the induced automorphism is inversion,
$g\gamma_{1,1}g^{-1}=\gamma_{1,-1}$, that is, $B^{-1}\zeta B=\zeta^{-1}$;
thus conjugation by $B$ is a rotation sending $i$ to $-i$, so
$B\in\RR j+\RR k$ and $\Ree(B)=0$. The other lifts $\gamma_{q,m}\circ g$ are
the maps $v\mapsto(qA)\,v\,(B\zeta^{-m})$, and $qA$ and $B\zeta^{-m}$ are
again of these two kinds. By Lemma~\ref{per:lem:rigidity}, no lift has a
fixed point. Hence $\bar g$ is fixed-point free, a contradiction.

The same argument applies to any fibration of $\E$ as in
Definition~\ref{def:fibred}, since its closed fibre is again diffeomorphic
to $\N$ by Remark~\ref{rem:rigidity-outline}(b).
\end{remark}

\begin{lemma}[The classical associated-group theorem in a homotopy four-sphere]\label{lem:As-twoknot}
Let $K$ be a smooth oriented two-knot in a smooth homotopy four-sphere
$\Sigma$, with exterior $X$, knot group $\pi=\pi_1(X)$, positive meridian
$m\in\pi$ and peripheral subgroup $P\le\pi$. Identify $Q(\Sigma,K)$ with the
peripheral coset quandle $P\backslash\pi$, $Pg*Ph=P\,gh^{-1}mh$
\cite{Joyce,WinterHigher}, as in \eqref{eq:coset}. Then $P=\langle m\rangle$, and
\[
 \kappa\colon\As(Q(\Sigma,K))\to\pi,\;\;\; e_{Pg}\longmapsto g^{-1}mg,
\]
is a well-defined isomorphism carrying each generator $e_x$ to a meridian.
\end{lemma}
\begin{proof}
For oriented codimension-two knots in the standard sphere, the
isomorphism between the associated group of the fundamental quandle
and the knot group is classical; see
\cite[Theorem~4.1]{Kamada2014} and the earlier work
\cite{Joyce,FennRourke1992}; see also \cite{WinterHigher}.
We include the argument to make explicit the formulation needed here
for a two-knot in an arbitrary smooth homotopy four-sphere.

Write $Q=Q(\Sigma,K)$. The subgroup $P$ is the image of
$\pi_1(\partial X)\cong\pi_1(S^2\times S^1)\cong\ZZ$, which is generated by
the meridian, so $P=\langle m\rangle$ and $g^{-1}mg$ depends only on $Pg$.
Since
\[
 \kappa(e_{Pg*Ph})=(gh^{-1}mh)^{-1}\,m\,(gh^{-1}mh)
 =\kappa(e_{Ph})^{-1}\kappa(e_{Pg})\kappa(e_{Ph}),
\]
$\kappa$ is a homomorphism. By van Kampen's theorem
$\pi/\ncl{m}=\pi_1(\Sigma)=1$; as the image of $\kappa$ contains every
conjugate of $m$, $\kappa$ is surjective.

Let $\As(Q)$ act on $Q$ on the right by $x\cdot e_y=x*y$. The defining
relations and induction on word length give $w^{-1}e_xw=e_{x\cdot w}$ for all
$w\in\As(Q)$ and $x\in Q$. Because $Pg*Ph=Pg\cdot\kappa(e_{Ph})$, this action
is the composite of $\kappa$ with the action of $\pi$ on $P\backslash\pi$ by
right multiplication. Hence every $w\in\mathcal K:=\ker\kappa$ fixes $Q$
pointwise and therefore commutes with every $e_x$; thus $\mathcal K$ is
central.

The right multiplications by the conjugates of $m$ generate the action of
$\ncl{m}=\pi$, which is transitive on $P\backslash\pi$, so $Q$ is connected.
Hence all $e_x$ are conjugate in $\As(Q)$, and $\As(Q)^{\mathrm{ab}}$ is
infinite cyclic, generated by the class of any $e_x$; the assignment
$e_x\mapsto1$ defines an epimorphism onto $\ZZ$. Alexander duality in the
homology four-sphere $\Sigma$ (the proof of \cite[Theorem~3.44]{Hatcher}
uses only the homology of the ambient sphere) gives $H_1(X;\ZZ)\cong\ZZ$,
generated by the class of $m=\kappa(e_P)$, and
$H_2(X;\ZZ)\cong\widetilde H^1(S^2;\ZZ)=0$. Hence $\kappa$ induces an
isomorphism on abelianizations, and $\mathcal K\subseteq[\As(Q),\As(Q)]$.

The five-term exact sequence in group homology for the central extension
\cite[p.~47]{Brown1982}
$1\to\mathcal K\to\As(Q)\to\pi\to1$ reads
\[
 H_2(\As(Q);\ZZ)\to H_2(\pi;\ZZ)\to\mathcal K\to
 H_1(\As(Q);\ZZ)\to H_1(\pi;\ZZ)\to0,
\]
where $\mathcal K/[\mathcal K,\As(Q)]=\mathcal K$ because $\mathcal K$ is
central. The fourth map is injective, so $\mathcal K$ is a quotient of
$H_2(\pi;\ZZ)$. By Hopf's theorem \cite[Chapter~II]{Brown1982},
$H_2(\pi;\ZZ)$ is a quotient of
$H_2(X;\ZZ)=0$. Thus $\mathcal K=1$.
\end{proof}

\begin{remark}[An algebraic proof of $\As(P_7)\cong G_7$]\label{rem:As-algebraic}
The proof of Lemma~\ref{lem:As-twoknot} uses the knot only through three
properties of the marked group $(\pi,m)$: $m$ normally generates $\pi$,
$\pi^{\mathrm{ab}}$ is infinite cyclic generated by the class of $m$, and
$H_2(\pi;\ZZ)=0$. For $(G_7,t)$ the first two are
Lemma~\ref{lem:group-data}. For the third, apply the Wang sequence
\cite[Example~2.48]{Hatcher} to the mapping torus of a self-homotopy
equivalence of $B\Gamma$ inducing $\theta$. This mapping torus is a
classifying space for $G_7=\Gamma\rtimes_\theta\langle t\rangle$, and its
homology sequence contains the exact segment
\[
 H_2(\Gamma;\ZZ)\to H_2(G_7;\ZZ)\to H_1(\Gamma;\ZZ)
 \xrightarrow{\;1-\theta_*\;}H_1(\Gamma;\ZZ).
\]
Here $H_2(\Gamma;\ZZ)=0$, since the Schur multipliers of $Q_8$ and $C_7$
vanish and $Q_8^{\mathrm{ab}}\otimes C_7=0$; here we use the standard
direct-product formula for Schur multipliers \cite{Karpilovsky1987}. On
$H_1(\Gamma;\ZZ)\cong(\ZZ/2)^2\oplus\ZZ/7$, with basis $\bar i,\bar j$ of the
first summand, $1-\theta_*$ is given by
$\left(\begin{smallmatrix}1&-1\\-1&0\end{smallmatrix}\right)$, of
determinant $-1$, on $(\ZZ/2)^2$ and by multiplication by $2$ on $\ZZ/7$.
Hence it is injective, and $H_2(G_7;\ZZ)=0$. Applied to the coset quandle
\eqref{eq:coset}, the argument shows algebraically that $e_x\mapsto x^{-1}tx$
is an isomorphism $\As(P_7)\to G_7$; the coset enumeration of
Remark~\ref{rem:As-direct} is an independent check. Consequently, for every
smooth oriented two-knot $K$ in a smooth homotopy four-sphere $\Sigma$ and
every quandle isomorphism $\psi\colon Q(\Sigma,K)\to P_7$, the composite
$\pi_1(X)\xrightarrow{\kappa^{-1}}\As(Q(\Sigma,K))\xrightarrow{\psi_*}
\As(P_7)\to G_7$ is an isomorphism carrying the meridian $g^{-1}mg$ to
$x^{-1}tx$ with $x=\psi(Pg)\in\Gamma$.
\end{remark}

\section{Weight loops and the finite fibred reduction}\label{sec:rigidity}
Choose representatives of the two attaching maps so that
\begin{equation}\label{eq:gluck}
 \phi_- =\phi_+\circ\tau,\;\;\;
 \tau(x,e^{2\pi is})=(A(s)x,e^{2\pi is}),
\end{equation}
with $A$ from Lemma~\ref{lem:frame}.
For the explicit frames of Lemma~\ref{lem:frame}, this identity holds exactly.
Since $u_-(s)=\exp(\pi j\,r(s))\,u(s)$ and the two exponentials commute,
\[
 \widetilde c(s)\,u_-(s)\,a\,u_-(s)^{-1}
   =\widetilde c(s)\,u(s)\bigl(A(s)(a)\bigr)u(s)^{-1}
   \;\;\;(a=i,j,k).
\]
Thus, if a frame is regarded as the linear map sending $a\in\{i,j,k\}$ to
the corresponding frame vector, the frame defining $\lambda_-$ is the frame
\eqref{eq:frame} precomposed with $A(s)$; this is \eqref{eq:gluck}.
The surgery pairs depend only on the isotopy classes of the attaching
maps. Since $A$ represents the nonzero element of
$\pi_1(SO(3))\cong\ZZ/2$, the map $\tau$ is the Gluck boundary twist
relative to the reference attachment $\phi_+$
\cite[Section~6, p.~312]{Gluck1962}.
Consequently the pairs in \eqref{eq:candidates} are Gluck
reconstructions of one another; see also
\cite[Section~14.1]{Hillman} and \cite[Section~5.2]{GS}.
Figure~\ref{fig:gluck} shows the twist $\tau$.
The two pairs have the same exterior and quandle. Neither fact, however,
determines the ambient smooth types or whether the pairs are equivalent:
Theorem~\ref{std:thm:main} shows that both ambient manifolds are standard, and
Corollary~\ref{eval:cor:pairs} shows that the two pairs are not equivalent.

\begin{figure}[tbp]
\centering
\begin{tikzpicture}[
  fr/.style={-{Stealth[length=1.35mm,width=1.1mm]},line width=0.65pt}]
\foreach \k/\phi/\lb in {0/30/{s=0},1/120/{s=\frac14},2/210/{s=\frac12},3/300/{s=\frac34},4/390/{s=1}} {
  \begin{scope}[shift={(1.42*\k,0)}]
  \shade[ball color=diagramlight!70!white] (0,0) circle[radius=0.56];
  \draw[diagramblue!45,line width=0.4pt] (0,0) circle[radius=0.56];
  \draw[diagramblue!40,line width=0.4pt] (-0.56,0) arc[start angle=180,end angle=360,x radius=0.56,y radius=0.15];
  \draw[diagramblue!30,line width=0.4pt,densely dashed] (0.56,0) arc[start angle=0,end angle=180,x radius=0.56,y radius=0.15];
  \draw[gray!60,densely dotted] (0,-0.76) -- (0,0.76);
  \fill (0,0.56) circle[radius=0.035] (0,-0.56) circle[radius=0.035];
  \pgfmathtruncatemacro{\back}{cos(\phi)<0 ? 1 : 0}
  \ifnum\back=1
    \draw[liftblue,line width=0.9pt,densely dashed] plot[domain=-90:90,samples=31]
      ({0.56*cos(\x)*sin(\phi)},{0.56*sin(\x)-0.15*cos(\x)*cos(\phi)});
    \draw[liftblue,fill=white,line width=0.8pt] ({0.56*sin(\phi)},{-0.15*cos(\phi)}) circle[radius=0.05];
  \else
    \draw[liftblue,line width=0.9pt] plot[domain=-90:90,samples=31]
      ({0.56*cos(\x)*sin(\phi)},{0.56*sin(\x)-0.15*cos(\x)*cos(\phi)});
    \fill[liftblue] ({0.56*sin(\phi)},{-0.15*cos(\phi)}) circle[radius=0.055];
  \fi
  \node[slab] at (0,-1.02) {$\lb$};
  \end{scope}}
\node[slab,anchor=east] at (-0.06,0.7) {$j$};
\node[slab,anchor=east] at (-0.07,-0.7) {$-j$};
\draw[-{Stealth[length=1.6mm]},gray!65!black] (0.35,1.02) -- (5.45,1.02);
\node[slab,anchor=south] at (2.9,1.06) {rotation through $2\pi r(s)$ about the $j$-axis};
\node[slab,anchor=west] at (6.33,0) {$\equiv s=0$};
\node[panelcap] at (2.84,-1.62) {(a) $\tau$ on the slices $S^2\times\{e^{2\pi is}\}$};
\begin{scope}[shift={(8.35,0)}]
\fill[diagramlight] (-0.3,0.15) rectangle (3.9,0.78);
\fill[liftblue!8] (-0.3,-0.85) rectangle (3.9,-0.22);
\draw[gray!55] (-0.3,0.45) -- (3.9,0.45);
\draw[gray!55] (-0.3,-0.55) -- (3.9,-0.55);
\foreach \k in {0,...,4} {
  \pgfmathsetmacro{\ang}{90*\k}
  \draw[fr,liftblue] (0.9*\k,0.45) -- ++(0:0.3);
  \draw[fr,orbitred] (0.9*\k,0.45) -- ++(90:0.3);
  \draw[fr,liftblue] (0.9*\k,-0.55) -- ++(\ang:0.3);
  \draw[fr,orbitred] (0.9*\k,-0.55) -- ++(\ang+90:0.3);
  \fill (0.9*\k,0.45) circle[radius=0.03] (0.9*\k,-0.55) circle[radius=0.03];}
\node[slab,anchor=east] at (-0.36,0.45) {$\id$};
\node[slab,anchor=east] at (-0.36,-0.55) {$d\tau$};
\node[slab,text=inkgray,anchor=west] at (3.95,0.45) {constant};
\node[slab,text=inkgray,anchor=west,align=left] at (3.95,-0.55) {one full\\ turn};
\node[slab] at (0,-1.02) {$s=0$};
\node[slab] at (3.6,-1.02) {$s=1$};
\node[panelcap] at (1.8,-1.62) {(b) a frame of $T_jS^2$ along $\{j\}\times S^1$};
\end{scope}
\end{tikzpicture}
\caption{The Gluck twist $\tau$ of \eqref{eq:gluck}, drawn as if $r(s)=s$.
(a) Five slices of $S^2\times S^1$: on the slice at time $s$, $\tau$ is the
rotation $A(s)$ about the $j$-axis through the angle $2\pi r(s)$. The marked
semicircle on $S^2$ therefore makes one full turn as $s$ runs once around the
circle, while the poles $\pm j$ stay fixed. Since the loop $A$ is essential in
$SO(3)$, $\tau$ is the Gluck boundary twist and
$\phi_-=\phi_+\circ\tau$ is the Gluck reconstruction. (b) Along the circle
$\{j\}\times S^1$ over a fixed pole, the derivative of $\tau$ turns the normal
plane $T_jS^2$ once. Hence $\tau$ changes the closing sign of this circle with
its product framing and interchanges the two spin structures of $S^2\times S^1$,
only one of which extends over $S^2\times D^2$; this is used in the proof of
Corollary~\ref{eval:cor:pairs}.}
\label{fig:gluck}
\end{figure}

The \emph{degree} of an element $g\in G_7$ is its image $\deg g\in\ZZ$ under
the epimorphism $G_7\to G_7^{\mathrm{ab}}\cong\ZZ$ sending $t$ to $1$
(Lemma~\ref{lem:group-data}); thus $\deg(xt^n)=n$ for $x\in\Gamma$, and the
elements of degree one form the coset $\Gamma t$. Under the identification $\pi_1(\M)=G_7$ in \eqref{per:eq:torus}, the deck transformations
$\gamma_{q,a}$ preserve the $\RR$-coordinate and $D$ raises it by one, so
$\deg$ coincides with $p_*\colon\pi_1(\M)\to\pi_1(S^1)=\ZZ$. Degree is
invariant under conjugation; hence an oriented unbased loop in $\M$ has a
well-defined degree, which changes sign when the orientation is reversed.
Such a loop \emph{represents} $g$ if it represents the conjugacy class of
$g$, and an element or oriented loop has \emph{positive degree one} if its
degree is $+1$.

The distinction between conjugacy classes and positive strict weight
orbits (Definition~\ref{def:weight}) is important. The following lemma
is the parameter-seven specialization of
\cite[Lemma~18.2 and Proposition~18.4(ii), pp.~356--359]{QextC},
in the conventions of \eqref{eq:group} and \eqref{per:eq:torus}.
The cited proof includes the twisted-conjugacy and involution calculations.

\begin{lemma}[Two loop classes, one pair of framing choices]\label{lem:weight-loops}
The positive degree-one elements of $G_7$ form exactly two conjugacy
classes, represented by $t$ and $(-1)t$. Every such element normally generates
$G_7$. Every embedded loop representing such an element is ambiently
isotopic to a section of $\M\to S^1$. The orientation-preserving
involution
$
                         \Phi[v,s]=[vi,s]
$
interchanges the two conjugacy classes. Consequently every framed
surgery on a loop of degree $\pm1$ in $\M$ is represented by one of the pairs
in \eqref{eq:candidates}, up to the orientation of the belt sphere.
\end{lemma}
\begin{proof}
This is \cite[Lemma~18.2 and Proposition~18.4(ii), pp.~356--359]{QextC}
specialized to $n=7$ and to the conventions of \eqref{eq:group} and
\eqref{per:eq:torus}.
\end{proof}

\begin{proof}[Self-contained verification of Lemma~\ref{lem:weight-loops}]
\emph{Conjugacy classes.} The elements of degree one are the products $xt$
with $x\in\Gamma$. For $y\in\Gamma$, the convention \eqref{eq:group} gives
\[
 y^{-1}(xt)\,y=y^{-1}x\,\theta^{-1}(y)\,t,
 \;\;\;
 t^{-1}(xt)\,t=\theta(x)\,t .
\]
Taking $y=\theta(x)^{-1}$ in the first formula shows that $\theta(x)t$ is
$\Gamma$-conjugate to $xt$ (indeed, for this $y$ one has
$y^{-1}x\,\theta^{-1}(y)=\theta(x)\,x\,x^{-1}=\theta(x)$). Thus conjugation by the stable letter does not
enlarge a twisted-conjugacy orbit. Hence the conjugacy classes in $\Gamma t$ correspond to the orbits of
the twisted conjugation
$x\mapsto y^{-1}x\,\theta^{-1}(y)$ of $\Gamma$ on itself. Since $\theta^{-1}=\alpha^{-1}\times(c\mapsto c^{-1})$,
these orbits are products of orbits in the two factors. In $C_7$ the orbit
of $c^b$ is $\{c^{b-2a}\}=C_7$. In $Q_8$, the element $y^{-1}\alpha^{-1}(y)$
equals $1$, $j$, $k$ and $i$ for $y=\pm1$, $\pm i$, $\pm j$ and $\pm k$,
respectively, so the orbit of $1$ is $\{1,i,j,k\}$; multiplying by the
central element $-1$ shows that $\{-1,-i,-j,-k\}$ is the orbit of $-1$. Therefore the elements of positive degree one form exactly two
conjugacy classes,
\[
 \{1,i,j,k\}\times C_7\cdot t
 \;\;\text{and}\;\;
 \{-1,-i,-j,-k\}\times C_7\cdot t ,
\]
with twenty-eight elements each, represented by $t$ and $(-1)t$.

\emph{Normal generation.} Let $w$ have degree one. In $G_7/\ncl{w}$ the
image of $t$ lies in the image of $\Gamma$, so this quotient is a quotient
of $\Gamma$. Its abelianization is a quotient of
$G_7^{\mathrm{ab}}/\langle[w]\rangle=0$. Thus $G_7/\ncl{w}$ is a perfect
quotient of the solvable group $\Gamma$, and hence it is trivial.

\emph{The involution.} The map $\widetilde\Phi(v,s)=(vi,s)$ of
$S^3\times\RR$ commutes with every $\gamma_{q,a}$, because $i$ commutes
with $\zeta$. Since $f^{-1}(w)=\omega wj^{-1}$ and $i^{-1}j^{-1}i=j$,
\[
 \widetilde\Phi D\widetilde\Phi^{-1}(v,s)
 =\bigl(\omega v\,i^{-1}j^{-1}i,\;s+1\bigr)
 =(\omega vj,\;s+1)
 =\gamma_{-1,0}D(v,s).
\]
Hence $\widetilde\Phi$ normalizes the deck group and descends to a
diffeomorphism $\Phi$ of $\M$, which preserves $p$ and the orientation.
Moreover $\Phi^2=\id$, because $\widetilde\Phi^{2}=\gamma_{-1,0}$.
On $\pi_1(\M)=G_7$, conjugation by $\widetilde\Phi$ induces the
automorphism that fixes $\Gamma$ pointwise and sends $t$ to $(-1)t$; it respects the
relation $t^{-1}gt=\theta(g)$ because $-1$ is central. Up to an inner
automorphism this is $\Phi_*$, so $\Phi$ interchanges the two conjugacy
classes.

\emph{Isotopy to a section.} Freely homotopic oriented embedded circles
in a smooth four-manifold are ambiently isotopic, by the unframed
circle-isotopy assertion of Lemma~\ref{lem:circle-spin}; its
transversality and isotopy-extension proof is given there.
This codimension-three observation was used directly for sections and
section surgeries of fibred two-knots by Teragaito
\cite{Teragaito1989Addendum}.
The section $\gamma$ represents $t$, and the section $\Phi\circ\gamma$
represents $(-1)t$.

\emph{Framed surgeries.} Reversing the orientation of a loop of degree
$-1$ changes only the orientation of the belt sphere. A loop of degree one
is ambiently isotopic to $\gamma$ or to $\Phi\circ\gamma$; in the second
case apply $\Phi$. The isotopy and $\Phi$ carry the given framing to one of
the two framing classes $\lambda_\pm$ of Lemma~\ref{lem:frame}, and framed
surgery depends only on the isotopy class of the framed circle. Hence the
result is diffeomorphic as a pair to $(\Sigma_+(7),K_+)$ or to
$(\Sigma_-(7),K_-)$, up to the orientation of the belt sphere.

\emph{Framings of the opposite orientation class.} Lemma~\ref{lem:frame}
classifies the framings of $\gamma$ that induce on its normal bundle the
same orientation as $\lambda_+$, and the preceding paragraph applies to
framings of this kind. A framing inducing the opposite orientation occurs,
for example, after the orientation of a loop of degree $-1$ is reversed,
which reverses the normal orientation induced from the loop and from $\M$.
It also occurs after transport by an orientation-reversing diffeomorphism,
as in Remark~\ref{rem:rigidity-outline}(d). Such a framing is represented
by $T\circ(\id_{S^1}\times r)$, where $T$ is a tubular embedding
representing $\lambda_+$ or $\lambda_-$ and $r=\diag(-1,1,1)$. With
attaching maps $\phi_T(x,e^{2\pi is})=T(s,x)$ as in
Section~\ref{sec:construction},
\[
 \phi_{T\circ(\id\times r)}=\phi_T\circ\bigl(r|_{S^2}\times\id_{S^1}\bigr).
\]
Since $r|_{S^2}\times\id_{D^2}$ extends $r|_{S^2}\times\id_{S^1}$ over
$S^2\times D^2$, the identity of $\M\setminus\Int\nu\gamma$ together with
$r|_{S^2}\times\id_{D^2}$ is a diffeomorphism between the two surgered
manifolds. It carries belt sphere onto belt sphere and reverses the
orientation of the latter. Hence every framing of a loop of degree $\pm1$
yields $(\Sigma_+(7),K_+)$ or $(\Sigma_-(7),K_-)$ up to the orientation of
the belt sphere.
\end{proof}

\begin{remark}[Conjugacy classes versus strict weight orbit]\label{rem:two-classes}
The two positive degree-one conjugacy classes are represented by $t$ and
$(-1)t$ and are interchanged by the involution $\Phi$; see
\cite[Lemma~18.2(i),(iii), pp.~356--358]{QextC}. Thus there are two conjugacy
classes but one positive strict weight orbit. This distinction is the only
weight-class information used below.
\end{remark}

The classification reduction in the next theorem is the parameter-seven
specialization of \cite[Theorem~18.5(iii), pp.~359--360]{QextC}, rather than a
new result here. The questions addressed in the present paper are whether
the two resulting homotopy four-spheres are standard and whether the two
resulting knot pairs are equivalent.

\begin{theorem}[Fibred rigidity]\label{thm:rigidity}
Every smooth oriented fibred two-knot with quandle $P_7$ in a smooth
homotopy four-sphere determines a pair diffeomorphic, possibly after reversing
the knot orientation, to one of the pairs
\[
                     (\Sigma_+(7),K_+),\;\;\;(\Sigma_-(7),K_-).
\]
Thus there are at most two types under equivalence that permits reversal of
the knot orientation.
\end{theorem}
\begin{proof}
This is the case $n=7$ of \cite[Theorem~18.5(iii), pp.~359--360]{QextC}.

Independently of \cite{QextC}, Remark~\ref{rem:rigidity-outline} below
gives a complete proof within the present paper. Apart from standard
differential topology and the coset description of the fundamental
quandle \cite{Joyce,WinterHigher}, its only external inputs are
Perelman's elliptization theorem
\cite{PerelmanEntropy,PerelmanSurg,PerelmanExt}, used in step~(b), and
the fact that every diffeomorphism of $\N$ is isotopic to an isometry
\cite[Theorem~1.2.1]{HKMR2012}, used in step~(c). The algebraic and
fibrewise steps rest on Lemma~\ref{lem:As-twoknot},
Remark~\ref{rem:As-algebraic}, the self-contained verification of
Lemma~\ref{lem:weight-loops} and the circle-isotopy statement of
Lemma~\ref{lem:circle-spin}.
\end{proof}

\begin{remark}[Structure of the argument behind Theorem~\ref{thm:rigidity}]
\label{rem:rigidity-outline}
For the reader's convenience we indicate how Theorem~\ref{thm:rigidity}
follows from the material above. Besides the knot--quandle dictionary of
Section~\ref{sec:construction}, the argument uses Perelman's elliptization
theorem \cite{PerelmanEntropy,PerelmanSurg,PerelmanExt} and the
generalized Smale conjecture, which holds for all prism and quaternionic
manifolds \cite{HKMR2012}, in particular for $\N$. Let $K\subset\Sigma$ be
a smooth oriented fibred two-knot with $Q(\Sigma,K)\cong P_7$, exterior
$X$, fibre $F$ and closed monodromy $\hat h$.

(a) \emph{Group and fibre.} By the dictionary and
Proposition~\ref{prop:exists}, $\pi_1(X)\cong\As(P_7)\cong G_7$, and
meridians correspond to conjugates of $t$.
Lemma~\ref{lem:As-twoknot} and Remark~\ref{rem:As-algebraic} prove this
step for an arbitrary smooth oriented two-knot in a smooth homotopy
four-sphere: any quandle isomorphism $Q(\Sigma,K)\cong P_7$ induces an
isomorphism $\pi_1(X)\cong G_7$ that carries each meridian to an element
$x^{-1}tx$ with $x\in\Gamma$.
The bundle projection induces an
epimorphism $\pi_1(X)\to\ZZ$ that sends a meridian to $1$, hence the degree
homomorphism. The exact homotopy sequence of the bundle identifies
$\pi_1(F)=\pi_1(\widehat F)$ with its kernel $\Gamma$, and under this
identification $\hat h$ induces the outer class of conjugation by a
meridian, namely $[\theta]$.

(b) \emph{The closed fibre is $\N$.} By elliptization,
$\widehat F=S^3/\rho(\Gamma)$ for a free orthogonal action $\rho$, which is
orientation-preserving because $\widehat F$ is orientable. Write the elements
of $SO(4)=(S^3\times S^3)/\{\pm(1,1)\}$ as
$[(x,y)]\colon v\mapsto xvy^{-1}$. The unique involution $-1$ of $\Gamma$
acts by $-\id=[(-1,1)]$, the only free involution in $SO(4)$. If
$q\in\{\pm i,\pm j,\pm k\}$ and $\rho(q)=[(x,y)]$, then
$(x^2,y^2)=\pm(-1,1)$, so $\rho(q)$ has the form $[(x,1)]$ or $[(1,y)]$.
Elements of these two forms commute, whereas $\rho(i)$ commutes neither with
$\rho(j)$ nor with $\rho(k)$. Hence, after conjugating, if necessary, by the
orientation-reversing isometry $v\mapsto\bar v$, which exchanges the two
factors, $\rho(Q_8)$ acts by left multiplication by a subgroup of $S^3$
isomorphic to $Q_8$, and such a subgroup is conjugate to the standard
$Q_8$. If $\rho(c)=[(x,y)]$, then commutation with $\rho(Q_8)$ forces
$x=\pm1$, so $\rho(c)$ is right multiplication by an element of order
seven. After a further conjugation in $SO(4)$ and an automorphism of
$\Gamma$, the action $\rho$ becomes \eqref{eq:action}. Thus $\widehat F$ is
diffeomorphic to $\N$, possibly reversing orientation.

(c) \emph{The closed monodromy.}
Since $X$ is orientable, $\hat h$ preserves the orientation of
$\widehat F$: the mapping torus of an orientation-reversing diffeomorphism
of a connected orientable manifold is nonorientable. Conjugating by a
diffeomorphism $\widehat F\to\N$, whether or not it preserves orientation,
therefore gives an element of $\Diff^+(\N)$, to which the injectivity
statement below applies.
Transport $\hat h$ to $\N$ by such a
diffeomorphism. Its outer class changes by an automorphism of $\Gamma$.
Since $\Out(\Gamma)=\Out(Q_8)\times\Aut(C_7)$ and $[\theta]=([\alpha],-1)$,
the transported class is $([\alpha]^{\pm1},-1)$, the class of
$\bar f^{\pm1}$. Every orientation-preserving isometry of $\N$ is induced
by a map $v\mapsto AvB$ in which $A$ lies in the normalizer of $Q_8$ in
$S^3$, the binary octahedral group, and
$B\in\{e^{i\phi}\}\cup j\{e^{i\phi}\}$. Such a map induces the trivial outer
automorphism exactly when $A\in Q_8$ and $B\in\{e^{i\phi}\}$, and it then
lies in the identity component $\{[v\mapsto ve^{i\phi}]\}$ of
$\Isom^+(\N)$. Hence $\pi_0\Isom^+(\N)\to\Out(\Gamma)$ is injective. By the
Smale conjecture for $\N$ (only its $\pi_0$-part \cite[Theorem~1.2.1]{HKMR2012} is used here; in
the terminology of \cite[Section~1.1, Table~1]{HKMR2012}, $\N$ is a
quaternionic manifold, so the full conjecture
\cite[Theorem~1.2.2]{HKMR2012} also applies), the same holds for $\pi_0\Diff^+(\N)$, so the
transported closed monodromy is isotopic to $\bar f$ or to $\bar f^{-1}$.

(d) \emph{Reconstruction.} Cap each fibre of $X$ by a ball, that is, form
$\widehat X=X\cup_{\partial}(D^3\times S^1)$. Then $\widehat X$ is the
mapping torus of $\hat h$, so it is diffeomorphic to $\M$ by (b) and (c);
the mapping torus of $\bar f^{-1}$ is $\M$ with the base direction
reversed. The pair $(\Sigma,K)$ is recovered from $\widehat X$ by framed
surgery on the core circle $\{0\}\times S^1$, whose belt sphere is $K$. This
circle is a section of $\widehat X\to S^1$, so it has degree $\pm1$, and
Lemma~\ref{lem:weight-loops} identifies $(\Sigma,K)$, up to the orientation
of $K$, with $(\Sigma_+(7),K_+)$ or with $(\Sigma_-(7),K_-)$.
\end{remark}

Combining this theorem with Corollary~\ref{eval:cor:pairs}(i) gives exactly
two types under the same notion of equivalence, which permits reversal of the
knot orientation (Corollary~\ref{eval:cor:pairs}(ii)).

The theorem does not assert that every smooth realization of $P_7$ has a
fibred exterior. Its precise consequence for the standard four-sphere is
\begin{equation}\label{eq:fibred}
 \begin{gathered}
 P_7\text{ has a smooth fibred realization in the standard }S^4\\
 \Longleftrightarrow\;\;
 \Sigma_+(7)\diff S^4\ \text{or}\ \Sigma_-(7)\diff S^4.
 \end{gathered}
\end{equation}
This equivalence is recorded in \cite[Corollary~18.6(i)]{QextC}. By Theorem~\ref{std:thm:main}
both alternatives on the right-hand side hold.

\section{An order-six representative of the closed monodromy class}\label{per:sec:periodic}
Retain the quaternionic data, the positive-meridian convention and the
presentation \eqref{per:eq:torus} of Section~\ref{sec:construction}.
An embedded loop of degree $\pm1$ under $p_*$ will be called a
\emph{weight circle}; Lemma~\ref{lem:weight-loops} shows that these
loops normally generate $G_7$ and that their framed surgeries are
represented by the two candidates in \eqref{eq:candidates}.

We first record the elementary rigidity statement used in all subsequent
computations.

\begin{lemma}[Two-sided multiplication rigidity]\label{per:lem:rigidity}
Let $A$, $B$, $C$ and $E$ be unit quaternions. Then $AvB=CvE$ for all $v\in S^3$ if
and only if $(C,E)=(\epsilon A,\epsilon B)$ for some $\epsilon=\pm1$.
Moreover the isometry $v\mapsto AvB$ of $S^3$ has a fixed point if and
only if
\begin{equation}\label{per:eq:fixcrit}
                     \Ree(A)=\Ree(B).
\end{equation}
When a fixed point exists and $A\neq\pm1$, the fixed-point set is a great circle.
\end{lemma}
\begin{proof}
These are standard consequences of the unit-quaternion description of
$SO(4)$; see \cite{DuVal1964,ConwaySmith2003} and
\cite[Section~2]{MecchiaSeppi2015}. In the convention used here, the
only pairs acting trivially by $v\mapsto AvB$ are $(1,1)$ and $(-1,-1)$.
If $AvB=CvE$ for all $v$, then
$\lambda:=C^{-1}A=EB^{-1}$
satisfies $\lambda v=v\lambda$ for all $v\in S^3$, so $\lambda=\pm1$; the
converse is clear. Next, $AvB=v$ is equivalent to $v^{-1}Av=B^{-1}$; two
unit quaternions are conjugate in $S^3$ exactly when their real parts agree,
and $\Ree(B^{-1})=\Ree(\overline B)=\Ree(B)$, which gives
\eqref{per:eq:fixcrit}. The solution set is a coset of the centralizer of
$A$, which is a great circle for $A\neq\pm1$.
\end{proof}

The sixfold cover and the suspension Seifert circle fibration are
described in \cite[Proposition~21.2(iv), pp.~379--380]{QextC}.
For the general theory of $\mathbb S^3\times\mathbb E^1$ mapping-torus
geometry and of circle orbifold bundles over geometric three-orbifolds, see
\cite[Section~11.4, pp.~225--226; Section~13.5, pp.~263--264]{Hillman}.
The distinction between periodic closed monodromy and periodic
monodromy of the knot exterior is emphasized separately in
\cite[Section~16.3, p.~316]{Hillman}.

The order-six closed-monodromy data below are the parameter-seven case of
\cite[Proposition~15.24, pp.~275--276]{QextC}; the sixfold-cover description
is \cite[Proposition~21.2(iv), pp.~379--380]{QextC}.

\begin{theorem}[Order-six periodicity and the geometry of $\M$]
\label{per:thm:periodic}
With the data of Section~\ref{sec:construction} and \eqref{per:eq:torus}, the following statements hold:
\begin{enumerate}[label=(\roman*),leftmargin=2.2em]
\item $f^n(v)=\omega^{-n}vj^n$ for all $n$, and
      $f^6=\gamma_{-1,0}$, that is, $f^6(v)=-v$. Moreover $f$ has
      order $12$ in $\Isom(S^3)$.
\item $f^n\in\Gamma$ if and only if $6\mid n$. Hence, the closed
      monodromy $\bar f$ has order exactly six in $\operatorname{Diff}(\N)$,
      and $\theta$ has order exactly six in $\operatorname{Aut}(\Gamma)$.
\item $G_7$ acts freely, properly discontinuously, cocompactly and
      isometrically on $S^3\times\RR$; thus $\M$ is a closed
      four-manifold with the geometry $\mathbb S^3\times\mathbb E^1$, and
      it admits a metric of positive scalar curvature.
\item The subgroup $\Gamma\times\langle T\rangle$, where $T(v,s)=(v,s+6)$,
      equals $\langle\Gamma,D^6\rangle$ and has index six in $G_7$. The
      corresponding cover is canonically described by
      \begin{equation}\label{per:eq:sixfold}
       \N\times S^1\to\M,\;\;\;
       \M=(\N\times S^1)\big/\bigl\langle(\bar f^{-1},
             \text{rotation by }\tfrac{2\pi}{6})\bigr\rangle,
      \end{equation}
      where the $\ZZ/6$-action is free and $S^1=\RR/6\ZZ$.
\end{enumerate}
\end{theorem}
\begin{proof}
Parts~(i), (ii) and the closed-monodromy component of~(iv) are the
$n=7$ specialization of \cite[Proposition~15.24, pp.~275--276]{QextC}; the
sixfold product cover and suspension description are
\cite[Proposition~21.2(iv), pp.~379--380]{QextC}. For~(iii), an element of nonzero base degree translates the $\mathbb R$
coordinate in \eqref{per:eq:torus}, whereas an element of degree zero
belongs to the freely acting finite group $\Gamma$.
Thus the action is free, proper and cocompact. The round product metric on
$S^3\times\mathbb R$ is invariant and has positive scalar curvature, so it
descends to $\M$.
\end{proof}

\begin{proof}[Direct verification of (i), (ii) and (iv)]
Since $f(v)=\omega^{-1}vj$, induction gives $f^n(v)=\omega^{-n}vj^n$ for all
$n\in\ZZ$. As $\omega^3=1$ and $j^2=-1$, this gives $f^6(v)=-v$, that is,
$f^6=\gamma_{-1,0}$. By Lemma~\ref{per:lem:rigidity}, $f^n=\gamma_{q,a}$
holds exactly when $(\omega^{-n},j^n)=\pm(q,\zeta^{-a})$. The second
coordinate then lies in $\RR+\RR i$, which forces $j^n=\pm1$; hence $n$ is
even, $\zeta^{-a}=\pm1$, and $a=0$ because $\zeta$ has odd order. The first
coordinate then lies in $\pm Q_8=Q_8$, which forces $3\mid n$, because
$\omega$ and $\omega^2$ have real part $-\tfrac12$. Hence $f^n\in\Gamma$ if
and only if $6\mid n$, and $\bar f$ has order exactly six. The lifts of $\id_{\N}$ to $S^3$ are precisely the deck transformations.
Thus $\bar f^{\,n}=\id_{\N}$ if and only if $f^n\in\Gamma$. Similarly,
$f^n=\id$ exactly when $(\omega^{-n},j^n)=\pm(1,1)$; since
$-1\notin\langle\omega\rangle$, this means $\omega^{-n}=1=j^n$, that is,
$12\mid n$. Finally $\theta^n=\alpha^n\times(c\mapsto c^{(-1)^n})$, and
$\alpha$ has order three, so $\theta$ has order six.

For~(iv), $D^6(v,s)=(f^{-6}(v),s+6)=\gamma_{-1,0}T(v,s)$, and $T$ commutes
with every $\gamma_{q,a}$. Hence
$\langle\Gamma,D^6\rangle=\Gamma\times\langle T\rangle$. This subgroup is the
kernel of the base degree modulo six, so it is normal of index six. Its
quotient of $S^3\times\RR$ is $\N\times\RR/6\ZZ$, on which $D$ induces
$([v],s)\mapsto(\bar f^{-1}[v],s+1)$. This map rotates the circle factor by
$2\pi/6$, so it generates a free $\ZZ/6$-action, and the quotient is $\M$;
this is \eqref{per:eq:sixfold}.
\end{proof}

\noindent\emph{The product-circle marking of the sixfold cover.}
In the product cover \eqref{per:eq:sixfold}, let $z$ be the positive
generator of the $S^1$ factor. Its lift to $S^3\times\RR$ is the pure
translation $T(v,s)=(v,s+6)$. Thus, under the inclusion of the covering
group in $G_7$,
\[
 z=(-1)t^6,\;\;\; t^6=(-1)z,
 \;\;\;
 \pi_1(\N\times S^1)=\Gamma\times\langle z\rangle.
\]
In particular, the section lifted to this cover represents $t^6$,
whose projection to $\Gamma=\pi_1(\N)$ is the central involution
$-1\in Q_8$, acting by $\gamma_{-1,0}$. This specifies the product marking relevant to the
sixfold branched-cover surgery; compare
\cite[Remark~20.3, p.~377]{QextC}.

\begin{proposition}[The fixed-point sets of the powers]
\label{per:prop:fixed}\label{nf:prop:fixed}
For $1\leq n\leq5$, the fixed-point set of $\bar f^{\,n}$ on $\N$ is given by
\[
 \Fix(\bar f^{\,n})=\emptyset\;\;(n=1,2,4,5),\;\;\;
 \Fix(\bar f^{\,3})=\mathcal C_1\sqcup\mathcal C_2\sqcup\mathcal C_3,
\]
where the $\mathcal C_r$ are disjoint embedded geodesic circles. The
involution $\bar f^{\,3}$ fixes each $\mathcal C_r$ pointwise, and
$\bar f$ permutes the three circles cyclically. The preimage of
$\Fix(\bar f^{\,3})$ in $S^3$ is the union of the $42$ great circles
$\Fix(\gamma_{q,a}f^3)$ with $q\in\{\pm i,\pm j,\pm k\}$ and
$a\in\ZZ/7$. The action of $\Gamma$ on this set of circles has three orbits
of size fourteen.
\end{proposition}
\begin{proof}
For the fixed-point sets we use the quaternionic criterion of
Lemma~\ref{per:lem:rigidity}; compare
\cite[Propositions~15.24 and~21.2(iii)]{QextC}. A fixed point of $\bar f^{\,n}$ exists exactly when, for some
$q\in Q_8$ and $a\in\ZZ/7$, the equation
\[
       q^{-1}\omega^{-n}v j^n\zeta^a=v
\]
has a solution. By Lemma~\ref{per:lem:rigidity}, this is equivalent to
\[
       \Ree(q^{-1}\omega^{-n})=\Ree(j^n\zeta^a).
\]
For $n\in\{1,2,4,5\}$, the left-hand side is always $\pm\tfrac12$: the
element $q^{-1}\omega^{-n}$ lies in $Q_8\omega\cup Q_8\omega^2$, and these
sixteen elements of the binary tetrahedral group $\langle i,j,\omega\rangle$
are exactly the quaternions $\tfrac12(\pm1\pm i\pm j\pm k)$.
For $n=1,5$, the right-hand side is zero. For $n=2,4$, it is
$(-1)^{n/2}\cos(2\pi a/7)$, which cannot equal $\pm\tfrac12$: a unit
complex number with real part $\pm\tfrac12$ has order three or six,
whereas $\zeta^a$ has order dividing seven. Hence all four powers are
fixed-point free. It remains to count the components of the fixed set
of the third power.

A point $[v]\in\N$ is fixed by $\bar f^{\,3}$ exactly when
$f^3(v)=\gamma_{q,a}(v)$ for some $(q,a)\in Q_8\times\mathbb Z/7$.
Equivalently,
$
        q^{-1}v\,j^3\zeta^a=v
$,
because $\omega^{-3}=1$. By Lemma~\ref{per:lem:rigidity}, this equation has
a solution exactly when
\[
        \Ree(q^{-1})=\Ree(j^3\zeta^a)=0,
\]
that is, precisely for
$q\in\{\pm i,\pm j,\pm k\}$ and arbitrary $a\in\mathbb Z/7$. For each of
these $42$ choices, the corresponding isometry has a great circle as its fixed-point set.
Each of them is an involution, since $q^{-2}=-1=(j^3\zeta^a)^2$.
As $(q,a)$ ranges over these choices, the resulting circles form the same collection as the
$42$ circles $\Fix(\gamma_{q,a}f^3)$ appearing in the statement (inverting the
$\Gamma$-factor merely relabels the set of pairs).

Distinct elements of the coset $\Gamma f^3$ cannot fix a common point: their
quotient would be a nonidentity element of the freely acting deck group
$\Gamma$ fixing that point. Hence, these $42$ great circles are pairwise
disjoint. For $h\in\Gamma$ and $g\in\Gamma f^3$, one has
$h\,\Fix(g)=\Fix(hgh^{-1})$ with $hgh^{-1}\in\Gamma f^3$.
Hence if $h$ carries a point of $\Fix(g)$ into $\Fix(g')$ for some
$g'\in\Gamma f^3$, then $hgh^{-1}=g'$ by the preceding disjointness
argument. Thus only the stabilizer of a circle identifies points of that
circle, and circles in different $\Gamma$-orbits have disjoint images.
Conjugation by a deck transformation gives
\[
 \gamma_{p,b}(\gamma_{q,a}f^3)\gamma_{p,b}^{-1}
       =\gamma_{pqp^{-1},\,a+2b}f^3,
\]
because $f^3\gamma_{p,b}f^{-3}=\gamma_{p,-b}$. The first coordinate ranges
over the $Q_8$-conjugacy class $\{q,-q\}$ and the second over all of
$\mathbb Z/7$. Thus, the $42$ circles split into three $\Gamma$-orbits of size
$14$, indexed by the three axes $\{\pm i\}$, $\{\pm j\}$ and
$\{\pm k\}$. Their images are therefore three disjoint embedded geodesic
circles $\mathcal C_1$, $\mathcal C_2$ and $\mathcal C_3$ in $\N$.

More explicitly, put $C_{q,a}=\Fix(\gamma_{q,a}f^3)$.
The disjointness of the fixed circles shows that the setwise stabilizer of this circle
is the centralizer of $\gamma_{q,a}f^3$ in $\Gamma$.
The displayed conjugation formula therefore gives
\[
 \Stab_{\Gamma}(C_{q,a})
 =\{\gamma_{p,0}:p\in C_{Q_8}(q)\}
 =\{\gamma_{p,0}:p\in\langle q\rangle\}
 \cong C_4.
\]
Indeed, $a+2b=a$ forces $b=0$ in $\ZZ/7$, and the centralizer
in $Q_8$ of a noncentral element $q$ is $\langle q\rangle$.
This agrees with the orbit count $56/14=4$.
This order-four stabilizer acts freely on $C_{q,a}$, since the
deck action is free. A circle isometry that reverses orientation has
a fixed point, so the stabilizer acts by rotations.
Thus each of the fourteen great circles in one $\Gamma$-orbit maps
with covering degree four onto its image in $\N$.
Only the stabilizer identifies points of a given circle, and circles
in different $\Gamma$-orbits have disjoint images. Consequently
these images are embedded circles. Since $S^3\to\N$ is a local
isometry, they are geodesic.

Finally,
$f\Fix(\gamma_{q,a}f^3)=\Fix(\gamma_{\alpha(q),-a}f^3)$ and $\alpha$ cycles
$i\mapsto j\mapsto k\mapsto i$, so $\bar f$ cyclically permutes the three
components. By construction $\bar f^3$ fixes each component pointwise. This
refines the exceptional-orbit description in
\cite[Proposition~21.2(iv), pp.~379--380]{QextC}.
\end{proof}

Figure~\ref{per:fig:fixtest} displays the comparison of real parts made in this
proof.

\begin{figure}[tbp]
\centering
\begin{tikzpicture}[
  ldot/.style={circle,fill=liftblue,inner sep=0pt,minimum size=4.4pt},
  rdot/.style={circle,draw=liftorange,line width=0.9pt,inner sep=0pt,minimum size=7.6pt}]
\begin{scope}
\shade[ball color=diagramlight!60!white] (0,0) circle[radius=1.3];
\draw[diagramblue!45] (0,0) circle[radius=1.3];
\foreach \h/\c in {0.65/liftblue,0/diagramblue,-0.65/diagramblue} {
  \pgfmathsetmacro{\rx}{1.3*sqrt(1-(\h/1.3)^2)}
  \draw[\c!75,line width=0.55pt] (-\rx,\h) arc[start angle=180,end angle=360,x radius=\rx,y radius=0.12];
  \draw[\c!60,densely dashed,line width=0.45pt] (\rx,\h) arc[start angle=0,end angle=180,x radius=\rx,y radius=0.12];}
\fill (0,1.3) circle[radius=0.045] (0,-1.3) circle[radius=0.045];
\node[slab,anchor=south] at (0,1.33) {$1$};
\node[slab,anchor=north] at (0,-1.33) {$-1$};
\node[slab,anchor=west] at (1.2,0.76) {$\Ree=\tfrac12$};
\node[slab,anchor=west] at (1.36,0.0) {$\Ree=0$};
\node[slab,anchor=west] at (1.2,-0.76) {$\Ree=-\tfrac12$};
\fill[liftblue] (-0.78,0.545) circle[radius=0.055];
\fill[liftorange] (0.7,0.535) circle[radius=0.055];
\node[slab,text=liftblue,anchor=north] at (-0.78,0.48) {$A$};
\node[slab,text=liftorange!85!black,anchor=north] at (0.72,0.48) {$B^{-1}$};
\draw[-{Stealth[length=1.4mm]},gray!60!black] (-0.72,0.62) .. controls (-0.36,0.98) and (0.3,0.98) .. (0.64,0.62);
\node[slab,text=gray!55!black] at (0,1.03) {$v^{-1}(\cdot)v$};
\node[panelcap] at (0,-2.35) {(a) conjugacy classes in $S^3$\\ are latitude spheres};
\end{scope}
\begin{scope}[shift={(3.55,0)}]
\foreach \n/\y in {1/1.1,2/0.55,3/0,4/-0.55,5/-1.1} {
  \ifnum\n=3
    \fill[orbitred!8,rounded corners=2pt] (-0.25,\y-0.22) rectangle (4.65,\y+0.22);
  \else
    \fill[diagramlight!80,rounded corners=2pt] (-0.25,\y-0.22) rectangle (4.65,\y+0.22);
  \fi
  \draw[gray!55] (0,\y) -- (4.4,\y);
  \foreach \x in {0,1.1,2.2,3.3,4.4} \draw[gray!55] (\x,\y-0.05) -- (\x,\y+0.05);
  \node[slab,anchor=east] at (-0.3,\y) {$n=\n$};}
\foreach \y in {1.1,0.55,-0.55,-1.1} {\node[ldot] at (1.1,\y) {}; \node[ldot] at (3.3,\y) {};}
\foreach \x in {0,2.2,4.4} \node[ldot] at (\x,0) {};
\foreach \y in {1.1,0,-1.1} \node[rdot] at (2.2,\y) {};
\foreach \a in {0,...,3} {
  \pgfmathsetmacro{\xx}{2.2*(1-cos(360*\a/7))}
  \node[rdot] at (\xx,0.55) {};}
\foreach \a in {0,...,3} {
  \pgfmathsetmacro{\xx}{2.2*(1+cos(360*\a/7))}
  \node[rdot] at (\xx,-0.55) {};}
\draw[orbitred,line width=0.8pt] (2.2,0) circle[radius=0.21];
\node[slab,text=orbitred,anchor=west,fill=white,inner sep=1pt] at (2.47,0.2) {$6\times7=42$};
\foreach \x/\lb in {0/-1,1.1/-\tfrac12,2.2/0,3.3/\tfrac12,4.4/1} \node[tlab] at (\x,-1.45) {$\lb$};
\node[ldot] at (0.35,-1.82) {}; \node[slab,anchor=west] at (0.58,-1.82) {$\Ree A$, $A=q^{-1}\omega^{-n}$, $q\in Q_8$};
\node[rdot] at (0.35,-2.12) {}; \node[slab,anchor=west] at (0.58,-2.12) {$\Ree B$, $B=j^n\zeta^a$, $a\in\ZZ/7$};
\node[panelcap] at (2.2,-2.6) {(b) comparing real parts for $\bar f^{\,n}$};
\end{scope}
\end{tikzpicture}
\caption{The fixed-point test of Lemma~\ref{per:lem:rigidity} in the proof of
Proposition~\ref{per:prop:fixed}. (a) Two unit quaternions are conjugate exactly
when their real parts agree, so the conjugacy classes of $S^3$ are the latitude
two-spheres (shown schematically as ellipses) $\{\Ree=c\}$ and the poles $\pm1$.
The isometry $v\mapsto AvB$ fixes $v$ exactly when $v^{-1}Av=B^{-1}$, so it has a
fixed point exactly when $A$ and $B^{-1}$ lie on the same latitude, that is,
when $\Ree A=\Ree B$. (b) A point of $\N$ is fixed by $\bar f^{\,n}$ when one of
the isometries $v\mapsto q^{-1}\omega^{-n}vj^n\zeta^a$ has a fixed point. For
$n=1,2,4,5$ the dots lie at $\pm\frac12$, which no circle reaches; for $n=2,4$ the
circles lie at $(-1)^{n/2}\cos(2\pi a/7)$; repeated real parts are shown only
once. For $n=3$ (shaded row) the dots at $0$ come from $q\in\{\pm i,\pm j,\pm k\}$
and meet the circles from all seven values of $a$; these $42$ isometries are
involutions, and their fixed great circles project to
$\mathcal C_1\sqcup\mathcal C_2\sqcup\mathcal C_3$.}
\label{per:fig:fixtest}
\end{figure}

\section{The canonical Seifert structure of \texorpdfstring{$\M$}{M7}}\label{per:sec:action}
The periodicity of the closed monodromy representative $\bar f$
(Theorem~\ref{per:thm:periodic}), not of the punctured-fibre monodromy
(Remark~\ref{rem:closed-vs-punctured-monodromy}), converts the
mapping-torus direction into a circle action. Write
$\Lam=\langle\Gamma,f\rangle\subset\Isom(S^3)$.

\begin{theorem}[The canonical circle action]\label{per:thm:action}
Translations $T_r(v,s)=(v,s+r)$ along the $\RR$-factor descend to a
smooth, effective action of $S^1=\RR/6\ZZ$ on $\M$ with the following
properties:
\begin{enumerate}[label=(\roman*),leftmargin=2.2em]
\item The action has no fixed points. The isotropy group at $[v,s]$ is
      cyclic of order $6/k_0$, where
      $k_0=\min\{n>0:\bar f^{\,n}[v]=[v]\}$, and the orbit through
      $[v,s]$ has degree $k_0$ under $p_*$.
\item We have $k_0\in\{3,6\}$. Principal orbits have $k_0=6$; the exceptional
      orbits are those with $[v]\in\Fix(\bar f^{\,3})$; they have
      isotropy of order two and degree three, and their union is a single
      embedded torus
      \[
       \mathcal E=\bigl\{[v,s]:[v]\in\Fix(\bar f^{\,3})\bigr\}\cong T^2.
      \]
\item The class of a principal orbit is $(-1)t^{6}\in G_7$, which is
      central and has degree six.
\item $\Lam$ has order $336$, and the orbit map realizes $\M$ as a
      Seifert-fibred four-manifold over the spherical three-orbifold
      $
       \M/S^1=\N/\langle\bar f\rangle=S^3/\Lam
      $,
      whose singular set is one circle with isotropy of order two.
\item There is a central extension
      \begin{equation}\label{per:eq:centralext}
       1\to\ZZ\bigl\langle(-1)t^{6}\bigr\rangle
        \to G_7\to\Lam\to1.
      \end{equation}
\end{enumerate}
\end{theorem}
\begin{proof}
The effective period-six suspension action, the isotropy formula, and the
orbit-space identification with $\N/\langle\bar f\rangle$ are precisely the
$n=7$ case of \cite[Proposition~21.2(iv), pp.~379--380]{QextC}. Together
with Proposition~\ref{per:prop:fixed}, this gives~(i) and shows that the
only exceptional orbits occur over $\Fix(\bar f^{\,3})$ with isotropy
$C_2$. Since $\bar f$ cyclically permutes the three components of
$\Fix(\bar f^{\,3})$ and $\bar f^{\,3}$ fixes them pointwise, their mapping
torus is a single torus $\mathcal E\cong T^2$, proving~(ii).

Equivalently, $\mathcal E=\Fix(T_3)$: by~(i), the time-three map fixes
$[v,s]$ exactly when $k_0\mid3$, that is, when $k_0=3$. Since $T_3$ is an
isometry of the metric induced by the round product metric on $S^3\times\RR$
(Theorem~\ref{per:thm:periodic}(iii)), $\mathcal E$ is a closed totally
geodesic, in particular smoothly embedded, submanifold
\cite[Chapter~II]{Kobayashi1972}.

For~(iii), a principal orbit lifts from $(v_0,0)$ to $(v_0,6)$. Because
$f^{-6}=-\id$, the corresponding deck element is
$\gamma_{-1,0}D^6$, hence its class is $(-1)t^6$. This element is central:
$-1\in Z(\Gamma)$ is fixed by $\theta$, and $\theta^6=\id$.

For~(iv)--(v), the least positive power of $f$ contained in $\Gamma$ is
$f^6$, so $[\Lam:\Gamma]=6$ and $|\Lam|=336$. Projection to the
$S^3$-coordinate gives a surjection $G_7\to\Lam$ with
\[
 \ker(G_7\to\Lam)
 =\bigl\langle\gamma_{-1,0}D^6\bigr\rangle
 =\ZZ\bigl\langle(-1)t^6\bigr\rangle,
\]
which proves \eqref{per:eq:centralext}.

For the kernel, note that $\gamma_{q,a}D^n$ acts on the $S^3$-coordinate
by $\gamma_{q,a}f^{-n}$. If this map is the identity, then $f^n\in\Gamma$,
so $6\mid n$ by Theorem~\ref{per:thm:periodic}(ii); then
$f^{-n}=\gamma_{(-1)^{n/6},0}$ forces $(q,a)=((-1)^{n/6},0)$, and, since
$\gamma_{-1,0}$ commutes with $D$, the element is
$(\gamma_{-1,0}D^6)^{n/6}=T_n$. Conversely, $T_6=\gamma_{-1,0}D^6$ acts
trivially on the $S^3$-coordinate, and it has infinite order because it
translates the $\RR$-coordinate by six. Its class $(-1)t^6$ is central
by~(iii). The orbit map is therefore
$\M\to S^3/\Lam$. Moreover $\mathcal E/S^1$ is a single circle. At a point
of this circle the time-three isotropy acts on a transverse
three-dimensional slice as an involution with one-dimensional fixed
subspace; in suitable coordinates it is
$(\xi,\eta)\mapsto(\xi,-\eta)$ on $\RR\times\RR^2$. Hence the singular
locus has local group $C_2$, while principal orbits have the product local
model. This establishes the asserted Seifert fibration and the isomorphism
$\pi_1^{\mathrm{orb}}(S^3/\Lam)\cong\Lam$.
\end{proof}

\begin{proof}[Direct verification of the suspension action]
Each translation $T_r$ commutes with every $\gamma_{q,a}$ and with $D$, so it
descends to $\M$, and $T_6=\gamma_{-1,0}D^6$ is a deck transformation by
Theorem~\ref{per:thm:periodic}(iv). Hence $\RR/6\ZZ$ acts smoothly on $\M$.
If $T_r[v,s]=[v,s]$, then $(v,s+r)=\gamma_{q,a}D^n(v,s)$ for some $(q,a)$
and some $n\in\ZZ$. Comparing $\RR$-coordinates gives $r=n$, and then
$v=\gamma_{q,a}f^{-n}(v)$, that is, $\bar f^{\,n}[v]=[v]$. Conversely, if
$\bar f^{\,n}[v]=[v]$, then $T_n[v,s]=[v,s]$. Thus the isotropy group at $[v,s]$ is $k_0\ZZ/6\ZZ$, which is cyclic of order $6/k_0$; here $k_0\mid6$
because $\bar f^{\,6}=\id$. In particular every isotropy group is finite,
so there are no fixed points, and the orbit $r\mapsto[v,s+r]$,
$0\leq r\leq k_0$, has degree $k_0$ under $p_*$. By
Proposition~\ref{per:prop:fixed}, $k_0\notin\{1,2\}$; hence
$k_0\in\{3,6\}$, and $k_0=3$ exactly when $[v]\in\Fix(\bar f^{\,3})$. Points
outside the three circles $\mathcal C_1$, $\mathcal C_2$ and $\mathcal C_3$
have trivial isotropy, so the action is effective.
\end{proof}

\begin{remark}[Terminology]\label{per:rem:seifert-terminology}
In Theorem~\ref{per:thm:action}(iv), ``Seifert-fibred over the three-orbifold
$S^3/\Lam$'' means that the orbit map of the fixed-point-free circle action is
an orbifold circle bundle, that is, a Seifert fibring with typical fibre $S^1$
in the sense of \cite{LeeRaymond2010}, as in Lemma~\ref{std:lem:exact}. This
differs from the four-dimensional usage of \cite[Section~7.7]{Hillman}, where a
Seifert-fibred four-manifold is the total space of an orbifold bundle with
general fibre a torus or a Klein bottle over a two-orbifold. By
\cite[Section~7.7]{Hillman} every $\mathbb S^3\times\mathbb E^1$-manifold, and
hence $\M$, is also Seifert fibred in that sense; that structure is not used here.
\end{remark}

\noindent
The orbit orbifold $S^3/\Lam$ and its singular circle are identified explicitly
in Section~\ref{wh:sec:whitehead} (Theorem~\ref{wh:thm:whitehead}).
The surgery section $\gamma$ has base degree one, whereas every orbit of
this action has degree three or six. Thus $\gamma$ is not an orbit, even
up to free homotopy and reversal of orientation. The orbit-surgery lemma
below does not directly extend this action over either surgery on
$\gamma$; the reduction used in Section~\ref{std:sec:main} is required.

\subsection{Circle actions and orbit surgeries}
The following classical facts about circle actions and framed surgery on
orbits are used in Section~\ref{std:sec:main}.

\begin{lemma}[Centrality of evaluation loops]\label{per:lem:gottlieb}
Let a path-connected topological group $G$ act continuously on a space $M$,
let $x\in M$, and define $\mathrm{ev}_x:G\to M$ by $g\mapsto g\cdot x$.
Then $(\mathrm{ev}_x)_*\pi_1(G,1)$ lies in the centre of $\pi_1(M,x)$.
\end{lemma}
\begin{proof}
This is the group-action case of Gottlieb's evaluation-subgroup
centrality theorem
\cite[Theorem~I.4 and the following paragraph, p.~843]{Gottlieb}.
Let $u:[0,1]\to G$ be a loop based at the identity
and let $\alpha:[0,1]\to M$ be a loop based at $x$.
The continuous map
\[
 H:[0,1]^2\to M,\;\;\; H(r,s)=u(r)\cdot\alpha(s)
\]
restricts to the boundary loop
$\beta*\alpha*\beta^{-1}*\alpha^{-1}$, where
$\beta(r)=u(r)\cdot x$.
The map $H$ fills this loop, so $[\beta]$ commutes with every $[\alpha]$.
\end{proof}

\noindent
\emph{Primitive orbits and evaluation loops.}
Orient circle orbits by the positive direction of the action.
If a circle orbit has isotropy $C_b$, the evaluation map
$S^1\to S^1/C_b$ has degree $b$.
Thus, if $e$ denotes the class of one positive traversal of the
geometric orbit, the full evaluation loop represents $e^b$,
and the lemma asserts the centrality of $e^b$.
For the action of Theorem~\ref{per:thm:action}, a full turn
takes time six and represents
$z=(-1)t^6$ at every basepoint, after transport to the chosen basepoint.
At a principal orbit this is one traversal; at an exceptional orbit
it is two traversals. Hence the class of a primitive exceptional orbit satisfies
$
                     p_*(e)=3,\;\;\; e^2=z.
$
Such an $e$ is not central. Indeed, every element of base degree
three has the form $h t^3$ with $h\in\Gamma$, and $h$ commutes with
$c$. The convention $t^{-1}ct=c^{-1}$ gives
$
                    e^{-1}ce=c^{-1}\neq c.
$

\begin{lemma}[Both framings of an orbit admit equivariant surgery]
\label{fa:lem:orbit-surgery}
Let $X$ be a connected oriented smooth four-manifold with an effective
smooth circle action, and let $O\subset\Int X$ be a circle orbit.
For either of its two oriented normal-framing classes, the action
extends to an effective smooth circle action on the manifold obtained
by surgery along $O$.
\end{lemma}
\begin{proof}
We give the local extension argument; for general background on
circle actions and equivariant modifications, see
\cite{FintushelSC,Fintushel1978,Pao}.
The isotropy of $O$ is a proper closed subgroup of $S^1$, so it is
$C_b$ for some $b\geq1$.
Choose an invariant Riemannian metric, obtained by averaging an arbitrary
metric over the circle \cite[Chapter~VI, Theorem~2.1]{Bredon}.
The equivariant tubular-neighbourhood theorem gives an invariant
neighbourhood equivariantly diffeomorphic to $S^1\times_{C_b}D^3$
\cite[Chapter~VI, Theorem~2.2 and Corollary~2.4,
pp.~306--309]{Bredon}.
The connected circle group preserves the orientation of $X$,
and its isotropy fixes the orbit pointwise because the group is
abelian. Hence the slice representation takes values in $SO(3)$.
It is faithful: an element in its kernel fixes the tube pointwise,
and an isometry equal to the identity on a nonempty open subset of
a connected manifold is the identity everywhere (the fixed-point set of an isometry is a closed totally geodesic submanifold
\cite[Chapter~II]{Kobayashi1972}, so a component containing a nonempty open
set is open and closed).
Thus a generator of $C_b$ acts on $D^3$ by a rotation
$R_{2\pi a/b}$ about an axis, with $\gcd(a,b)=1$.
When $b=1$, choose an arbitrary axis and take $a=0$.

Orient the normal bundle of $O$ so that the positive direction of $O$
followed by this normal orientation is the orientation of $X$. We identify
the slice with $\RR^3$ by an orthonormal basis that is positively oriented
for this normal orientation and whose third vector spans the axis, and
$R_\theta$ denotes the rotation through the angle $\theta$ in the first two
coordinates. The two possible directions along the axis replace $a$ by
$-a$, which is immaterial below.

Here is an explicit product trivialization of this tube.
For $h=e^{2\pi i k/b}$ use the associated-bundle relation
\[
 (u,y)\sim(u h^{-1},R_{2\pi a k/b}y).
\]
The map
\[
 F_a:S^1\times_{C_b}D^3\to S^1\times D^3,\;\;\;
 F_a([e^{2\pi i s},y])
   =(e^{2\pi i b s},R_{2\pi a s}y)
\]
is well defined by that relation. Its inverse is obtained by choosing
a $b$th root of the first coordinate and undoing the indicated
rotation; the relation makes the inverse independent of this choice.
Local smooth choices of the root show that $F_a$ is a diffeomorphism.

The local coordinates $(s,y)\mapsto[e^{2\pi is},y]$ of the tube preserve
orientation by the convention fixed above, and in them the Jacobian matrix of
$F_a$ is block lower triangular, with diagonal blocks $b>0$ and
$R_{2\pi as}\in SO(3)$. Hence $F_a$ preserves orientation, and it is an
oriented framing of $O$; the same argument applies to $F_{a+b}$ below.
In these coordinates the circle action is
\begin{equation}\label{fa:eq:orbit-weights}
 t\cdot(z,y)=(e^{2\pi i b t}z,R_{2\pi a t}y),
 \;\;\; t\in\RR/\ZZ.
\end{equation}

For this framing, define the following action on the replacement
$S^2\times D^2$:
\[
             t\cdot(x,w)=(R_{2\pi a t}x,e^{2\pi i b t}w).
\]
Under the surgery identification
$(x,z)\in S^2\times\partial D^2\mapsto
(z,x)\in S^1\times\partial D^3$, its boundary action agrees with
\eqref{fa:eq:orbit-weights}.
Both actions preserve radial collar coordinates and are independent
of the collar parameter, so invariant product collars give a smooth
action on the manifold obtained by the oriented surgery.
Equivalently, the weighted formula extends over the surgery-trace
handle $D^2\times D^3$, and invariant rounding of its corners gives
the same extension.

The integer $a+b$ gives the same $C_b$ slice representation, and
the corresponding trivialization satisfies
\[
 F_{a+b}\circ F_a^{-1}(z,y)=(z,R_{\arg z}y).
\]
The full-rotation loop $z\mapsto R_{\arg z}$ represents the nonzero
class of $\pi_1(SO(3))\cong\ZZ/2$: its lift to the unit-quaternion
double cover starting at $1$ ends at $-1$.
The oriented normal bundle over the circle is trivial, and its
framing classes are classified by this group. Hence $F_a$ and
$F_{a+b}$ represent the two framing classes.
Replacing $a$ by $a+b$ in the action on the replacement piece therefore
performs the surgery with the other framing.
This weight change is the local form of the mechanism in Pao's argument
\cite{Pao}: after a Gluck twist along an invariant two-sphere, the circle
action on the exterior extends across the reglued piece with one rotation
weight changed by the other; see \cite[proof of Theorem~3.2]{FukudaGluck}, and
compare the equivariant disc-bundle constructions of \cite{FintushelSC}
summarized in \cite[Section~5.2]{GalazGarcia2012}.

For either choice, write $a'$ for the chosen weight, so
$a'\in\{a,a+b\}$ and $\gcd(a',b)=1$.
The extended action is effective already on the replacement piece:
if $t\in\RR/\ZZ$ acts trivially there, then $a't,bt\in\ZZ$,
and coprimality forces $t=0$ in $\RR/\ZZ$.
Its new fixed set consists of the two poles of $S^2\times\{0\}$ when
$a'\neq0$, and is all of $S^2\times\{0\}$ when $a'=0$.
These fixed points are allowed by
Remark~\ref{fa:rem:Pao}.
\end{proof}

\begin{remark}[The circle-action recognition theorem]\label{fa:rem:Pao}
A closed simply connected smooth four-manifold admitting an effective
smooth circle action is diffeomorphic to a connected sum of copies of
$S^4$, $\pm\CP^2$, and $S^2\times S^2$; see
\cite[Section~5.1, Theorem~5.1]{GalazGarcia2012}, based on the
classification of Fintushel and Pao \cite{Fintushel1978,Pao}. Consequently,
a smooth homotopy four-sphere with an effective smooth circle action is
diffeomorphic to $S^4$. No restriction on the fixed-point set is required.
As recorded in \cite[Section~5.1]{GalazGarcia2012}, this smooth statement
combines Fintushel's classification \cite{FintushelSC,Fintushel1978} and the
work of Pao \cite{Pao} with Perelman's proof of the three-dimensional
Poincar\'e conjecture \cite{PerelmanEntropy,PerelmanSurg,PerelmanExt}.
The validity of Fintushel's classification in the smooth category is
attributed there to \cite{FintushelSternSunukjian2009}.
Fintushel's identification of the four-manifold was originally established
modulo that conjecture, which enters through the simply connected
three-dimensional orbit space.
\end{remark}

\section{The cyclic branched tower and its calibrated spin structures}\label{nf:sec:spin}

\subsection{Group and homology data for the cyclic covers}
For $m\geq1$, let $\Sigma^{(m)}_\eps$ denote the $m$-fold cyclic branched
cover of $(\Sigma_\eps(7),K_\eps)$. The group and homology assertions in the following theorem are the
parameter-seven specialization of
\cite[Proposition~20.1 and Remark~20.3, pp.~374--377]{QextC}.

\begin{theorem}[Six-periodic group and homology data]\label{per:thm:tower}
For every $m\geq1$ and either framing, $\Sigma^{(m)}_\eps$ is a closed
smooth oriented four-manifold with $\chi=2$ and
\begin{equation}\label{per:eq:towergroups}
 \pi_1\bigl(\Sigma^{(m)}_\eps\bigr)\cong
 \begin{cases}
  1, & m\equiv1,5\pmod 6,\\
  C_7, & m\equiv2,4\pmod 6,\\
  Q_8, & m\equiv3\pmod 6,\\
  \Gamma=Q_8\times C_7, & m\equiv0\pmod 6.
 \end{cases}
\end{equation}
Consequently $b_1=b_2=0$ in every case, each $\Sigma^{(m)}_\eps$ is a
rational homology four-sphere, and
$H_2(\Sigma^{(m)}_\eps;\ZZ)\cong H_1(\Sigma^{(m)}_\eps;\ZZ)
\cong\pi_1(\Sigma^{(m)}_\eps)^{\mathrm{ab}}$. For
$m\equiv\pm1\pmod 6$ the cover $\Sigma^{(m)}_\eps$ is itself a smooth
homotopy four-sphere.
\end{theorem}
\begin{proof}
This is the parameter-seven specialization of
\cite[Proposition~20.1 and Remark~20.3, pp.~374--377]{QextC}. The statement
is independent of the framing label.
\end{proof}

\begin{proof}[Direct verification]
Let $E^{(m)}\to\E$ be the $m$-fold cyclic covering corresponding to the
subgroup $\Gamma\rtimes\langle t^m\rangle$ of elements whose degree is
divisible by $m$. The branched cover is
$\Sigma^{(m)}_\eps=E^{(m)}\cup(S^2\times D^2)$, attached along the lift of
$\phi_\eps$ through $(x,z)\mapsto(x,z^m)$; compare
\eqref{nf:eq:lifted-surgery} below. Hence
\[
 \chi\bigl(\Sigma^{(m)}_\eps\bigr)
 =m\,\chi(\E)+\chi(S^2\times D^2)-\chi(S^2\times S^1)=2 .
\]
The circle $\{x\}\times\partial D^2$ of the cap represents a conjugate of
$t^m$. Since $t^{-m}gt^m=\theta^m(g)$ for $g\in\Gamma$, killing $t^m$ imposes
exactly the relations $g=\theta^m(g)$, and van Kampen's theorem gives
\[
 \pi_1\bigl(\Sigma^{(m)}_\eps\bigr)
 \cong(\Gamma\rtimes\langle t^m\rangle)/\ncl{t^m}
 \cong\Gamma/N_m,\;\;\;
 N_m=\bigl\langle g^{-1}\theta^m(g)\bigm| g\in\Gamma\bigr\rangle ;
\]
the subgroup $N_m$ is normal, since, as observed in
Definition~\ref{def:weight}, the subgroup generated by the elements
$g^{-1}\varphi(g)$ is normal for every automorphism $\varphi$.
For $g=(q,c^a)$ one has
$g^{-1}\theta^m(g)=\bigl(q^{-1}\alpha^m(q),\,c^{((-1)^m-1)a}\bigr)$. Taking
$g=(q,1)$ and $g=(1,c^a)$ shows that $N_m=A_m\times B_m$, where $A_m$ is
generated by the elements $q^{-1}\alpha^m(q)$ and $B_m$ by the elements
$c^{((-1)^m-1)a}$. If $3\nmid m$, then $\alpha^m\in\{\alpha,\alpha^2\}$, and
$q=i,j$ give two elements of $\{\pm i,\pm j,\pm k\}$ on different axes
(namely $-k,-i$ or $j,k$), which generate $Q_8$; if $3\mid m$, then
$A_m=1$. Moreover $B_m=C_7$ for odd $m$, because $2$ is invertible modulo
$7$, and $B_m=1$ for even $m$. Hence $N_m$ equals $\Gamma$, $Q_8$, $C_7$ or
$1$ for $m\equiv\pm1$, $\pm2$, $3$ or $0$ modulo six, which gives
\eqref{per:eq:towergroups}. Since $\pi_1$ is finite, $b_1=b_3=0$, and
$\chi=2$ forces $b_2=0$. The identification $H_2\cong H_1$ is explained
below. For $m\equiv\pm1\pmod6$, the manifold $\Sigma^{(m)}_\eps$ is simply
connected with $\chi=2$, so the argument of Section~\ref{sec:construction}
shows that it is a homotopy four-sphere.
\end{proof}

The displayed isomorphism between $H_2$ and $H_1$ is understood as an
isomorphism of abstract finite abelian groups. More precisely, writing
$X=\Sigma^{(m)}_\eps$, integral Poincar\'e duality and the universal
coefficient theorem give
\[
 H_2(X;\ZZ)\cong H^2(X;\ZZ)
 \cong\Ext^1_{\ZZ}(H_1(X;\ZZ),\ZZ)
 \cong\Hom(H_1(X;\ZZ),\QQ/\ZZ).
\]
Here $H_2(X;\ZZ)$ is finite because $b_2(X)=0$, so the
$\Hom(H_2(X;\ZZ),\ZZ)$ term in the universal coefficient theorem
vanishes. Every finite abelian group is abstractly isomorphic to its character
group. This gives the stated isomorphism $H_2\cong H_1$, but the
identification involves a choice. The six-periodicity in
Theorem~\ref{per:thm:tower} concerns the group and homology data.
Smooth identifications of covering manifolds, pairs or deck actions
require separate geometric arguments.

\noindent
For $m\equiv\pm1\pmod6$, Corollary~20.2 of \cite[p.~376]{QextC}
identifies the covering surgery pair, up to orientation, with one of the two
framing-labelled candidates. Hence Theorem~\ref{std:thm:main} gives
$\Sigma^{(m)}_\eps\diff S^4$.

This identification can also be read off from
Lemma~\ref{lem:weight-loops}. In the notation of \eqref{nf:eq:mm} below,
$(X_{\eps,m},R_{\eps,m})$ is the surgery pair of
$(M_m,\gamma_m,\lambda_{\eps,m})$. If $m\equiv1\pmod6$, then
$\bar f^{\,m}=\bar f$, so $M_m$ is the mapping torus $\M$ of
\eqref{eq:mappingtorus}, and $\gamma_m$ is an embedded circle of degree one.
If $m\equiv5\pmod6$, then $\bar f^{\,m}=\bar f^{-1}$, and
$[x,s]\mapsto[x,1-s]$ is a diffeomorphism $M_m\to\M$ that carries
$\gamma_m$ to an embedded circle of degree $-1$. In both cases
Lemma~\ref{lem:weight-loops} identifies the surgery pair, up to the
orientation of the belt sphere, with $(\Sigma_+(7),K_+)$ or with
$(\Sigma_-(7),K_-)$. The label is determined in Remark~\ref{std:rem:label-preserved}.

\subsection{Cover notation}
Write
\begin{equation}\label{nf:eq:cover-notation}
 X_{\eps,m}=\Sigma_\eps^{(m)},\;\;\;
 R_{\eps,m}=K_\eps^{(m)}\subset X_{\eps,m}
 \;\;\;(m\geq1)
\end{equation}
for the cyclic branched cover and its ramification sphere. Thus
$(X_{\eps,1},R_{\eps,1})=(\Sigma_\eps(7),K_\eps)$.
For the discussion of spin structures, an \emph{orientation-preserving diffeomorphism
of pairs} preserves both the ambient orientation and the sphere
orientation, hence also the orientation of the normal disc. This is a stricter
condition than equivalence in Definition~\ref{def:equivalence}, which
permits reversal of the ambient orientation; the equivalence relation
of that definition is unchanged.

The degree-$m$ unbranched cover of $\M$, with its base rescaled to length
one, is
\begin{equation}\label{nf:eq:mm}
 M_m=\N\times[0,1]/(x,1)\sim(\bar f^m(x),0),\;\;\;
 D_m(v,s)=(f^{-m}(v),s+1).
\end{equation}
The inverse image $\gamma_m$ of $\gamma$ is a single circle, since
$\gamma$ has degree one. Its class is $t_m=D_m$, corresponding to $t^m$
in $G_7$. Let $\lambda_{\eps,m}$ be the lifted framing. Then
\begin{equation}\label{nf:eq:lifted-surgery}
 (X_{\eps,m},R_{\eps,m})
 =\text{the surgery pair of }(M_m,\gamma_m,\lambda_{\eps,m}).
\end{equation}
This also follows directly from the local model for branching
$(x,z)\mapsto(x,z^m)$ on $S^2\times D^2$.
The difference of the two lifted framing classes is
\[
 [\lambda_{-,m}]-[\lambda_{+,m}]
       =m\pmod2\;\;\text{in }\pi_1(SO(3))\cong\ZZ/2.
\]
Indeed, the difference loop $A$ of Lemma~\ref{lem:frame} is pulled back
along the degree-$m$ map $\gamma_m\to\gamma$. Consequently, for even
$m$ the two lifted framings are homotopic and the surgery pairs
$(X_{+,m},R_{+,m})$ and $(X_{-,m},R_{-,m})$ are related by an
orientation-preserving diffeomorphism fixed away from a collar of the common
covering exterior (the map constructed below is the identity on the common covering
exterior and near $S^2\times\{0\}$; it differs from the identity only in
a collar of $\partial(S^2\times D^2)$ inside the cap). Explicitly, write $\phi_{\eps,m}$ for the lifted attaching maps. The
lift of the boundary twist \eqref{eq:gluck} through the branched covering
$(x,z)\mapsto(x,z^m)$ is
$\tau_m(x,e^{2\pi is})=(A(ms)x,e^{2\pi is})$, and
$\phi_{-,m}=\phi_{+,m}\circ\tau_m$. For even $m$ the loop $s\mapsto A(ms)$ is
null-homotopic in $SO(3)$, so it extends to a smooth map
$\widetilde A\colon D^2\to SO(3)$ that equals the identity near $0$.
We also choose $\widetilde A(\rho e^{2\pi is})=A(ms)$ for $\rho$ near $1$,
so that the map below is independent of the radial coordinate near
$\partial D^2$ and is therefore compatible with the product collars used
in the gluing.
The diffeomorphism $(x,w)\mapsto(\widetilde A(w)x,w)$ of $S^2\times D^2$
restricts to $\tau_m$ on the boundary and to the identity on
$S^2\times\{0\}$; together with the identity on the common covering exterior
it is an orientation-preserving diffeomorphism
$(X_{-,m},R_{-,m})\to(X_{+,m},R_{+,m})$ that preserves the orientation of the
sphere. No equivariance with respect to the branched-cover
deck actions is asserted. For odd $m$ the two framing classes remain
distinct, which alone does not imply inequivalence of their surgery
pairs.

\subsection{Framed circles and spin signs}
On an oriented spin four-manifold $(W,s)$, let $(C,\lambda)$ be an
oriented embedded circle with an oriented normal framing. After
orthonormalization, the tangent vector and the three normal vectors give a
loop of oriented frames. Define
\[
 \spn_s(C,\lambda)\in\{+1,-1\}
\]
to be $+1$ when that loop lifts to a closed loop of spin frames, and
$-1$ otherwise. This sign records whether the lift closes; it is not the
holonomy of a chosen connection. The elementary transformation rules are
\begin{equation}\label{nf:eq:spin-rules}
 \begin{split}
 \spn_{s+a}(C,\lambda)&=(-1)^{\langle a,[C]\rangle}
                              \spn_s(C,\lambda),\\
 \spn_s(F(C),F_*\lambda)&=\spn_{F^*s}(C,\lambda)
 \end{split}
\end{equation}
for $a\in H^1(W;\ZZ/2)$ and an orientation-preserving diffeomorphism $F$.
These are the usual obstruction-theoretic transformation rules for spin
structures; compare \cite[Sections~5.6--5.7]{GS}. Changing the
normal-framing class changes the sign.

\subsection{Explicit spin structures on the covering mapping tori}
Use the double covering
\[
 \Spin(4)=S^3_L\times S^3_R\to SO(4),\;\;\;
 (A,B):x\longmapsto AxB^{-1}.
\]
Write $\kappa=(-1,-1)$ for the nontrivial kernel element. The element
$(-1,1)$ projects to $-I_4$, which represents the derivative of
$\gamma_{-1,0}$ in the trivialization \eqref{nf:eq:trivialization} below;
it is the lift $\widetilde\gamma_{-1,0}$ in \eqref{nf:eq:deck-spin}.

On $S^3\times\RR$, trivialize the tangent bundle by
\begin{equation}\label{nf:eq:trivialization}
 J_{(v,s)}(w,a)=w+av\in\HH,
 \;\;\; w\in T_vS^3,
\end{equation}
using the orientation for which the base direction precedes the fibre
frame. Reversing the ambient orientation convention consistently does
not change any sign comparison below. In this trivialization, lifts of
the derivatives of the deck generators are
\begin{equation}\label{nf:eq:deck-spin}
 \widetilde\gamma_{q,a}=(q,\zeta^a),\;\;\;
 d_m=(\omega^m,j^m).
\end{equation}
They satisfy the deck relations exactly: $(q,a)\mapsto(q,\zeta^a)$ is a
homomorphism $\Gamma\to S^3_L\times S^3_R$, because $Q_8$ and $C_7$ are
placed in different factors, and
\[
 d_m\,\widetilde\gamma_{q,a}\,d_m^{-1}
 =(\omega^mq\omega^{-m},\,j^m\zeta^aj^{-m})
 =\widetilde\gamma_{\alpha^{-m}(q),\,(-1)^ma},
\]
which is the lift of
$D_m\gamma_{q,a}D_m^{-1}=\gamma_{\alpha^{-m}(q),\,(-1)^ma}$, the relation
computed from \eqref{nf:eq:mm} and \eqref{per:eq:torus}. Since the deck
group $\Gamma\rtimes\langle D_m\rangle$ of $M_m$ has no further
relations, the assignment extends to a homomorphism from the deck group
to $\Spin(4)$ lifting the derivative representation. This homomorphism
therefore specifies a spin structure $s_m$ on $M_m$. Explicitly, if
$\widetilde\rho_m:\Gamma\rtimes\langle D_m\rangle\to\Spin(4)$ denotes
this lifted derivative representation and $\rho_m$ its projection to
$SO(4)$, then in the trivialization \eqref{nf:eq:trivialization} the
oriented orthonormal frame bundle of $M_m$, for the descended product
metric, is the quotient of $(S^3\times\RR)\times SO(4)$ by
$
 \delta\cdot(x,A)=\bigl(\delta x,\rho_m(\delta)A\bigr).
$
Replacing $SO(4)$ by $\Spin(4)$ and $\rho_m$ by $\widetilde\rho_m$
gives a principal $\Spin(4)$-bundle mapping two-to-one onto this oriented
orthonormal frame bundle, hence a spin structure $s_m$ on $M_m$.
The structures $s_m$ are pullbacks of $s_1$: the covering map is induced
by $(v,s)\mapsto(v,ms)$, and restriction of the lifted deck
representation sends $D^m$ to $d_1^m=d_m$ while leaving
$\widetilde\gamma_{q,a}$ unchanged. The positive rescaling of the base
tangent direction is homotopic to the identity through positive
rescalings and therefore does not change the spin-structure comparison.
Let
$
 a_m\in H^1(M_m;\ZZ/2)
$
be the mod-two reduction of the degree class determined by the positive
orientation of the base.

\noindent\emph{Six-step comparison of the calibrated spin structures.}
Since $\bar f^6=\id$, the displayed mapping-torus models give a natural
identification $M_{m+6}=M_m$ over the base circle. In their common
universal cover, however,
$
 D_{m+6}=\gamma_{-1,0}D_m.
$
Consequently the section markings satisfy
$t_{m+6}=(-1)t_m$ under this identification. With the spin lifts of
\eqref{nf:eq:deck-spin}, one has
$
 d_{m+6}=(\omega^m,-j^m)
   =\kappa\,\widetilde\gamma_{-1,0}\,d_m.
$
The lifts agree on $\Gamma$, while their values on the positive base
generator differ by $\kappa$. Hence, after identifying the two
mapping tori and their degree classes,
$
 s_{m+6}=s_m+a_m.
$
Thus the natural six-step identification shifts the calibrated spin
structure by the base character. The formula for the closing sign in
Lemma~\ref{nf:lem:calibration} compares each structure with its own
lifted section and framing, so it is compatible with this shift.

\begin{lemma}[The calibrated framing signs]\label{nf:lem:calibration}
Set $\eta_+=0$ and $\eta_-=1$. Then
\[
 \spn_{s_m}(\gamma_m,\lambda_{\eps,m})
       =(-1)^{m\eta_\eps}.
\]
\end{lemma}
\begin{proof}
Shearing the tangent vector of $\gamma$ to the base direction does not
change the closing sign: the vertical vectors $V_a(s)$, together with any
vector of positive base component, form a positively oriented frame, and
$\widetilde c$ is stationary near the seam, so the shear is a homotopy of
closed loops of oriented frames. Put $u_+:=u$, and let $u_-$ be as in
Lemma~\ref{lem:frame}. At $(\widetilde c(s),s)$ the trivialization
\eqref{nf:eq:trivialization} sends the base direction to $\widetilde c(s)$
and $V_a(s)$ to $\widetilde c(s)u_\eps(s)\,a\,u_\eps(s)^{-1}$, so
the full frame is the image of the basis $(1,i,j,k)$ under
$x\mapsto\widetilde c(s)u_\eps(s)\,x\,u_\eps(s)^{-1}$. It
therefore lifts to the path
\[
 (\widetilde c(s)u_\eps(s),u_\eps(s))\in\Spin(4).
\]
For the $+$ frame this path starts at $(f(1),1)$ and ends at $(j,j)$.
Since $\omega f(1)=j$, the endpoint is
\[
 (j,j)=d_1(f(1),1).
\]
Thus, the lift closes in the spin bundle of $M_1$, giving the sign $+1$.
For the $-$ frame the endpoint is $(-j,-j)=\kappa(j,j)$, giving the sign
$-1$. The lifted circle in the degree-$m$ cover traverses the original
framed loop $m$ times. Pulling back its two-sheeted spin-frame cover
raises the closing sign to the $m$th power, proving the formula.
\end{proof}

\noindent\emph{Consistency with the six-step comparison.}
Write $c_+=1$ and $c_-=\kappa$. The proof above shows that one traversal
of the framed loop lifts to a path from $(f(1),1)$ to
$c_\eps d_1(f(1),1)$. Since $c_\eps$ is central, $n$ consecutive
traversals end at $c_\eps^{\,n}d_n(f(1),1)$. Under the identification
$M_{m+6}=M_m$, the circle $\gamma_{m+6}$ is a section representing
$D_{m+6}=\gamma_{-1,0}D_m$, and $s_m$ lifts this deck element to
$\widetilde\gamma_{-1,0}d_m$. Its frame lift ends at
$c_\eps^{\,m+6}d_{m+6}(f(1),1)
=c_\eps^{\,m}\kappa\,\widetilde\gamma_{-1,0}d_m(f(1),1)$, so
$\spn_{s_m}(\gamma_{m+6},\lambda_{\eps,m+6})=-(-1)^{m\eta_\eps}$.
Since $s_{m+6}=s_m+a_m$ and $\langle a_m,[\gamma_{m+6}]\rangle=1$, the
rule \eqref{nf:eq:spin-rules} gives
$\spn_{s_{m+6}}(\gamma_{m+6},\lambda_{\eps,m+6})
=(-1)^{m\eta_\eps}=(-1)^{(m+6)\eta_\eps}$, as
Lemma~\ref{nf:lem:calibration} requires.

\begin{lemma}[Circle comparison and the spin-extension sign]
\label{lem:circle-spin}
Freely homotopic oriented embedded circles in a smooth oriented
four-manifold are ambiently isotopic. On a spin four-manifold $(W,s)$,
the spin structure extends across surgery on a framed circle $(C,\lambda)$
precisely when $\spn_s(C,\lambda)=-1$.
\end{lemma}
\begin{proof}
The circle-isotopy assertion is the standard codimension-three
observation used for sections of fibred two-knots in
\cite{Teragaito1989Addendum}. We give the general argument here, including
the relative transversality details needed in the earlier section-surgery
comparison.

Choose parametrizations of the two oriented circles and a smooth free
homotopy $H:S^1\times I\to W$ between them that is stationary near the ends.
Relative parametric transversality \cite[Chapter~3]{Hirsch} permits a
perturbation, fixed near the ends, such that $d(H_t)_z\neq0$ for every
$(z,t)$: the zero-derivative condition has codimension four in a
two-dimensional parameter space. The same perturbation can make the map
$
 (z,z',t)\longmapsto\bigl(H_t(z),H_t(z')\bigr),\;\;\; z\neq z',
$
transverse to the diagonal in $W\times W$. That diagonal has codimension
four and this parameter space has dimension three, so its preimage is
empty. Near $z=z'$, uniform local injectivity follows from the immersion
condition and compactness of $S^1\times I$; pairs away from the diagonal
are covered by the preceding multijet-transversality argument.
Each $H_t$ is therefore an injective immersion of the compact
circle, hence an embedding. The resulting smooth isotopy extends to an
ambient isotopy \cite[Chapter~8]{Hirsch}.

The trace of surgery on $(C,\lambda)$ is obtained by attaching a five-dimensional
two-handle $D^2\times D^3$ to $W\times I$. The normal framing allows us to
express the obstruction to extending the spin structure over that handle
in terms of
the spin structure induced on its attaching circle
\cite[Sections~5.6--5.7]{GS}. Extension holds precisely when this circle
has the bounding spin structure. For an oriented disc, the boundary tangent vector together with the inward
radial vector makes one full rotation in the disc plane. Stabilizing by
constant transverse vectors still gives a spin lift that does not close. Thus
the bounding structure has the closing sign
$-1$, proving the criterion. Equivalently, the spin structure on the
unchanged exterior extends over the surgery cap exactly under this
condition.
\end{proof}

The circle-isotopy assertion of Lemma~\ref{lem:circle-spin} concerns
unframed oriented embeddings. An ambient isotopy transports a chosen
normal framing. When comparing it with an independently specified
framing of the terminal circle, their difference in
$\pi_1(SO(3))\cong\ZZ/2$ must also be recorded: a nonzero difference
reverses the closing sign. In particular, free homotopy of the
underlying circles suffices to compare their framed spin signs only
after this framing comparison has been made.

\noindent\emph{Spin extension throughout the cyclic tower.}
For every $m\geq1$ and either framing, put
\[
 \widehat s_{m,\eps}:=s_m+(1+m\eta_\eps)a_m,
 \;\;\;\text{with coefficients in }\ZZ/2.
\]
Since $\langle a_m,[\gamma_m]\rangle=1$, the transformation rule
\eqref{nf:eq:spin-rules} and Lemma~\ref{nf:lem:calibration} give
\[
 \spn_{\widehat s_{m,\eps}}(\gamma_m,\lambda_{\eps,m})
   =(-1)^{1+m\eta_\eps}(-1)^{m\eta_\eps}=-1.
\]
Lemma~\ref{lem:circle-spin} therefore shows that every $X_{\eps,m}$
is spin. The torsor of spin structures has one element when $3\nmid m$ and
four elements when $3\mid m$, because spin structures form a torsor over
$H^1(X_{\eps,m};\ZZ/2)$ and the latter group is determined by
Theorem~\ref{per:thm:tower} \cite[Section~5.6]{GS}.
The structure $\widehat s_{m,\eps}$ is one extending structure, but it need
not be the only one. By \eqref{nf:eq:spin-rules} and
Lemma~\ref{lem:circle-spin}, a spin structure $\widehat s_{m,\eps}+b$
on $M_m$, with $b\in H^1(M_m;\ZZ/2)$, extends across the surgery exactly
when $\langle b,[\gamma_m]\rangle=0$; restriction to the common exterior
identifies these extending structures bijectively with the spin structures
on $X_{\eps,m}$.
Here restriction from $M_m$ to the exterior
$E^{(m)}=M_m\setminus\Int\nu\gamma_m$ is a bijection on spin structures,
since $H^k(M_m,E^{(m)};\ZZ/2)\cong H_{4-k}(S^1;\ZZ/2)=0$ for $k=1,2$ by
excision and Lefschetz duality; and a spin structure on $E^{(m)}$ has at
most one extension over the cap $S^2\times D^2$, since
$H^1(S^2\times D^2,S^2\times S^1;\ZZ/2)\cong H_3(S^2;\ZZ/2)=0$.
Since $H_1(M_m;\ZZ)$ is the direct sum of $\ZZ$ and
the coinvariants of $\theta^m$ on
$\Gamma^{\mathrm{ab}}\cong(\ZZ/2)^2\oplus\ZZ/7$, the class $a_m$ spans
$H^1(M_m;\ZZ/2)$ when $3\nmid m$, and $\widehat s_{m,\eps}$ is then the
only extending structure; when $3\mid m$ there are four. For $m=3$ one has
$1+3\eta_\eps\equiv1-\eta_\eps\pmod 2$, so
$\widehat s_{3,\eps}=s_3+(1-\eta_\eps)a_3$. This is the particular
extending structure fixed in \eqref{quat:eq:extending-spin}; by
Corollary~\ref{quat:cor:bordism-bit}, the spin-bordism class of
$X_{\eps,3}$ over $BQ_8$ does not depend on which spin structure is
used.

\section{The orbit orbifold is a Whitehead surgery orbifold}
\label{wh:sec:whitehead}
This section explicitly identifies the orbit orbifold $X=\M/S^1$ of the
canonical circle action and the position of the surgery circle $\gamma$.

\begin{itemize}[leftmargin=1.6em]
\item For the mirror image fixed below, the orbit orbifold $X=S^3/\Lambda$
      is obtained by $+3$ surgery on one component of the Whitehead link, with
      cone angle $\pi$ along the other component
      (Theorem~\ref{wh:thm:whitehead}).
\item The unit section over any band sum of a meridian of the surgery curve
      with a meridian of the singular circle is a weight circle, and hence a
      representative of $\gamma$ in the sense of Lemma~\ref{lem:weight-loops}
      (Theorem~\ref{wh:thm:section}).
\end{itemize}

Let $X=\M/S^1=S^3/\Lam$ be the orbit orbifold of
Theorem~\ref{per:thm:action}(iv), with $\Lam=\langle\Gamma,f\rangle$ of
order $336$, and let $\mathcal S\subset|X|$ be the singular circle of isotropy order
two, namely the image of the exceptional torus. Here $|X|$ denotes the
underlying space.
\subsection{Product form of the groups}
\noindent\emph{Conventions.}
A pair $(A,B)$ of unit quaternions denotes the isometry $v\mapsto AvB$ of
$S^3$. Thus
\[
 (A,B)\circ(C,D)=(AC,DB),\;\;\;
 (C,D)(A,B)(C,D)^{-1}=(CAC^{-1},\,D^{-1}BD),
\]
and $(A,B)$ and $(-A,-B)$ represent the same isometry. Let $L$ and $R$ be finite
subgroups of $S^3$ containing $-1$. We write $L\pmx R$ for the group of all
isometries $(A,B)$ with $A\in L$ and $B\in R$; its order is $|L||R|/2$. This differs from the convention $(A,B)\colon x\mapsto AxB^{-1}$ used in
Section~\ref{nf:sec:spin} and in \cite{MecchiaSeppi2015,MecchiaSchilling2025}:
the pair $(A,B)$ here is the pair $(A,B^{-1})$ there. Since $R$ is a group,
$L\pmx R$ is the same set of isometries in both conventions, so the
identifications with the families of \cite[Table~1]{MecchiaSchilling2025}
below are unaffected.

Let $\Tst=\langle i,j,\omega\rangle$ be the binary tetrahedral group and
$\Qtt=\langle\zeta,j\rangle$ the binary dihedral group of order $28$. Let
$C_{14}=\langle-\zeta\rangle$ be its cyclic subgroup of index two. In this
notation $\gamma_{q,a}=(q,\zeta^{-a})$ and $f=(\omega^{-1},j)$.

\begin{proposition}[Product form]\label{wh:prop:product}
\begin{enumerate}[label=(\roman*),leftmargin=2.2em]
\item The groups in the construction are full products:
      \[
       \Gamma=Q_8\pmx C_{14},\;\;\;
       \Lambda=\Tst\pmx\Qtt.
      \]
      Their orders are $56$ and $336$, respectively.
\item $[\Lambda,\Lambda]=\Gamma$ and $\Lambda^{\mathrm{ab}}=\Lambda/\Gamma$
      is cyclic of order six, generated by the image of $f$.
\item The index-two subgroup $\Lambda^+:=\Tst\pmx C_{14}$ acts freely on
      $S^3$, and $\Lambda^+\cong\Tst\times C_7$.
\item The nontrivial elements of $\Lambda$ with a fixed point are the $42$
      involutions $(A,B)$ with $A\in\{\pm i,\pm j,\pm k\}$ and
      $B\in\{\pm\zeta^mj\}$; they lie in the coset $f^3\Gamma$, and none
      of them lies in $\Lambda^+$.
\end{enumerate}
\end{proposition}

\begin{proof}
(i) The generators $\gamma_{q,a}=(q,\zeta^{-a})$ and $f=(\omega^{-1},j)$ lie
in $\Tst\pmx\Qtt$. This group has order $24\cdot28/2=336=|\Lambda|$
(Theorem~\ref{per:thm:action}(iv)), so the inclusion is an equality. In the
same way, $\Gamma\subseteq Q_8\pmx C_{14}$, and both groups have order $56$.
Since $\omega^3=1$, we have $f^3=(1,j^3)=(1,-j)$.

(ii) The group $\Gamma$ is normal and $\Lambda/\Gamma$ is cyclic of order
six, generated by $f$ (Theorem~\ref{per:thm:periodic}(ii)). Hence
$[\Lambda,\Lambda]\subseteq\Gamma$. Conversely, $[\Tst,\Tst]=Q_8$, and
$(1,\zeta)$ and $(1,j)$ have commutator $(1,\zeta^{\pm2})$. So
$[\Lambda,\Lambda]$ contains $Q_8\times1$ and $1\times C_7$, which together
generate $\Gamma$; note that $(1,-1)=(-1,1)$.

(iii) By Lemma~\ref{per:lem:rigidity}, $(A,B)$ has a fixed point exactly
when $\Ree A=\Ree B$. For $A\in\Tst$ we have $\Ree A\in\{0,\pm\frac12,\pm1\}$.
For $B\in C_{14}$ we have $\Ree B\in\{\pm\cos(\pi m/7)\}$. These sets meet
only in $\pm1$, which gives the identity. The map
$\Tst\times C_7\to\Lambda^+$ given by $(A,c)\mapsto(A,c)$ is injective because
$C_7$ has odd order. The two groups have the same order.

(iv) For $B\in\Qtt\setminus C_{14}=\{\pm\zeta^mj\}$ we have $\Ree B=0$, so
$(A,B)$ has a fixed point exactly when $A$ is one of the six elements of
order four in $\Tst$. Then $(A,B)^2=(-1,-1)$ is the identity. This gives
$6\cdot14/2=42$ elements; they are the elements $\gamma_{q,a}f^3$ identified in the
proof of Proposition~\ref{per:prop:fixed}, so they lie in $f^3\Gamma$.

\end{proof}

\noindent
In the notation of Du Val \cite{DuVal1964} used in
\cite{MecchiaSeppi2015,MecchiaSeppi2019,MecchiaSchilling2025}, a full product
$L\pmx R$ is written $(L/L,R/R)$. Thus, $\Gamma=(D^*_8/D^*_8,C_{14}/C_{14})$
is a prism-manifold group and $\Lambda^+=(\Tst/\Tst,C_{14}/C_{14})$ is a
tetrahedral-manifold group. Moreover, $\Lambda=(\Tst/\Tst,D^*_{28}/D^*_{28})$ and
$\Lambda_0:=\langle\Gamma,f^3\rangle=(D^*_8/D^*_8,D^*_{28}/D^*_{28})$ (see
Lemma~\ref{wh:lem:underlying}(i) below). These groups belong to Families~2, 5, 14
and~10 of \cite[Table~1]{MecchiaSchilling2025}, respectively, up to interchanging the
two factors. This interchange corresponds to the ``bis'' convention of \cite{MecchiaSchilling2025}
(see also \cite[Section~3.2]{MecchiaSeppi2019}) and is induced by an
orientation-reversing isometry of $S^3$.

\begin{lemma}[The quotient model of $\M$]\label{wh:lem:model}
Let $\psi:\Lambda\to\ZZ/6\subset\RR/6\ZZ$ be the homomorphism with kernel
$\Gamma$ and $\psi(f)=-1$. Then
\[
 \M\diff(S^3\times\RR/6\ZZ)\big/\Lambda,\;\;\;
 \lambda\cdot(v,s)=(\lambda v,\,s+\psi(\lambda)),
\]
and the action is free. The canonical action of
Theorem~\ref{per:thm:action} is translation of the second factor, and its
orbit map is the projection to $S^3/\Lambda=X$.
\end{lemma}

\begin{proof}
In \eqref{per:eq:torus}, the map $T:=\gamma_{-1,0}D^6$ is $(v,s)\mapsto(v,s+6)$,
because $f^{-6}=-\id$, and it is central. An element of $G_7$ with base
degree $k$ acts on $S^3$ by an element of $f^{-k}\Gamma$. Such an element is
trivial on $S^3$ only when $6\mid k$, so the kernel of $G_7\to\Isom(S^3)$ is
$\langle T\rangle$.
Indeed, for $k=6l$ the element $\gamma_{q,a}D^{6l}$ acts by
$\gamma_{q,a}f^{-6l}=\gamma_{q,a}\gamma_{(-1)^l,0}$, which is trivial only
when $(q,a)=((-1)^l,0)$; the element is then $(\gamma_{-1,0}D^6)^l=T^l$. The image of this map is $\Lambda$, and the translation
part modulo $6$ is the homomorphism $\psi$. Hence
$\M=\bigl((S^3\times\RR)/\langle T\rangle\bigr)\big/(G_7/\langle T\rangle)$
is the displayed quotient. Freeness follows from
Proposition~\ref{wh:prop:product}(iv): an element with a fixed point on $S^3$
is an involution in $f^3\Gamma$, and it translates the circle factor by
$3\ne0$.
\end{proof}

\subsection{The underlying space and the branched covers of the singular circle}

\begin{lemma}[The underlying space is $\spLt$]\label{wh:lem:underlying}
\begin{enumerate}[label=(\roman*),leftmargin=2.2em]
\item The nontrivial elements of $\Lambda$ with a fixed point on $S^3$
      are exactly the $42$ involutions $\gamma_{q,a}f^3$ with
      $q\in\{\pm i,\pm j,\pm k\}$ and $a\in\ZZ/7$. The subgroup they
      generate is $\Lambda_0=\langle\Gamma,f^3\rangle=Q_8\pmx\Qtt$, of order
      $112$; it is normal, and $\Lambda/\Lambda_0\cong C_3$ is generated by
      the image of $f$.
\item $\pi_1(|X|)\cong C_3$, and $|X|$ is a closed orientable
      three-manifold homeomorphic, hence diffeomorphic, to $\spLt$.
\item $|S^3/\Lambda_0|$ is homeomorphic to $S^3$, and $|X|$ is its
      quotient by a free action of $C_3=\Lambda/\Lambda_0$.
\item The singular circle $\mathcal S$ is a tame knot in $|X|$.
\end{enumerate}
\end{lemma}

\begin{proof}
(i) Every element of $\Lambda$ has the form $\gamma_{q,a}f^n$ with
$0\le n\le5$, since $f^6=\gamma_{-1,0}$ by
Theorem~\ref{per:thm:periodic}(i). For $n=0$ only the identity has a fixed
point, because $\Gamma$ acts freely; for $n\in\{1,2,4,5\}$ no element
has one, because $\Fix(\bar f^{\,n})=\emptyset$ on $\N$
(Proposition~\ref{per:prop:fixed}); for $n=3$ the fixed-point criterion
of Lemma~\ref{per:lem:rigidity} gives exactly $q\in\{\pm i,\pm j,\pm k\}$,
as in the proof of Proposition~\ref{per:prop:fixed}. These elements have
order two by Proposition~\ref{wh:prop:product}(iv), where $(A,B)^2=(-1,-1)$ is
computed directly. For any two of these involutions,
\[
 (\gamma_{q,a}f^3)(\gamma_{q',a'}f^3)^{-1}
 =\gamma_{q,a}\gamma_{q',a'}^{-1}=\gamma_{qq'^{-1},\,a-a'}.
\]
Taking $q=q'$ gives $C_7$, and taking $(q,q')=(i,j)$, $(j,k)$ and $(k,i)$ gives
$-k$, $-i$ and $-j$, which generate $Q_8$; hence the subgroup generated by the
involutions contains $\Gamma$ and also $f^3=\gamma_{i,0}^{-1}(\gamma_{i,0}f^3)$. Conversely, every
$\gamma_{q,a}f^3$ lies in $\langle\Gamma,f^3\rangle$. This subgroup is
normal, since $f$ normalizes $\Gamma$ and commutes with $f^3$, and it
has order $2\cdot56$ because $f^3\notin\Gamma$ and $f^6\in\Gamma$
(Theorem~\ref{per:thm:periodic}(ii)). The quotient has order three and is
generated by $f$. In the notation of Proposition~\ref{wh:prop:product},
$f^3=(1,-j)$, so $\Lambda_0=Q_8\pmx\langle C_{14},j\rangle=Q_8\pmx\Qtt$.

(ii) Armstrong's theorem \cite{Armstrong1968} identifies $\pi_1(|X|)$
with the quotient of $\Lambda$ by the normal subgroup generated by the
elements having fixed points. By (i), this quotient is
$\Lambda/\Lambda_0\cong C_3$.
Locally, $|X|$ is a quotient of a ball by a rotation group of order at
most two, and $D^3/C_2\cong D^3$; hence $|X|$ is a closed topological
three-manifold. It is orientable because $\Lambda$ acts on $S^3$ by
two-sided unit-quaternion multiplications, which preserve orientation.
By the results of Moise and Munkres \cite{Moise1977,Munkres1960}, it has a smooth
structure, unique up to diffeomorphism. By elliptization \cite{PerelmanEntropy,PerelmanSurg,PerelmanExt}, a
closed three-manifold with
fundamental group $C_3$ is a spherical space form with cyclic fundamental
group of order three, that is, $L(3,1)$ or $L(3,2)$ \cite[Part~III]{Wolf2011}; these are diffeomorphic as
unoriented manifolds.

(iii) The same argument applied to $\Lambda_0$, whose elements with fixed
points are the same $42$ involutions and generate $\Lambda_0$ by (i),
shows that $|S^3/\Lambda_0|$ is a simply connected closed three-manifold,
hence $S^3$ by the three-dimensional Poincar\'e conjecture
\cite{PerelmanEntropy,PerelmanSurg,PerelmanExt}. If $g\in\Lambda\setminus\Lambda_0$ carried a
$\Lambda_0$-orbit to itself, say $g v=hv$ with $h\in\Lambda_0$, then
$h^{-1}g\in\Lambda\setminus\Lambda_0$ would have a fixed point, contrary
to (i). So $\Lambda/\Lambda_0$ acts freely.

(iv) In an orbifold chart $D^3/C_2$, the singular set is the image of the
rotation axis, which the homeomorphism
$(re^{i\theta},x)\mapsto(re^{2i\theta},x)$ onto $D^3$ carries to a
straight segment.
\end{proof}

Write $E_{\mathcal S}=|X|\setminus\Int\nu\mathcal S$ for the exterior of the
singular circle.

\begin{theorem}[Branched covers of the singular circle]
\label{wh:thm:branched}
\begin{enumerate}[label=(\roman*),leftmargin=2.2em]
\item $N':=S^3/\Lambda^+$ is a spherical space form with
      $\pi_1N'\cong\Tst\times C_7$ and $H_1(N')\cong\ZZ/21$. The deck
      involution $\iota$ of $N'\to X$ has a connected fixed set, which maps
      homeomorphically onto $\mathcal S$. Hence, $N'$ is the double cover of
      $|X|\cong\spLt$ branched along $\mathcal S$.
\item Let $\widetilde X=|S^3/\Lambda_0|$. This is $S^3$, and it is the
      universal cover of $|X|$ (Lemma~\ref{wh:lem:underlying}). The preimage
      of $\mathcal S$ in this cover is a three-component link $L_0$, whose components
      are permuted cyclically by the deck group $C_3=\Lambda/\Lambda_0$. The
      manifold $\N=S^3/\Gamma$ is the double cover of $S^3$ branched along
      $L_0$, with covering involution $\bar f^{\,3}$.
\item $\mathcal S$ is null-homotopic in $|X|$, and
      $H_1(E_{\mathcal S})\cong\ZZ\langle\mu\rangle\oplus\ZZ/3$. Here $\mu$ is
      a meridian, and the torsion summand maps isomorphically to
      $H_1(|X|)$.
\end{enumerate}
\end{theorem}

\begin{proof}
(i) The group $\Lambda^+$ acts freely by
Proposition~\ref{wh:prop:product}(iii). The abelianization of
$\Tst\times C_7$ is $C_3\times C_7$. We have $X=N'/(\Lambda/\Lambda^+)$, and
the nontrivial element acts by an involution $\iota$. The fixed set of
$\iota$ is the image of the fixed circles of the $42$ involutions. These
circles are pairwise disjoint: a common point would be fixed by a nontrivial
element of $\Lambda^+$. The $42$ involutions form one $\Lambda^+$-conjugacy class. Put $\xi:=-\zeta$;
this is a generator of $C_{14}$, of order $14$, with $\xi^7=-1$, and
$\Qtt\setminus C_{14}=\{\xi^mj:m\in\ZZ/14\}$. With the conventions of
Proposition~\ref{wh:prop:product}, the involutions are the elements
$(A,\xi^mj)$ with $A\in\{i,j,k\}$ and $m\in\ZZ/14$, since $(-A,B)=(A,-B)$ and
$-\xi^mj=\xi^{m+7}j$. Conjugation by $(\omega,1)$ permutes $i$, $j$ and $k$
cyclically; conjugation by $(1,\zeta)$ sends $(A,\xi^mj)$ to
$(A,\zeta^{-1}\xi^mj\zeta)=(A,\xi^{m-2}j)$, because $j\zeta=\zeta^{-1}j$ and
$\zeta^{-2}=\xi^{-2}$; and conjugation by $(j,1)$ sends $(i,\xi^mj)$ to
$(-i,\xi^mj)=(i,\xi^{m+7}j)$. The translations $m\mapsto m-2$ and
$m\mapsto m+7$ generate all translations of $\ZZ/14$. Hence, the fixed circles form one $\Lambda^+$-orbit, so $\Fix(\iota)$ is
the image in $N'$ of a single fixed circle $\Fix(g)$, where $g$ is one of
the $42$ involutions. If $h\in\Lambda^+$ satisfies $h\Fix(g)=\Fix(g)$,
then $hgh^{-1}$ fixes $\Fix(g)$ pointwise. Since the fixed circles are
pairwise disjoint, the only elements of $\Lambda$ fixing a point of
$\Fix(g)$ are $1$ and $g$; hence $hgh^{-1}=g$. Thus, the stabilizer of
$\Fix(g)$ in $\Lambda^+$ is the centralizer of $g$ in $\Lambda^+$, of
order $168/42=4$, and it acts freely on the circle $\Fix(g)$ because
$\Lambda^+$ acts freely on $S^3$. The quotient of a circle by a free
action of a finite group is a circle, so $\Fix(\iota)$ is an embedded
circle in $N'$, and the restriction of the quotient map $N'\to X$ to $\Fix(\iota)$
maps this circle bijectively onto the branch locus $\mathcal S$.

(ii) The components of the preimage of $\mathcal S$ in $\widetilde X$ are the
images of the $42$ circles modulo $\Lambda_0$. These correspond to the
$\Lambda_0$-conjugacy classes of the involutions. An element $(C,D)$ of
$\Lambda_0=Q_8\pmx\Qtt$ has $C\in Q_8$, and conjugation by $C$ preserves each
of the sets $\{\pm i\}$, $\{\pm j\}$ and $\{\pm k\}$; so the left axis of an
involution is a conjugacy-class invariant, and there are at least three classes. By the
computation in (i), applied with $(1,\zeta)$, with $(p,1)$ for an element
$p\in Q_8$ anticommuting with the left axis (so that
$(A,\xi^mj)\mapsto(A,-\xi^mj)=(A,\xi^{m+7}j)$), and with $(1,j)$ (which
sends $(A,\xi^mj)$ to $(A,\xi^{-m}j)$), the
$14$ involutions with a given left axis are conjugate in $\Lambda_0$. So
there are exactly three classes, of $14$ elements each, and $L_0$ has three
components. Each component is an embedded circle. Indeed, let $g$ be one of the
involutions. If $h\in\Lambda_0$ preserves $\Fix(g)$, then $hgh^{-1}$ fixes
$\Fix(g)$ pointwise, so $hgh^{-1}=g$ as in~(i). By
Lemma~\ref{wh:lem:underlying}(i) and the disjointness of the $42$ fixed
circles, the only elements of $\Lambda_0$ fixing a point of $\Fix(g)$ are $1$
and $g$. Hence $\Stab_{\Lambda_0}(\Fix(g))/\langle g\rangle$ acts freely on the
circle $\Fix(g)$, and the image of $\Fix(g)$ in $\widetilde X$ is an embedded
circle, tame by the chart argument of Lemma~\ref{wh:lem:underlying}(iv).
The element $f=(\omega^{-1},j)$ conjugates the left axes
cyclically, so it permutes the three components cyclically; compare the
proof of Proposition~\ref{per:prop:fixed}. Finally, $\Lambda_0/\Gamma$ is
generated by $f^3$. This proves the statement about $\N$.

(iii) By (ii), the universal covering
$\widetilde X\to |X|$ has degree three and the full preimage of
$\mathcal S$ has exactly three components. Hence, each of those components
maps with degree one onto $\mathcal S$. Equivalently, a loop running once
along $\mathcal S$ has a closed lift to the universal cover, so the element of
$\pi_1(|X|)$ that it represents is trivial. Thus, $\mathcal S$ is null-homotopic. For a null-homologous knot in a rational homology
sphere, the Mayer--Vietoris sequence gives
$H_1(E_{\mathcal S})\cong\ZZ\langle\mu\rangle\oplus H_1(|X|)$, where the projection
to $\ZZ\langle\mu\rangle$ is given by intersection with a Seifert surface for
$\mathcal S$; compare \cite[Section~2.3]{GordonLidman2014}.
\end{proof}

\noindent
By \cite[Theorem~1.1]{MecchiaSchilling2025}, every spherical three-manifold has
a unique hyperelliptic involution up to conjugacy, and the singular set of its
quotient is a Montesinos link with at most three tangles
\cite[Theorem~1.2]{MecchiaSchilling2025}; for the prism manifold $\N$ this
involution is $\bar f^{\,3}$, and $L_0$ is the corresponding three-tangle
Montesinos link, in agreement with Montesinos' description of prism
manifolds as double branched covers \cite{Montesinos1973}.

\begin{remark}[The Seifert structure of the orbit orbifold]\label{wh:rem:seifert}
In the stabilizer calculations below, we fix a Hopf circle in $S^3$ and
take its setwise stabilizer in the specified finite group. The isotropy
group at each point of $X$ is trivial or cyclic of order two.

The group
$\Lambda=\Tst\pmx\Qtt$ preserves the Hopf fibration $v\mapsto vU(1)$ of
$S^3$, because $\Qtt\subset U(1)\cup U(1)j$ normalizes $U(1)$. Under the Hopf
map $h(v)=vi\bar v$, an element $(A,B)$ acts on the base $S^2$ by
$x\mapsto AxA^{-1}$ if $B\in C_{14}$ and by $x\mapsto-AxA^{-1}$ if
$B\in\Qtt\setminus C_{14}$. The induced group is the pyritohedral group
$A_4\times\{\pm1\}$ of order $24$, and the kernel is
$C_{14}=\{(1,B):B\in C_{14}\}$. Hence, $X$ is a Seifert-fibred orbifold over
the two-orbifold $S^2/(A_4\times\{\pm1\})$, a disc with one cone point of
order three and a mirror boundary carrying one corner reflector of order two
(the orbifold $3{*}2$ in Conway's notation \cite{Conway1992}). The $42$
involutions act on
the base by reflections, so $\mathcal S$ lies over the mirror boundary:
over a mirror point the fibre of $X$ is an interval whose two endpoints lie
on $\mathcal S$. The fibre over the cone point of order three has stabilizer
$\langle-\omega\rangle\pmx C_{14}\cong C_{42}$, which acts freely on it, so
this fibre is exceptional, with multiplicity $42/14=3$. Traversing this
exceptional fibre once represents a generator of its stabilizer, an element
of order $42$ in $\Lambda=\stdpi(X)$; the cube of this element generates the
kernel $C_{14}$, that is, it is the class of a regular fibre up to
orientation. The fibre over
the corner reflector has stabilizer $\langle(i,\zeta),(j,j)\rangle$, which is dihedral
of order $56$, since $(j,j)$ inverts $(i,\zeta)$. The same computation for
$\Lambda_0=Q_8\pmx\Qtt$ gives a Seifert structure on $S^3/\Lambda_0$ over the
orbifold ${*}222$, a disc with three corner reflectors of order two, with
one singular circle over each of the three mirror edges; these circles form the
three-component link $L_0$ of Theorem~\ref{wh:thm:branched}(ii). The same invariants follow from the closed formulas in
\cite{MecchiaSeppi2015} and are tabulated in \cite{MecchiaSchilling2025}:
in \cite[Table~5]{MecchiaSchilling2025} the quotient of the
tetrahedral manifold $N'$ by the involution with extension $\Lambda$
(Family~14, $m=7$) is listed with base orbifold $D^2(3;2)$, local
invariants $7/2$ and $7/3$, and Euler class $-7/12$; in
\cite[Table~7]{MecchiaSchilling2025} the quotient of the prism manifold
$\N$ by the
involution with extension $\Lambda_0$ (Family~10, $m=7$) is listed with
base orbifold $D^2(;2,2,2)$, local invariants $7/2$, $7/2$ and $7/2$, and Euler
class $-7/4$, up to the orientation convention.
\end{remark}

Figure~\ref{wh:fig:quotients} summarizes the quotients of this subsection and
their Seifert base orbifolds.

\begin{figure}[tbp]
\centering
\begin{tikzpicture}[
  qb/.style={draw=diagramblue!80,rounded corners=3pt,fill=diagramlight,align=center,
    font=\footnotesize,inner sep=3pt,line width=0.55pt},
  freearr/.style={-{Stealth[length=2.1mm,width=1.5mm]},draw=diagramblue,line width=0.65pt},
  brarr/.style={-{Stealth[length=2.1mm,width=1.5mm]},draw=orbitred,line width=0.65pt,densely dashed},
  mirror/.style={draw=diagramblue,line width=1.4pt,line join=round},
  cone2/.style={circle,fill=#1,inner sep=0pt,minimum size=4.6pt},
  qarr/.style={-{Stealth[length=1.6mm,width=1.2mm]},gray!70!black,line width=0.5pt}]
\newcommand{\conethree}[2]{\fill[#2] ($#1+(90:0.085)$) -- ($#1+(210:0.085)$) -- ($#1+(330:0.085)$) -- cycle;}
\begin{scope}
\node[qb] (S) at (0,1.55) {$S^3$};
\node[qb] (N) at (0,0) {$\N=S^3/\Gamma$\\ $|\Gamma|=56$};
\node[qb] (Np) at (-1.95,-2.0) {$N'=S^3/\Lambda^+$\\ $|\Lambda^+|=168$};
\node[qb] (L) at (1.95,-2.0) {$S^3/\Lambda_0\cong(S^3,L_0)$\\ $|\Lambda_0|=112$};
\node[qb,fill=highlight!60] (X) at (0,-4.0) {$X=S^3/\Lambda\cong(\spLt,\mathcal S)$\\ $|\Lambda|=336$};
\draw[freearr] (S) -- node[right=1.5pt,slab]{$\Gamma$, free} (N);
\draw[freearr] (N) -- node[above left,slab,align=right]{$\langle\bar f^{\,2}\rangle\cong C_3$\\ free} (Np);
\draw[brarr] (N) -- node[above right,slab,align=left]{$C_2$; ramification set\\ $\mathcal C_1\sqcup\mathcal C_2\sqcup\mathcal C_3$} (L);
\draw[brarr] (Np) -- node[pos=0.32,below left,slab,align=right]{$C_2$; ramification\\ set $\Fix(\iota)$} (X);
\draw[freearr] (L) -- node[pos=0.32,below right,slab,align=left]{$C_3$, free;\\ permutes $L_0$} (X);
\draw[freearr] (-2.3,-5.05) -- ++(0.5,0) node[right,slab,text=black] {free quotient};
\draw[brarr] (0.35,-5.05) -- ++(0.5,0) node[right,slab,text=black] {branched, $C_2$};
\node[panelcap] at (0,-5.6) {(a) the quotients};
\end{scope}
\begin{scope}[shift={(5.75,0.75)}]
\foreach \name/\x/\y/\top/\grp/\sym in {
    base_gamma/0/0/{$\N$}/{$V_4$}/{$S^2(2,2,2)$},
    base_lplus/3.5/0/{$N'$}/{$A_4$}/{$S^2(2,3,3)$},
    base_lzero/0/-3.3/{$S^3/\Lambda_0$}/{$V_4\times\{\pm1\}$}/{${*}222$},
    base_lambda/3.5/-3.3/{$X$}/{$A_4\times\{\pm1\}$}/{$3{*}2$}} {
  \begin{scope}[shift={(\x,\y)}]
  \node[slab,anchor=south] at (0.55,0.9) {\top};
  \node[slab,text=inkgray] at (0,-1.02) {\grp};
  \node[slab] at (2.0,-1.02) {\sym};
  \draw[qarr] (0.92,0) -- (1.42,0);
  \end{scope}}
\begin{scope}[shift={(0,0)}]
\shade[ball color=diagramlight!55!white] (0,0) circle[radius=0.780];
\draw[diagramblue!45,line width=0.4pt] (0,0) circle[radius=0.780];
\draw[gray!55,densely dotted,line width=0.4pt] (-0.426,-0.272) -- (0.426,0.272);
\draw[gray!55,densely dotted,line width=0.4pt] (0.653,-0.178) -- (-0.653,0.178);
\draw[gray!55,densely dotted,line width=0.4pt] (-0.000,0.709) -- (0.000,-0.709);
\begin{scope}\clip (0,0) circle[radius=0.784];
\draw[diagramblue!45,line width=0.4pt] plot[smooth] coordinates {(-0.000,0.709) (-0.051,0.721) (-0.102,0.728) (-0.152,0.731) (-0.202,0.729) (-0.250,0.723) (-0.297,0.712) (-0.341,0.697) (-0.384,0.678) (-0.424,0.655)};
\draw[diagramblue!25,line width=0.35pt,densely dashed] plot[smooth] coordinates {(-0.424,0.655) (-0.462,0.627) (-0.497,0.596) (-0.528,0.560) (-0.557,0.522) (-0.582,0.480) (-0.603,0.435) (-0.621,0.388) (-0.635,0.338) (-0.645,0.286) (-0.651,0.233) (-0.653,0.178) (-0.651,0.121) (-0.645,0.065) (-0.635,0.007) (-0.621,-0.050) (-0.603,-0.107) (-0.582,-0.164) (-0.557,-0.219) (-0.528,-0.273) (-0.497,-0.325) (-0.462,-0.376) (-0.424,-0.424) (-0.384,-0.469) (-0.341,-0.512) (-0.297,-0.551) (-0.250,-0.587) (-0.202,-0.619) (-0.152,-0.648) (-0.102,-0.673) (-0.051,-0.693) (-0.000,-0.709) (0.051,-0.721) (0.102,-0.728) (0.152,-0.731) (0.202,-0.729) (0.250,-0.723) (0.297,-0.712) (0.341,-0.697) (0.384,-0.678) (0.424,-0.655)};
\draw[diagramblue!45,line width=0.4pt] plot[smooth] coordinates {(0.424,-0.655) (0.462,-0.627) (0.497,-0.596) (0.528,-0.560) (0.557,-0.522) (0.582,-0.480) (0.603,-0.435) (0.621,-0.388) (0.635,-0.338) (0.645,-0.286) (0.651,-0.233) (0.653,-0.178) (0.651,-0.121) (0.645,-0.065) (0.635,-0.007) (0.621,0.050) (0.603,0.107) (0.582,0.164) (0.557,0.219) (0.528,0.273) (0.497,0.325) (0.462,0.376) (0.424,0.424) (0.384,0.469) (0.341,0.512) (0.297,0.551) (0.250,0.587) (0.202,0.619) (0.152,0.648) (0.102,0.673) (0.051,0.693) (0.000,0.709)};
\end{scope}
\begin{scope}\clip (0,0) circle[radius=0.784];
\draw[diagramblue!25,line width=0.35pt,densely dashed] plot[smooth] coordinates {(0.000,-0.709) (0.033,-0.686) (0.067,-0.658) (0.100,-0.626) (0.132,-0.590) (0.163,-0.551) (0.194,-0.508) (0.223,-0.462) (0.251,-0.414) (0.277,-0.362) (0.301,-0.309) (0.324,-0.254) (0.345,-0.197) (0.363,-0.138) (0.380,-0.079) (0.394,-0.020) (0.405,0.040) (0.415,0.099) (0.421,0.158) (0.425,0.216) (0.426,0.272) (0.425,0.327) (0.421,0.380) (0.415,0.430) (0.405,0.478) (0.394,0.523) (0.380,0.564) (0.363,0.603) (0.345,0.637) (0.324,0.667) (0.301,0.694) (0.277,0.716) (0.251,0.734) (0.223,0.747) (0.194,0.755)};
\draw[diagramblue!45,line width=0.4pt] plot[smooth] coordinates {(0.194,0.755) (0.163,0.759) (0.132,0.758) (0.100,0.753) (0.067,0.743) (0.033,0.728) (0.000,0.709) (-0.033,0.686) (-0.067,0.658) (-0.100,0.626) (-0.132,0.590) (-0.163,0.551) (-0.194,0.508) (-0.223,0.462) (-0.251,0.414) (-0.277,0.362) (-0.301,0.309) (-0.324,0.254) (-0.345,0.197) (-0.363,0.138) (-0.380,0.079) (-0.394,0.020) (-0.405,-0.040) (-0.415,-0.099) (-0.421,-0.158) (-0.425,-0.216) (-0.426,-0.272) (-0.425,-0.327) (-0.421,-0.380) (-0.415,-0.430) (-0.405,-0.478) (-0.394,-0.523) (-0.380,-0.564) (-0.363,-0.603) (-0.345,-0.637) (-0.324,-0.667) (-0.301,-0.694) (-0.277,-0.716) (-0.251,-0.734) (-0.223,-0.747) (-0.194,-0.755)};
\draw[diagramblue!25,line width=0.35pt,densely dashed] plot[smooth] coordinates {(-0.194,-0.755) (-0.163,-0.759) (-0.132,-0.758) (-0.100,-0.753) (-0.067,-0.743) (-0.033,-0.728) (-0.000,-0.709)};
\end{scope}
\begin{scope}\clip (0,0) circle[radius=0.784];
\draw[diagramblue!45,line width=0.4pt] plot[smooth] coordinates {(0.653,-0.178) (0.685,-0.156) (0.712,-0.133) (0.735,-0.109) (0.753,-0.085) (0.760,-0.072) (0.772,-0.047) (0.778,-0.022) (0.780,0.004)};
\draw[diagramblue!25,line width=0.35pt,densely dashed] plot[smooth] coordinates {(0.780,0.004) (0.777,0.029) (0.769,0.054) (0.756,0.079) (0.739,0.104) (0.717,0.128) (0.691,0.151) (0.661,0.173) (0.626,0.194) (0.588,0.214) (0.546,0.232) (0.500,0.249) (0.452,0.265) (0.400,0.279) (0.347,0.291) (0.291,0.302) (0.233,0.310) (0.174,0.317) (0.114,0.322) (0.053,0.324) (-0.008,0.325) (-0.070,0.324) (-0.130,0.320) (-0.190,0.315) (-0.249,0.308) (-0.306,0.299) (-0.362,0.288) (-0.415,0.275) (-0.465,0.261) (-0.513,0.245) (-0.557,0.227) (-0.599,0.208) (-0.636,0.188) (-0.669,0.167) (-0.699,0.144) (-0.724,0.121) (-0.744,0.097) (-0.760,0.072) (-0.772,0.047) (-0.778,0.022) (-0.780,-0.004)};
\draw[diagramblue!45,line width=0.4pt] plot[smooth] coordinates {(-0.780,-0.004) (-0.777,-0.029) (-0.769,-0.054) (-0.756,-0.079) (-0.739,-0.104) (-0.717,-0.128) (-0.691,-0.151) (-0.661,-0.173) (-0.626,-0.194) (-0.588,-0.214) (-0.546,-0.232) (-0.500,-0.249) (-0.452,-0.265) (-0.400,-0.279) (-0.347,-0.291) (-0.291,-0.302) (-0.233,-0.310) (-0.144,-0.319) (-0.083,-0.323) (-0.022,-0.325) (0.039,-0.325) (0.100,-0.322) (0.160,-0.318) (0.220,-0.312) (0.278,-0.304) (0.334,-0.294) (0.388,-0.282) (0.440,-0.268) (0.489,-0.253) (0.536,-0.236) (0.578,-0.218) (0.618,-0.198) (0.653,-0.178)};
\end{scope}
\fill[orbitred] (-0.426,-0.272) circle[radius=0.058];
\fill[orbitred,opacity=0.38] (0.426,0.272) circle[radius=0.058];
\fill[liftblue] (0.653,-0.178) circle[radius=0.058];
\fill[liftblue,opacity=0.38] (-0.653,0.178) circle[radius=0.058];
\fill[dgreen] (-0.000,0.709) circle[radius=0.058];
\fill[dgreen,opacity=0.38] (0.000,-0.709) circle[radius=0.058];\end{scope}
\begin{scope}[shift={(3.5,0)}]
\shade[ball color=diagramlight!55!white] (0,0) circle[radius=0.780];
\draw[diagramblue!45,line width=0.4pt] (0,0) circle[radius=0.780];
\begin{scope}\clip (0,0) circle[radius=0.784];
\draw[diagramblue!45,line width=0.4pt] plot[smooth] coordinates {(-0.000,0.709) (-0.051,0.721) (-0.102,0.728) (-0.152,0.731) (-0.202,0.729) (-0.250,0.723) (-0.297,0.712) (-0.341,0.697) (-0.384,0.678) (-0.424,0.655)};
\draw[diagramblue!25,line width=0.35pt,densely dashed] plot[smooth] coordinates {(-0.424,0.655) (-0.462,0.627) (-0.497,0.596) (-0.528,0.560) (-0.557,0.522) (-0.582,0.480) (-0.603,0.435) (-0.621,0.388) (-0.635,0.338) (-0.645,0.286) (-0.651,0.233) (-0.653,0.178) (-0.651,0.121) (-0.645,0.065) (-0.635,0.007) (-0.621,-0.050) (-0.603,-0.107) (-0.582,-0.164) (-0.557,-0.219) (-0.528,-0.273) (-0.497,-0.325) (-0.462,-0.376) (-0.424,-0.424) (-0.384,-0.469) (-0.341,-0.512) (-0.297,-0.551) (-0.250,-0.587) (-0.202,-0.619) (-0.152,-0.648) (-0.102,-0.673) (-0.051,-0.693) (-0.000,-0.709) (0.051,-0.721) (0.102,-0.728) (0.152,-0.731) (0.202,-0.729) (0.250,-0.723) (0.297,-0.712) (0.341,-0.697) (0.384,-0.678) (0.424,-0.655)};
\draw[diagramblue!45,line width=0.4pt] plot[smooth] coordinates {(0.424,-0.655) (0.462,-0.627) (0.497,-0.596) (0.528,-0.560) (0.557,-0.522) (0.582,-0.480) (0.603,-0.435) (0.621,-0.388) (0.635,-0.338) (0.645,-0.286) (0.651,-0.233) (0.653,-0.178) (0.651,-0.121) (0.645,-0.065) (0.635,-0.007) (0.621,0.050) (0.603,0.107) (0.582,0.164) (0.557,0.219) (0.528,0.273) (0.497,0.325) (0.462,0.376) (0.424,0.424) (0.384,0.469) (0.341,0.512) (0.297,0.551) (0.250,0.587) (0.202,0.619) (0.152,0.648) (0.102,0.673) (0.051,0.693) (0.000,0.709)};
\end{scope}
\begin{scope}\clip (0,0) circle[radius=0.784];
\draw[diagramblue!25,line width=0.35pt,densely dashed] plot[smooth] coordinates {(0.000,-0.709) (0.033,-0.686) (0.067,-0.658) (0.100,-0.626) (0.132,-0.590) (0.163,-0.551) (0.194,-0.508) (0.223,-0.462) (0.251,-0.414) (0.277,-0.362) (0.301,-0.309) (0.324,-0.254) (0.345,-0.197) (0.363,-0.138) (0.380,-0.079) (0.394,-0.020) (0.405,0.040) (0.415,0.099) (0.421,0.158) (0.425,0.216) (0.426,0.272) (0.425,0.327) (0.421,0.380) (0.415,0.430) (0.405,0.478) (0.394,0.523) (0.380,0.564) (0.363,0.603) (0.345,0.637) (0.324,0.667) (0.301,0.694) (0.277,0.716) (0.251,0.734) (0.223,0.747) (0.194,0.755)};
\draw[diagramblue!45,line width=0.4pt] plot[smooth] coordinates {(0.194,0.755) (0.163,0.759) (0.132,0.758) (0.100,0.753) (0.067,0.743) (0.033,0.728) (0.000,0.709) (-0.033,0.686) (-0.067,0.658) (-0.100,0.626) (-0.132,0.590) (-0.163,0.551) (-0.194,0.508) (-0.223,0.462) (-0.251,0.414) (-0.277,0.362) (-0.301,0.309) (-0.324,0.254) (-0.345,0.197) (-0.363,0.138) (-0.380,0.079) (-0.394,0.020) (-0.405,-0.040) (-0.415,-0.099) (-0.421,-0.158) (-0.425,-0.216) (-0.426,-0.272) (-0.425,-0.327) (-0.421,-0.380) (-0.415,-0.430) (-0.405,-0.478) (-0.394,-0.523) (-0.380,-0.564) (-0.363,-0.603) (-0.345,-0.637) (-0.324,-0.667) (-0.301,-0.694) (-0.277,-0.716) (-0.251,-0.734) (-0.223,-0.747) (-0.194,-0.755)};
\draw[diagramblue!25,line width=0.35pt,densely dashed] plot[smooth] coordinates {(-0.194,-0.755) (-0.163,-0.759) (-0.132,-0.758) (-0.100,-0.753) (-0.067,-0.743) (-0.033,-0.728) (-0.000,-0.709)};
\end{scope}
\begin{scope}\clip (0,0) circle[radius=0.784];
\draw[diagramblue!45,line width=0.4pt] plot[smooth] coordinates {(0.653,-0.178) (0.685,-0.156) (0.712,-0.133) (0.735,-0.109) (0.753,-0.085) (0.760,-0.072) (0.772,-0.047) (0.778,-0.022) (0.780,0.004)};
\draw[diagramblue!25,line width=0.35pt,densely dashed] plot[smooth] coordinates {(0.780,0.004) (0.777,0.029) (0.769,0.054) (0.756,0.079) (0.739,0.104) (0.717,0.128) (0.691,0.151) (0.661,0.173) (0.626,0.194) (0.588,0.214) (0.546,0.232) (0.500,0.249) (0.452,0.265) (0.400,0.279) (0.347,0.291) (0.291,0.302) (0.233,0.310) (0.174,0.317) (0.114,0.322) (0.053,0.324) (-0.008,0.325) (-0.070,0.324) (-0.130,0.320) (-0.190,0.315) (-0.249,0.308) (-0.306,0.299) (-0.362,0.288) (-0.415,0.275) (-0.465,0.261) (-0.513,0.245) (-0.557,0.227) (-0.599,0.208) (-0.636,0.188) (-0.669,0.167) (-0.699,0.144) (-0.724,0.121) (-0.744,0.097) (-0.760,0.072) (-0.772,0.047) (-0.778,0.022) (-0.780,-0.004)};
\draw[diagramblue!45,line width=0.4pt] plot[smooth] coordinates {(-0.780,-0.004) (-0.777,-0.029) (-0.769,-0.054) (-0.756,-0.079) (-0.739,-0.104) (-0.717,-0.128) (-0.691,-0.151) (-0.661,-0.173) (-0.626,-0.194) (-0.588,-0.214) (-0.546,-0.232) (-0.500,-0.249) (-0.452,-0.265) (-0.400,-0.279) (-0.347,-0.291) (-0.291,-0.302) (-0.233,-0.310) (-0.144,-0.319) (-0.083,-0.323) (-0.022,-0.325) (0.039,-0.325) (0.100,-0.322) (0.160,-0.318) (0.220,-0.312) (0.278,-0.304) (0.334,-0.294) (0.388,-0.282) (0.440,-0.268) (0.489,-0.253) (0.536,-0.236) (0.578,-0.218) (0.618,-0.198) (0.653,-0.178)};
\end{scope}
\fill[orbitred] (-0.426,-0.272) circle[radius=0.058];
\fill[orbitred,opacity=0.38] (0.426,0.272) circle[radius=0.058];
\fill[orbitred] (0.653,-0.178) circle[radius=0.058];
\fill[orbitred,opacity=0.38] (-0.653,0.178) circle[radius=0.058];
\fill[orbitred] (-0.000,0.709) circle[radius=0.058];
\fill[orbitred,opacity=0.38] (0.000,-0.709) circle[radius=0.058];
\fill[liftblue] (0.131,0.235) -- (0.057,0.107) -- (0.205,0.107) -- cycle;
\fill[liftblue,opacity=0.38] (-0.623,-0.379) -- (-0.697,-0.506) -- (-0.550,-0.506) -- cycle;
\fill[liftblue,opacity=0.38] (0.623,-0.270) -- (0.550,-0.397) -- (0.697,-0.397) -- cycle;
\fill[liftblue,opacity=0.38] (-0.131,0.754) -- (-0.205,0.627) -- (-0.057,0.627) -- cycle;
\fill[dgreen,opacity=0.38] (-0.131,-0.065) -- (-0.205,-0.192) -- (-0.057,-0.192) -- cycle;
\fill[dgreen] (0.623,0.549) -- (0.550,0.421) -- (0.697,0.421) -- cycle;
\fill[dgreen] (-0.623,0.440) -- (-0.697,0.312) -- (-0.550,0.312) -- cycle;
\fill[dgreen] (0.131,-0.584) -- (0.057,-0.712) -- (0.205,-0.712) -- cycle;\end{scope}
\begin{scope}[shift={(0,-3.3)}]
\shade[ball color=diagramlight!55!white] (0,0) circle[radius=0.780];
\draw[diagramblue!45,line width=0.4pt] (0,0) circle[radius=0.780];
\begin{scope}\clip (0,0) circle[radius=0.784];
\draw[diagramblue!45,line width=0.8pt,densely dashed] plot[smooth] coordinates {(-0.424,0.655) (-0.462,0.627) (-0.497,0.596) (-0.528,0.560) (-0.557,0.522) (-0.582,0.480) (-0.603,0.435) (-0.621,0.388) (-0.635,0.338) (-0.645,0.286) (-0.651,0.233) (-0.653,0.178) (-0.651,0.121) (-0.645,0.065) (-0.635,0.007) (-0.621,-0.050) (-0.603,-0.107) (-0.582,-0.164) (-0.557,-0.219) (-0.528,-0.273) (-0.497,-0.325) (-0.462,-0.376) (-0.424,-0.424) (-0.384,-0.469) (-0.341,-0.512) (-0.297,-0.551) (-0.250,-0.587) (-0.202,-0.619) (-0.152,-0.648) (-0.102,-0.673) (-0.051,-0.693) (-0.000,-0.709) (0.051,-0.721) (0.102,-0.728) (0.152,-0.731) (0.202,-0.729) (0.250,-0.723) (0.297,-0.712) (0.341,-0.697) (0.384,-0.678) (0.424,-0.655)};
\end{scope}
\begin{scope}\clip (0,0) circle[radius=0.784];
\draw[diagramblue!45,line width=0.8pt,densely dashed] plot[smooth] coordinates {(0.000,-0.709) (0.033,-0.686) (0.067,-0.658) (0.100,-0.626) (0.132,-0.590) (0.163,-0.551) (0.194,-0.508) (0.223,-0.462) (0.251,-0.414) (0.277,-0.362) (0.301,-0.309) (0.324,-0.254) (0.345,-0.197) (0.363,-0.138) (0.380,-0.079) (0.394,-0.020) (0.405,0.040) (0.415,0.099) (0.421,0.158) (0.425,0.216) (0.426,0.272) (0.425,0.327) (0.421,0.380) (0.415,0.430) (0.405,0.478) (0.394,0.523) (0.380,0.564) (0.363,0.603) (0.345,0.637) (0.324,0.667) (0.301,0.694) (0.277,0.716) (0.251,0.734) (0.223,0.747) (0.194,0.755)};
\draw[diagramblue!45,line width=0.8pt,densely dashed] plot[smooth] coordinates {(-0.194,-0.755) (-0.163,-0.759) (-0.132,-0.758) (-0.100,-0.753) (-0.067,-0.743) (-0.033,-0.728) (-0.000,-0.709)};
\end{scope}
\begin{scope}\clip (0,0) circle[radius=0.784];
\draw[diagramblue!45,line width=0.8pt,densely dashed] plot[smooth] coordinates {(0.780,0.004) (0.777,0.029) (0.769,0.054) (0.756,0.079) (0.739,0.104) (0.717,0.128) (0.691,0.151) (0.661,0.173) (0.626,0.194) (0.588,0.214) (0.546,0.232) (0.500,0.249) (0.452,0.265) (0.400,0.279) (0.347,0.291) (0.291,0.302) (0.233,0.310) (0.174,0.317) (0.114,0.322) (0.053,0.324) (-0.008,0.325) (-0.070,0.324) (-0.130,0.320) (-0.190,0.315) (-0.249,0.308) (-0.306,0.299) (-0.362,0.288) (-0.415,0.275) (-0.465,0.261) (-0.513,0.245) (-0.557,0.227) (-0.599,0.208) (-0.636,0.188) (-0.669,0.167) (-0.699,0.144) (-0.724,0.121) (-0.744,0.097) (-0.760,0.072) (-0.772,0.047) (-0.778,0.022) (-0.780,-0.004)};
\end{scope}
\fill[orbitred,opacity=0.38] (0.426,0.272) circle[radius=0.058];
\fill[liftblue,opacity=0.38] (-0.653,0.178) circle[radius=0.058];
\fill[dgreen,opacity=0.38] (0.000,-0.709) circle[radius=0.058];
\fill[liftorange!30] plot coordinates {(-0.43,-0.27) (-0.36,-0.29) (-0.28,-0.30) (-0.20,-0.31) (-0.12,-0.32) (-0.04,-0.32) (0.04,-0.32) (0.12,-0.32) (0.20,-0.31) (0.28,-0.30) (0.35,-0.29) (0.42,-0.27) (0.49,-0.25) (0.55,-0.23) (0.61,-0.21) (0.65,-0.18) (0.65,-0.10) (0.64,-0.03) (0.62,0.05) (0.60,0.13) (0.57,0.20) (0.53,0.27) (0.49,0.34) (0.44,0.41) (0.38,0.47) (0.33,0.53) (0.27,0.58) (0.20,0.62) (0.14,0.66) (0.07,0.69) (-0.00,0.71) (-0.04,0.68) (-0.09,0.64) (-0.13,0.59) (-0.17,0.54) (-0.21,0.48) (-0.25,0.41) (-0.29,0.34) (-0.32,0.27) (-0.34,0.20) (-0.37,0.12) (-0.39,0.04) (-0.41,-0.04) (-0.42,-0.12) (-0.42,-0.20)} -- cycle;
\begin{scope}\clip (0,0) circle[radius=0.784];
\draw[diagramblue,line width=1.25pt] plot[smooth] coordinates {(-0.000,0.709) (-0.051,0.721) (-0.102,0.728) (-0.152,0.731) (-0.202,0.729) (-0.250,0.723) (-0.297,0.712) (-0.341,0.697) (-0.384,0.678) (-0.424,0.655)};
\draw[diagramblue,line width=1.25pt] plot[smooth] coordinates {(0.424,-0.655) (0.462,-0.627) (0.497,-0.596) (0.528,-0.560) (0.557,-0.522) (0.582,-0.480) (0.603,-0.435) (0.621,-0.388) (0.635,-0.338) (0.645,-0.286) (0.651,-0.233) (0.653,-0.178) (0.651,-0.121) (0.645,-0.065) (0.635,-0.007) (0.621,0.050) (0.603,0.107) (0.582,0.164) (0.557,0.219) (0.528,0.273) (0.497,0.325) (0.462,0.376) (0.424,0.424) (0.384,0.469) (0.341,0.512) (0.297,0.551) (0.250,0.587) (0.202,0.619) (0.152,0.648) (0.102,0.673) (0.051,0.693) (0.000,0.709)};
\end{scope}
\begin{scope}\clip (0,0) circle[radius=0.784];
\draw[diagramblue,line width=1.25pt] plot[smooth] coordinates {(0.194,0.755) (0.163,0.759) (0.132,0.758) (0.100,0.753) (0.067,0.743) (0.033,0.728) (0.000,0.709) (-0.033,0.686) (-0.067,0.658) (-0.100,0.626) (-0.132,0.590) (-0.163,0.551) (-0.194,0.508) (-0.223,0.462) (-0.251,0.414) (-0.277,0.362) (-0.301,0.309) (-0.324,0.254) (-0.345,0.197) (-0.363,0.138) (-0.380,0.079) (-0.394,0.020) (-0.405,-0.040) (-0.415,-0.099) (-0.421,-0.158) (-0.425,-0.216) (-0.426,-0.272) (-0.425,-0.327) (-0.421,-0.380) (-0.415,-0.430) (-0.405,-0.478) (-0.394,-0.523) (-0.380,-0.564) (-0.363,-0.603) (-0.345,-0.637) (-0.324,-0.667) (-0.301,-0.694) (-0.277,-0.716) (-0.251,-0.734) (-0.223,-0.747) (-0.194,-0.755)};
\end{scope}
\begin{scope}\clip (0,0) circle[radius=0.784];
\draw[diagramblue,line width=1.25pt] plot[smooth] coordinates {(0.653,-0.178) (0.685,-0.156) (0.712,-0.133) (0.735,-0.109) (0.753,-0.085) (0.760,-0.072) (0.772,-0.047) (0.778,-0.022) (0.780,0.004)};
\draw[diagramblue,line width=1.25pt] plot[smooth] coordinates {(-0.780,-0.004) (-0.777,-0.029) (-0.769,-0.054) (-0.756,-0.079) (-0.739,-0.104) (-0.717,-0.128) (-0.691,-0.151) (-0.661,-0.173) (-0.626,-0.194) (-0.588,-0.214) (-0.546,-0.232) (-0.500,-0.249) (-0.452,-0.265) (-0.400,-0.279) (-0.347,-0.291) (-0.291,-0.302) (-0.233,-0.310) (-0.144,-0.319) (-0.083,-0.323) (-0.022,-0.325) (0.039,-0.325) (0.100,-0.322) (0.160,-0.318) (0.220,-0.312) (0.278,-0.304) (0.334,-0.294) (0.388,-0.282) (0.440,-0.268) (0.489,-0.253) (0.536,-0.236) (0.578,-0.218) (0.618,-0.198) (0.653,-0.178)};
\end{scope}
\fill[orbitred] (-0.426,-0.272) circle[radius=0.058];
\fill[liftblue] (0.653,-0.178) circle[radius=0.058];
\fill[dgreen] (-0.000,0.709) circle[radius=0.058];\end{scope}
\begin{scope}[shift={(3.5,-3.3)}]
\shade[ball color=diagramlight!55!white] (0,0) circle[radius=0.780];
\draw[diagramblue!45,line width=0.4pt] (0,0) circle[radius=0.780];
\begin{scope}\clip (0,0) circle[radius=0.784];
\draw[diagramblue!45,line width=0.8pt,densely dashed] plot[smooth] coordinates {(-0.424,0.655) (-0.462,0.627) (-0.497,0.596) (-0.528,0.560) (-0.557,0.522) (-0.582,0.480) (-0.603,0.435) (-0.621,0.388) (-0.635,0.338) (-0.645,0.286) (-0.651,0.233) (-0.653,0.178) (-0.651,0.121) (-0.645,0.065) (-0.635,0.007) (-0.621,-0.050) (-0.603,-0.107) (-0.582,-0.164) (-0.557,-0.219) (-0.528,-0.273) (-0.497,-0.325) (-0.462,-0.376) (-0.424,-0.424) (-0.384,-0.469) (-0.341,-0.512) (-0.297,-0.551) (-0.250,-0.587) (-0.202,-0.619) (-0.152,-0.648) (-0.102,-0.673) (-0.051,-0.693) (-0.000,-0.709) (0.051,-0.721) (0.102,-0.728) (0.152,-0.731) (0.202,-0.729) (0.250,-0.723) (0.297,-0.712) (0.341,-0.697) (0.384,-0.678) (0.424,-0.655)};
\end{scope}
\begin{scope}\clip (0,0) circle[radius=0.784];
\draw[diagramblue!45,line width=0.8pt,densely dashed] plot[smooth] coordinates {(0.000,-0.709) (0.033,-0.686) (0.067,-0.658) (0.100,-0.626) (0.132,-0.590) (0.163,-0.551) (0.194,-0.508) (0.223,-0.462) (0.251,-0.414) (0.277,-0.362) (0.301,-0.309) (0.324,-0.254) (0.345,-0.197) (0.363,-0.138) (0.380,-0.079) (0.394,-0.020) (0.405,0.040) (0.415,0.099) (0.421,0.158) (0.425,0.216) (0.426,0.272) (0.425,0.327) (0.421,0.380) (0.415,0.430) (0.405,0.478) (0.394,0.523) (0.380,0.564) (0.363,0.603) (0.345,0.637) (0.324,0.667) (0.301,0.694) (0.277,0.716) (0.251,0.734) (0.223,0.747) (0.194,0.755)};
\draw[diagramblue!45,line width=0.8pt,densely dashed] plot[smooth] coordinates {(-0.194,-0.755) (-0.163,-0.759) (-0.132,-0.758) (-0.100,-0.753) (-0.067,-0.743) (-0.033,-0.728) (-0.000,-0.709)};
\end{scope}
\begin{scope}\clip (0,0) circle[radius=0.784];
\draw[diagramblue!45,line width=0.8pt,densely dashed] plot[smooth] coordinates {(0.780,0.004) (0.777,0.029) (0.769,0.054) (0.756,0.079) (0.739,0.104) (0.717,0.128) (0.691,0.151) (0.661,0.173) (0.626,0.194) (0.588,0.214) (0.546,0.232) (0.500,0.249) (0.452,0.265) (0.400,0.279) (0.347,0.291) (0.291,0.302) (0.233,0.310) (0.174,0.317) (0.114,0.322) (0.053,0.324) (-0.008,0.325) (-0.070,0.324) (-0.130,0.320) (-0.190,0.315) (-0.249,0.308) (-0.306,0.299) (-0.362,0.288) (-0.415,0.275) (-0.465,0.261) (-0.513,0.245) (-0.557,0.227) (-0.599,0.208) (-0.636,0.188) (-0.669,0.167) (-0.699,0.144) (-0.724,0.121) (-0.744,0.097) (-0.760,0.072) (-0.772,0.047) (-0.778,0.022) (-0.780,-0.004)};
\end{scope}
\fill[orbitred,opacity=0.38] (0.426,0.272) circle[radius=0.058];
\fill[orbitred,opacity=0.38] (-0.653,0.178) circle[radius=0.058];
\fill[orbitred,opacity=0.38] (0.000,-0.709) circle[radius=0.058];
\fill[liftblue,opacity=0.38] (-0.623,-0.379) -- (-0.697,-0.506) -- (-0.550,-0.506) -- cycle;
\fill[liftblue,opacity=0.38] (0.623,-0.270) -- (0.550,-0.397) -- (0.697,-0.397) -- cycle;
\fill[liftblue,opacity=0.38] (-0.131,0.754) -- (-0.205,0.627) -- (-0.057,0.627) -- cycle;
\fill[liftblue,opacity=0.38] (-0.131,-0.065) -- (-0.205,-0.192) -- (-0.057,-0.192) -- cycle;
\fill[liftorange!34] plot coordinates {(-0.43,-0.27) (-0.39,-0.28) (-0.36,-0.29) (-0.32,-0.30) (-0.28,-0.30) (-0.24,-0.31) (-0.20,-0.31) (-0.16,-0.32) (-0.12,-0.32) (-0.08,-0.32) (-0.04,-0.32) (-0.00,-0.33) (0.04,-0.32) (0.08,-0.32) (0.12,-0.32) (0.16,-0.32) (0.16,-0.29) (0.16,-0.26) (0.16,-0.23) (0.16,-0.20) (0.16,-0.17) (0.16,-0.14) (0.15,-0.10) (0.15,-0.07) (0.15,-0.04) (0.15,-0.01) (0.14,0.02) (0.14,0.05) (0.14,0.09) (0.13,0.12) (0.13,0.15) (0.10,0.16) (0.07,0.18) (0.04,0.19) (0.01,0.20) (-0.02,0.21) (-0.05,0.22) (-0.07,0.24) (-0.10,0.25) (-0.13,0.26) (-0.16,0.27) (-0.19,0.28) (-0.22,0.28) (-0.25,0.29) (-0.27,0.30) (-0.30,0.31) (-0.32,0.27) (-0.33,0.23) (-0.34,0.20) (-0.36,0.16) (-0.37,0.12) (-0.38,0.08) (-0.39,0.04) (-0.40,-0.00) (-0.41,-0.04) (-0.41,-0.08) (-0.42,-0.12) (-0.42,-0.16) (-0.42,-0.20) (-0.43,-0.23)} -- cycle;
\draw[liftorange!85!black,line width=0.5pt,densely dotted] plot coordinates {(0.131,0.150) (0.136,0.102) (0.141,0.055) (0.146,0.007) (0.150,-0.041) (0.153,-0.088) (0.156,-0.136) (0.158,-0.182) (0.159,-0.229) (0.160,-0.274) (0.160,-0.318)};
\draw[liftorange!85!black,line width=0.5pt,densely dotted] plot coordinates {(0.131,0.150) (0.087,0.169) (0.043,0.188) (-0.001,0.206) (-0.045,0.224) (-0.089,0.241) (-0.133,0.256) (-0.176,0.271) (-0.219,0.285) (-0.261,0.297) (-0.301,0.309)};
\begin{scope}\clip (0,0) circle[radius=0.784];
\draw[diagramblue,line width=1.25pt] plot[smooth] coordinates {(-0.000,0.709) (-0.051,0.721) (-0.102,0.728) (-0.152,0.731) (-0.202,0.729) (-0.250,0.723) (-0.297,0.712) (-0.341,0.697) (-0.384,0.678) (-0.424,0.655)};
\draw[diagramblue,line width=1.25pt] plot[smooth] coordinates {(0.424,-0.655) (0.462,-0.627) (0.497,-0.596) (0.528,-0.560) (0.557,-0.522) (0.582,-0.480) (0.603,-0.435) (0.621,-0.388) (0.635,-0.338) (0.645,-0.286) (0.651,-0.233) (0.653,-0.178) (0.651,-0.121) (0.645,-0.065) (0.635,-0.007) (0.621,0.050) (0.603,0.107) (0.582,0.164) (0.557,0.219) (0.528,0.273) (0.497,0.325) (0.462,0.376) (0.424,0.424) (0.384,0.469) (0.341,0.512) (0.297,0.551) (0.250,0.587) (0.202,0.619) (0.152,0.648) (0.102,0.673) (0.051,0.693) (0.000,0.709)};
\end{scope}
\begin{scope}\clip (0,0) circle[radius=0.784];
\draw[diagramblue,line width=1.25pt] plot[smooth] coordinates {(0.194,0.755) (0.163,0.759) (0.132,0.758) (0.100,0.753) (0.067,0.743) (0.033,0.728) (0.000,0.709) (-0.033,0.686) (-0.067,0.658) (-0.100,0.626) (-0.132,0.590) (-0.163,0.551) (-0.194,0.508) (-0.223,0.462) (-0.251,0.414) (-0.277,0.362) (-0.301,0.309) (-0.324,0.254) (-0.345,0.197) (-0.363,0.138) (-0.380,0.079) (-0.394,0.020) (-0.405,-0.040) (-0.415,-0.099) (-0.421,-0.158) (-0.425,-0.216) (-0.426,-0.272) (-0.425,-0.327) (-0.421,-0.380) (-0.415,-0.430) (-0.405,-0.478) (-0.394,-0.523) (-0.380,-0.564) (-0.363,-0.603) (-0.345,-0.637) (-0.324,-0.667) (-0.301,-0.694) (-0.277,-0.716) (-0.251,-0.734) (-0.223,-0.747) (-0.194,-0.755)};
\end{scope}
\begin{scope}\clip (0,0) circle[radius=0.784];
\draw[diagramblue,line width=1.25pt] plot[smooth] coordinates {(0.653,-0.178) (0.685,-0.156) (0.712,-0.133) (0.735,-0.109) (0.753,-0.085) (0.760,-0.072) (0.772,-0.047) (0.778,-0.022) (0.780,0.004)};
\draw[diagramblue,line width=1.25pt] plot[smooth] coordinates {(-0.780,-0.004) (-0.777,-0.029) (-0.769,-0.054) (-0.756,-0.079) (-0.739,-0.104) (-0.717,-0.128) (-0.691,-0.151) (-0.661,-0.173) (-0.626,-0.194) (-0.588,-0.214) (-0.546,-0.232) (-0.500,-0.249) (-0.452,-0.265) (-0.400,-0.279) (-0.347,-0.291) (-0.291,-0.302) (-0.233,-0.310) (-0.144,-0.319) (-0.083,-0.323) (-0.022,-0.325) (0.039,-0.325) (0.100,-0.322) (0.160,-0.318) (0.220,-0.312) (0.278,-0.304) (0.334,-0.294) (0.388,-0.282) (0.440,-0.268) (0.489,-0.253) (0.536,-0.236) (0.578,-0.218) (0.618,-0.198) (0.653,-0.178)};
\end{scope}
\fill[orbitred] (-0.426,-0.272) circle[radius=0.058];
\fill[orbitred] (0.653,-0.178) circle[radius=0.058];
\fill[orbitred] (-0.000,0.709) circle[radius=0.058];
\fill[liftblue] (0.131,0.235) -- (0.057,0.107) -- (0.205,0.107) -- cycle;
\fill[liftblue] (0.623,0.549) -- (0.550,0.421) -- (0.697,0.421) -- cycle;
\fill[liftblue] (-0.623,0.440) -- (-0.697,0.312) -- (-0.550,0.312) -- cycle;
\fill[liftblue] (0.131,-0.584) -- (0.057,-0.712) -- (0.205,-0.712) -- cycle;\end{scope}
\begin{scope}[shift={(2.0,0)}]
\shade[ball color=diagramlight!55!white] (0,0) circle[radius=0.42];
\draw[diagramblue!45,line width=0.4pt] (0,0) circle[radius=0.42];
\node[cone2=orbitred] at (-0.18,0.15) {}; \node[cone2=liftblue] at (0.18,0.15) {};
\node[cone2=dgreen] at (0,-0.2) {};
\end{scope}
\begin{scope}[shift={(5.5,0)}]
\shade[ball color=diagramlight!55!white] (0,0) circle[radius=0.42];
\draw[diagramblue!45,line width=0.4pt] (0,0) circle[radius=0.42];
\node[cone2=orbitred] at (0,0.2) {};
\conethree{(-0.17,-0.12)}{liftblue}\conethree{(0.17,-0.12)}{dgreen}
\end{scope}
\begin{scope}[shift={(2.0,-3.3)}]
\filldraw[fill=liftorange!30,mirror] (90:0.48) to[bend right=14] (210:0.48) to[bend right=14] (330:0.48) to[bend right=14] cycle;
\foreach \a/\c in {90/dgreen,210/orbitred,330/liftblue} \node[cone2=\c] at (\a:0.48) {};
\end{scope}
\begin{scope}[shift={(5.5,-3.3)}]
\filldraw[fill=liftorange!34,mirror] (0,0.5) .. controls (0.62,0.02) and (0.45,-0.45) .. (0,-0.45)
  .. controls (-0.45,-0.45) and (-0.62,0.02) .. cycle;
\node[cone2=orbitred] at (0,0.5) {};
\conethree{(0,-0.08)}{liftblue}
\end{scope}
\node[slab,text=inkgray,align=center] at (2.75,-5.12) {dots: $2$-fold points;\ triangles: $3$-fold points;\\
  heavy lines: mirrors;\ same colour: one orbit};
\node[panelcap] at (2.75,-6.35) {(b) the actions on the base $S^2$ and the Seifert bases};
\end{scope}
\end{tikzpicture}
\caption{The quotients of $S^3$ in Section~\ref{wh:sec:whitehead}. (a) The top arrow
is the free quotient by the noncyclic group $\Gamma=Q_8\times C_7$; the remaining
arrows are quotients by the indicated cyclic groups. All displayed singular links
have local isotropy $C_2$. The maps $\N\to N'$ and $|S^3/\Lambda_0|\to|X|$ are free
threefold quotients; $\N\to S^3/\Lambda_0$ is the double cover branched along the
image $L_0$ of $\Fix(\bar f^{\,3})=\mathcal C_1\sqcup\mathcal C_2\sqcup\mathcal C_3$, and
$N'\to X$ is the double cover branched along the single circle $\Fix(\iota)$, which
maps onto $\mathcal S$ (Lemma~\ref{wh:lem:underlying} and
Theorem~\ref{wh:thm:branched}). (b) All four groups preserve the Hopf fibration
$v\mapsto vU(1)$, so these quotients are Seifert fibred over the corresponding
quotients of $S^2$ (Remark~\ref{wh:rem:seifert}). Each sphere shows the induced
group acting on the base $S^2$, viewed with the octant of $i+j+k$ in front:
$V_4$ consists of the rotations through $\pi$ about the coordinate axes, $A_4$ adds
the rotations through $2\pi/3$ about the diagonals, and the factor $\{\pm1\}$ turns
the three coordinate great circles into mirrors. Points of the same colour form one
orbit and become one cone point or corner reflector of the quotient (right of each
sphere); a fundamental domain is shaded for the two groups with mirrors. Thus $\N$,
$N'$, $S^3/\Lambda_0$ and $X$ fibre over $S^2(2,2,2)$, $S^2(2,3,3)$, ${*}222$ and
$3{*}2$. The singular circles $L_0$ and $\mathcal S$ lie over the mirror
boundaries.}
\label{wh:fig:quotients}
\end{figure}

\begin{remark}[Underlying spaces and hyperelliptic involutions]\label{wh:rem:tables}
The orbifold $S^3/\Lambda_0$ fibres over $D^2(;2,2,2)$, a disc without cone
points. Its underlying space is therefore $S^3$ by
\cite[Proposition~2.10]{Dunbar1981}; see also
\cite[Section~5.3]{MecchiaSchilling2025}. The orbifold $X$ fibres over
$D^2(3;2)$, whose cone point has the nonintegral local invariant $7/3$.
Thus, $|X|\not\cong S^3$ by \cite[Proposition~2.11]{Dunbar1981}.
The involution $\bar f^{\,3}$ of $\N$, with extension $\Lambda_0$
(Family~10), is hyperelliptic, whereas the involution $\iota$ of $N'$, with
extension $\Lambda$ (Family~14), is not
\cite[Tables~11 and~13]{MecchiaSchilling2025}.
\end{remark}

\subsection{Identification with the Whitehead link}

Let $W=J\cup W'$ be a Whitehead link ($5^2_1$ in \cite{Rolfsen1976}) in
its standard five-crossing diagram. We label the components so that $J$ is
the component without a self-crossing and $W'$ is the component carrying the
self-crossing. The Whitehead link admits an orientation-preserving ambient isotopy
interchanging its two components; compare the component interchange used
explicitly in \cite[Section~5]{GordonLidman2014}. Thus, this choice is only a
convenient labelling convention and introduces no asymmetry into the link
type.

\smallskip\noindent\emph{Mirror convention.}
The two mirror images of the Whitehead link can be distinguished by
Proposition~\ref{wh:prop:lift}(iv) below. Throughout, $W$ denotes the mirror
image with $m=2$ in the notation of Proposition~\ref{wh:prop:lift}: the one
whose clasp, in the standard diagram of $J$ in the solid torus
$S^3\setminus\Int\nu W'$, has two negative self-crossings ($w=-2$);
equivalently, the one for which the double cover of $S^3_{+3}(J)$ branched
along $W'$ has first homology of order $21$. In the terminology of
\cite[Section~5]{GordonLidman2014} this is the Whitehead link with a negative
clasp. Lemma~\ref{wh:lem:abel} and
Proposition~\ref{wh:prop:lift} are stated for both mirror images;
Theorem~\ref{wh:thm:whitehead} and all results depending on it refer to this
chosen mirror image.
Here $w$ is the writhe of the clasp diagram in the solid torus; it is
not the self-writhe of $J$ in the five-crossing diagram, which is zero.
Figure~\ref{std:fig:belt} is drawn in this mirror convention.

\smallskip\noindent
Let $\OW$ be the orbifold whose underlying space is
$S^3_{+3}(J)\cong\spLt$ and whose singular locus is $W'$, with isotropy of
order two.

\begin{lemma}[The abelianization of the orbifold group]\label{wh:lem:abel}
For either mirror image of the Whitehead link,
$H_1^{\mathrm{orb}}(\OW)=\pi_1^{\mathrm{orb}}(\OW)^{\mathrm{ab}}$ is cyclic of
order six. The images of a meridian $\mu_J$ of $J$ and of a meridian
$\mu_{W'}$ of $W'$ have orders three and two, and their product is a
generator.
\end{lemma}

\begin{proof}
The orbifold group is $\pi_1(S^3\setminus(J\cup W'))$ modulo the normal
closure of the surgery slope $\mu_J^3\lambda_J$ and of $\mu_{W'}^{\,2}$.
By Alexander duality \cite[Theorem~3.44]{Hatcher},
$H_1(S^3\setminus(J\cup W'))=\ZZ\mu_J\oplus\ZZ\mu_{W'}$,
and the Seifert longitude of $J$ represents $\lambda_J=\lk(J,W')\,\mu_{W'}=0$ in this group.
Hence
$H_1^{\mathrm{orb}}(\OW)=\ZZ^2/\langle3\mu_J,2\mu_{W'}\rangle\cong\ZZ/3\oplus\ZZ/2$.
\end{proof}

\begin{proposition}[The double branched cover]\label{wh:prop:lift}
Let $W=J\cup W'$ be either mirror image of the Whitehead link, and let $Y$ be
the double cover of $S^3_{+3}(J)$ branched along $W'$. Let
$w\in\{2,-2\}$ be the writhe of the standard clasp diagram of $J$ in the
solid torus $V=S^3\setminus\Int\nu W'$, that is, the sum of the signs of its
two self-crossings (the signs are equal, as they are at every clasp), and put $m=-w$.
\begin{enumerate}[label=(\roman*),leftmargin=2.2em]
\item The double cover of $S^3$ branched along $W'$ is $S^3$, and the
      preimage $\widetilde J=\widetilde J_1\cup\widetilde J_2$ of $J$ in this cover is
      the torus link $T(2,2m)$: its components are unknotted and cobound an
      annulus with an unknotted core and framing $m$.
\item The surgery slope $3\mu_J+\lambda_J$ lifts on each $\widetilde J_i$ to
      the slope $(3+m)\widetilde\mu_i+\widetilde\lambda_i^{\,0}$, where
      $\widetilde\lambda_i^{\,0}$ is the Seifert longitude of $\widetilde J_i$.
      Hence, $Y=S^3_{3+m,\,3+m}\bigl(T(2,2m)\bigr)$ and
      $|H_1(Y)|=|(3+m)^2-m^2|=|9+6m|$.
\item $Y$ is Seifert fibred over $S^2(2,3,3)$: the components of $T(2,2m)$
      are regular fibres of the Seifert fibration of $S^3$ by
      $(1,m)$-torus curves, whose fibre framing is $m$, and the lifted slopes
      are $3\widetilde\mu_i+\lambda_{\mathrm{fib}}$.
\item For $m=2$ one has $|H_1(Y)|=21$, and for $m=-2$ one has $|H_1(Y)|=3$.
      Thus, the two mirror images are distinguished by $|H_1(Y)|$.
\end{enumerate}
\end{proposition}

\noindent
Parts (i)--(iii) are applications of the Montesinos trick
\cite{Montesinos1975}; Figure~\ref{wh:fig:montesinos} shows it for the mirror image
fixed above. The double cover of $S^3$ branched along one component
of the Whitehead link is $S^3$, and the other component lifts to the
two-component chain link, i.e.,\ a $(2,\pm4)$-torus link
\cite[Section~5]{GordonLidman2014}; the Seifert fibration in (iii) is obtained by
standard Seifert surgery on regular fibres \cite{Montesinos1973}.
In the terminology of \cite[Section~5]{GordonLidman2014}, the case $w=2$,
that is, $m=-2$, is the Whitehead link with a positive clasp, for which that
section computes the lifts of $J$ and its blackboard longitude (see the
proof of~(i) below); the case $m=2$ used in Theorem~\ref{wh:thm:whitehead} is
its mirror image.
The corrigendum \cite{GordonLidman2017} to \cite{GordonLidman2014} corrects
the statement and proof of Theorem~1.4 there concerning double branched
covers of $(p,2)$-cables and does not affect the Whitehead-link computations of
\cite[Section~5]{GordonLidman2014} used here.

\begin{proof}
(i) Since $W'$ is unknotted, we have $\Sigma_2(S^3,W')=S^3$. Since
$\lk(J,W')=0$, the preimage of $J$ has two components,
$\widetilde J_1$ and $\widetilde J_2$, interchanged by the deck involution. Let
$\widetilde\lambda_i$ be the lift of the Seifert longitude $\lambda_J$ to
$\partial\nu\widetilde J_i$. For $w=2$, the calculation in \cite[Section~5]{GordonLidman2014}
shows, with the two components interchanged by the symmetry of the Whitehead
link, that $\widetilde J$ is the torus link $T(2,-4)$ (denoted $T_{-2,4}$
there) and that the blackboard longitude $\lambda_J+w\mu_J$ lifts to the
Seifert longitudes of the $\widetilde J_i$. Hence
$\lk(\widetilde J_i,\widetilde\lambda_i)=-w=m$, and $\widetilde J$ is the
boundary of an annulus with an unknotted core and framing $m$, that is,
$T(2,2m)$. The case $w=-2$ follows by applying an orientation-reversing
diffeomorphism of $S^3$, which changes the signs of $w$, of $m$ and of all
linking numbers.

\emph{A self-contained check of~(i).} The identification of $\widetilde J$
can also be read off from the diagram. Project the clasp diagram of $J$ in the
annulus along the axis $W'$. The double cover branched along $W'$ is
$\vartheta\mapsto2\vartheta$ in cylindrical coordinates, so it commutes with
this projection: the lifted diagram consists of two copies of the band of $J$,
each running half way around, and two copies of the clasp, each now between
$\widetilde J_1$ and $\widetilde J_2$ and with the same crossing signs. Hence
each $\widetilde J_i$ has a crossingless diagram, so it is unknotted and its
blackboard framing is its Seifert framing; since the blackboard longitude
$\lambda_J+w\mu_J$ lifts to the blackboard longitudes,
$\widetilde\lambda_i=\widetilde\lambda_i^{\,0}-w\widetilde\mu_i$. With the
lifted orientations the four mixed crossings have signs summing to $2w$, so
$\lk(\widetilde J_1,\widetilde J_2)=w$. Seifert's algorithm applied to the
clasp diagram of $J$ gives three Seifert circles and two bands, hence a disc
$\Delta$ with $\partial\Delta=J$; exactly two of the circles encircle the
axis, so $\Delta$ meets $W'$ transversely in two points. Its preimage
$\widetilde\Delta$, the double cover of a disc branched at two points, is an
annulus with $\partial\widetilde\Delta=\widetilde J_1\sqcup\widetilde J_2$, and
the lifted orientations are its boundary orientation. Its core is isotopic to
$\widetilde J_1$, hence unknotted. For the boundary orientation of an annulus
with $k$ full twists the two boundary circles have linking number $-k$, so
$k=-w=m$ and $\widetilde J=T(2,2m)$, for either sign of $w$.

(ii) The lift of $\mu_J$ is the meridian $\widetilde\mu_i$ of $\widetilde J_i$,
and the lift of $\lambda_J$ is $\widetilde\lambda_i$, with
$\lk(\widetilde J_i,\widetilde\lambda_i)=m$ by (i), so
$\widetilde\lambda_i=\widetilde\lambda_i^{\,0}+m\widetilde\mu_i$ and the
lifted slope is $(3+m)\widetilde\mu_i+\widetilde\lambda_i^{\,0}$. The double
cover of $S^3_{+3}(J)$ branched along $W'$ is obtained from
$\Sigma_2(S^3,W')=S^3$ by surgery along $\widetilde J$ with the lifted
slopes, because the surgery solid torus is disjoint from $W'$ and its boundary
torus lifts to two tori. With parallel orientations, the components of
$T(2,2m)$ have linking number $m$, so the linking matrix of the surgery is
$\bigl(\begin{smallmatrix}3+m&m\\m&3+m\end{smallmatrix}\bigr)$, and
its determinant is $(3+m)^2-m^2$.

(iii) Two parallel $(1,m)$-torus curves on the standard torus form
$T(2,2m)$ and have linking number $m$. They are regular fibres of the Seifert
fibration of $S^3$ whose fibres are the $(1,m)$-curves; it has one exceptional
fibre, of multiplicity $|m|=2$, and the fibre framing of a regular fibre is
$m$ (the framing of a $(p,q)$-torus knot induced by a neighbouring fibre is
$pq$ \cite{Moser1971}), so
$\lambda_{\mathrm{fib}}=\widetilde\lambda_i^{\,0}+m\widetilde\mu_i$
and the lifted slope is $3\widetilde\mu_i+\lambda_{\mathrm{fib}}$. It meets
the fibre slope three times, so each of the two surgeries replaces a regular
fibre by an exceptional fibre of multiplicity three, and $Y$ is Seifert
fibred over $S^2(2,3,3)$.

(iv) Substitute $m=\pm2$ in (ii).
\end{proof}

\begin{figure}[tbp]
\centering
\begin{tikzpicture}[
  knot/.style={line width=1.15pt,line cap=round,line join=round,draw=diagramblue!70!black}]
\begin{scope}[scale=0.62]
\shade[inner color=diagramblue!22,outer color=diagramlight,even odd rule] (0,0) circle[radius=2.95] (0,0) circle[radius=1.05];
\draw[diagramblue!45,line width=0.5pt] (0,0) circle[radius=2.95] (0,0) circle[radius=1.05];
\draw[knot] plot[smooth,tension=0.5] coordinates {(1.19,1.85) (1.10,1.67) (0.97,1.52) (0.80,1.38) (0.62,1.29) (0.42,1.24) (0.21,1.26) (0.01,1.28) (-0.19,1.27) (-0.39,1.22) (-0.57,1.15) (-0.74,1.04) (-0.90,0.91) (-1.03,0.76) (-1.14,0.59) (-1.21,0.41) (-1.26,0.21) (-1.28,0.01) (-1.27,-0.19) (-1.22,-0.39) (-1.14,-0.57) (-1.04,-0.74) (-0.91,-0.90) (-0.76,-1.03) (-0.59,-1.14) (-0.40,-1.21) (-0.21,-1.26) (-0.01,-1.28) (0.19,-1.27) (0.39,-1.22) (0.57,-1.14) (0.74,-1.04) (0.90,-0.91) (1.03,-0.76) (1.14,-0.59) (1.21,-0.40) (1.26,-0.21) (1.28,-0.01) (1.27,0.19) (1.22,0.39) (1.21,0.60) (1.24,0.81) (1.24,1.02) (1.17,1.21) (0.98,1.39)};
\draw[knot] plot[smooth,tension=0.5] coordinates {(0.80,1.50) (0.61,1.59) (0.41,1.65) (0.21,1.69) (0.01,1.70) (-0.19,1.69) (-0.39,1.65) (-0.59,1.70) (-0.73,1.85) (-0.74,2.05) (-0.61,2.20) (-0.41,2.26) (-0.21,2.29) (-0.00,2.30) (0.20,2.29) (0.41,2.26) (0.61,2.22) (0.81,2.15) (1.00,2.07) (1.18,1.97) (1.35,1.86) (1.51,1.73) (1.69,1.61) (1.87,1.51) (2.04,1.42) (2.22,1.30) (2.38,1.17) (2.51,1.01) (2.59,0.82) (2.65,0.63) (2.69,0.43) (2.71,0.22) (2.72,0.02) (2.71,-0.18) (2.69,-0.38) (2.66,-0.58) (2.61,-0.78) (2.54,-0.97) (2.46,-1.16) (2.37,-1.34) (2.26,-1.51) (2.14,-1.68) (2.01,-1.83) (1.87,-1.98) (1.71,-2.11) (1.55,-2.24) (1.38,-2.35) (1.20,-2.44) (1.01,-2.52) (0.82,-2.59) (0.63,-2.65) (0.43,-2.69) (0.22,-2.71) (0.02,-2.72) (-0.18,-2.71) (-0.39,-2.69) (-0.59,-2.66) (-0.78,-2.61) (-0.97,-2.54) (-1.16,-2.46) (-1.34,-2.37) (-1.52,-2.26) (-1.68,-2.14) (-1.84,-2.01) (-1.98,-1.86) (-2.11,-1.71) (-2.24,-1.55) (-2.35,-1.38) (-2.44,-1.20) (-2.52,-1.01) (-2.59,-0.82) (-2.65,-0.62) (-2.69,-0.42) (-2.71,-0.22) (-2.72,-0.02) (-2.71,0.18) (-2.69,0.39) (-2.66,0.59) (-2.60,0.78) (-2.54,0.98) (-2.46,1.16) (-2.37,1.34) (-2.26,1.52) (-2.14,1.68) (-2.01,1.84) (-1.86,1.98) (-1.71,2.11) (-1.55,2.24) (-1.38,2.35) (-1.20,2.44) (-1.01,2.52) (-0.82,2.59) (-0.62,2.65) (-0.42,2.69) (-0.22,2.71) (-0.02,2.72) (0.18,2.71) (0.39,2.69) (0.59,2.66) (0.80,2.58) (0.99,2.46) (1.12,2.29) (1.20,2.07)};
\fill[orbitred] (0,0) circle[radius=0.09];
\draw[orbitred,line width=0.8pt] (0,0) circle[radius=0.22];
\node[slab,text=orbitred,below=3pt] at (0,-0.22) {$W'$};
\node[lab] at (-2.78,-1.55) {$J$};
\node[lab,fill=white,rounded corners=1.5pt,draw=diagramblue!50,inner sep=1.5pt] at (-2.05,-2.68) {$+3$};
\node[slab,text=diagramblue] at (2.45,-2.25) {$V$};
\node[panelcap] at (0,-3.5) {(a) $J\subset V=S^3\setminus\Int\nu W'$};
\end{scope}

\draw[darrow] (5.45,0.25) -- node[above=2pt,slab,align=center]{double cover\\ branched\\ along $W'$} node[below=2pt,slab]{$\vartheta\mapsto2\vartheta$} (3.2,0.25);
\begin{scope}[shift={(8.6cm,0)},scale=0.62]
\shade[inner color=diagramblue!22,outer color=diagramlight,even odd rule] (0,0) circle[radius=2.95] (0,0) circle[radius=1.05];
\draw[diagramblue!45,line width=0.5pt] (0,0) circle[radius=2.95] (0,0) circle[radius=1.05];
\draw[knot,draw=liftblue] plot[smooth,tension=0.5] coordinates {(-0.62,-2.10) (-0.57,-1.89) (-0.49,-1.68) (-0.39,-1.51) (-0.27,-1.34) (-0.09,-1.28) (0.11,-1.27) (0.31,-1.24) (0.50,-1.18) (0.68,-1.08) (0.84,-0.96) (0.98,-0.82) (1.10,-0.66) (1.19,-0.48) (1.25,-0.28) (1.28,-0.09) (1.27,0.11) (1.24,0.31) (1.18,0.50) (1.08,0.68) (0.96,0.84) (0.82,0.98) (0.72,1.15) (0.70,1.35) (0.65,1.54) (0.57,1.60)};
\draw[knot,draw=liftblue] plot[smooth,tension=0.5] coordinates {(0.36,1.66) (0.16,1.69) (-0.04,1.70) (-0.25,1.71) (-0.35,1.89) (-0.38,2.09) (-0.30,2.27) (-0.10,2.30) (0.10,2.30) (0.31,2.28) (0.51,2.24) (0.71,2.19) (0.91,2.15) (1.10,2.19) (1.30,2.24) (1.50,2.24) (1.68,2.14) (1.83,2.01) (1.98,1.87) (2.11,1.71) (2.24,1.55) (2.35,1.38) (2.44,1.20) (2.52,1.01) (2.59,0.82) (2.65,0.63) (2.69,0.43) (2.71,0.22) (2.72,0.02) (2.71,-0.18) (2.69,-0.39) (2.66,-0.59) (2.61,-0.78) (2.54,-0.97) (2.46,-1.16) (2.37,-1.34) (2.26,-1.52) (2.14,-1.68) (2.01,-1.84) (1.86,-1.98) (1.71,-2.11) (1.55,-2.24) (1.38,-2.35) (1.20,-2.44) (1.01,-2.52) (0.82,-2.59) (0.62,-2.65) (0.42,-2.69) (0.22,-2.71) (0.02,-2.72) (-0.18,-2.71) (-0.39,-2.68) (-0.55,-2.54) (-0.62,-2.32)};
\draw[knot,draw=liftorange] plot[smooth,tension=0.5] coordinates {(0.62,2.10) (0.57,1.89) (0.49,1.68) (0.39,1.51) (0.27,1.34) (0.09,1.28) (-0.11,1.27) (-0.31,1.24) (-0.50,1.18) (-0.68,1.08) (-0.84,0.96) (-0.98,0.82) (-1.10,0.66) (-1.19,0.48) (-1.25,0.28) (-1.28,0.09) (-1.27,-0.11) (-1.24,-0.31) (-1.18,-0.50) (-1.08,-0.68) (-0.96,-0.84) (-0.82,-0.98) (-0.72,-1.15) (-0.70,-1.35) (-0.65,-1.54) (-0.57,-1.60)};
\draw[knot,draw=liftorange] plot[smooth,tension=0.5] coordinates {(-0.36,-1.66) (-0.16,-1.69) (0.04,-1.70) (0.25,-1.71) (0.35,-1.89) (0.38,-2.09) (0.30,-2.27) (0.10,-2.30) (-0.10,-2.30) (-0.31,-2.28) (-0.51,-2.24) (-0.71,-2.19) (-0.91,-2.15) (-1.10,-2.19) (-1.30,-2.24) (-1.50,-2.24) (-1.68,-2.14) (-1.83,-2.01) (-1.98,-1.87) (-2.11,-1.71) (-2.24,-1.55) (-2.35,-1.38) (-2.44,-1.20) (-2.52,-1.01) (-2.59,-0.82) (-2.65,-0.63) (-2.69,-0.43) (-2.71,-0.22) (-2.72,-0.02) (-2.71,0.18) (-2.69,0.39) (-2.66,0.59) (-2.61,0.78) (-2.54,0.97) (-2.46,1.16) (-2.37,1.34) (-2.26,1.52) (-2.14,1.68) (-2.01,1.84) (-1.86,1.98) (-1.71,2.11) (-1.55,2.24) (-1.38,2.35) (-1.20,2.44) (-1.01,2.52) (-0.82,2.59) (-0.62,2.65) (-0.42,2.69) (-0.22,2.71) (-0.02,2.72) (0.18,2.71) (0.39,2.68) (0.55,2.54) (0.62,2.32)};
\fill[orbitred] (0,0) circle[radius=0.09];
\draw[orbitred,line width=0.8pt] (0,0) circle[radius=0.22];
\draw[-{Stealth[length=1.6mm]},orbitred] (0.45,0.0) arc[start angle=0,end angle=160,radius=0.45];
\node[slab,text=orbitred] at (0.56,0.47) {$\pi$};
\draw[knot,draw=liftblue,-{Stealth[length=1.6mm]}] plot[smooth] coordinates {(1.25,-0.28) (1.28,-0.09) (1.27,0.11)};
\draw[knot,draw=liftorange,-{Stealth[length=1.6mm]}] plot[smooth] coordinates {(-1.27,-0.11) (-1.28,0.09) (-1.25,0.28)};
\node[lab,text=liftblue] at (3.05,1.25) {$\widetilde J_1$};
\node[lab,text=liftorange] at (-3.05,-1.25) {$\widetilde J_2$};
\node[lab,text=liftblue,fill=white,rounded corners=1.5pt,draw=liftblue!50,inner sep=1.5pt] at (2.75,-1.8) {$5$};
\node[lab,text=liftorange!85!black,fill=white,rounded corners=1.5pt,draw=liftorange!60,inner sep=1.5pt] at (-2.75,1.8) {$5$};
\node[panelcap] at (0,-3.5) {(b) $\widetilde J=\widetilde J_1\cup\widetilde J_2=T(2,4)$};
\end{scope}
\end{tikzpicture}
\caption{The Montesinos trick of Proposition~\ref{wh:prop:lift}, for the mirror image fixed in Section~\ref{wh:sec:whitehead} ($w=-2$, $m=2$). (a) The solid torus $V=S^3\setminus\Int\nu W'$, drawn as a thickened annulus; $W'$ is the axis through the central hole, perpendicular to the page and closed up through $\infty$. The component $J$, with coefficient $+3$, is the Whitehead pattern in $V$; its clasp has two negative crossings. (b) The double cover of $S^3$ branched along $W'$ is again $S^3$; in cylindrical coordinates about the axis it is $\vartheta\mapsto2\vartheta$, and its deck involution is the rotation by $\pi$ about the axis. The two lifts $\widetilde J_1$ and $\widetilde J_2$ of $J$ are unknots that clasp twice and form $T(2,4)$. The arrows choose parallel orientations, with linking number $+2$; the orientations induced as the boundary of the lifted annulus instead have linking number $-2$. The slope $3\mu_J+\lambda_J$ lifts to $5\widetilde\mu_i+\widetilde\lambda_i^{\,0}$, and the linking matrix $\bigl(\begin{smallmatrix}5&2\\2&5\end{smallmatrix}\bigr)$ (for parallel orientations) has determinant $21=|H_1(N')|$.}
\label{wh:fig:montesinos}
\end{figure}

\begin{theorem}[The Whitehead model]\label{wh:thm:whitehead}
For the mirror image of the Whitehead link fixed above, there is a
homeomorphism of orbifolds $X\cong\OW$, possibly reversing orientation. In
particular,
\[
 (|X|,\mathcal S)\ \cong\ \bigl(S^3_{+3}(J),\,W'\bigr).
\]
Under this homeomorphism, the singular circle is the Whitehead partner of a $(+3)$-surgery curve.
Equivalently, $\mathcal S$ lies in one solid torus of a genus-one Heegaard
splitting of $\spLt$ as a Whitehead double of its core, with the twisting
of the Whitehead link.

\end{theorem}

\begin{proof}
By the mirror convention, $m=2$ in Proposition~\ref{wh:prop:lift};
equivalently, the double cover of $S^3_{+3}(J)$ branched along $W'$ has
first homology of order $21$. Let $Y$ be this double cover of
$S^3_{+3}(J)$ branched along $W'$. It exists because $W'$ is null-homologous:
$\lk(J,W')=0$. By Proposition~\ref{wh:prop:lift}, $Y$ is Seifert fibred over
$S^2(2,3,3)$ with $|H_1(Y)|=21$. In the notation of normalized Seifert invariants \cite{Orlik1972},
$Y=M(0;b;(2,1),(3,\beta_2),(3,\beta_3))$ with $\beta_i\in\{1,2\}$, and
$|H_1(Y)|=|18b+9+6(\beta_2+\beta_3)|=21$. The only solutions are $b=0$ with
$\beta_2=\beta_3=1$, and $b=-3$ with $\beta_2=\beta_3=2$, which are the two
orientations of one manifold. Indeed, reversing the orientation sends the normalized invariants
$M\bigl(0;b;(\alpha_1,\beta_1),\ldots,(\alpha_n,\beta_n)\bigr)$ to
$M\bigl(0;-n-b;(\alpha_1,\alpha_1-\beta_1),\ldots,(\alpha_n,\alpha_n-\beta_n)\bigr)$
\cite{Orlik1972}, \cite[Section~2.4, formula~(2.1)]{GordonLidman2014}; here
$n=3$.

The Hopf fibration $v\mapsto vi\bar v$ descends to $N'=S^3/\Lambda^+$:
its base action is the tetrahedral rotation group $A_4$, so the base is
$S^2(2,3,3)$; compare Remark~\ref{wh:rem:seifert} and \cite{Scott1983}.
Since $|H_1(N')|=21$ by Theorem~\ref{wh:thm:branched}(i), the same calculation
with normalized Seifert invariants identifies $Y$ with $N'$ up to orientation.
The homeomorphism may be replaced by a diffeomorphism
\cite{Moise1977,Munkres1960}.

Transport the covering involution of $Y$ to $N'$ by this diffeomorphism.
It is orientation-preserving even if the chosen diffeomorphism reverses
orientation. By \cite[Theorem~E and Section~2.1]{DinkelbachLeeb}, the
transported action is smoothly conjugate to an isometric action for the
standard spherical quotient metric on $N'$. Here the use of the fixed
quotient metric is justified by the uniqueness of an elliptic geometric
structure up to diffeomorphism, as recalled in Section~2.1 of that paper;
compare the classical spherical-space-form classification
\cite{ThrelfallSeifert1931,ThrelfallSeifert1933,Wolf2011} and
\cite[Section~8]{deRham1950}. An isometry of $N'$ lifts through its universal
Riemannian covering to an isometry of $S^3$ normalizing $\Lambda^+$.
Thus the covering involution has an orientation-preserving lift $(C,D)$,
and the full lifted group contains $\Lambda^+$ with index two. Therefore
$\OW\cong S^3/\Lambda_W$, where
$\Lambda_W=\langle\Lambda^+,(C,D)\rangle\subset SO(4)$ contains $\Lambda^+$ as a subgroup of index two and normalizes it. We also have
$\Lambda_W\cong\pi_1^{\mathrm{orb}}(\OW)$, so
$\Lambda_W^{\mathrm{ab}}\cong C_6$ by Lemma~\ref{wh:lem:abel}. No element of $O(4)\setminus SO(4)$
normalizes $\Lambda^+$, because an orientation-reversing isometry of $S^3$
interchanges the left and right factors, whereas the two projections of
$\Lambda^+$ are the nonisomorphic groups $\Tst$ and $C_{14}$. The
standard normalizers of binary polyhedral and cyclic subgroups of the unit
quaternions \cite{DuVal1964,ConwaySmith2003} show that normalizing $\Lambda^+$
requires
\[
   C\in N_{S^3}(\Tst)=O^*,\;\;\;
   D\in N_{S^3}(C_{14})=U(1)\cup U(1)j.
\]

\emph{Step 1: $C\in\Tst$.} Suppose instead that $C\in O^*\setminus\Tst$. Then
conjugation by $(C,D)$ restricts to the nontrivial outer automorphism of
$\Tst$ and inverts $\Tst/Q_8\cong C_3$. In
$\Lambda_W^{\mathrm{ab}}$ the image $x$ of $(\omega,1)$ therefore satisfies
$x=x^{-1}$ and $x^3=1$, so $x=1$. The image of
$Q_8\times1\subseteq[\Lambda^+,\Lambda^+]$ is also trivial; hence the entire
left factor $\Tst\times1$ has trivial image in the abelianization.

Now $O^*/Q_8\cong S_3$ and $\Tst/Q_8\cong A_3$. Thus
$C\notin\Tst$ implies $C^2\in Q_8$. Since $(C,D)$ represents the nontrivial
coset of the index-two extension $\Lambda_W/\Lambda^+$, we also have
$(C,D)^2\in\Lambda^+$, and consequently $D^2\in C_{14}$. If $y$ denotes the
image of $(C,D)$ in $\Lambda_W^{\mathrm{ab}}$, this abelianization is therefore
generated by the image of $C_{14}$ and by $y$, with $y^2$ lying in the image
of $C_{14}$. The image of $C_{14}$ has order dividing $14$, so this
abelianization has no $3$-torsion. This contradicts
$\Lambda_W^{\mathrm{ab}}\cong C_6$. Hence, $C\in\Tst$, and by replacing $(C,D)$
with $(C^{-1},1)(C,D)$, we may take $C=1$.

\emph{Step 2: $D\notin U(1)$.} Suppose that $D\in U(1)$. Then the finite cyclic
group $\langle C_{14},D\rangle\subset U(1)$ contains $C_{14}$ with index two,
so it is $C_{28}$ and $\Lambda_W=\Tst\pmx C_{28}$. Since
$[\Tst,\Tst]=Q_8$ contains $-1$, abelianization kills the diagonal relation
$(-1,-1)$ on the left and therefore also kills the order-two element of
$C_{28}$. Hence
\[
  \Lambda_W^{\mathrm{ab}}
     \cong (\Tst)^{\mathrm{ab}}\times
            \bigl(C_{28}/\langle-1\rangle\bigr)
     \cong C_3\times C_{14},
\]
which has order $42$, contradicting
$\Lambda_W^{\mathrm{ab}}\cong C_6$. Thus, $D\in U(1)j$, and
$R=\langle C_{14},D\rangle$ is binary dihedral of order $28$.

\emph{Step 3: conjugation to $\Lambda$.} Conjugation by $(1,u)$ with
$u\in U(1)$ fixes $\Lambda^+$ elementwise and sends $e^{i\varphi}j$ to
$e^{i\varphi}u^{-2}j$. Choosing $u^2=e^{i\varphi}$ gives $D=j$, so $\Lambda_W$
is conjugate to $\Tst\pmx\Qtt=\Lambda$. Hence
$\OW\cong S^3/\Lambda=X$.

For the other mirror image, the corresponding double cover has first homology of
order $3$ (Proposition~\ref{wh:prop:lift}(iv)), so it is not homeomorphic
to $N'$, and the corresponding orbifold is not homeomorphic to $X$: an
orbifold homeomorphism would carry the orbifold cover $N'\to X$ to the
double cover branched along $W'$, since $\Lambda^{\mathrm{ab}}\cong C_6$
has exactly one subgroup of index two and both covers correspond to it.

For the Heegaard description, note that $J$ is unknotted, so the complement of
$\nu J$ in $S^3$ is a solid torus. Its core is a meridian circle of $J$, and
$W'$ is the Whitehead double of that core inside it, untwisted with respect
to the framing induced from $S^3$. The surgery solid torus is the other
Heegaard solid torus.
\end{proof}

\subsection{The surgery circle as a section over a band sum}

Say that an embedded loop $\ell$ in the free part of $X$ is \emph{admissible}
if its class in $\Lambda=\pi_1^{\mathrm{orb}}(X)$ maps to a generator of
$\Lambda^{\mathrm{ab}}\cong C_6$. The class is defined up to conjugacy after
choosing a base path. The restriction $\pi^{-1}(\ell)\to\ell$ of the orbit
map is a principal circle bundle over a circle, hence trivial, and its
sections differ by multiples of the orbit. Under
$p_*:H_1(\M)\xrightarrow{\cong}\ZZ$ the orbit has image $6$. The composite
$G_7\to\Lambda\to\Lambda/\Gamma\cong C_6$ of the surjection of
Theorem~\ref{per:thm:action}(v) with the quotient map kills $\Gamma$ and
sends $t$ to the class of $f^{-1}$, a generator; so it is the base degree
$p_*$ reduced modulo $6$, up to the choice of generator. The orbit map
$\M\to X$ induces this surjection, and a section of $\pi^{-1}(\ell)$ is a
loop in $\M$ whose class maps to the class of $\ell$ in $\Lambda$. For
an admissible $\ell$, the resulting class in $C_6$ is a generator, so every section has base
degree $\equiv\pm1\pmod6$.
Orient $\ell$ so that its sections have base degree $\equiv1\pmod6$.
Here and below, uniqueness of a section means uniqueness up to homotopy
through sections, or, equivalently, up to isotopy through sections. Indeed,
once a trivialization of $\pi^{-1}(\ell)\to\ell$ is chosen, sections are
graphs of smooth maps $\ell\to S^1$. Their homotopy classes form a
$\ZZ$-torsor, and changing the winding number by $k$ changes the base
degree by $6k$. Consequently, precisely one such class has degree one.
Choose any representative of this class and call it the \emph{unit
section} $\gamma_\ell$. A smooth homotopy through sections is an isotopy
of their embedded graphs; isotopy extension therefore makes the
pair obtained by framed surgery independent of this representative,
with the framings transported by the isotopy \cite[Chapter~8]{Hirsch}.

\begin{theorem}[Band-sum description of $\gamma$]\label{wh:thm:section}
\begin{enumerate}[label=(\roman*),leftmargin=2.2em]
\item For every admissible $\ell$, the unit section $\gamma_\ell$ is a
      weight circle. Hence, by Lemma~\ref{lem:weight-loops}, its framed
      surgeries are represented by \eqref{eq:candidates}, up to the
      orientation of the belt sphere.
\item In the Whitehead model of Theorem~\ref{wh:thm:whitehead}, let $\ell$
      be any band sum, in $E_{\mathcal S}$, of a meridian $\mu_J$ of the
      surgery curve and a meridian $\mu_{W'}$ of $\mathcal S=W'$. Then
      $\ell$ is admissible and $\lk(\ell,\mathcal S)=\pm1$.
\item In particular, $\Sigma_+(7)$ and $\Sigma_-(7)$ are the two framed
      surgeries on $\gamma_\ell$ for such an $\ell$, with the two
      framings identified only as an unordered pair.
\end{enumerate}
\end{theorem}

\begin{proof}
(i) By construction $\gamma_\ell$ is embedded and has degree one under
$p_*$. Lemma~\ref{lem:weight-loops} shows that every such element normally
generates $G_7$, so $\gamma_\ell$ is a weight circle in the group-theoretic
sense. The same lemma identifies the unordered pair of framed-surgery
outputs.

(ii) The homeomorphism of Theorem~\ref{wh:thm:whitehead} induces an
isomorphism $\pi_1^{\mathrm{orb}}(\OW)\cong\Lambda$, hence an isomorphism of
abelianizations $H_1^{\mathrm{orb}}(\OW)\cong\Lambda^{\mathrm{ab}}=C_6$. By
Lemma~\ref{wh:lem:abel}, the image of $\mu_J$ has order three and the image
of $\mu_{W'}$ has order two. (The latter is also visible in $\Lambda$: the
meridian $\mu_{W'}$ maps to one of the $42$ involutions, whose left part
lies in $Q_8$ and whose right part lies in $\Qtt\setminus C_{14}$.)
A band sum represents a based product
$\mu_J\cdot g\mu_{W'}^{\pm1}g^{-1}$. Its image is the product of these two
images, which generates $C_6$. For the linking number,
$\lk(\mu_J,W')=0$ and $\lk(\mu_{W'},W')=\pm1$. Here $\lk(\cdot,\mathcal S)$ is the intersection number with a Seifert
surface for the null-homologous knot $\mathcal S$ in the rational homology
sphere $|X|$ (Theorem~\ref{wh:thm:branched}(iii)); it depends only on the
class in $H_1(E_{\mathcal S})$ and is therefore additive under band sums.
Since $\lk(J,W')=0$ in $S^3$, the knot $W'$ bounds a Seifert surface in
$S^3\setminus\nu J\subset|X|$. This surface is disjoint from a meridian $\mu_J$
chosen on $\partial\nu J$; this gives $\lk(\mu_J,W')=0$.

(iii) Combine (i) and (ii). The unordered pair is the pair described in
Lemma~\ref{lem:weight-loops} and Theorem~\ref{thm:rigidity}.
\end{proof}

Figure~\ref{wh:fig:section} illustrates admissibility and the unit section.

\begin{figure}[tbp]
\centering
\begin{tikzpicture}[
  ida/.style={-{Stealth[length=1.4mm]},gray!60!black},
  idb/.style={-{Stealth[length=1.4mm] Stealth[length=1.4mm]},gray!60!black},
  strand/.style={line width=1.1pt,line cap=round},
  loopst/.style={draw=liftblue,line width=0.9pt,line cap=round,line join=round},
  over/.style={preaction={draw=white,line width=3.6pt}}]
\begin{scope}
\def\ex{0.38}\def\ey{0.17}
\fill[bandfill,opacity=0.55] (0.356,-0.06) rectangle (1.244,0.06);
\draw[loopst] (0.356,0.06) -- (1.244,0.06) arc[start angle=159.3,end angle=0,x radius=\ex,y radius=\ey];
\draw[loopst] (-0.38,0) arc[start angle=-180,end angle=-339.3,x radius=\ex,y radius=\ey];
\draw[loopst] (1.244,-0.06) -- (0.356,-0.06);
\draw[strand,diagramblue!80!black,over] (0,-0.85) -- (0,0.85);
\draw[strand,orbitred,over,line width=1.5pt] (1.6,-0.85) -- (1.6,0.85);
\draw[loopst,over] (1.98,0) arc[start angle=0,end angle=-159.3,x radius=\ex,y radius=\ey];
\draw[loopst,over] (0.356,-0.06) arc[start angle=-20.7,end angle=-180,x radius=\ex,y radius=\ey];
\draw[-{Stealth[length=1.6mm,width=1.25mm]},liftblue,line width=0.9pt] (0.76,0.06) -- (0.84,0.06);
\node[slab,anchor=south] at (0,0.87) {$J$, $+3$};
\node[slab,anchor=south,text=orbitred,align=center] at (1.6,0.87) {$W'$, cone\\[-1pt] angle $\pi$};
\draw[gray!55,line width=0.4pt] (0.56,0.08) -- (0.56,0.5) node[above,slab,text=inkgray,inner sep=1pt] {band};
\node[slab,text=liftblue] at (0.8,-1.1) {$\ell=\mu_J\natural\mu_{W'}$};
\node[panelcap] at (0.8,-1.75) {(a) a band sum of meridians};
\end{scope}
\begin{scope}[shift={(3.8,0)}]
\fill[diagramlight] (0,0) circle[radius=1.05];
\draw[gray!40] (0,0) circle[radius=0.9];
\foreach \k in {0,...,5} \coordinate (c\k) at ({90-60*\k}:0.9);
\draw[gray!45,line width=0.45pt] (c0)--(c1)--(c2)--(c3)--(c4)--(c5)--cycle;
\draw[-{Stealth[length=1.8mm,width=1.4mm]},liftblue,line width=0.9pt,shorten <=3pt,shorten >=3pt] (c0) -- (c2);
\draw[-{Stealth[length=1.8mm,width=1.4mm]},orbitred,line width=0.9pt,shorten <=3pt,shorten >=3pt] (c2) -- (c5);
\fill (c0) circle[radius=0.06];
\fill[liftblue] (c2) circle[radius=0.07] (c4) circle[radius=0.07];
\fill[orbitred] (c3) circle[radius=0.07];
\foreach \k in {1,5} {\fill (c\k) circle[radius=0.045]; \draw[line width=0.8pt] (c\k) circle[radius=0.12];}
\foreach \k in {0,...,5} \node[slab] at ({90-60*\k}:1.2) {$\k$};
\node[slab,anchor=west,text=liftblue] at (1.32,0.42) {$\chi(\mu_J)\in\{2,4\}$: order $3$};
\node[slab,anchor=west,text=orbitred] at (1.32,0.07) {$\chi(\mu_{W'})=3$: order $2$};
\node[slab,anchor=west] at (1.32,-0.28) {$\chi(\ell)=2+3=5$: a generator};
\node[slab,anchor=west,text=inkgray] at (1.32,-0.63) {generators: $1,5$ (ringed)};
\node[panelcap] at (1.45,-1.75) {(b) $\Lambda^{\mathrm{ab}}\cong C_6$, additively};
\end{scope}
\begin{scope}[shift={(9.35,-0.95)}]
\shade[bottom color=diagramlight,top color=diagramblue!12] (0,0) rectangle (2.2,2.0);
\draw[ida] (0,0) -- (1.2,0); \draw[gray!60!black] (1.1,0) -- (2.2,0);
\draw[ida] (0,2.0) -- (1.2,2.0); \draw[gray!60!black] (1.1,2.0) -- (2.2,2.0);
\draw[idb] (0,0) -- (0,1.1); \draw[gray!60!black] (0,1.0) -- (0,2.0);
\draw[idb] (2.2,0) -- (2.2,1.1); \draw[gray!60!black] (2.2,1.0) -- (2.2,2.0);
\draw[orbitred,line width=1.2pt,-{Stealth[length=1.6mm]}] (0,0.55) -- (2.2,0.55);
\draw[liftblue,line width=0.9pt,-{Stealth[length=1.5mm]}] (0,1.0) -- (1.1,2.0) (1.1,0) -- (2.2,1.0);
\draw[gray!70,densely dashed,line width=0.9pt] (1.65,0) -- (1.65,2.0);
\node[slab,text=orbitred,anchor=west] at (2.3,0.55) {$k=0$: $\gamma_\ell$, degree $1$};
\node[slab,text=liftblue,anchor=west] at (2.3,1.2) {$k=1$: degree $7$};
\node[slab,text=gray!60!black,anchor=west] at (2.3,1.85) {orbit: degree $6$};
\node[slab,below] at (1.1,-0.05) {$\ell$};
\node[slab,left] at (-0.05,1.0) {$S^1$};
\node[panelcap] at (1.75,-0.8) {(c) sections of $T_\ell=\pi^{-1}(\ell)$: degree $1+6k$};
\end{scope}
\end{tikzpicture}
\caption{Admissible loops and unit sections (Theorem~\ref{wh:thm:section}). (a) In
the Whitehead model a band sum $\ell=\mu_J\natural\mu_{W'}$ of a meridian of the
surgery curve $J$ and a meridian of the singular circle $W'$, drawn locally: each
meridian passes behind its strand at the top and in front of it at the bottom, and
the shaded band joins them. (b) The abelianization
$\chi\colon\Lambda\cong\pi_1^{\mathrm{orb}}(\OW)\to\Lambda^{\mathrm{ab}}\cong C_6$, written
additively. A meridian of $J$ maps to an element of order three and a
meridian of $W'$ to the element of order two (Lemma~\ref{wh:lem:abel}). Hence the
band sum maps to a generator, here $2+3=5$, and is admissible. (c) After orienting
$\ell$ so that section degrees are $1$ modulo six and using the unit section as the
zero-winding reference, the torus $T_\ell=\pi^{-1}(\ell)$ is a trivial circle
bundle, drawn as a square with opposite sides identified. Since the orbit has
degree six, a section that winds $k$ more times around the orbit direction than
the unit section $\gamma_\ell$ has degree $1+6k$; exactly one class of sections has
degree one.}
\label{wh:fig:section}
\end{figure}

\subsection{The circle orbibundle model}

For a closed orientable three-orbifold $\mathcal O$ with finite orbifold
fundamental group $\Pi$, universal orbifold cover $S^3$, and a character
$\chi:\Pi\to S^1$ injective on every isotropy subgroup, define
\[
 M(\mathcal O,\chi)=(S^3\times S^1)/\Pi,\;\;\;
 g\cdot(v,z)=(gv,\chi(g)z).
\]
This diagonal action is free: an element fixing $(v,z)$ lies in the isotropy
subgroup at $v$ and in $\ker\chi$, hence is the identity. Rotation of the
second factor commutes with the action and defines the orbit map to
$\mathcal O$.

\begin{corollary}[Whitehead orbibundle description]\label{wh:cor:input}
Let $\chi$ be either of the two surjections
$\pi_1^{\mathrm{orb}}(\OW)\to C_6\subset S^1$. Then $\chi(\mu_J)$ has order
three and $\chi(\mu_{W'})$ has order two. Up to orientation,
\[
 \M\diff M(\OW,\chi).
\]
For every band sum $\ell=\mu_J\natural\mu_{W'}$, the manifolds
$\Sigma_\pm(7)$ are the two framed surgeries on the unit section
$\gamma_\ell$. Consequently, the framed-surgery pair is determined by a
single two-component link diagram: the Whitehead link, with coefficient
$+3$ on $J$ and cone angle $\pi$ on $W'$, together with a band and a choice of
framing parity.
\end{corollary}

\begin{proof}
Lemma~\ref{wh:lem:model} gives $\M\diff M(X,\psi)$.
The space $M(\OW,\chi)$ is defined: by Theorem~\ref{wh:thm:whitehead} the
universal orbifold cover of $\OW$ is $S^3$, and every isotropy subgroup of
$\OW$ is conjugate to $\langle\mu_{W'}\rangle$, on which $\chi$ is injective
because $\chi(\mu_{W'})$ has order two. The orbifold
homeomorphism $\OW\cong X$ of Theorem~\ref{wh:thm:whitehead} is covered by
a diffeomorphism $\Phi:S^3\to S^3$ of the universal covers that is
equivariant with respect to an isomorphism
$\varphi:\pi_1^{\mathrm{orb}}(\OW)\to\Lambda$: the smooth conjugacy of the
covering involution of $Y\cong N'$ to an isometry lifts to the universal
covers equivariantly with respect to the lifted groups, and the remaining
identifications in that proof are isometries of $S^3$. For a character
$\chi=\psi\circ\varphi$, the map $(x,u)\mapsto(\Phi(x),u)$ descends to a
diffeomorphism $M(\OW,\chi)\to M(X,\psi)$ whose effect on orientation is the same as that of $\Phi$. The group $\Lambda^{\mathrm{ab}}$
is $C_6$, and its only automorphisms are $\pm1$. So the induced character is
$\chi^{\pm1}$, and replacing $\chi$ by $\chi^{-1}$ reverses the circle. The
orders of $\chi(\mu_J)$ and $\chi(\mu_{W'})$ were computed in the proof of
Theorem~\ref{wh:thm:section}(ii). The remaining statement follows from
Theorem~\ref{wh:thm:section}(iii).
\end{proof}

\section{Surgery on a section of a fibre torus: both surgery manifolds are standard}
\label{std:sec:main}

This section proves that both surgery manifolds are diffeomorphic to the
standard four-sphere.

\begin{theorem}[Both surgery manifolds are standard]\label{std:thm:main}
\[
 \Sigma_+(7)\diff S^4\;\;\;\text{and}\;\;\;\Sigma_-(7)\diff S^4.
\]
In particular, neither resulting manifold is an exotic four-sphere.
\end{theorem}

The proof uses a weight circle on an invariant torus of the canonical
circle action. Surgery on the weight circle compresses this torus
to a two-sphere. Surgery on that sphere completes a torus replacement for
which the circle action on the unchanged exterior extends over the new
piece. The orbit space of the resulting action is the Dehn filling described in
Proposition~\ref{std:prop:torus}; no circle action on the intermediate
candidate is assumed. The candidate is recovered by surgery on the core
circle of the sphere surgery. When the new base orbifold has a cyclic or odd
dihedral orbifold fundamental group, the recognition criteria identify the
candidate with $S^4$ (Propositions~\ref{std:prop:cyclic}
and~\ref{std:prop:dihedral}). In the Whitehead model of
Section~\ref{wh:sec:whitehead}, a clasp belt produces these two cases for
consecutive framings (Theorem~\ref{std:thm:belt}). These framings realize both
normal-framing classes of the weight circle (Lemma~\ref{std:lem:frame}), so
no identification of either belt framing with the specified framing
$\lambda_+$ is required.
Figure~\ref{fig:standardness-picture} shows the whole argument in one picture.
\begin{figure}[tbp]
\centering
\begin{tikzpicture}[
  jstrand/.style={line width=0.9pt,line cap=round,line join=round,draw=diagramblue!75!black},
  wstrand/.style={line width=0.9pt,line cap=round,line join=round,draw=orbitred},
  beltst/.style={line width=0.85pt,line cap=round,line join=round,draw=belt},
  orbit/.style={draw=gray!35!black,line width=0.5pt},
  orbitback/.style={draw=gray!55,line width=0.4pt,densely dashed},
  sect/.style={draw=orbitred,line width=1.1pt},
  sectback/.style={draw=orbitred!55,line width=0.75pt,densely dashed},
  trefoilst/.style={draw=orbitred,line width=0.95pt,line cap=round},
  res/.style={draw=#1!75,fill=#1!7,rounded corners=3pt,line width=0.55pt,align=flush center,
    font=\scriptsize,inner sep=3pt,text width=2.85cm},
  icon/.style={draw=diagramblue!30,fill=white,rounded corners=4pt,line width=0.4pt,
    minimum width=1.5cm,minimum height=1.35cm}]
\begin{scope}[scale=0.56]
\draw[jstrand] (1.8153,0.8394) arc[start angle=24.8142, end angle=196.2199, radius=2.0];
\draw[jstrand] (-1.8153,-0.8394) arc[start angle=204.8142, end angle=349.0271, radius=2.0];
\draw[jstrand] (1.9983,-0.0830) arc[start angle=357.6215, end angle=376.2199, radius=2.0];
\draw[wstrand] plot[domain=1.6122:4.0546, samples=70] ({3.3*cos(\x r)},{1.5*sin(\x r)*cos(\x r)});
\draw[wstrand] plot[domain=4.1630:5.2618, samples=35] ({3.3*cos(\x r)},{1.5*sin(\x r)*cos(\x r)});
\draw[wstrand] plot[domain=5.3702:6.0178, samples=24] ({3.3*cos(\x r)},{1.5*sin(\x r)*cos(\x r)});
\draw[wstrand] plot[domain=6.2153:7.8126, samples=48] ({3.3*cos(\x r)},{1.5*sin(\x r)*cos(\x r)});
\draw[beltst] (3.0991,0.2712) arc[start angle=64.6764, end angle=118.9091, x radius=1.05, y radius=0.3];
\draw[beltst] (1.8518,0.1949) arc[start angle=139.4810, end angle=405.0515, x radius=1.05, y radius=0.3];
\coordinate (beltpt) at (3.559,0.15);
\node[slab,text=diagramblue!75!black] at (0,-2.45) {$J$, $+3$};
\node[slab,text=orbitred,align=center,anchor=south east] at (-2.3,1.35) {$W'$, cone\\[-1pt] angle $\pi$};
\node[slab,text=belt] at (3.95,-0.45) {$\ell_A$};
\end{scope}
\node[panelcap,anchor=north] at (0,-2.15) {(a) the base $\OW$ and the clasp belt};
\begin{scope}[shift={(3.95,1.45)},scale=0.62]
\shade[left color=diagramlight,right color=diagramblue!22] plot[smooth cycle] coordinates {(1.700,0.000) (1.696,0.079) (1.685,0.157) (1.666,0.232) (1.640,0.305) (1.609,0.374) (1.572,0.440) (1.531,0.501) (1.486,0.559) (1.438,0.613) (1.387,0.663) (1.333,0.709) (1.276,0.753) (1.218,0.793) (1.157,0.831) (1.095,0.866) (1.030,0.898) (0.964,0.928) (0.896,0.956) (0.827,0.981) (0.757,1.004) (0.685,1.025) (0.612,1.043) (0.538,1.059) (0.463,1.073) (0.387,1.085) (0.310,1.095) (0.233,1.103) (0.156,1.108) (0.078,1.111) (0.000,1.112) (-0.078,1.111) (-0.156,1.108) (-0.233,1.103) (-0.310,1.095) (-0.387,1.085) (-0.463,1.073) (-0.538,1.059) (-0.612,1.043) (-0.685,1.025) (-0.757,1.004) (-0.827,0.981) (-0.896,0.956) (-0.964,0.928) (-1.030,0.898) (-1.095,0.866) (-1.157,0.831) (-1.218,0.793) (-1.276,0.753) (-1.333,0.709) (-1.387,0.663) (-1.438,0.613) (-1.486,0.559) (-1.531,0.501) (-1.572,0.440) (-1.609,0.374) (-1.640,0.305) (-1.666,0.232) (-1.685,0.157) (-1.696,0.079) (-1.700,0.000) (-1.696,-0.079) (-1.685,-0.157) (-1.666,-0.232) (-1.640,-0.305) (-1.609,-0.374) (-1.572,-0.440) (-1.531,-0.501) (-1.486,-0.559) (-1.438,-0.613) (-1.387,-0.663) (-1.333,-0.709) (-1.276,-0.753) (-1.218,-0.793) (-1.157,-0.831) (-1.095,-0.866) (-1.030,-0.898) (-0.964,-0.928) (-0.896,-0.956) (-0.827,-0.981) (-0.757,-1.004) (-0.685,-1.025) (-0.612,-1.043) (-0.538,-1.059) (-0.463,-1.073) (-0.387,-1.085) (-0.310,-1.095) (-0.233,-1.103) (-0.156,-1.108) (-0.078,-1.111) (-0.000,-1.112) (0.078,-1.111) (0.156,-1.108) (0.233,-1.103) (0.310,-1.095) (0.387,-1.085) (0.463,-1.073) (0.538,-1.059) (0.612,-1.043) (0.685,-1.025) (0.757,-1.004) (0.827,-0.981) (0.896,-0.956) (0.964,-0.928) (1.030,-0.898) (1.095,-0.866) (1.157,-0.831) (1.218,-0.793) (1.276,-0.753) (1.333,-0.709) (1.387,-0.663) (1.438,-0.613) (1.486,-0.559) (1.531,-0.501) (1.572,-0.440) (1.609,-0.374) (1.640,-0.305) (1.666,-0.232) (1.685,-0.157) (1.696,-0.079)};
\fill[white] plot[smooth] coordinates {(0.778,-0.000) (0.764,0.012) (0.746,0.026) (0.725,0.040) (0.701,0.055) (0.673,0.071) (0.643,0.086) (0.609,0.101) (0.573,0.116) (0.534,0.130) (0.493,0.143) (0.450,0.156) (0.405,0.167) (0.358,0.177) (0.310,0.187) (0.260,0.194) (0.210,0.201) (0.158,0.206) (0.106,0.209) (0.053,0.212) (0.000,0.212) (-0.053,0.212) (-0.106,0.209) (-0.158,0.206) (-0.210,0.201) (-0.260,0.194) (-0.310,0.187) (-0.358,0.177) (-0.405,0.167) (-0.450,0.156) (-0.493,0.143) (-0.534,0.130) (-0.573,0.116) (-0.609,0.101) (-0.643,0.086) (-0.673,0.071) (-0.701,0.055) (-0.725,0.040) (-0.746,0.026) (-0.764,0.012) (-0.778,-0.000)} -- plot[smooth] coordinates {(-0.778,0.000) (-0.764,-0.012) (-0.746,-0.026) (-0.725,-0.040) (-0.701,-0.055) (-0.673,-0.071) (-0.643,-0.086) (-0.609,-0.101) (-0.573,-0.116) (-0.534,-0.130) (-0.493,-0.143) (-0.450,-0.156) (-0.405,-0.167) (-0.358,-0.177) (-0.310,-0.187) (-0.260,-0.194) (-0.210,-0.201) (-0.158,-0.206) (-0.106,-0.209) (-0.053,-0.212) (-0.000,-0.212) (0.053,-0.212) (0.106,-0.209) (0.158,-0.206) (0.210,-0.201) (0.260,-0.194) (0.310,-0.187) (0.358,-0.177) (0.405,-0.167) (0.450,-0.156) (0.493,-0.143) (0.534,-0.130) (0.573,-0.116) (0.609,-0.101) (0.643,-0.086) (0.673,-0.071) (0.701,-0.055) (0.725,-0.040) (0.746,-0.026) (0.764,-0.012) (0.778,0.000)} -- cycle;
\draw[diagramblue!55,line width=0.45pt] plot[smooth cycle] coordinates {(1.700,0.000) (1.696,0.079) (1.685,0.157) (1.666,0.232) (1.640,0.305) (1.609,0.374) (1.572,0.440) (1.531,0.501) (1.486,0.559) (1.438,0.613) (1.387,0.663) (1.333,0.709) (1.276,0.753) (1.218,0.793) (1.157,0.831) (1.095,0.866) (1.030,0.898) (0.964,0.928) (0.896,0.956) (0.827,0.981) (0.757,1.004) (0.685,1.025) (0.612,1.043) (0.538,1.059) (0.463,1.073) (0.387,1.085) (0.310,1.095) (0.233,1.103) (0.156,1.108) (0.078,1.111) (0.000,1.112) (-0.078,1.111) (-0.156,1.108) (-0.233,1.103) (-0.310,1.095) (-0.387,1.085) (-0.463,1.073) (-0.538,1.059) (-0.612,1.043) (-0.685,1.025) (-0.757,1.004) (-0.827,0.981) (-0.896,0.956) (-0.964,0.928) (-1.030,0.898) (-1.095,0.866) (-1.157,0.831) (-1.218,0.793) (-1.276,0.753) (-1.333,0.709) (-1.387,0.663) (-1.438,0.613) (-1.486,0.559) (-1.531,0.501) (-1.572,0.440) (-1.609,0.374) (-1.640,0.305) (-1.666,0.232) (-1.685,0.157) (-1.696,0.079) (-1.700,0.000) (-1.696,-0.079) (-1.685,-0.157) (-1.666,-0.232) (-1.640,-0.305) (-1.609,-0.374) (-1.572,-0.440) (-1.531,-0.501) (-1.486,-0.559) (-1.438,-0.613) (-1.387,-0.663) (-1.333,-0.709) (-1.276,-0.753) (-1.218,-0.793) (-1.157,-0.831) (-1.095,-0.866) (-1.030,-0.898) (-0.964,-0.928) (-0.896,-0.956) (-0.827,-0.981) (-0.757,-1.004) (-0.685,-1.025) (-0.612,-1.043) (-0.538,-1.059) (-0.463,-1.073) (-0.387,-1.085) (-0.310,-1.095) (-0.233,-1.103) (-0.156,-1.108) (-0.078,-1.111) (-0.000,-1.112) (0.078,-1.111) (0.156,-1.108) (0.233,-1.103) (0.310,-1.095) (0.387,-1.085) (0.463,-1.073) (0.538,-1.059) (0.612,-1.043) (0.685,-1.025) (0.757,-1.004) (0.827,-0.981) (0.896,-0.956) (0.964,-0.928) (1.030,-0.898) (1.095,-0.866) (1.157,-0.831) (1.218,-0.793) (1.276,-0.753) (1.333,-0.709) (1.387,-0.663) (1.438,-0.613) (1.486,-0.559) (1.531,-0.501) (1.572,-0.440) (1.609,-0.374) (1.640,-0.305) (1.666,-0.232) (1.685,-0.157) (1.696,-0.079)};
\draw[diagramblue!55,line width=0.45pt] plot[smooth] coordinates {(0.778,-0.000) (0.764,0.012) (0.746,0.026) (0.725,0.040) (0.701,0.055) (0.673,0.071) (0.643,0.086) (0.609,0.101) (0.573,0.116) (0.534,0.130) (0.493,0.143) (0.450,0.156) (0.405,0.167) (0.358,0.177) (0.310,0.187) (0.260,0.194) (0.210,0.201) (0.158,0.206) (0.106,0.209) (0.053,0.212) (0.000,0.212) (-0.053,0.212) (-0.106,0.209) (-0.158,0.206) (-0.210,0.201) (-0.260,0.194) (-0.310,0.187) (-0.358,0.177) (-0.405,0.167) (-0.450,0.156) (-0.493,0.143) (-0.534,0.130) (-0.573,0.116) (-0.609,0.101) (-0.643,0.086) (-0.673,0.071) (-0.701,0.055) (-0.725,0.040) (-0.746,0.026) (-0.764,0.012) (-0.778,-0.000)};
\draw[diagramblue!55,line width=0.45pt] plot[smooth] coordinates {(-0.778,0.000) (-0.764,-0.012) (-0.746,-0.026) (-0.725,-0.040) (-0.701,-0.055) (-0.673,-0.071) (-0.643,-0.086) (-0.609,-0.101) (-0.573,-0.116) (-0.534,-0.130) (-0.493,-0.143) (-0.450,-0.156) (-0.405,-0.167) (-0.358,-0.177) (-0.310,-0.187) (-0.260,-0.194) (-0.210,-0.201) (-0.158,-0.206) (-0.106,-0.209) (-0.053,-0.212) (-0.000,-0.212) (0.053,-0.212) (0.106,-0.209) (0.158,-0.206) (0.210,-0.201) (0.260,-0.194) (0.310,-0.187) (0.358,-0.177) (0.405,-0.167) (0.450,-0.156) (0.493,-0.143) (0.534,-0.130) (0.573,-0.116) (0.609,-0.101) (0.643,-0.086) (0.673,-0.071) (0.701,-0.055) (0.725,-0.040) (0.746,-0.026) (0.764,-0.012) (0.778,0.000)};
\draw[orbitback] plot[smooth] coordinates {(1.700,0.000) (1.699,0.020)};
\draw[orbit] plot[smooth] coordinates {(1.699,0.020) (1.685,0.099) (1.651,0.173) (1.600,0.240) (1.533,0.297) (1.454,0.340) (1.366,0.369) (1.274,0.381) (1.180,0.377) (1.089,0.356) (1.005,0.320) (0.932,0.270) (0.873,0.208) (0.830,0.137) (0.806,0.060) (0.801,-0.020)};
\draw[orbitback] plot[smooth] coordinates {(0.801,-0.020) (0.815,-0.099) (0.849,-0.173) (0.900,-0.240) (0.967,-0.297) (1.046,-0.340) (1.134,-0.369) (1.226,-0.381) (1.297,-0.380) (1.389,-0.363) (1.475,-0.330) (1.551,-0.284) (1.614,-0.224) (1.661,-0.155) (1.690,-0.079) (1.700,-0.000)};
\draw[orbitback] plot[smooth] coordinates {(1.093,0.690) (1.086,0.765) (1.068,0.830) (1.028,0.890) (0.986,0.919)};
\draw[orbit] plot[smooth] coordinates {(0.986,0.919) (0.935,0.930) (0.878,0.923) (0.819,0.898) (0.758,0.856) (0.700,0.798) (0.646,0.728) (0.599,0.648) (0.561,0.562) (0.533,0.474) (0.518,0.387) (0.515,0.305) (0.524,0.232) (0.546,0.171) (0.579,0.125) (0.621,0.096)};
\draw[orbitback] plot[smooth] coordinates {(0.621,0.096) (0.672,0.084) (0.729,0.092) (0.788,0.117) (0.849,0.159) (0.907,0.217) (0.948,0.268) (0.997,0.346) (1.037,0.431) (1.068,0.519) (1.086,0.607) (1.093,0.690)};
\draw[orbitback] plot[smooth] coordinates {(-1.093,0.690) (-1.086,0.765) (-1.068,0.830) (-1.028,0.890) (-0.986,0.919)};
\draw[orbit] plot[smooth] coordinates {(-0.986,0.919) (-0.935,0.930) (-0.878,0.923) (-0.819,0.898) (-0.758,0.856) (-0.700,0.798) (-0.646,0.728) (-0.599,0.648) (-0.561,0.562) (-0.533,0.474) (-0.518,0.387) (-0.515,0.305) (-0.524,0.232) (-0.546,0.171) (-0.579,0.125) (-0.621,0.096)};
\draw[orbitback] plot[smooth] coordinates {(-0.621,0.096) (-0.672,0.084) (-0.729,0.092) (-0.788,0.117) (-0.849,0.159) (-0.907,0.217) (-0.948,0.268) (-0.997,0.346) (-1.037,0.431) (-1.068,0.519) (-1.086,0.607) (-1.093,0.690)};
\draw[orbitback] plot[smooth] coordinates {(-1.700,0.000) (-1.699,0.020)};
\draw[orbit] plot[smooth] coordinates {(-1.699,0.020) (-1.685,0.099) (-1.651,0.173) (-1.600,0.240) (-1.533,0.297) (-1.454,0.340) (-1.366,0.369) (-1.274,0.381) (-1.180,0.377) (-1.089,0.356) (-1.005,0.320) (-0.932,0.270) (-0.873,0.208) (-0.830,0.137) (-0.806,0.060) (-0.801,-0.020)};
\draw[orbitback] plot[smooth] coordinates {(-0.801,-0.020) (-0.815,-0.099) (-0.849,-0.173) (-0.900,-0.240) (-0.967,-0.297) (-1.046,-0.340) (-1.134,-0.369) (-1.226,-0.381) (-1.297,-0.380) (-1.389,-0.363) (-1.475,-0.330) (-1.551,-0.284) (-1.614,-0.224) (-1.661,-0.155) (-1.690,-0.079) (-1.700,-0.000)};
\draw[orbit] plot[smooth] coordinates {(-1.093,-0.690) (-1.086,-0.607) (-1.068,-0.519) (-1.037,-0.431) (-0.997,-0.346) (-0.948,-0.268) (-0.893,-0.201) (-0.834,-0.147) (-0.773,-0.109) (-0.714,-0.088) (-0.659,-0.086) (-0.610,-0.102)};
\draw[orbitback] plot[smooth] coordinates {(-0.610,-0.102) (-0.569,-0.135) (-0.539,-0.185) (-0.521,-0.249) (-0.514,-0.325) (-0.521,-0.408) (-0.539,-0.496) (-0.569,-0.584) (-0.610,-0.669) (-0.659,-0.747) (-0.714,-0.814) (-0.773,-0.868) (-0.834,-0.906) (-0.893,-0.927) (-0.948,-0.929) (-0.997,-0.913)};
\draw[orbit] plot[smooth] coordinates {(-0.997,-0.913) (-1.037,-0.880) (-1.068,-0.830) (-1.086,-0.765) (-1.093,-0.690)};
\draw[orbit] plot[smooth] coordinates {(1.093,-0.690) (1.086,-0.607) (1.068,-0.519) (1.037,-0.431) (0.997,-0.346) (0.948,-0.268) (0.893,-0.201) (0.834,-0.147) (0.773,-0.109) (0.714,-0.088) (0.659,-0.086) (0.610,-0.102)};
\draw[orbitback] plot[smooth] coordinates {(0.610,-0.102) (0.569,-0.135) (0.539,-0.185) (0.521,-0.249) (0.514,-0.325) (0.521,-0.408) (0.539,-0.496) (0.569,-0.584) (0.610,-0.669) (0.659,-0.747) (0.714,-0.814) (0.773,-0.868) (0.834,-0.906) (0.893,-0.927) (0.948,-0.929) (0.997,-0.913)};
\draw[orbit] plot[smooth] coordinates {(0.997,-0.913) (1.037,-0.880) (1.068,-0.830) (1.086,-0.765) (1.093,-0.690)};
\draw[sect] plot[smooth] coordinates {(1.461,0.337) (1.458,0.391) (1.447,0.445) (1.429,0.498) (1.405,0.550) (1.373,0.602) (1.335,0.652) (1.290,0.700) (1.239,0.747) (1.182,0.792) (1.119,0.835) (1.051,0.875) (0.978,0.912) (0.900,0.947) (0.817,0.979) (0.731,1.008) (0.641,1.033) (0.547,1.055) (0.452,1.073) (0.354,1.088) (0.254,1.100) (0.153,1.107) (0.051,1.111) (-0.051,1.111) (-0.153,1.107) (-0.254,1.100) (-0.354,1.088) (-0.452,1.073) (-0.547,1.055) (-0.641,1.033) (-0.731,1.008) (-0.817,0.979) (-0.900,0.947) (-0.978,0.912) (-1.051,0.875) (-1.119,0.835) (-1.182,0.792) (-1.239,0.747) (-1.290,0.700) (-1.335,0.652) (-1.373,0.602) (-1.405,0.550) (-1.429,0.498) (-1.447,0.445) (-1.458,0.391) (-1.461,0.337) (-1.458,0.283) (-1.447,0.229) (-1.429,0.176) (-1.405,0.124) (-1.373,0.072) (-1.335,0.022) (-1.290,-0.027) (-1.239,-0.073) (-1.182,-0.118) (-1.119,-0.161) (-1.051,-0.201) (-0.978,-0.239) (-0.900,-0.273) (-0.817,-0.305) (-0.731,-0.334) (-0.641,-0.359) (-0.547,-0.381) (-0.452,-0.399) (-0.354,-0.414) (-0.254,-0.426) (-0.153,-0.433) (-0.051,-0.437) (0.051,-0.437) (0.153,-0.433) (0.254,-0.426) (0.354,-0.414) (0.452,-0.399) (0.547,-0.381) (0.641,-0.359) (0.731,-0.334) (0.817,-0.305) (0.900,-0.273) (0.978,-0.239) (1.051,-0.201) (1.119,-0.161) (1.182,-0.118) (1.239,-0.073) (1.290,-0.027) (1.335,0.022) (1.373,0.072) (1.405,0.124) (1.429,0.176) (1.447,0.229) (1.458,0.283) (1.461,0.337)};
\coordinate (glabel) at (0.254,1.100);
\coordinate (olabel) at (1.093,-0.690);
\draw[orbitred!70,line width=0.4pt] (glabel) -- ++(1.2,0.6) node[right,slab,text=orbitred,inner sep=1pt] {$\gamma$};
\end{scope}
\node[slab,anchor=east] at (3.35,2.3) {$T_{\ell_A}=\pi^{-1}(\ell_A)\subset\M$};
\draw[darrow,shorten >=0.6pt] (3.75,0.78) to[out=-90,in=35] node[pos=0.3,right,slab, yshift=-0.2cm]{$\pi$} (beltpt);
\node[panelcap,align=center,anchor=north] at (4.5,-2.15) {(b) the unit section\\ over the belt};
\begin{scope}[shift={(7.25,0)}]
\node[res=rigidcol] (r1) at (1.3,1.1) {$Q_0$: singular trefoil\\ in $S^3$,\ $\stdpi(Q_0)\cong S_3$};
\node[res=stdcol] (r2) at (1.3,-0.95) {$Q_{-1}$: singular core\\ in $L(4,1)$,\ $\stdpi(Q_{-1})\cong\ZZ/8$};
\node[icon] (ic1) at (3.85,1.1) {};
\begin{scope}[shift={(ic1.center)},scale=0.9]
\draw[trefoilst] plot[smooth] coordinates {(-0.064,0.235) (-0.079,0.242) (-0.094,0.249) (-0.110,0.254) (-0.126,0.259) (-0.142,0.263) (-0.158,0.266) (-0.175,0.267) (-0.192,0.268) (-0.208,0.268) (-0.222,0.266) (-0.238,0.263) (-0.254,0.259) (-0.270,0.254) (-0.285,0.246) (-0.299,0.237) (-0.312,0.227) (-0.323,0.215) (-0.333,0.202) (-0.342,0.188) (-0.348,0.172) (-0.353,0.156) (-0.355,0.140) (-0.356,0.123) (-0.354,0.107) (-0.352,0.091) (-0.347,0.075) (-0.342,0.059) (-0.335,0.044) (-0.328,0.032) (-0.319,0.018) (-0.309,0.004) (-0.299,-0.008) (-0.287,-0.020) (-0.275,-0.032) (-0.263,-0.043) (-0.249,-0.053) (-0.236,-0.062) (-0.222,-0.071) (-0.207,-0.079) (-0.193,-0.087) (-0.178,-0.094) (-0.162,-0.100) (-0.147,-0.106) (-0.131,-0.111) (-0.115,-0.115) (-0.099,-0.119) (-0.083,-0.123) (-0.066,-0.125) (-0.052,-0.127) (-0.036,-0.129) (-0.019,-0.130) (-0.003,-0.130) (0.014,-0.130) (0.030,-0.129) (0.047,-0.128) (0.063,-0.126) (0.080,-0.123) (0.096,-0.120)};
\draw[trefoilst] plot[smooth] coordinates {(0.236,-0.062) (0.249,-0.053) (0.263,-0.043) (0.275,-0.032) (0.287,-0.020) (0.299,-0.008) (0.309,0.004) (0.319,0.018) (0.328,0.032) (0.336,0.046) (0.342,0.059) (0.347,0.075) (0.352,0.091) (0.354,0.107) (0.356,0.123) (0.355,0.140) (0.353,0.156) (0.348,0.172) (0.342,0.188) (0.333,0.202) (0.323,0.215) (0.312,0.227) (0.299,0.237) (0.285,0.246) (0.270,0.254) (0.254,0.259) (0.238,0.263) (0.222,0.266) (0.205,0.268) (0.192,0.268) (0.175,0.267) (0.158,0.266) (0.142,0.263) (0.126,0.259) (0.110,0.254) (0.094,0.249) (0.079,0.242) (0.064,0.235) (0.049,0.228) (0.035,0.219) (0.021,0.210) (0.008,0.201) (-0.006,0.191) (-0.018,0.180) (-0.031,0.169) (-0.042,0.157) (-0.054,0.145) (-0.065,0.133) (-0.075,0.120) (-0.084,0.109) (-0.093,0.095) (-0.103,0.082) (-0.111,0.067) (-0.119,0.053) (-0.127,0.038) (-0.134,0.023) (-0.140,0.008) (-0.146,-0.008) (-0.152,-0.023)};
\draw[trefoilst] plot[smooth] coordinates {(-0.172,-0.173) (-0.170,-0.190) (-0.168,-0.206) (-0.165,-0.222) (-0.161,-0.239) (-0.156,-0.254) (-0.151,-0.270) (-0.144,-0.285) (-0.136,-0.300) (-0.128,-0.314) (-0.120,-0.325) (-0.109,-0.338) (-0.097,-0.350) (-0.085,-0.360) (-0.071,-0.370) (-0.056,-0.377) (-0.041,-0.383) (-0.025,-0.388) (-0.008,-0.390) (0.008,-0.390) (0.025,-0.388) (0.041,-0.383) (0.056,-0.377) (0.071,-0.370) (0.085,-0.360) (0.097,-0.350) (0.109,-0.338) (0.120,-0.325) (0.129,-0.312) (0.136,-0.300) (0.144,-0.285) (0.151,-0.270) (0.156,-0.254) (0.161,-0.239) (0.165,-0.222) (0.168,-0.206) (0.170,-0.190) (0.172,-0.173) (0.172,-0.157) (0.172,-0.140) (0.172,-0.123) (0.170,-0.107) (0.168,-0.090) (0.165,-0.074) (0.161,-0.058) (0.157,-0.042) (0.153,-0.026) (0.147,-0.010) (0.142,0.005) (0.136,0.018) (0.129,0.033) (0.122,0.048) (0.114,0.063) (0.105,0.077) (0.097,0.091) (0.087,0.104) (0.077,0.118) (0.067,0.131) (0.056,0.143)};\end{scope}
\node[icon] (ic2) at (3.85,-0.95) {};
\begin{scope}[shift={(ic2.center)}]
\draw[draw=diagramblue!75!black,line width=0.95pt] (-0.2,0)++(66.3:0.32) arc[start angle=66.3,end angle=396.3,radius=0.32];
\draw[draw=orbitred,line width=0.95pt] (0.2,0)++(246.3:0.32) arc[start angle=246.3,end angle=576.3,radius=0.32];
\node[tlab,fill=white,draw=diagramblue!50,rounded corners=1pt,inner sep=1pt] at (-0.55,-0.42) {$4$};
\end{scope}
\node[draw=diagramblue,line width=0.7pt,rounded corners=3pt,fill=highlight!60,font=\footnotesize,inner sep=3pt] (s4a) at (5.5,1.1) {$S^4$};
\node[draw=diagramblue,line width=0.7pt,rounded corners=3pt,fill=highlight!60,font=\footnotesize,inner sep=3pt] (s4b) at (5.5,-0.95) {$S^4$};
\draw[darrow] (ic1) -- (s4a);
\draw[darrow] (ic2) -- (s4b);
\node[tlab,text=rigidcol,anchor=north] at ($(ic1.south)+(0,-0.05)$) {odd dihedral};
\node[tlab,text=stdcol,anchor=north] at ($(ic2.south)+(0,-0.05)$) {cyclic};
\node[panelcap,anchor=north] at (2.85,-2.15) {(c) the two framed surgeries on $\gamma$};
\end{scope}
\draw[darrow] (5.2,1.65) to[out=0,in=180] node[pos=0.3,above,slab, yshift=0.1cm]{$\lambda(\lambda_A)$} (r1.west);
\draw[darrow] (5.2,1.3) to[out=-20,in=180] node[pos=0.62,left,slab, xshift=-0.2cm]{$\lambda(\lambda_A-\mu_A)$} (r2.west);
\end{tikzpicture}
\caption{The standardness proof of Section~\ref{std:sec:main} in one picture.
(a) By Theorem~\ref{wh:thm:whitehead} the orbit orbifold of the canonical circle action
on $\M$ is $\OW$: the Whitehead link with coefficient $+3$ on $J$ and cone angle $\pi$
along $W'$. The clasp belt $\ell_A$ is a band sum of meridians of $J$ and $W'$
(Theorem~\ref{std:thm:belt}(i)). (b) Over $\ell_A$ the orbits sweep out the fibre torus
$T_{\ell_A}$, and its unit section $\gamma$ is a weight circle, whose two framed
surgeries are $\Sigma_\pm(7)$ (Theorem~\ref{wh:thm:section}). (c) By
Proposition~\ref{std:prop:torus} and Lemma~\ref{std:lem:frame}, the consecutive base
framings $\lambda_A$ and $\lambda_A-\mu_A$ turn the two framed surgeries on $\gamma$
into circle surgeries over the Dehn fillings $Q_0$ and $Q_{-1}$, described in
Theorem~\ref{std:thm:belt}(ii) and~(iii); here $s_*=-1$ and $r=4$ for the mirror image fixed
in Section~\ref{wh:sec:whitehead} (Remark~\ref{std:rem:sstar}). In the icons, the singular circle is red; for $Q_{-1}$, the blue circle
represents the surgery curve $J$ with
coefficient $4$. The odd dihedral and cyclic recognition criteria then give $S^4$ in
both cases (Theorem~\ref{std:thm:main}).}
\label{fig:standardness-picture}
\end{figure}

\subsection{Surgery on a section of an invariant torus}
\label{std:sec:torus}

We first consider the following general setting.

\begin{definition}[Fibre-torus data]\label{std:def:data}
Let $Y$ be a closed connected oriented smooth four-manifold with a smooth
effective circle action without fixed points, and let $\pi:Y\to Q=Y/S^1$ be
the orbit map onto the closed orientable three-orbifold $Q$. Let $\ell\subset
Q$ be a smoothly embedded loop in the free stratum, let
$T_\ell=\pi^{-1}(\ell)$, and let $\gamma\subset T_\ell$ be a section of the
circle bundle $T_\ell\to\ell$. A \emph{framing} of $\ell$ is a slope $\ell''$
on $\partial\nu\ell$ meeting the meridian $m$ once; the framings form a
$\ZZ$-torsor under $\ell''\mapsto\ell''+m$. For a framing $\ell''$, let
\[
 Q(\ell'')=(Q\setminus\Int\nu\ell)\cup(S^1\times D^2)
\]
be the Dehn filling in which $\{\mathrm{pt}\}\times\partial D^2$ is glued to
$\ell''$.
\end{definition}

Given a framing $\ell''$, choose a tubular neighbourhood with an identification $\nu\ell\cong S^1_a\times D^2$
carrying $\ell$ to $S^1_a\times\{0\}$ and $\ell''$ to $S^1_a\times\{1\}$. The
action over $\nu\ell$ is free, so $\pi^{-1}(\nu\ell)$ is a principal circle
bundle over a solid torus and admits an equivariant trivialization.
The inverse of the fibre coordinate of $\gamma$ extends over $\nu\ell$ by
radial retraction. Multiplying the fibre coordinate by this extension gives
equivariant product coordinates
\[
 P:=\pi^{-1}(\nu\ell)=S^1_a\times D^2\times S^1_h,\;\;\;
 \gamma=S^1_a\times\{0\}\times\{1\},\;\;\;
 T_\ell=S^1_a\times\{0\}\times S^1_h,
\]
in which the circle rotates the factor $S^1_h$. Let $\lambda(\ell'')$ be the
product framing of the normal bundle of $\gamma$, whose fibre is
$T_0D^2\oplus T_1S^1_h$.

\begin{lemma}[Consecutive framings]\label{std:lem:frame}
The homotopy class of $\lambda(\ell'')$ depends only on $\ell''$, and
$\lambda(\ell'')$ and $\lambda(\ell''+m)$ are the two distinct
normal-framing classes of $\gamma$. Consequently, for every framing $\ell''$,
the two framed surgeries on $\gamma$ are the surgeries with framings
$\lambda(\ell'')$ and $\lambda(\ell''+m)$.
\end{lemma}

\begin{proof}
Tubular maps with a prescribed framing form a connected family. Two
equivariant trivializations in which $\gamma$ is the constant section differ
by a map $\nu\ell\to S^1$ that is identically one on $\ell$, and such a map is
homotopic to the constant map rel $\ell$. Hence, the homotopy class of
$\lambda(\ell'')$ is well defined. Replacing $\ell''$ by $\ell''+m$ changes the
tubular map by $(a,z)\mapsto(a,e^{2\pi ia}z)$ for $a\in\RR/\ZZ$. This fixes the
$h$-direction and turns the $D^2$-directions once around $\gamma$, so the two
frames differ by the loop $a\mapsto\diag(R_{2\pi a},1)$ in $SO(3)$. That loop
represents the generator of $\pi_1(SO(3))\cong\ZZ/2$, and there are exactly
two framing classes (Lemma~\ref{lem:frame}).
\end{proof}

Figure~\ref{did:fig:framing-parity} makes explicit the passage from an
integer choice of base slope to a parity choice of normal framing in the
four-manifold.
\begin{figure}[tbp]
\centering
\begin{tikzpicture}[
  sdot/.style={circle,fill=#1,inner sep=0pt,minimum size=7pt},
  cls/.style={draw=#1!75,fill=#1!8,rounded corners=3pt,line width=0.6pt,
    align=flush center,font=\footnotesize,inner sep=4pt,text width=3.6cm},
  fr/.style={-{Stealth[length=1.35mm,width=1.1mm]},line width=0.65pt}]
\draw[gray!55,line width=0.6pt] (-0.6,0) -- (8.6,0);
\foreach \s/\x in {-2/0,-1/2,0/4,1/6,2/8} {
  \pgfmathtruncatemacro{\odd}{abs(mod(\s,2))}
  \ifnum\odd=1 \def\c{dviolet}\else\def\c{diagramblue}\fi
  \node[sdot=\c] (s\x) at (\x,0) {};
  \node[slab,anchor=south] at (\x,0.15) {$s=\s$};
  \node[tlab,text=inkgray,anchor=north] at (\x,-0.15) {$\ell''+(\s)m$};}
\draw[-{Stealth[length=1.6mm]},gray!65!black,line width=0.55pt] (4.15,0.55) to[bend left=30]
  node[above=2pt,slab,text=black] {$+m$: the normal $D^2$ turns once} (5.85,0.55);
\node[cls=diagramblue] (c0) at (1.6,-2.0) {class $0$: the reference\\ framing $\lambda(\ell'')$};
\node[cls=dviolet] (c1) at (6.4,-2.0) {class $1$: the other\\ framing $\lambda(\ell''+m)$};
\foreach \x in {0,4,8} \draw[diagramblue!60,line width=0.5pt,-{Stealth[length=1.5mm]}] (\x,-0.48) -- (c0.north);
\foreach \x in {2,6} \draw[dviolet!60,line width=0.5pt,-{Stealth[length=1.5mm]}] (\x,-0.48) -- (c1.north);
\begin{scope}[shift={(10.0,-0.25)}]
\fill[diagramlight,rounded corners=2pt] (-0.25,-0.42) rectangle (3.85,0.42);
\draw[gray!55] (-0.25,0) -- (3.85,0);
\foreach \k in {0,...,4} {
  \pgfmathsetmacro{\ang}{90*\k}
  \draw[fr,liftblue] (0.9*\k,0) -- ++(\ang:0.3);
  \draw[fr,orbitred] (0.9*\k,0) -- ++(\ang+90:0.3);
  \fill (0.9*\k,0) circle[radius=0.03];}
\node[slab] at (0,-0.65) {$a=0$};
\node[slab] at (3.6,-0.65) {$a=1$};
\node[slab,align=center] at (1.8,0.85) {$a\longmapsto\diag(R_{2\pi a},1)\in SO(3)$};
\node[slab,align=center,text=inkgray] at (1.8,-1.25) {the generator of $\pi_1(SO(3))\cong\ZZ/2$};
\end{scope}
\end{tikzpicture}
\caption{Integer base framings give two normal-framing classes
(Lemma~\ref{std:lem:frame}). Here $m$ denotes the meridian slope of $\ell$, as in
Definition~\ref{std:def:data}, rather than a covering degree. Choose $\ell''$ as the
reference slope. Moving by one meridian turns the normal disc $D^2$ of $\gamma$ once
(right): the two frames differ by the loop $a\mapsto\diag(R_{2\pi a},1)$, which
generates $\pi_1(SO(3))$. Moving by $s$ meridians therefore changes the normal
framing by $s$ times the generator, and the arrows reduce the integers modulo two.
Consecutive slopes represent both choices. The integers $0$ and $-1$ used for the
clasp belt are of opposite parity, although their identification with the
explicit labels $+$ and $-$ need not be fixed by the orbit-space model.}
\label{did:fig:framing-parity}
\end{figure}

\begin{proposition}[Surgery on a section of a fibre torus]
\label{std:prop:torus}
In the setting of Definition~\ref{std:def:data}, let $\Sigma$ be the result
of surgery on $\gamma$ with framing $\lambda(\ell'')$.
\begin{enumerate}[label=(\roman*),leftmargin=2.2em]
\item $\Sigma$ contains a smoothly embedded two-sphere $S_T$ with trivial
      normal bundle, obtained from $T_\ell$ by surgery along $\gamma$.
\item Surgery on $S_T$ gives
      \[
       Y'(\ell''):=(Y\setminus\Int P)\cup_{\partial P}(S^1_m\times V).
      \]
      Here $S^1_m=\partial D^2$, $V$ is the solid torus bounded by
      $S^1_a\times S^1_h$ in which $S^1_a\times\{\mathrm{pt}\}$ bounds a disc,
      and the gluing is the identity on
      $\partial P=S^1_a\times S^1_m\times S^1_h$.
\item The action on $Y\setminus\Int P$ extends to a smooth effective circle
      action on $Y'(\ell'')$ without fixed points, free on $S^1_m\times V$.
      Its orbit space is $Q(\ell'')$. Its nonfree orbits, with their isotropy
      groups and slice representations, are those of $Y$.
\item $\Sigma$ is obtained from $Y'(\ell'')$ by a framed surgery on an
      embedded circle $\delta$, the core of the sphere surgery.
\end{enumerate}
\end{proposition}

\noindent
The underlying nonequivariant factorization is the round two-handle
factorization of an integral generalized logarithmic transform
\cite[Lemma~2]{BaykurSunukjian}. The following local proof also establishes
compatibility with the circle action and the orbit-space Dehn filling.

\begin{proof}
Put $Y_0=D^2\times S^1_h$, a solid torus with core
$C_0=\{0\}\times S^1_h$, so that $P=S^1_a\times Y_0$,
$\gamma=S^1_a\times\{y_0\}$ with $y_0=(0,1)$, and $T_\ell=S^1_a\times C_0$.
Let $B\subset Y_0$ be a small round ball about $y_0$. The core $C_0$ crosses
$\partial B$ in two points $n$ and $s$. Surgery with the product framing
replaces $S^1_a\times B$ by $D^2\times\partial B$, so
\[
 \Sigma=(Y\setminus\Int P)\cup P',\;\;\;
 P'=\bigl(S^1_a\times(Y_0\setminus\Int B)\bigr)\cup_{S^1_a\times\partial B}
    (D^2\times\partial B).
\]

(i) Let $c=C_0\setminus\Int B$, an arc from $n$ to $s$, and put
\[
 S_T=(S^1_a\times c)\cup\bigl(D^2\times\{n,s\}\bigr).
\]
This is an annulus capped by two discs, and it is the result of surgering
$T_\ell$ along $\gamma$. Let $N_c$ be a regular neighbourhood of $c$ in
$Y_0\setminus\Int B$, meeting $\partial B$ in two polar caps $D_n$ and $D_s$.
After the corners have been rounded,
$\nu S_T=(S^1_a\times N_c)\cup(D^2\times(D_n\cup D_s))$ is a tubular
neighbourhood of $S_T$, and the product structures of the two pieces agree
along their intersection. So the normal bundle is trivial.

(ii) Let $N_0=B\cup N_c$, a solid-torus regular neighbourhood of $C_0$. Its
boundary is the union of the annuli $A=\partial B\setminus\Int(D_n\cup D_s)$
and $A'=\partial N_c\setminus\Int B$; the core of $A$ is the equator of
$\partial B$, a meridian circle of $C_0$. Since $Y_0\setminus\Int N_0$ is a
collar $\partial Y_0\times[0,1]$, we have
\[
 P'\setminus\Int\nu S_T
 =\bigl(S^1_a\times(Y_0\setminus\Int N_0)\bigr)\cup_{S^1_a\times A}(D^2\times A).
\]
Write $\partial Y_0=S^1_m\times S^1_h$. In the collar coordinates
we have $A=S^1_m\times I$ for an arc $I\subset S^1_h$, and therefore
\[
 P'\setminus\Int\nu S_T\ \cong\
 S^1_m\times\Bigl(\bigl(S^1_a\times S^1_h\times[0,1]\bigr)
   \cup_{S^1_a\times I\times\{0\}}(D^2\times I)\Bigr)
 = S^1_m\times\bigl(V\setminus\Int B'\bigr),
\]
where $B'$ is a ball. Indeed, attaching the three-dimensional two-handle
$D^2\times I$ along $S^1_a\times I\times\{0\}$ turns
$S^1_a\times S^1_h\times[0,1]$ into a solid torus minus a ball, whose meridian
is $S^1_a\times\{\mathrm{pt}\}$. The remaining boundary component is
$S^1_m\times\partial B'=\partial\nu S_T$, and the circles
$\{x\}\times S^1_m$ are meridians of $S_T$: near the caps they are parallel to
$\partial D_n$, and along the annulus they are meridians of the arc $c$.
Surgery on $S_T$ glues in $B'\times S^1_m$, which gives (ii).

(iii) Write $V=D^2_a\times S^1_h$ with $\partial D^2_a=S^1_a$. Rotation of the
factor $S^1_h$ is a free action on $S^1_m\times V$ that agrees with the given
action on $\partial P$. Its orbit space is the solid torus $S^1_m\times D^2_a$,
attached to $Q\setminus\Int\nu\ell$ along
$\partial\nu\ell=S^1_a\times S^1_m$ with meridian $S^1_a\times\{\mathrm{pt}\}$.
By the choice of the tubular map this curve is $\ell''$. Nothing changes
outside $P$.

(iv) Surgery on the circle $\delta=\{0\}\times S^1_m$ in $B'\times S^1_m$,
with the framing given by the product structure, reverses the sphere surgery.
\end{proof}

Figure~\ref{std:fig:torus} illustrates the proof.

\begin{figure}[tbp]
\centering
\begin{tikzpicture}[
  orbit/.style={draw=gray!35!black,line width=0.55pt},
  orbitback/.style={draw=gray!55,line width=0.45pt,densely dashed},
  sect/.style={draw=orbitred,line width=1.15pt},
  sectback/.style={draw=orbitred!55,line width=0.8pt,densely dashed},
  sp/.style={draw=#1!75,fill=#1!9,rounded corners=3pt,line width=0.55pt,
    font=\footnotesize,inner sep=4pt,align=center},
  sp/.default=diagramblue,
  orbsp/.style={font=\footnotesize,inner sep=4pt,text=diagramblue!85!black}]
\begin{scope}
\shade[left color=diagramlight,right color=diagramblue!22] plot[smooth cycle] coordinates {(1.700,0.000) (1.696,0.079) (1.685,0.157) (1.666,0.232) (1.640,0.305) (1.609,0.374) (1.572,0.440) (1.531,0.501) (1.486,0.559) (1.438,0.613) (1.387,0.663) (1.333,0.709) (1.276,0.753) (1.218,0.793) (1.157,0.831) (1.095,0.866) (1.030,0.898) (0.964,0.928) (0.896,0.956) (0.827,0.981) (0.757,1.004) (0.685,1.025) (0.612,1.043) (0.538,1.059) (0.463,1.073) (0.387,1.085) (0.310,1.095) (0.233,1.103) (0.156,1.108) (0.078,1.111) (0.000,1.112) (-0.078,1.111) (-0.156,1.108) (-0.233,1.103) (-0.310,1.095) (-0.387,1.085) (-0.463,1.073) (-0.538,1.059) (-0.612,1.043) (-0.685,1.025) (-0.757,1.004) (-0.827,0.981) (-0.896,0.956) (-0.964,0.928) (-1.030,0.898) (-1.095,0.866) (-1.157,0.831) (-1.218,0.793) (-1.276,0.753) (-1.333,0.709) (-1.387,0.663) (-1.438,0.613) (-1.486,0.559) (-1.531,0.501) (-1.572,0.440) (-1.609,0.374) (-1.640,0.305) (-1.666,0.232) (-1.685,0.157) (-1.696,0.079) (-1.700,0.000) (-1.696,-0.079) (-1.685,-0.157) (-1.666,-0.232) (-1.640,-0.305) (-1.609,-0.374) (-1.572,-0.440) (-1.531,-0.501) (-1.486,-0.559) (-1.438,-0.613) (-1.387,-0.663) (-1.333,-0.709) (-1.276,-0.753) (-1.218,-0.793) (-1.157,-0.831) (-1.095,-0.866) (-1.030,-0.898) (-0.964,-0.928) (-0.896,-0.956) (-0.827,-0.981) (-0.757,-1.004) (-0.685,-1.025) (-0.612,-1.043) (-0.538,-1.059) (-0.463,-1.073) (-0.387,-1.085) (-0.310,-1.095) (-0.233,-1.103) (-0.156,-1.108) (-0.078,-1.111) (-0.000,-1.112) (0.078,-1.111) (0.156,-1.108) (0.233,-1.103) (0.310,-1.095) (0.387,-1.085) (0.463,-1.073) (0.538,-1.059) (0.612,-1.043) (0.685,-1.025) (0.757,-1.004) (0.827,-0.981) (0.896,-0.956) (0.964,-0.928) (1.030,-0.898) (1.095,-0.866) (1.157,-0.831) (1.218,-0.793) (1.276,-0.753) (1.333,-0.709) (1.387,-0.663) (1.438,-0.613) (1.486,-0.559) (1.531,-0.501) (1.572,-0.440) (1.609,-0.374) (1.640,-0.305) (1.666,-0.232) (1.685,-0.157) (1.696,-0.079)};
\fill[white] plot[smooth] coordinates {(0.778,-0.000) (0.764,0.012) (0.746,0.026) (0.725,0.040) (0.701,0.055) (0.673,0.071) (0.643,0.086) (0.609,0.101) (0.573,0.116) (0.534,0.130) (0.493,0.143) (0.450,0.156) (0.405,0.167) (0.358,0.177) (0.310,0.187) (0.260,0.194) (0.210,0.201) (0.158,0.206) (0.106,0.209) (0.053,0.212) (0.000,0.212) (-0.053,0.212) (-0.106,0.209) (-0.158,0.206) (-0.210,0.201) (-0.260,0.194) (-0.310,0.187) (-0.358,0.177) (-0.405,0.167) (-0.450,0.156) (-0.493,0.143) (-0.534,0.130) (-0.573,0.116) (-0.609,0.101) (-0.643,0.086) (-0.673,0.071) (-0.701,0.055) (-0.725,0.040) (-0.746,0.026) (-0.764,0.012) (-0.778,-0.000)} -- plot[smooth] coordinates {(-0.778,0.000) (-0.764,-0.012) (-0.746,-0.026) (-0.725,-0.040) (-0.701,-0.055) (-0.673,-0.071) (-0.643,-0.086) (-0.609,-0.101) (-0.573,-0.116) (-0.534,-0.130) (-0.493,-0.143) (-0.450,-0.156) (-0.405,-0.167) (-0.358,-0.177) (-0.310,-0.187) (-0.260,-0.194) (-0.210,-0.201) (-0.158,-0.206) (-0.106,-0.209) (-0.053,-0.212) (-0.000,-0.212) (0.053,-0.212) (0.106,-0.209) (0.158,-0.206) (0.210,-0.201) (0.260,-0.194) (0.310,-0.187) (0.358,-0.177) (0.405,-0.167) (0.450,-0.156) (0.493,-0.143) (0.534,-0.130) (0.573,-0.116) (0.609,-0.101) (0.643,-0.086) (0.673,-0.071) (0.701,-0.055) (0.725,-0.040) (0.746,-0.026) (0.764,-0.012) (0.778,0.000)} -- cycle;
\draw[diagramblue!55,line width=0.45pt] plot[smooth cycle] coordinates {(1.700,0.000) (1.696,0.079) (1.685,0.157) (1.666,0.232) (1.640,0.305) (1.609,0.374) (1.572,0.440) (1.531,0.501) (1.486,0.559) (1.438,0.613) (1.387,0.663) (1.333,0.709) (1.276,0.753) (1.218,0.793) (1.157,0.831) (1.095,0.866) (1.030,0.898) (0.964,0.928) (0.896,0.956) (0.827,0.981) (0.757,1.004) (0.685,1.025) (0.612,1.043) (0.538,1.059) (0.463,1.073) (0.387,1.085) (0.310,1.095) (0.233,1.103) (0.156,1.108) (0.078,1.111) (0.000,1.112) (-0.078,1.111) (-0.156,1.108) (-0.233,1.103) (-0.310,1.095) (-0.387,1.085) (-0.463,1.073) (-0.538,1.059) (-0.612,1.043) (-0.685,1.025) (-0.757,1.004) (-0.827,0.981) (-0.896,0.956) (-0.964,0.928) (-1.030,0.898) (-1.095,0.866) (-1.157,0.831) (-1.218,0.793) (-1.276,0.753) (-1.333,0.709) (-1.387,0.663) (-1.438,0.613) (-1.486,0.559) (-1.531,0.501) (-1.572,0.440) (-1.609,0.374) (-1.640,0.305) (-1.666,0.232) (-1.685,0.157) (-1.696,0.079) (-1.700,0.000) (-1.696,-0.079) (-1.685,-0.157) (-1.666,-0.232) (-1.640,-0.305) (-1.609,-0.374) (-1.572,-0.440) (-1.531,-0.501) (-1.486,-0.559) (-1.438,-0.613) (-1.387,-0.663) (-1.333,-0.709) (-1.276,-0.753) (-1.218,-0.793) (-1.157,-0.831) (-1.095,-0.866) (-1.030,-0.898) (-0.964,-0.928) (-0.896,-0.956) (-0.827,-0.981) (-0.757,-1.004) (-0.685,-1.025) (-0.612,-1.043) (-0.538,-1.059) (-0.463,-1.073) (-0.387,-1.085) (-0.310,-1.095) (-0.233,-1.103) (-0.156,-1.108) (-0.078,-1.111) (-0.000,-1.112) (0.078,-1.111) (0.156,-1.108) (0.233,-1.103) (0.310,-1.095) (0.387,-1.085) (0.463,-1.073) (0.538,-1.059) (0.612,-1.043) (0.685,-1.025) (0.757,-1.004) (0.827,-0.981) (0.896,-0.956) (0.964,-0.928) (1.030,-0.898) (1.095,-0.866) (1.157,-0.831) (1.218,-0.793) (1.276,-0.753) (1.333,-0.709) (1.387,-0.663) (1.438,-0.613) (1.486,-0.559) (1.531,-0.501) (1.572,-0.440) (1.609,-0.374) (1.640,-0.305) (1.666,-0.232) (1.685,-0.157) (1.696,-0.079)};
\draw[diagramblue!55,line width=0.45pt] plot[smooth] coordinates {(0.778,-0.000) (0.764,0.012) (0.746,0.026) (0.725,0.040) (0.701,0.055) (0.673,0.071) (0.643,0.086) (0.609,0.101) (0.573,0.116) (0.534,0.130) (0.493,0.143) (0.450,0.156) (0.405,0.167) (0.358,0.177) (0.310,0.187) (0.260,0.194) (0.210,0.201) (0.158,0.206) (0.106,0.209) (0.053,0.212) (0.000,0.212) (-0.053,0.212) (-0.106,0.209) (-0.158,0.206) (-0.210,0.201) (-0.260,0.194) (-0.310,0.187) (-0.358,0.177) (-0.405,0.167) (-0.450,0.156) (-0.493,0.143) (-0.534,0.130) (-0.573,0.116) (-0.609,0.101) (-0.643,0.086) (-0.673,0.071) (-0.701,0.055) (-0.725,0.040) (-0.746,0.026) (-0.764,0.012) (-0.778,-0.000)};
\draw[diagramblue!55,line width=0.45pt] plot[smooth] coordinates {(-0.778,0.000) (-0.764,-0.012) (-0.746,-0.026) (-0.725,-0.040) (-0.701,-0.055) (-0.673,-0.071) (-0.643,-0.086) (-0.609,-0.101) (-0.573,-0.116) (-0.534,-0.130) (-0.493,-0.143) (-0.450,-0.156) (-0.405,-0.167) (-0.358,-0.177) (-0.310,-0.187) (-0.260,-0.194) (-0.210,-0.201) (-0.158,-0.206) (-0.106,-0.209) (-0.053,-0.212) (-0.000,-0.212) (0.053,-0.212) (0.106,-0.209) (0.158,-0.206) (0.210,-0.201) (0.260,-0.194) (0.310,-0.187) (0.358,-0.177) (0.405,-0.167) (0.450,-0.156) (0.493,-0.143) (0.534,-0.130) (0.573,-0.116) (0.609,-0.101) (0.643,-0.086) (0.673,-0.071) (0.701,-0.055) (0.725,-0.040) (0.746,-0.026) (0.764,-0.012) (0.778,0.000)};
\draw[orbitback] plot[smooth] coordinates {(1.700,0.000) (1.699,0.020)};
\draw[orbit] plot[smooth] coordinates {(1.699,0.020) (1.685,0.099) (1.651,0.173) (1.600,0.240) (1.533,0.297) (1.454,0.340) (1.366,0.369) (1.274,0.381) (1.180,0.377) (1.089,0.356) (1.005,0.320) (0.932,0.270) (0.873,0.208) (0.830,0.137) (0.806,0.060) (0.801,-0.020)};
\draw[orbitback] plot[smooth] coordinates {(0.801,-0.020) (0.815,-0.099) (0.849,-0.173) (0.900,-0.240) (0.967,-0.297) (1.046,-0.340) (1.134,-0.369) (1.226,-0.381) (1.297,-0.380) (1.389,-0.363) (1.475,-0.330) (1.551,-0.284) (1.614,-0.224) (1.661,-0.155) (1.690,-0.079) (1.700,-0.000)};
\draw[orbitback] plot[smooth] coordinates {(1.093,0.690) (1.086,0.765) (1.068,0.830) (1.028,0.890) (0.986,0.919)};
\draw[orbit] plot[smooth] coordinates {(0.986,0.919) (0.935,0.930) (0.878,0.923) (0.819,0.898) (0.758,0.856) (0.700,0.798) (0.646,0.728) (0.599,0.648) (0.561,0.562) (0.533,0.474) (0.518,0.387) (0.515,0.305) (0.524,0.232) (0.546,0.171) (0.579,0.125) (0.621,0.096)};
\draw[orbitback] plot[smooth] coordinates {(0.621,0.096) (0.672,0.084) (0.729,0.092) (0.788,0.117) (0.849,0.159) (0.907,0.217) (0.948,0.268) (0.997,0.346) (1.037,0.431) (1.068,0.519) (1.086,0.607) (1.093,0.690)};
\draw[orbitback] plot[smooth] coordinates {(-1.093,0.690) (-1.086,0.765) (-1.068,0.830) (-1.028,0.890) (-0.986,0.919)};
\draw[orbit] plot[smooth] coordinates {(-0.986,0.919) (-0.935,0.930) (-0.878,0.923) (-0.819,0.898) (-0.758,0.856) (-0.700,0.798) (-0.646,0.728) (-0.599,0.648) (-0.561,0.562) (-0.533,0.474) (-0.518,0.387) (-0.515,0.305) (-0.524,0.232) (-0.546,0.171) (-0.579,0.125) (-0.621,0.096)};
\draw[orbitback] plot[smooth] coordinates {(-0.621,0.096) (-0.672,0.084) (-0.729,0.092) (-0.788,0.117) (-0.849,0.159) (-0.907,0.217) (-0.948,0.268) (-0.997,0.346) (-1.037,0.431) (-1.068,0.519) (-1.086,0.607) (-1.093,0.690)};
\draw[orbitback] plot[smooth] coordinates {(-1.700,0.000) (-1.699,0.020)};
\draw[orbit] plot[smooth] coordinates {(-1.699,0.020) (-1.685,0.099) (-1.651,0.173) (-1.600,0.240) (-1.533,0.297) (-1.454,0.340) (-1.366,0.369) (-1.274,0.381) (-1.180,0.377) (-1.089,0.356) (-1.005,0.320) (-0.932,0.270) (-0.873,0.208) (-0.830,0.137) (-0.806,0.060) (-0.801,-0.020)};
\draw[orbitback] plot[smooth] coordinates {(-0.801,-0.020) (-0.815,-0.099) (-0.849,-0.173) (-0.900,-0.240) (-0.967,-0.297) (-1.046,-0.340) (-1.134,-0.369) (-1.226,-0.381) (-1.297,-0.380) (-1.389,-0.363) (-1.475,-0.330) (-1.551,-0.284) (-1.614,-0.224) (-1.661,-0.155) (-1.690,-0.079) (-1.700,-0.000)};
\draw[orbit] plot[smooth] coordinates {(-1.093,-0.690) (-1.086,-0.607) (-1.068,-0.519) (-1.037,-0.431) (-0.997,-0.346) (-0.948,-0.268) (-0.893,-0.201) (-0.834,-0.147) (-0.773,-0.109) (-0.714,-0.088) (-0.659,-0.086) (-0.610,-0.102)};
\draw[orbitback] plot[smooth] coordinates {(-0.610,-0.102) (-0.569,-0.135) (-0.539,-0.185) (-0.521,-0.249) (-0.514,-0.325) (-0.521,-0.408) (-0.539,-0.496) (-0.569,-0.584) (-0.610,-0.669) (-0.659,-0.747) (-0.714,-0.814) (-0.773,-0.868) (-0.834,-0.906) (-0.893,-0.927) (-0.948,-0.929) (-0.997,-0.913)};
\draw[orbit] plot[smooth] coordinates {(-0.997,-0.913) (-1.037,-0.880) (-1.068,-0.830) (-1.086,-0.765) (-1.093,-0.690)};
\draw[orbit] plot[smooth] coordinates {(1.093,-0.690) (1.086,-0.607) (1.068,-0.519) (1.037,-0.431) (0.997,-0.346) (0.948,-0.268) (0.893,-0.201) (0.834,-0.147) (0.773,-0.109) (0.714,-0.088) (0.659,-0.086) (0.610,-0.102)};
\draw[orbitback] plot[smooth] coordinates {(0.610,-0.102) (0.569,-0.135) (0.539,-0.185) (0.521,-0.249) (0.514,-0.325) (0.521,-0.408) (0.539,-0.496) (0.569,-0.584) (0.610,-0.669) (0.659,-0.747) (0.714,-0.814) (0.773,-0.868) (0.834,-0.906) (0.893,-0.927) (0.948,-0.929) (0.997,-0.913)};
\draw[orbit] plot[smooth] coordinates {(0.997,-0.913) (1.037,-0.880) (1.068,-0.830) (1.086,-0.765) (1.093,-0.690)};
\draw[sect] plot[smooth] coordinates {(1.461,0.337) (1.458,0.391) (1.447,0.445) (1.429,0.498) (1.405,0.550) (1.373,0.602) (1.335,0.652) (1.290,0.700) (1.239,0.747) (1.182,0.792) (1.119,0.835) (1.051,0.875) (0.978,0.912) (0.900,0.947) (0.817,0.979) (0.731,1.008) (0.641,1.033) (0.547,1.055) (0.452,1.073) (0.354,1.088) (0.254,1.100) (0.153,1.107) (0.051,1.111) (-0.051,1.111) (-0.153,1.107) (-0.254,1.100) (-0.354,1.088) (-0.452,1.073) (-0.547,1.055) (-0.641,1.033) (-0.731,1.008) (-0.817,0.979) (-0.900,0.947) (-0.978,0.912) (-1.051,0.875) (-1.119,0.835) (-1.182,0.792) (-1.239,0.747) (-1.290,0.700) (-1.335,0.652) (-1.373,0.602) (-1.405,0.550) (-1.429,0.498) (-1.447,0.445) (-1.458,0.391) (-1.461,0.337) (-1.458,0.283) (-1.447,0.229) (-1.429,0.176) (-1.405,0.124) (-1.373,0.072) (-1.335,0.022) (-1.290,-0.027) (-1.239,-0.073) (-1.182,-0.118) (-1.119,-0.161) (-1.051,-0.201) (-0.978,-0.239) (-0.900,-0.273) (-0.817,-0.305) (-0.731,-0.334) (-0.641,-0.359) (-0.547,-0.381) (-0.452,-0.399) (-0.354,-0.414) (-0.254,-0.426) (-0.153,-0.433) (-0.051,-0.437) (0.051,-0.437) (0.153,-0.433) (0.254,-0.426) (0.354,-0.414) (0.452,-0.399) (0.547,-0.381) (0.641,-0.359) (0.731,-0.334) (0.817,-0.305) (0.900,-0.273) (0.978,-0.239) (1.051,-0.201) (1.119,-0.161) (1.182,-0.118) (1.239,-0.073) (1.290,-0.027) (1.335,0.022) (1.373,0.072) (1.405,0.124) (1.429,0.176) (1.447,0.229) (1.458,0.283) (1.461,0.337)};
\coordinate (glabel) at (0.254,1.100);
\coordinate (olabel) at (1.093,-0.690);
\draw[orbitred!70,line width=0.4pt] (glabel) -- ++(0.35,0.55) node[above,slab,text=orbitred,inner sep=1pt] {$\gamma$};
\draw[gray!60,line width=0.4pt] (olabel) -- ++(0.5,-0.35) node[right,slab,text=gray!40!black,inner sep=1pt] {orbits};
\node[slab,anchor=south east] at (-1.42,0.74) {$T_\ell$};
\draw[-{Stealth[length=1.8mm]}] (0,-1.05) -- node[right=1.5pt,slab]{$\pi$} (0,-1.55);
\draw[line width=0.9pt,diagramblue] (0,-1.95) ellipse[x radius=1.25,y radius=0.33];
\node[slab,anchor=west] at (1.32,-1.95) {$\ell\subset Q$};
\node[panelcap] at (0,-2.9) {(a) $T_\ell=\pi^{-1}(\ell)$};
\end{scope}
\draw[darrow] (2.1,0) -- node[above=2pt,slab,align=center]{surgery\\ on $\gamma$} (3.2,0);
\begin{scope}[shift={(4.45,0.05)}]
\shade[left color=diagramlight,right color=diagramblue!22] (-0.8,-0.75) rectangle (0.8,0.75);
\shade[left color=diagramlight,right color=diagramblue!22] (0,-0.75) ellipse[x radius=0.8,y radius=0.22];
\fill[orbitred!20] (0,-0.75) ellipse[x radius=0.8,y radius=0.22];
\draw[orbitred!60,densely dashed] (0.8,-0.75) arc[start angle=0,end angle=180,x radius=0.8,y radius=0.22];
\shade[left color=diagramlight,right color=diagramblue!22] (-0.8,-0.75) -- (-0.8,0.75) -- (0.8,0.75) -- (0.8,-0.75)
  arc[start angle=0,end angle=-180,x radius=0.8,y radius=0.22] -- cycle;
\draw[orbitred!80!black,line width=0.5pt] (-0.8,-0.75) arc[start angle=180,end angle=360,x radius=0.8,y radius=0.22];
\draw[diagramblue!70,line width=0.5pt] (-0.8,-0.75) -- (-0.8,0.75) (0.8,-0.75) -- (0.8,0.75);
\foreach \x in {-0.55,0,0.55}
  \draw[gray!35!black,line width=0.55pt] (\x,{-0.75-0.22*sqrt(1-(\x/0.8)^2)}) -- (\x,{0.75-0.22*sqrt(1-(\x/0.8)^2)});
\fill[orbitred!30] (0,0.75) ellipse[x radius=0.8,y radius=0.22];
\draw[orbitred!80!black,line width=0.5pt] (0,0.75) ellipse[x radius=0.8,y radius=0.22];
\node[slab,anchor=west] at (0.86,0.82) {$D^2\times\{n\}$};
\node[slab,anchor=west] at (0.86,-0.85) {$D^2\times\{s\}$};
\node[slab,anchor=west] at (0.86,0) {$S^1_a\times c$};
\node[slab,text=gray!40!black,fill=diagramlight,inner sep=0.5pt] at (0.27,0.12) {$c$};
\node[slab] at (0,-1.45) {$S_T\cong S^2$};
\node[panelcap] at (0,-2.95) {(b) the sphere $S_T$};
\end{scope}
\begin{scope}[shift={(8.05,0.55)}]
\node[sp] (Y) at (0,0) {$Y$};
\node[sp=gray] (Sg) at (2.3,0) {$\Sigma$};
\node[sp] (Yp) at (4.6,0) {$Y'(\ell'')$};
\node[orbsp] (Q) at (0,-2.05) {$Q$};
\node[orbsp] (Qp) at (4.6,-2.05) {$Q(\ell'')$};
\draw[darrow] (Y) -- node[above=2.4mm,slab,align=center]{surgery on $\gamma$\\ framing $\lambda(\ell'')$} (Sg);
\draw[darrow] (Sg) -- node[above=2.4mm,slab,align=center]{surgery\\ on $S_T$} (Yp);
\draw[darrow,densely dashed] (Yp) to[bend left=38] node[below=1.5pt,slab]{surgery on $\delta$} (Sg);
\draw[darrow] (Y) -- node[left=1.5pt,slab]{$\pi$} (Q);
\draw[darrow] (Yp) -- node[right=1.5pt,slab]{$\pi'$} (Qp);
\draw[darrow] (Q) -- node[below=2pt,slab]{Dehn filling along $\ell''$} (Qp);
\node[sp,inner sep=0pt,minimum size=7pt] at (0.3,-2.95) {};
\node[slab,anchor=west,text=inkgray] at (0.48,-2.95) {with a circle action};
\node[sp=gray,inner sep=0pt,minimum size=7pt] at (3.3,-2.95) {};
\node[slab,anchor=west,text=inkgray] at (3.48,-2.95) {none asserted};
\node[panelcap] at (2.3,-3.45) {(c) the two surgeries and the orbit-space filling};
\end{scope}
\end{tikzpicture}
\caption{Surgery on a section of a fibre torus (Proposition~\ref{std:prop:torus}).
(a) Over a loop $\ell$ in the free stratum of $Q$ the orbits sweep out the torus
$T_\ell=\pi^{-1}(\ell)$ (meridian circles of the drawn torus), and $\gamma$ is a section
of $T_\ell\to\ell$. (b) Surgery on $\gamma$ cuts $T_\ell$ along $\gamma$, leaving the
annulus $S^1_a\times c$ whose boundary circles are parallel to $\gamma$, and caps it
with the discs $D^2\times\{n\}$ and $D^2\times\{s\}$; the result is the sphere $S_T$,
on which the former orbits have become the arcs $\{a\}\times c$. (c) Surgery on $S_T$
completes a torus surgery on $T_\ell$ across which the circle action extends; on
orbit spaces it is the Dehn filling $Q(\ell'')$, and $\Sigma$ is recovered from
$Y'(\ell'')$ by surgery on the core circle $\delta$. The vertical arrows are the orbit
maps of $Y$ and $Y'(\ell'')$ onto the orbit orbifolds $Q$ and $Q(\ell'')$ (unboxed); no
extension of the original action to the intermediate $\Sigma$ (grey) is asserted here.}
\label{std:fig:torus}
\end{figure}

\begin{remark}[Relation with torus surgery]\label{std:rem:BS}
The composite operation $Y\rightsquigarrow\Sigma\rightsquigarrow Y'(\ell'')$ is
the torus surgery on $T_\ell$ whose new meridian is the
$\lambda(\ell'')$-push-off of $\gamma$. Thus, Proposition~\ref{std:prop:torus}
is an equivariant refinement, in the present local setting, of the standard
round-handle factorization cited immediately before its proof. The circle
action is never used on $\Sigma$ itself. Torus surgery on tori in $S^4$ and
its relation to circle and sphere surgeries are also discussed in
\cite{Larson2018}. Iwase \cite{Iwase1988} showed that a Gluck twist of $S^4$
is a multiplicity-one logarithmic transformation on a stabilization of the
sphere; see \cite[arXiv preprint, Theorem~1.9]{NaylorSchwartz2020}. For circle
actions on $S^4$, the decomposition into a twin neighbourhood and the product of $S^1$ with a
knot exterior, and the effect of a framing change on the gluing, are described
in \cite[Section~3]{FintushelSternSunukjian2009}; earlier standardness
arguments for homotopy four-spheres obtained by regluing twin neighbourhoods
are given in \cite[Section~6]{Plotnick}.
\end{remark}

\subsection{Two recognition criteria}\label{std:sec:criteria}

\begin{lemma}[The fundamental group of the total space]\label{std:lem:exact}
Let $Y'$ be a closed connected oriented smooth four-manifold with a smooth
effective circle action without fixed points, and write $Q'$ for its orbit
orbifold. Let $h\in\pi_1(Y')$ be the class of a principal orbit. Then $h$ is central,
and the orbit map induces an exact sequence
\[
 \langle h\rangle\to\pi_1(Y')\to\stdpi(Q')
 \to1.
\]
If $O$ is an exceptional orbit with isotropy of order $k$, then $[O]^k$ is
conjugate to $h^{\pm1}$.
\end{lemma}

\begin{proof}
Centrality follows from Lemma~\ref{per:lem:gottlieb}. For exactness, use
the Borel model
\[
                 BQ'=ES^1\times_{S^1}Y'
\]
for the classifying space of the orbit orbifold. The diagonal circle
action on $ES^1\times Y'$ is free, so
\[
 S^1\to ES^1\times Y'\to BQ'
\]
is an ordinary principal circle bundle. Since $ES^1$ is contractible,
its homotopy exact sequence gives
\[
 \pi_1(S^1)\to\pi_1(Y')
       \to\pi_1(BQ')\to1.
\]
Here $\pi_1(BQ')=\stdpi(Q')$, and the image of the fibre homomorphism
is $\langle h\rangle$. This proves the asserted exactness even
when that homomorphism is not injective. The classifying-space
description of Seifert fibrations and orbifold fundamental groups is
given in \cite[Corollary~3.1.5 and Sections~4.2.1, 4.3.3,
pp.~83, 91, 95]{Haefliger1984}; compare also
\cite{LeeRaymond2010,Fintushel1978}. For $\M$ the sequence is the
central extension \eqref{per:eq:centralext}, and in dimension three
it specializes to the familiar sequence for Seifert fibrations
\cite{Scott1983}.

The last assertion follows from the slice model
$S^1\times_{C_k}D^3$, which deformation retracts onto $O$ and in
which a principal orbit winds $k$ times around $O$.
\end{proof}

In both criteria below, $\delta\subset Y'$ is an embedded circle such that some
framed surgery on it yields a homotopy four-sphere $\Sigma$. Then
\begin{equation}\label{std:eq:H1}
 H_1(Y')\cong\ZZ,\;\;\; [\delta]\ \text{is a generator},
\end{equation}
because removing a circle does not change the first homology group,
$Y'\setminus\nu\delta=\Sigma\setminus\nu S$ for the belt sphere $S$, and
Alexander duality in the homology four-sphere $\Sigma$ (the proof of
\cite[Theorem~3.44]{Hatcher} uses only the homology of the ambient sphere)
gives
$H_1(\Sigma\setminus S)\cong\ZZ$. The surgery on
$\delta$ kills $\pi_1(Y')$, so $\delta$ normally generates it. We call the
image of an element under $\pi_1(Y')\to H_1(Y')\cong\ZZ$ its \emph{degree}.

Abelianizing the exact sequence in Lemma~\ref{std:lem:exact} gives
\[
 \langle[h]\rangle\to H_1(Y')\cong\ZZ
 \to H_1^{\mathrm{orb}}(Q')\to0.
\]
Consequently, if $H_1^{\mathrm{orb}}(Q')$ is finite of order $d$, then
$\deg h=\pm d$. In particular, $h$ has infinite order. This gives
$d=n$ in the cyclic criterion and $d=2$ in the odd dihedral criterion,
and justifies both the injectivity assertion and the orbit-degree
calculation used below.

\begin{proposition}[Cyclic base]\label{std:prop:cyclic}
If $\stdpi(Q')$ is finite cyclic, then $Y'\diff S^3\times S^1$ and
$\Sigma\diff S^4$.
\end{proposition}

\begin{proof}
By Lemma~\ref{std:lem:exact}, $\pi_1(Y')$ is a central extension of a cyclic
group by the cyclic group $\langle h\rangle$. It is therefore abelian, so
$\pi_1(Y')\cong\ZZ$ by \eqref{std:eq:H1}. Let $n=|\stdpi(Q')|$. Since
$H_1(Y')/\langle h\rangle$ is finite of order $n$, we have $h=\tau^{\pm n}$ for
a generator $\tau$.

Since $h$ has infinite order, the action is injective in Chen's sense.
We verify the irreducibility hypothesis in the cited classification.
Injectivity implies that $Q'$ contains no bad two-suborbifold; hence it
admits a finite manifold cover
\cite[Section~2 and the proof of Theorem~1.4]{Chen2015}.
Since $\stdpi(Q')$ is finite, the universal orbifold cover of $Q'$ is a
closed simply connected three-manifold and is therefore diffeomorphic to
$S^3$ \cite{PerelmanEntropy,PerelmanSurg,PerelmanExt}. Its finite deck action is smoothly
conjugate to an isometric action by \cite[Theorem~E]{DinkelbachLeeb}.
Consequently $Q'$ is spherical and, in particular, irreducible, so $Y'$ is an
irreducible $S^1$-four-manifold in the terminology of \cite{Chen2015}.
Now $\pi_1(Y')\cong\ZZ$ is double-ended, and
\cite[Theorem~1.6(2)]{Chen2015} identifies $Y'$ with the mapping torus of a
periodic diffeomorphism $g$ of an elliptic three-manifold $F$.
The mapping-torus exact sequence embeds the finite group $\pi_1(F)$ in
$\pi_1(Y')\cong\ZZ$, so $\pi_1(F)=1$ and $F\diff S^3$.
Orientability of $Y'$ implies that $g$ preserves orientation. By Cerf's
theorem \cite{Cerf1968}, $g$ is isotopic to the identity, and hence
$Y'\diff S^3\times S^1$.

By \eqref{std:eq:H1}, $\delta$ represents a generator of
$\pi_1(Y')\cong\ZZ$, as does $\{\mathrm{pt}\}\times S^1$.
Freely homotopic embedded circles in a four-manifold are isotopic
(Lemma~\ref{lem:circle-spin}). Both normal-framing classes yield the
standard four-sphere, as follows. Fix
$x_0\in\partial D^4=S^3$, and choose a loop
$R_t$ in the stabilizer $SO(3)\subset SO(4)$ of $x_0$ such that $R_t$ represents the
nontrivial element of $\pi_1(SO(3))$. Then the map
\[
   F:D^4\times S^1\to D^4\times S^1,
   \;\;\;
   F(x,e^{2\pi i t})=(R_t x,e^{2\pi i t})
\]
is a diffeomorphism that fixes the attaching circle
$\{x_0\}\times S^1$ pointwise and changes its normal framing by the
nontrivial element of $\pi_1(SO(3))$. Hence, the two five-dimensional
two-handle attachments along this circle are diffeomorphic. For the product
framing the two-handle cancels the unique one-handle of
$D^4\times S^1$, leaving a five-ball; consequently the boundary after either
framed circle surgery is $S^4$ \cite{GS,Gluck1962}. Therefore, $\Sigma\diff S^4$.
\end{proof}

\begin{proposition}[Odd dihedral base]\label{std:prop:dihedral}
Suppose that $\stdpi(Q')$ is dihedral of order $2p$ with $p$ odd, and that
$Y'$ has an exceptional orbit $e$ with isotropy of order two. Then
$\Sigma\diff S^4$.
\end{proposition}

\begin{proof}
Since $p$ is odd, the abelianization of $\stdpi(Q')$ is $\ZZ/2$. By
Lemma~\ref{std:lem:exact} and \eqref{std:eq:H1}, $h$ has degree $\pm2$, and
$e$ has degree $\pm1$ because $e^2$ is conjugate to $h^{\pm1}$. Let $K$ be the
commutator subgroup of $\pi_1(Y')$. It maps onto the commutator subgroup
$C_p$ of the dihedral group, with kernel $K\cap\langle h\rangle$. This kernel
is trivial because $h^k$ has degree $\pm2k$. Thus, $K\cong C_p$, and
$\pi_1(Y')=K\rtimes\langle t\rangle$ for any element $t$ of degree one. The
image of $t$ in the dihedral group is a reflection, since its image in the
abelianization is nonzero, so $t$ acts on $K$ by inversion. For $y\in K$,
\[
 y\,t\,y^{-1}=y\,(t\,y^{-1}t^{-1})\,t=y^2t,
\]
and squaring is a bijection of $C_p$. Hence, all elements of degree one are
conjugate. Orient $\delta$ and $e$ to have degree one; they are then freely
homotopic, hence isotopic, and the isotopy transports the framing. So
$\Sigma$ is obtained by a framed surgery on the orbit $e$. By
Lemma~\ref{fa:lem:orbit-surgery}, the resulting manifold carries an effective smooth
circle action for either framing class. A smooth homotopy four-sphere
admitting an effective smooth circle action is diffeomorphic to $S^4$
\cite{Pao,FintushelHomotopy,Fintushel1978}; see also
Remark~\ref{fa:rem:Pao}.
\end{proof}

\subsection{A belt around a clasp of the Whitehead model}
\label{std:sec:belt}

We use the Whitehead model of Section~\ref{wh:sec:whitehead}:
$|\OW|=S^3_{+3}(J)$ with singular circle $W'$ of cone angle $\pi$, where
$J\cup W'$ is the Whitehead link in its standard five-crossing diagram. The
labelling is that of Section~\ref{wh:sec:whitehead}: $J$ is the component
without a self-crossing and $W'$ is the component carrying the self-crossing.
The two components of the Whitehead link are interchanged by an
orientation-preserving ambient isotopy. The arguments below use only this
diagram, so they also apply to the other mirror image with the same
coefficient $+3$ on $J$; mirroring the diagram changes only the sign of the
integer $s_*$ below (and hence the value of $r$), and the conclusions about
the cyclic and dihedral groups are unchanged.
The diagram has two bigon faces bounded by arcs of both components. At each
bigon, the two strands form a twist region with two crossings, and $W'$
passes over $J$ at one crossing and under $J$ at the other.

\begin{definition}[The clasp belt]\label{std:def:belt}
Fix one of the two bigons. Let $D\subset S^3$ be a small disc that is transverse to
the plane of the diagram and meets $J$ and $W'$ in one point each between
the two crossings of the bigon. Put $\ell_A=\partial D$
(Figure~\ref{std:fig:belt}). Let $b\subset D$ be the straight arc joining the
two points of $(J\cup W')\cap D$, and let $J\natural_bW'$ be the band sum along
the flat band in the diagram plane that contains $b$. For $s\in\ZZ$, let $Q_s$
be the orbifold obtained from $\OW$ by Dehn surgery on $\ell_A$ with slope
$s\mu_A+\lambda_A$, where $\mu_A$ and $\lambda_A$ are the meridian and the Seifert
longitude of $\ell_A$ in $S^3$.
\end{definition}

Since $\ell_A\subset S^3\setminus J$, its Seifert longitude is a framing of
$\ell_A$ in $|\OW|$. The slopes $s\mu_A+\lambda_A$ are thus exactly the framings of
$\ell_A$ in the sense of Definition~\ref{std:def:data}, and in that notation
$Q_s=\OW(\lambda_A+s\mu_A)$.

\begin{figure}[tbp]
\centering
\begin{tikzpicture}[
  jstrand/.style={line width=0.95pt,line cap=round,line join=round,draw=diagramblue!75!black},
  wstrand/.style={line width=0.95pt,line cap=round,line join=round,draw=orbitred},
  belt/.style={line width=0.85pt,line cap=round,line join=round,draw=belt},
  bigon/.style={fill=bigonfill},
  band/.style={fill=bandfill}]
\begin{scope}[scale=0.82]
\fill[bigon] (1.8731,0.7010) arc[start angle=20.5171, end angle=-20.5171, radius=2.0] -- plot[domain=-0.9672:0.9672, samples=40] ({3.3*cos(\x r)},{1.5*sin(\x r)*cos(\x r)}) -- cycle;
\fill[bigon] (-1.8731,-0.7010) arc[start angle=200.5171, end angle=159.4829, radius=2.0] -- plot[domain=2.1744:4.1088, samples=40] ({3.3*cos(\x r)},{1.5*sin(\x r)*cos(\x r)}) -- cycle;
\fill[band] (2.0,-0.075) rectangle (3.3,0.075);
\draw[jstrand,postaction={decorate,decoration={markings,mark=at position 0.62 with {\arrow{Stealth[length=1.6mm]}}}}] (1.8153,0.8394) arc[start angle=24.8142, end angle=196.2199, radius=2.0];
\draw[jstrand] (-1.8153,-0.8394) arc[start angle=204.8142, end angle=349.0271, radius=2.0];
\draw[jstrand] (1.9983,-0.0830) arc[start angle=357.6215, end angle=376.2199, radius=2.0];
\draw[wstrand,postaction={decorate,decoration={markings,mark=at position 0.5 with {\arrowreversed{Stealth[length=1.6mm]}}}}] plot[domain=1.6122:4.0546, samples=70] ({3.3*cos(\x r)},{1.5*sin(\x r)*cos(\x r)});
\draw[wstrand] plot[domain=4.1630:5.2618, samples=35] ({3.3*cos(\x r)},{1.5*sin(\x r)*cos(\x r)});
\draw[wstrand] plot[domain=5.3702:6.0178, samples=24] ({3.3*cos(\x r)},{1.5*sin(\x r)*cos(\x r)});
\draw[wstrand] plot[domain=6.2153:7.8126, samples=48] ({3.3*cos(\x r)},{1.5*sin(\x r)*cos(\x r)});
\draw[belt] (3.0991,0.2712) arc[start angle=64.6764, end angle=118.9091, x radius=1.05, y radius=0.3];
\draw[belt,postaction={decorate,decoration={markings,mark=at position 0.5 with {\arrowreversed{Stealth[length=1.6mm]}}}}] (1.8518,0.1949) arc[start angle=139.4810, end angle=405.0515, x radius=1.05, y radius=0.3];
\fill[black] (2.0,0) circle[radius=0.045];
\fill[black] (3.3,0) circle[radius=0.045];
\node[lab,text=diagramblue!75!black] at (0,-2.32) {$J$};
\node[lab,text=orbitred] at (-3.78,0) {$W'$};
\node[lab,text=belt] at (4.02,0.02) {$\ell_A$};
\node[lab] at (-0.36,0.5) {$c$};
\node[lab,font=\scriptsize] at (-2.62,0.02) {bigon};

\node[lab,font=\scriptsize] at (1.2,1.05) {$c_0$};
\node[lab,font=\scriptsize] at (1.2,-1.05) {$c_1$};
\node[lab,font=\scriptsize] at (-1.2,1.05) {$c_2$};
\node[lab,font=\scriptsize] at (-1.2,-1.05) {$c_3$};
\node[panelcap] at (0,-3.54) {(a) the diagram $J\cup W'$ and the belt $\ell_A$};
\end{scope}
\begin{scope}[shift={(8.9cm,0)},scale=1.75]
\begin{scope}[shift={(-2.55,0)}]
\clip (1.05,-1.35) rectangle (4.05,1.35);
\fill[bigon] (1.8731,0.7010) arc[start angle=20.5171, end angle=-20.5171, radius=2.0] -- plot[domain=-0.9672:0.9672, samples=40] ({3.3*cos(\x r)},{1.5*sin(\x r)*cos(\x r)}) -- cycle;
\fill[bigon] (-1.8731,-0.7010) arc[start angle=200.5171, end angle=159.4829, radius=2.0] -- plot[domain=2.1744:4.1088, samples=40] ({3.3*cos(\x r)},{1.5*sin(\x r)*cos(\x r)}) -- cycle;
\fill[belt!12,opacity=0.65] (2.65,0) ellipse[x radius=1.05,y radius=0.3];
\fill[band] (2.0,-0.075) rectangle (3.3,0.075);
\draw[jstrand,postaction={decorate,decoration={markings,mark=at position 0.62 with {\arrow{Stealth[length=1.8mm]}}}}] (1.8153,0.8394) arc[start angle=24.8142, end angle=196.2199, radius=2.0];
\draw[jstrand] (-1.8153,-0.8394) arc[start angle=204.8142, end angle=349.0271, radius=2.0];
\draw[jstrand] (1.9983,-0.0830) arc[start angle=357.6215, end angle=376.2199, radius=2.0];
\draw[wstrand,postaction={decorate,decoration={markings,mark=at position 0.5 with {\arrowreversed{Stealth[length=1.8mm]}}}}] plot[domain=1.6122:4.0546, samples=70] ({3.3*cos(\x r)},{1.5*sin(\x r)*cos(\x r)});
\draw[wstrand] plot[domain=4.1630:5.2618, samples=35] ({3.3*cos(\x r)},{1.5*sin(\x r)*cos(\x r)});
\draw[wstrand] plot[domain=5.3702:6.0178, samples=24] ({3.3*cos(\x r)},{1.5*sin(\x r)*cos(\x r)});
\draw[wstrand] plot[domain=6.2153:7.8126, samples=48] ({3.3*cos(\x r)},{1.5*sin(\x r)*cos(\x r)});
\draw[belt] (3.0991,0.2712) arc[start angle=64.6764, end angle=118.9091, x radius=1.05, y radius=0.3];
\draw[belt,postaction={decorate,decoration={markings,mark=at position 0.5 with {\arrowreversed{Stealth[length=1.8mm]}}}}] (1.8518,0.1949) arc[start angle=139.4810, end angle=405.0515, x radius=1.05, y radius=0.3];
\fill[black] (2.0,0) circle[radius=0.045];
\fill[black] (3.3,0) circle[radius=0.045];
\node[lab,text=diagramblue!75!black] at (1.45,1.1) {$J$};
\node[lab,text=orbitred] at (3.62,0.62) {$W'$};
\node[lab,text=belt] at (3.85,-0.3) {$\ell_A$};
\node[lab,font=\scriptsize] at (2.62,0.17) {$D$};
\node[lab,font=\scriptsize] at (2.62,-0.17) {$b$};
\node[lab,font=\scriptsize] at (2.05,0.95) {$c_0$};
\node[lab,font=\scriptsize] at (2.05,-0.95) {$c_1$};
\end{scope}
\node[panelcap] at (0,-1.66) {(b) the magnified bigon};
\end{scope}
\end{tikzpicture}
\caption{The clasp belt in the five-crossing Whitehead diagram. (a) The diagram $J\cup W'$. The component $J$ (dark blue) has no self-crossings, $c$ is the negative self-crossing of $W'$ (red),
and $c_0,\ldots,c_3$ are the four mixed crossings: $c_0,c_1$ are positive and $c_2,c_3$ are negative. The two bigon faces bounded by arcs of both components are shaded. The belt $\ell_A$ encircles the two strands of the right-hand bigon between
its crossings $c_0$ and $c_1$. The belt's upper arc passes behind both strands, and its lower arc passes in front of both; the shaded
strip is the flat band containing $b$. For the orientations shown, we have
$\lk(J,\ell_A)=-1$, $\lk(W',\ell_A)=+1$, and $\lk(J,W')=0$. (b) The magnified bigon: $D$ is the disc transverse to the diagram plane bounded by $\ell_A$, shown by its blue-tinted elliptic projection (the yellow regions are planar bigons), meeting $J$ and $W'$ in the two
marked points, and $b$ is the arc joining them. The crossings $c_0$ and $c_1$ form one full twist of the two strands through $D$.}
\label{std:fig:belt}
\end{figure}

\begin{theorem}[Two consecutive belt surgeries]\label{std:thm:belt}
\begin{enumerate}[label=(\roman*),leftmargin=2.2em]
\item $\ell_A$ lies in the free part of $\OW$, is a band sum of a meridian of
      $J$ and a meridian of $W'$, and is admissible in the sense of
      Theorem~\ref{wh:thm:section}.
\item $|Q_0|=S^3$, and the singular circle of $Q_0$ is the knot
      $J\natural_bW'$, which is a trefoil. Hence, $\stdpi(Q_0)$ is dihedral of
      order six.
\item There exists $s_*\in\{1,-1\}$ with the following property. Put
      $r=3-s_*\in\{2,4\}$. Then $|Q_{s_*}|=S^3_r(J)$ is a lens space, and the
      singular circle of $Q_{s_*}$ is the core of the Heegaard solid torus
      $S^3\setminus\Int\nu J$. Consequently $\stdpi(Q_{s_*})\cong\ZZ/2r$, the
      meridian of the singular circle maps to the unique element of order two, and
      the complement of the singular circle is a solid torus.
\end{enumerate}
\end{theorem}

\begin{proof}
(i) The loop $\ell_A$ lies in $S^3\setminus(J\cup W')$, hence in the free part
$|\OW|\setminus W'$. The disc $D$ meets $J$ and $W'$ once each, so $\ell_A$ is
isotopic, in the complement of $J\cup W'$, to the band sum of a meridian of
$J$ and a meridian of $W'$ along $b$. Such loops are admissible by
Theorem~\ref{wh:thm:section}(ii).

(ii) In $|\OW|=S^3_{+3}(J)$ the singular circle may be moved by handle slides
over $J$. Let $J^{(3)}$ be a $(+3)$-framed parallel of $J$, pushed slightly off
$\nu J$ and meeting $D$ in one point of the arc $b$, between $J$ and $W'$, and
let $b'\subset b$ be the subarc joining $W'$ to $J^{(3)}$. Slide $W'$ over $J$
along the flat band containing $b'$. The resulting knot
$W''=W'\natural_{b'}J^{(3)}$ is isotopic to $W'$ in $S^3_{+3}(J)$ by an
isotopy supported near this band, the surgery solid torus and the thin annulus
between $\partial\nu J$ and $J^{(3)}$. This support is disjoint from
$\ell_A=\partial D$, so the isotopy fixes $\ell_A$ together with its framings
$s\mu_A+\lambda_A$. Moreover $W''$ is disjoint from $D$: the strands of $W'$
and of $J^{(3)}$ pass through $D$ in opposite directions, and the band sum
along $b'$ removes both intersection points. Now $D$ meets only $J$, once, so
$\ell_A$ is a meridian of $J$ in the complement of $W''$. The slam dunk
\cite[Sections~5.1 and~5.3]{GS} replaces the pair $(J,3)\cup(\ell_A,s)$ by
$(J,3-1/s)$. For $s=0$ the new coefficient is $\infty$: both curves disappear,
$|Q_0|=S^3$, and the singular circle is $W''$. The slam dunk is performed relative to $W''$: its support can be
chosen in a sufficiently thin neighbourhood of $\nu J\cup D$
disjoint from $W''$. The resulting $S^3$ is the unfilled diagram,
with $W''$ unchanged. We now track the band as the chosen surgery-framed
parallel is moved back to $J$.

Make the initial choice of $J^{(3)}$ in a slightly enlarged tubular
neighbourhood $U$ of $J$, disjoint from $W'$. In zero-framed
coordinates $U=S^1\times D^2$, write
\[
 J^{(3)}=\{(\theta,\epsilon v(\theta)):\theta\in S^1\},
 \;\;\; 0<\epsilon<1,\;\;\; \deg v=3.
\]
The map $v:S^1\to S^1$ can be chosen constant on the small attaching
interval $I\subset S^1$ of the band, pointing along the original
flat band there. Arrange the portion of the band in $U$ to run
radially outward from $J^{(3)}$ on $I$, at normal radius at least
$\epsilon$. After the slam dunk, the original neighbourhood $U$
is again present in the unfilled $S^3$, and the annulus
\[
 A=\{(\theta,\rho\epsilon v(\theta)):
                  \theta\in S^1,\ 0\le\rho\le1\}
\]
is disjoint from $W'$ and meets the band only along its edge on
$J^{(3)}$. Move $J^{(3)}$ to
$J_\rho=\{(\theta,\rho\epsilon v(\theta)):\theta\in S^1\}$
while extending the band across the corresponding radial strip
in $A$. The extension is untwisted on $I$, since $v$ is constant
there. After the corners are smoothed, these band sums form an isotopy
of knots for $1\ge\rho\ge0$. At $\rho=0$ the attaching component
is $J$ and the extended band is the original flat band containing
$b$. Thus $W''$ is isotopic to $J\natural_bW'$.

In the diagram, the flat band between the two crossings of the bigon turns
both crossings into nugatory kinks. Explicitly, in the notation of
Figure~\ref{std:fig:belt}, the knot $J\natural_bW'$ runs along one edge of
the band from $J$ to $W'$, then once around $W'$, meeting in order
$c_0$ (over), $c$ (under), $c_3$ (over), $c_2$ (under), $c$ (over) and
$c_1$ (under). It then runs along the other edge of the band back to $J$ and once
around $J$ in the opposite sense, meeting $c_1$ (over), $c_3$ (under),
$c_2$ (over) and $c_0$ (under). The arc of $J$ and the arc of $W'$ that bound
the bigon between $c_0$ and the band form, together with that edge of the
band, a loop through $c_0$ containing no other crossing; the same holds for $c_1$.
Removing these two kinks leaves a three-crossing diagram whose crossings are
met in the cyclic order
\[
   c^{-},\;\; c_3^{+},\;\; c_2^{-},\;\; c^{+},\;\; c_3^{-},\;\; c_2^{+},
\]
where the exponent records an under-passage ($-$) or an over-passage ($+$).
Thus, each crossing is met once over and once under, over- and
under-passages alternate, and the cyclic crossing pattern is
$1\,2\,3\,1\,2\,3$; no
crossing is nugatory, because between the two passages through any crossing
each of the other two crossings occurs exactly once. This is the standard diagram of a trefoil, with the Wirtinger presentation
\cite[Chapter~3]{Rolfsen1976}
\begin{align*}
   \pi_1\bigl(S^3\setminus(J\natural_bW')\bigr)
   &\cong\bigl\langle x_1,x_2,x_3\bigm|
        x_2=x_3^{\,\eta}x_1x_3^{-\eta},\;
        x_3=x_1^{\,\eta}x_2x_1^{-\eta},\;
        x_1=x_2^{\,\eta}x_3x_2^{-\eta}\bigr\rangle\\
   &\cong\langle x,y\mid xyx=yxy\rangle
   \cong\langle a,b\mid a^2=b^3\rangle,
\end{align*}
where $\eta=\pm1$ is the common sign of the three crossings, the second
isomorphism eliminates $x_3$ by means of the second relation (the third
relation then becomes redundant), and the third isomorphism is given by $a=xyx$ and $b=xy$.
This group is nonabelian: imposing $a^2=b^3=1$ gives a surjection onto
$C_2*C_3$ (and, in particular, onto $S_3$). Hence, the knot is not the
unknot. Since it admits the displayed nonnugatory three-crossing diagram, it is a
trefoil (up to mirror image); compare the standard three-crossing diagram in
\cite{Rolfsen1976}. To compute the orbifold group directly,
use the Wirtinger meridians $x$ and $y$ in the middle presentation above. Imposing
the order-two meridian relation gives $x^2=y^2=1$, and with these relations
the braid relation $xyx=yxy$ is equivalent to $(xy)^3=1$. Therefore
\[
 \stdpi(Q_0)
  \cong\langle x,y\mid x^2=y^2=(xy)^3=1\rangle
  \cong S_3,
\]
the dihedral group of order six.

(iii) For $s=\pm1$, perform the Rolfsen twist along $D$ with $t=-s$ full twists
\cite[Chapter~9]{Rolfsen1976}, \cite[Section~5.3]{GS}. It changes the
coefficient of $\ell_A$ to the meridional slope $\infty$ (since $1/s+t=0$), so $\ell_A$ disappears. It
changes the coefficient of $J$ to $3+t\,\lk(J,\ell_A)^2=3-s$, and it inserts
$t$ full twists into the two strands through $D$. These two strands already
form one full twist, namely the two crossings of the bigon. For exactly one
sign $s=s_*$ the inserted twist cancels it, and the two crossings disappear by
Reidemeister moves. In the resulting diagram of $J\cup W'_*$ the component $J$ has no
self-crossings and meets $W'_*$ only at the two crossings $c_2$ and $c_3$ of
the other bigon, where $W'_*$ passes under $J$ at $c_2$ and over $J$ at
$c_3$. Let $D_J$ be the spanning disc of $J$ obtained from the flat disc in
the diagram plane by pushing its interior below every arc of $W'_*$, keeping
only a thin collar of $J$ near the plane. The knot $W'_*$ enters this collar
only at $c_2$ and $c_3$, and at $c_3$ it stays above $D_J$. Hence, $D_J$ meets
$W'_*$ in exactly one point, at $c_2$, so that $\lk(J,W'_*)=\pm1$.
Moreover, shrinking $D_J$ onto a small disc
about this point is an isotopy of $J$, in the complement of $W'_*$, onto a
meridian of $W'_*$; thus $J$ is a meridian of $W'_*$ (the linking
number alone would not establish this). Twisting
strands of different components preserves the knot type of each component,
so $W'_*$ is unknotted. Therefore, $J\cup W'_*$ is a Hopf link, and $W'_*$ is the
core of the solid torus $S^3\setminus\Int\nu J$.

After $r$-surgery on $J$, the complement of $W'_*$ is the surgery solid torus
$V_J$ together with a collar. Choose its core generator $a$ so that
$\mu_J=a$. Since $r\mu_J+\lambda_J$ bounds a meridian disc of $V_J$, we have
$\lambda_J=a^{-r}$. Put
$\varepsilon=\lk(J,W'_*)\in\{1,-1\}$. In the oriented Hopf-link exterior,
the preferred longitude of $J$ satisfies
$\lambda_J=\mu_{W'_*}^{\,\varepsilon}$, and therefore
$
  \mu_{W'_*}=a^{-\varepsilon r}.
$
Consequently,
$
 \stdpi(Q_{s_*})
 =\langle a\mid a^{2r}=1\rangle
 \cong\ZZ/2r,
$
and the meridian of the singular circle maps to its unique element of
order two. For the orientations of Figure~\ref{std:fig:belt}, the cancelling
twist gives $\varepsilon=-1$, so the oriented meridian is $a^r$ in the
infinite cyclic group of the singular-circle exterior. In particular, it
is $a^4$ when $s_*=-1$ and $r=4$, in agreement with the peripheral
calculation below.
\end{proof}

Figure~\ref{std:fig:moves} isolates the slopes and the two local moves used in
this proof, and Figure~\ref{std:fig:fillings} shows the two fillings for the
diagram of Figure~\ref{std:fig:belt}.

\begin{figure}[tbp]
\centering
\begin{tikzpicture}[
  ida/.style={-{Stealth[length=1.4mm]},gray!60!black},
  idb/.style={-{Stealth[length=1.4mm] Stealth[length=1.4mm]},gray!60!black},
  jst/.style={line width=1.1pt,line cap=round,draw=diagramblue!75!black},
  wst/.style={line width=1.1pt,line cap=round,draw=orbitred},
  beltst/.style={draw=liftblue,line width=0.85pt},
  casing/.style={draw=white,line width=3.6pt},
  over/.style={preaction={draw=white,line width=3.8pt}},
  oarrow/.style={postaction={decorate,decoration={markings,mark=at position #1 with
    {\arrow{Stealth[length=1.8mm,width=1.4mm]}}}}}]
\begin{scope}
\shade[bottom color=diagramlight,top color=diagramblue!12] (0,0) rectangle (2.0,2.0);
\draw[ida] (0,0) -- (1.1,0); \draw[gray!60!black] (1.0,0) -- (2.0,0);
\draw[ida] (0,2.0) -- (1.1,2.0); \draw[gray!60!black] (1.0,2.0) -- (2.0,2.0);
\draw[idb] (0,0) -- (0,1.05); \draw[gray!60!black] (0,0.95) -- (0,2.0);
\draw[idb] (2.0,0) -- (2.0,1.05); \draw[gray!60!black] (2.0,0.95) -- (2.0,2.0);
\draw[gray!60,line width=0.8pt,densely dashed] (0,0.35) -- (1.65,2.0) (1.65,0) -- (2.0,0.35);
\draw[orbitred,line width=1.2pt] (0,1.65) -- (1.65,0) (1.65,2.0) -- (2.0,1.65);
\draw[liftblue,line width=1.2pt] (0,1.0) -- (2.0,1.0);
\node[slab,below] at (1.0,-0.05) {$\lambda_A$};
\node[slab,left] at (-0.05,1.0) {$\mu_A$};
\node[slab,text=liftblue,anchor=west] at (2.08,1.0) {$s=0$};
\node[slab,text=orbitred,anchor=west] at (2.08,1.65) {$s=-1$};
\node[slab,text=gray!60!black,anchor=west,align=left] at (2.08,0.3) {$s=1$\\ (not used)};
\node[panelcap] at (1.2,-1.05) {(a) the slopes $\lambda_A+s\mu_A$};
\end{scope}
\begin{scope}[shift={(5.0,1.0)}]
\draw[beltst] (0.42,0) arc[start angle=0,end angle=180,x radius=0.42,y radius=0.15];
\draw[jst,over] (0,-0.8) -- (0,0.95);
\draw[beltst,over] (-0.42,0) arc[start angle=180,end angle=360,x radius=0.42,y radius=0.15];
\node[slab,anchor=south] at (0,0.98) {$J$, $3$};
\node[slab,text=liftblue,anchor=east] at (-0.46,0) {$\ell_A$, $0$};
\draw[darrow] (0.62,0) -- node[above=2pt,slab,align=center]{slam\\ dunk} (1.55,0);
\draw[jst,gray!45,densely dashed] (1.95,-0.8) -- (1.95,0.95);
\node[slab,anchor=south] at (1.95,0.98) {$J$, $\infty$};
\node[slab,align=center] at (0.95,-1.38) {$s=0$: meridional filling on $J$;\\ both curves are deleted; $W''\simeq J\natural_bW'$ remains};
\node[panelcap] at (0.95,-2.05) {(b) $s=0$: slam dunk};
\end{scope}
\begin{scope}[shift={(9.75,1.0)}]
\begin{scope}[shift={(0,0.35)}]
\draw[beltst] (0.62,-0.85) arc[start angle=0,end angle=180,x radius=0.62,y radius=0.16];
\draw[wst] plot[domain=-0.5:-0.3,samples=12] ({0.25*cos(360*(\x+0.5))},\x);
\draw[jst] plot[domain=0.5:0.3,samples=12] ({-0.25*cos(360*(\x+0.5))},\x);
\draw[jst] plot[domain=-0.2:-0.5,samples=20] ({-0.25*cos(360*(\x+0.5))},\x);
\draw[wst] plot[domain=0.2:0.5,samples=20] ({0.25*cos(360*(\x+0.5))},\x);
\draw[wst,over] plot[domain=-0.3:0.2,samples=30] ({0.25*cos(360*(\x+0.5))},\x);
\draw[jst,over] plot[domain=0.3:-0.2,samples=30] ({-0.25*cos(360*(\x+0.5))},\x);
\draw[jst,oarrow=0.55] (-0.25,0.85) -- (-0.25,0.5);
\draw[jst] (-0.25,-0.5) -- (-0.25,-1.15);
\draw[wst] (0.25,-1.15) -- (0.25,-0.5);
\draw[wst,oarrow=0.6] (0.25,0.5) -- (0.25,0.85);
\draw[beltst,over] (-0.62,-0.85) arc[start angle=180,end angle=360,x radius=0.62,y radius=0.16];
\node[slab,anchor=south east] at (-0.1,0.87) {$J$, $3$};
\node[slab,text=orbitred,anchor=south west] at (0.15,0.87) {$W'$};
\node[slab,text=liftblue,anchor=east] at (-0.66,-0.85) {$\ell_A$, $-1$};
\node[slab,text=inkgray,anchor=east,align=right] at (-0.42,0.15) {full\\ twist};
\end{scope}
\draw[darrow] (0.75,0.2) -- node[above=2pt,slab,align=center]{Rolfsen\\ twist} node[below=2pt,slab]{$t=1$} (1.75,0.2);
\draw[jst,oarrow=0.15] (2.2,1.2) -- (2.2,-0.8);
\draw[wst,oarrow=0.9] (2.7,-0.8) -- (2.7,1.2);
\node[slab,anchor=south east] at (2.33,1.22) {$J$, $4$};
\node[slab,text=orbitred,anchor=south west] at (2.58,1.22) {$W'_*$};
\node[slab,align=center] at (1.2,-1.38) {$\ell_A$: meridional slope ($\infty$);\\ $J$: $3+1\cdot(-1)^2=4$};
\node[panelcap] at (1.2,-2.05) {(c) $s=s_*=-1$: Rolfsen twist};
\end{scope}
\end{tikzpicture}
\caption{The local moves in the proof of Theorem~\ref{std:thm:belt}, for the mirror
image fixed in Section~\ref{wh:sec:whitehead}. (a) The framings of $\ell_A$ are the
slopes $\lambda_A+s\mu_A$ on $\partial\nu\ell_A$, drawn as a square with opposite sides
identified. By Lemma~\ref{std:lem:frame}, consecutive slopes induce the two
normal-framing classes of the unit section, so the fillings $s=0$ and $s=s_*=-1$
realize both framed surgeries on $\gamma$; the slope $s=1$ is not used. (b) For
$s=0$, after $W'$ has been slid over $J$ off the disc $D$, the belt is a meridian of
$J$, and the slam dunk replaces $(J,3)\cup(\ell_A,0)$ by a meridional ($\infty$) filling on $J$. Both
curves disappear, leaving $S^3$ with the trefoil $J\natural_bW'$ as singular circle.
(c) For $s=-1$ the two strands pass through $D$ in opposite directions and form one
full twist (the two crossings of the bigon, both positive for the orientations
shown), which contributes $+1$ to $\lk(J,W')$. A Rolfsen twist with $t=1$ along $D$
deletes $\ell_A$, changes the coefficient of $J$ to $3+t\,\lk(J,\ell_A)^2=4$, and
inserts a full twist contributing $-1$, which cancels the old one
(Remark~\ref{std:rem:sstar}).}
\label{std:fig:moves}
\end{figure}

\begin{figure}[tbp]
\centering
\begin{tikzpicture}[
  knot/.style={line width=1.0pt,line cap=round,line join=round,draw=orbitred},
  jknot/.style={line width=1.0pt,line cap=round,line join=round,draw=diagramblue!75!black}]
\begin{scope}[scale=0.6]
\fill[bigonfill] (1.99,0.17) -- (1.99,0.20) -- (1.99,0.22) -- (1.98,0.25) -- (1.98,0.28) -- (1.98,0.31) -- (1.97,0.33) -- (1.97,0.36) -- (1.96,0.39) -- (1.96,0.41) -- (1.95,0.44) -- (1.94,0.47) -- (1.94,0.49) -- (1.93,0.52) -- (1.92,0.55) -- (1.92,0.57) -- (1.91,0.60) -- (1.90,0.62) -- (1.89,0.65) -- (1.88,0.68) -- (1.87,0.70) -- (1.93,0.71) -- (2.04,0.73) -- (2.15,0.74) -- (2.26,0.75) -- (2.36,0.75) -- (2.45,0.75) -- (2.55,0.74) -- (2.63,0.72) -- (2.71,0.70) -- (2.79,0.68) -- (2.86,0.65) -- (2.93,0.61) -- (2.99,0.57) -- (3.05,0.53) -- (3.10,0.48) -- (3.15,0.43) -- (3.19,0.38) -- (3.22,0.32) -- (3.25,0.26) -- (3.27,0.20) -- (3.28,0.17) -- cycle;
\fill[bigonfill] (1.99,-0.17) -- (1.99,-0.20) -- (1.99,-0.22) -- (1.98,-0.25) -- (1.98,-0.28) -- (1.98,-0.31) -- (1.97,-0.33) -- (1.97,-0.36) -- (1.96,-0.39) -- (1.96,-0.41) -- (1.95,-0.44) -- (1.94,-0.47) -- (1.94,-0.49) -- (1.93,-0.52) -- (1.92,-0.55) -- (1.92,-0.57) -- (1.91,-0.60) -- (1.90,-0.62) -- (1.89,-0.65) -- (1.88,-0.68) -- (1.87,-0.70) -- (1.93,-0.71) -- (2.04,-0.73) -- (2.15,-0.74) -- (2.26,-0.75) -- (2.36,-0.75) -- (2.45,-0.75) -- (2.55,-0.74) -- (2.63,-0.72) -- (2.71,-0.70) -- (2.79,-0.68) -- (2.86,-0.65) -- (2.93,-0.61) -- (2.99,-0.57) -- (3.05,-0.53) -- (3.10,-0.48) -- (3.15,-0.43) -- (3.19,-0.38) -- (3.22,-0.32) -- (3.25,-0.26) -- (3.27,-0.20) -- (3.28,-0.17) -- cycle;
\fill[bandfill] (1.99,-0.17) rectangle (3.28,0.17);
\draw[knot] plot[smooth,tension=0.5] coordinates {(1.71,-0.66) (1.50,-0.61) (1.30,-0.54) (1.10,-0.47) (0.91,-0.40) (0.71,-0.32) (0.52,-0.23) (0.32,-0.14) (0.12,-0.05) (-0.08,0.04) (-0.28,0.13) (-0.48,0.21) (-0.67,0.30) (-0.87,0.38) (-1.06,0.46) (-1.26,0.53) (-1.46,0.60) (-1.71,0.66)};
\draw[knot] plot[smooth,tension=0.5] coordinates {(-2.04,0.73) (-2.25,0.75) (-2.45,0.75) (-2.66,0.72) (-2.86,0.65) (-3.03,0.55) (-3.17,0.40) (-3.26,0.22) (-3.30,0.02) (-3.27,-0.19) (-3.19,-0.37) (-3.06,-0.52) (-2.89,-0.64) (-2.69,-0.71) (-2.49,-0.74) (-2.29,-0.75) (-2.09,-0.73) (-1.88,-0.70) (-1.69,-0.66) (-1.48,-0.60) (-1.28,-0.54) (-1.08,-0.46) (-0.89,-0.39) (-0.69,-0.31) (-0.50,-0.22) (-0.30,-0.13) (-0.15,-0.07)};
\draw[knot] plot[smooth,tension=0.5] coordinates {(0.15,0.07) (0.34,0.15) (0.54,0.24) (0.73,0.32) (0.92,0.40) (1.11,0.48) (1.32,0.55) (1.51,0.61) (1.72,0.67) (1.92,0.71) (2.12,0.74) (2.32,0.75) (2.53,0.74) (2.73,0.70) (2.91,0.62) (3.08,0.50) (3.20,0.35) (3.24,0.17) (3.02,0.17) (2.81,0.17) (2.59,0.17) (2.38,0.17) (2.16,0.17) (1.99,0.20) (1.96,0.40) (1.93,0.54)};
\draw[knot] plot[smooth,tension=0.5] coordinates {(1.81,0.86) (1.71,1.04) (1.60,1.20) (1.47,1.36) (1.33,1.50) (1.17,1.62) (1.00,1.73) (0.82,1.82) (0.64,1.90) (0.44,1.95) (0.25,1.98) (0.05,2.00) (-0.15,1.99) (-0.35,1.97) (-0.54,1.92) (-0.73,1.86) (-0.92,1.78) (-1.09,1.68) (-1.25,1.56) (-1.40,1.43) (-1.54,1.28) (-1.66,1.12) (-1.76,0.95) (-1.85,0.77) (-1.91,0.58) (-1.96,0.39) (-1.99,0.19) (-2.00,-0.01) (-1.99,-0.21) (-1.96,-0.41) (-1.93,-0.54)};
\draw[knot] plot[smooth,tension=0.5] coordinates {(-1.81,-0.86) (-1.71,-1.04) (-1.60,-1.20) (-1.47,-1.36) (-1.33,-1.50) (-1.17,-1.62) (-1.00,-1.73) (-0.83,-1.82) (-0.64,-1.89) (-0.45,-1.95) (-0.25,-1.98) (-0.05,-2.00) (0.15,-1.99) (0.35,-1.97) (0.54,-1.93) (0.73,-1.86) (0.91,-1.78) (1.09,-1.68) (1.25,-1.56) (1.40,-1.43) (1.53,-1.28) (1.65,-1.12) (1.76,-0.95) (1.84,-0.77) (1.91,-0.58) (1.96,-0.39) (1.99,-0.19) (2.21,-0.17) (2.42,-0.17) (2.64,-0.17) (2.85,-0.17) (3.06,-0.17) (3.28,-0.17) (3.20,-0.35) (3.07,-0.51) (2.91,-0.63) (2.72,-0.70) (2.52,-0.74) (2.31,-0.75) (2.04,-0.73)};
\node[slab] at (0,0.62) {$c$};
\node[slab] at (1.27,1.12) {$c_0$};
\node[slab] at (1.27,-1.12) {$c_1$};
\node[slab] at (-1.27,1.12) {$c_2$};
\node[slab] at (-1.27,-1.12) {$c_3$};
\node[panelcap] at (0,-2.85) {(a) $s=0$: the knot $J\natural_bW'$};
\end{scope}
\draw[darrow] (2.35,0) -- node[above=2pt,slab,align=center]{remove\\ two kinks} (3.4,0);
\begin{scope}[shift={(5.7,0)},scale=0.6]
\draw[knot] plot[smooth,tension=0.5] coordinates {(-1.81,-0.86) (-1.71,-1.04) (-1.60,-1.20) (-1.47,-1.36) (-1.32,-1.50) (-1.16,-1.63) (-1.00,-1.73) (-0.82,-1.83) (-0.63,-1.90) (-0.43,-1.95) (-0.23,-1.99) (-0.03,-2.00) (0.17,-1.99) (0.37,-1.97) (0.56,-1.92) (0.75,-1.85) (0.94,-1.77) (1.11,-1.66) (1.27,-1.54) (1.42,-1.41) (1.56,-1.26) (1.67,-1.08) (1.66,-0.87) (1.53,-0.71) (1.36,-0.59) (1.16,-0.49) (0.97,-0.42) (0.78,-0.34) (0.59,-0.26) (0.40,-0.18) (0.21,-0.10) (0.03,-0.01) (-0.16,0.07) (-0.34,0.15) (-0.53,0.24) (-0.73,0.32) (-0.92,0.40) (-1.10,0.47) (-1.30,0.54) (-1.50,0.61) (-1.71,0.66)};
\draw[knot] plot[smooth,tension=0.5] coordinates {(-2.04,0.73) (-2.24,0.75) (-2.44,0.75) (-2.64,0.72) (-2.83,0.66) (-3.01,0.56) (-3.16,0.42) (-3.26,0.24) (-3.30,0.05) (-3.28,-0.16) (-3.21,-0.34) (-3.08,-0.50) (-2.92,-0.62) (-2.73,-0.70) (-2.53,-0.74) (-2.32,-0.75) (-2.12,-0.74) (-1.92,-0.71) (-1.72,-0.67) (-1.52,-0.61) (-1.32,-0.55) (-1.13,-0.48) (-0.93,-0.40) (-0.74,-0.33) (-0.55,-0.24) (-0.35,-0.16) (-0.15,-0.07)};
\draw[knot] plot[smooth,tension=0.5] coordinates {(0.15,0.07) (0.34,0.15) (0.53,0.24) (0.73,0.32) (0.91,0.40) (1.10,0.47) (1.31,0.56) (1.49,0.68) (1.64,0.83) (1.68,1.03) (1.59,1.21) (1.46,1.37) (1.31,1.51) (1.16,1.63) (0.99,1.74) (0.81,1.83) (0.62,1.90) (0.42,1.95) (0.22,1.99) (0.02,2.00) (-0.18,1.99) (-0.38,1.96) (-0.58,1.92) (-0.77,1.85) (-0.95,1.76) (-1.12,1.66) (-1.28,1.54) (-1.43,1.40) (-1.56,1.25) (-1.68,1.08) (-1.78,0.91) (-1.86,0.73) (-1.93,0.53) (-1.97,0.34) (-2.00,0.14) (-2.00,-0.06) (-1.98,-0.27) (-1.93,-0.54)};
\node[slab] at (0,0.62) {$c$};
\node[slab] at (-1.27,1.12) {$c_2$};
\node[slab] at (-1.27,-1.12) {$c_3$};
\node[panelcap] at (-0.6,-2.85) {(b) a trefoil: $\stdpi(Q_0)\cong S_3$};
\end{scope}
\begin{scope}[shift={(10.3,0)},scale=0.85]
\draw[jknot] plot[smooth,tension=0.5] coordinates {(0.69,0.72) (0.58,0.82) (0.44,0.90) (0.29,0.96) (0.14,0.99) (-0.02,1.00) (-0.17,0.99) (-0.32,0.95) (-0.47,0.88) (-0.60,0.80) (-0.72,0.70) (-0.82,0.58) (-0.90,0.44) (-0.96,0.29) (-0.99,0.14) (-1.00,-0.02) (-0.99,-0.17) (-0.95,-0.32) (-0.88,-0.47) (-0.80,-0.60) (-0.70,-0.72) (-0.58,-0.82) (-0.44,-0.90) (-0.29,-0.96) (-0.14,-0.99) (0.02,-1.00) (0.17,-0.99) (0.32,-0.95) (0.47,-0.88) (0.60,-0.80) (0.72,-0.70) (0.82,-0.58) (0.90,-0.44) (0.96,-0.29) (0.99,-0.14) (1.00,0.02) (0.99,0.17) (0.95,0.32) (0.86,0.52)};
\draw[knot] plot[smooth,tension=0.5] coordinates {(0.90,-0.66) (1.06,-0.68) (1.22,-0.66) (1.37,-0.60) (1.50,-0.51) (1.61,-0.39) (1.68,-0.25) (1.72,-0.10) (1.73,0.06) (1.69,0.22) (1.62,0.36) (1.52,0.49) (1.40,0.59) (1.25,0.65) (1.09,0.68) (0.93,0.67) (0.78,0.62) (0.64,0.54) (0.53,0.43) (0.44,0.30) (0.39,0.15) (0.37,-0.01) (0.39,-0.17) (0.45,-0.32) (0.54,-0.45) (0.67,-0.56)};
\draw[jknot,-{Stealth[length=1.6mm]}] plot[smooth] coordinates {(0.72,-0.70) (0.82,-0.58) (0.90,-0.44)};
\draw[knot,-{Stealth[length=1.6mm]}] plot[smooth] coordinates {(1.52,0.49) (1.40,0.59) (1.25,0.65)};
\node[lab,text=diagramblue!75!black] at (-1.3,0.75) {$J$};
\node[lab,fill=white,rounded corners=1.5pt,draw=diagramblue!50,inner sep=1.5pt] at (-1.25,-0.8) {$4$};
\node[lab,text=orbitred] at (2.05,0.45) {$W'_*$};
\node[panelcap] at (0.25,-2.01) {(c) $s=s_*=-1$: $\stdpi(Q_{s_*})\cong\ZZ/8$};
\end{scope}
\end{tikzpicture}
\caption{The two belt fillings of Theorem~\ref{std:thm:belt} for the diagram of Figure~\ref{std:fig:belt}. Singular circles are drawn in red. (a) For $s=0$ the slam dunk leaves $S^3$ with singular circle $J\natural_bW'$; along the flat band (grey) the crossings $c_0$ and $c_1$ become nugatory kinks (shaded loops). (b) Removing the two kinks leaves an alternating three-crossing diagram with crossings $c$, $c_2$ and $c_3$, a trefoil with all three oriented crossing signs negative in this drawing; hence $\stdpi(Q_0)\cong S_3$. (c) For $s=s_*=-1$ the Rolfsen twist cancels the full twist of the bigon, so that $W'_*$ is a meridian of $J$ and the coefficient of $J$ becomes $r=4$. The displayed orientations give $\lk(J,W'_*)=-1$; the diagram is a Hopf link, not merely a link with nonzero linking number. Thus $|Q_{s_*}|=L(4,1)$, the singular circle is the core of a Heegaard solid torus, and $\stdpi(Q_{s_*})\cong\ZZ/8$ (Remark~\ref{std:rem:sstar}).}
\label{std:fig:fillings}
\end{figure}

\begin{remark}[The value of $s_*$]\label{std:rem:sstar}
For the mirror image fixed in Section~\ref{wh:sec:whitehead}, we have $s_*=-1$
and $r=4$; thus $|Q_{s_*}|=L(4,1)$ and $\stdpi(Q_{s_*})\cong\ZZ/8$. Indeed, for
the orientations of Figure~\ref{std:fig:belt} the crossings $c_0$ and $c_1$ are
positive, so the right-hand bigon contributes $+1$ to $\lk(J,W')$, and
$\lk(J,\ell_A)=-\lk(W',\ell_A)=-1$. A Rolfsen twist with $t$ full twists along
$D$ changes $\lk(J,W')$ by $t\,\lk(J,\ell_A)\lk(W',\ell_A)=-t$ (compare
\cite[Section~5.3]{GS}), and the inserted twist cancels the bigon exactly when
it contributes $-1$; hence $t=1$ and $s_*=-t=-1$. Reversing an orientation
multiplies both contributions by the same sign, so the conclusion is
independent of the orientations; for the other mirror image we have $s_*=1$ and
$r=2$. The value of $s_*$ is not used in the proof of
Theorem~\ref{std:thm:main}.
\end{remark}

\begin{remark}[Peripheral presentations of the belt fillings]
\label{std:rem:belt-peripheral}
For the oriented diagram of Figure~\ref{std:fig:belt}, choose based
meridians $a$, $b$ and $c$ of $J$, $W'$ and $\ell_A$, respectively.
The Wirtinger presentation, with the other arc meridians eliminated,
can be written using the words
$
 u=b^{-1}ab,\;\; d=c^{-1}ac,\;\;
 w=c^{-1}a^{-1}bac,\;\; v=uwu^{-1}
$
and the three relations
$
 d=vuv^{-1},\;\; b=wvw^{-1},\;\;
 c=b^{-1}aca^{-1}b.
$
With these basings the preferred longitudes are
$
 \lambda_J=bv^{-1}c^{-1},\;\; \lambda_A=a^{-1}b.
$
Thus, if $E_s$ is the exterior of the singular circle in $|Q_s|$,
its group is obtained by adding
$
 a^3bv^{-1}c^{-1}=1,\;\; c^sa^{-1}b=1;
$
the orbifold group has the additional relation $b^2=1$.

For $s=0$, the second filling relation gives $b=a$, whence
$u=a$, $w=d$ and $v=ada^{-1}$. The remaining relations reduce to
$dad=ada$ and $c=a^5d^{-1}a^{-1}$; eliminating $c$ gives
\[
 \pi_1(E_0)\cong\langle a,d\mid dad=ada\rangle,
 \;\;\; \mu_{\mathcal S}=a.
\]
Here $d=(ad)a(ad)^{-1}$, so imposing $a^2=1$ also imposes
$d^2=1$ and yields the dihedral group of order six.

For $s=-1$, the belt relation gives $b=ac$, and hence
$u=d=c^{-1}ac$, $w=b$ and $v=dbd^{-1}$. The relation
$b=wvw^{-1}$ gives $v=b$, equivalently $[d,b]=1$.
Since $[d,b]=c^{-1}[a,c]c$, this is $[a,c]=1$.
The $J$-filling relation then gives $c=a^3$ and $b=a^4$;
all the other relations become redundant. Consequently,
\[
 \pi_1(E_{-1})\cong\ZZ\langle a\rangle,
 \;\;\; \mu_{\mathcal S}=a^4,
 \;\;\; \stdpi(Q_{-1})\cong\langle a\mid a^8=1\rangle.
\]
These presentations retain the singular meridian and confirm the
group and slope calculations of Theorem~\ref{std:thm:belt}.
\end{remark}

\subsection{Proof of Theorem~\ref{std:thm:main}}\label{std:sec:proof}

By Corollary~\ref{wh:cor:input} and Lemma~\ref{wh:lem:model}, there is a
diffeomorphism, possibly reversing orientation, from $\M$ to $M(\OW,\chi)$
that carries the canonical action to the rotation of the circle factor. Apply
the construction of Section~\ref{std:sec:torus} with
\[
 Y=M(\OW,\chi),\;\;\; Q=\OW,\;\;\; \ell=\ell_A,\;\;\;
 \gamma=\text{the unit section over }\ell_A.
\]
By Theorem~\ref{std:thm:belt}(i) and Theorem~\ref{wh:thm:section}, $\gamma$
is a weight circle, and the results of its two framed surgeries are $\Sigma_+(7)$ and
$\Sigma_-(7)$, up to diffeomorphism. By Lemma~\ref{std:lem:frame}, applied to
the consecutive framings $\lambda_A$ and $\lambda_A+s_*\mu_A$ of
Theorem~\ref{std:thm:belt} (Figure~\ref{std:fig:moves}(a)), the resulting manifolds are
$\Sigma(\lambda(\lambda_A))$ and $\Sigma(\lambda(\lambda_A+s_*\mu_A))$.

Let $s\in\{0,s_*\}$. By Proposition~\ref{std:prop:torus}, the surgery with
framing $\lambda(\lambda_A+s\mu_A)$ is a framed circle surgery on the
four-manifold $Y'(\lambda_A+s\mu_A)$, whose orbit space is $Q_s$ and whose only
nonfree orbits have isotropy of order two and form the exceptional torus over $W'$.
The result of this surgery is one of $\Sigma_\pm(7)$ and hence is a homotopy four-sphere by
Proposition~\ref{prop:exists}. For $s=0$, the group $\stdpi(Q_0)$ is dihedral of
order six by Theorem~\ref{std:thm:belt}(ii), and
Proposition~\ref{std:prop:dihedral} with $p=3$ gives $S^4$. For $s=s_*$,
the group $\stdpi(Q_{s_*})$ is cyclic by Theorem~\ref{std:thm:belt}(iii), and
Proposition~\ref{std:prop:cyclic} gives $S^4$. Thus both framed surgeries on $\gamma$ yield manifolds diffeomorphic to
$S^4$. The possible orientation
reversal does not matter, because $S^4$ admits orientation-reversing
diffeomorphisms. \qed

\section{The two knots are inequivalent}
\label{quat:sec:quaternionic}

The two knot pairs have the same exterior $\E$. Following Plotnick's approach
\cite{Plotnick1986}, we distinguish the two pairs by an invariant of a cyclic
branched cover. Write
\[
        X^{(3)}_\epsilon:=X_{\epsilon,3},\;\;\; \epsilon\in\{+,-\},
\]
for the third cyclic branched cover of $(\Sigma_\epsilon(7),K_\epsilon)$ in the notation of \eqref{nf:eq:cover-notation}. By Theorem~\ref{per:thm:tower},
$\pi_1(X^{(3)}_\epsilon)\cong Q_8$; the deck group $C_3$ acts on it by the
automorphism induced by $\alpha$.

The group and rational homology data are given in
\cite[Remark~20.3, pp.~376--377]{QextC}; the regular and irregular dihedral
covers and the sixfold cyclic covers have the smooth models described in
\cite[Theorem~21.3]{QextC}. The secondary spin-bordism invariant defined below distinguishes their
third cyclic branched covers.
Figure~\ref{quat:fig:strategy} outlines the argument.

\begin{figure}[tbp]
\centering
\begin{tikzpicture}[
  qb/.style={draw=#1!75,fill=white,rounded corners=3pt,line width=0.55pt,
    align=flush center,inner sep=3.5pt,font=\footnotesize,text width=4.7cm},
  qb/.default=ineqcol,
  sb/.style={qb=stdcol,text width=#1},
  sb/.default=4.7cm,
  flow/.style={-{Stealth[length=2.1mm,width=1.5mm]},draw=#1,line width=0.6pt},
  flow/.default=diagramblue]
\node[qb] (A) at (0,0) {$(\Sigma_\varepsilon(7),K_\varepsilon)$,\ \ $\varepsilon\in\{+,-\}$};
\node[qb] (B) at (0,-1.85) {$X^{(3)}_\varepsilon$: spin, $\pi_1=Q_8$, signature $0$
  (Proposition~\ref{quat:prop:X3-spin})};
\node[qb,fill=highlight!55,line width=0.8pt] (G) at (0,-6.2) {$\delta_\varepsilon=\mathrm{sec}(X^{(3)}_\varepsilon)=1+\eta_\varepsilon$:
  \ $\delta_+=1\neq0=\delta_-$, so $K_+$ and $K_-$ are inequivalent
  (Theorem~\ref{eval:thm:bit}, Corollary~\ref{eval:cor:pairs})};
\node[sb=5.2cm] (C) at (8.85,-1.85) {$(M_3,\widehat s_{3,\varepsilon},\bar u)$: the degree-three
  mapping torus with an extending spin structure};
\node[sb=5.2cm] (D) at (8.85,-3.4) {point preimages in $\mathcal Q=S^3/Q_8$ and the $D_{14}$
  transfer (Proposition~\ref{eval:prop:transfer})};
\node[sb=2.65cm] (E1) at (7.33,-4.9) {$M_3^{\mathrm{irr}}\cong\mathcal Q\times\RR/\ZZ$:\\ bit $\eta_\varepsilon$};
\node[sb=2.65cm] (E2) at (10.37,-4.9) {$M_3^{\mathrm{reg}}\cong\mathcal Q\times\RR/2\ZZ$:\\ bit $1$};
\node[sb=2.6cm] (F) at (8.85,-6.2) {$\mathfrak p(M_3)=\eta_\varepsilon+1$};
\coordinate (Atop) at ($(A.north)+(0,0.35)$);
\coordinate (Ctop) at ($(C.north)+(0,0.35)$);
\begin{scope}[on background layer]
\node[fill=ineqcol!6,draw=ineqcol!30,rounded corners=5pt,line width=0.4pt,inner sep=0.2cm,
  fit=(A)(B)(G)(Atop)] {};
\node[fill=stdcol!6,draw=stdcol!30,rounded corners=5pt,line width=0.4pt,inner sep=0.2cm,
  fit=(C)(D)(E1)(E2)(F)(Ctop)] {};
\end{scope}
\node[font=\scriptsize\bfseries,text=ineqcol,anchor=south west,inner sep=1pt] at ($(A.north west)+(-0.02,0.08)$) {the third cyclic branched covers};
\node[font=\scriptsize\bfseries,text=stdcol,anchor=south west,inner sep=1pt] at ($(C.north west)+(-0.02,0.08)$) {the computation on $M_3$};
\draw[flow=ineqcol] (A) -- node[right=1.5pt,slab,align=left]{third cyclic\\ branched cover} (B);
\draw[flow] (C) -- node[above=2pt,slab,align=center]{surgery trace: a spin\\ bordism over $BQ_8$} (B);
\draw[flow=ineqcol] (B) -- node[right=1.5pt,slab,align=left]{its class in\\ $\widetilde\Omega^{\Spin}_4(BQ_8)\cong\ZZ/2$\\
  (Corollary~\ref{quat:cor:bordism-bit})} (G);
\draw[flow=stdcol] (C) -- (D);
\draw[flow=stdcol] (D) -- (E1);
\draw[flow=stdcol] (D) -- (E2);
\draw[flow=stdcol] (E1) -- (F);
\draw[flow=stdcol] (E2) -- (F);
\draw[flow] (F) -- node[pos=0.55,above=2pt,slab]{Lemma~\ref{eval:lem:point}} (G);
\end{tikzpicture}
\caption{Outline of the proof of inequivalence. The secondary invariant
$\delta_\varepsilon\in\ZZ/2$ of the third cyclic branched cover is the torsion
coordinate of its spin-bordism class over $BQ_8$ (left). The surgery trace runs from
the degree-three mapping torus $M_3$ to $X^{(3)}_\varepsilon$ and identifies their
classes in $\Omega^{\Spin}_4(BQ_8)$. The class is then computed on $M_3$ (right) from
point preimages, through the odd dihedral transfer, on two covers that are products
of $\mathcal Q=S^3/Q_8$ with a circle.}
\label{quat:fig:strategy}
\end{figure}

\subsection{The third cyclic cover is spin}

\begin{proposition}[Spin structures on the third cyclic cover]
\label{quat:prop:X3-spin}
For both specified surgery framings, $X^{(3)}_\epsilon$ is spin. More
precisely, on the degree-three cover $M_3$ of the mapping torus, the spin
structure
\begin{equation}\label{quat:eq:extending-spin}
   \widehat s_{3,\epsilon}:=s_3+(1-\eta_\epsilon)a_3,
   \;\;\; \eta_+=0,\;\; \eta_-=1,
\end{equation}
extends across the circle surgery that produces
$X^{(3)}_\epsilon$.

Furthermore,
\begin{equation}\label{quat:eq:X3-homology}
 \pi_1(X^{(3)}_\epsilon)=Q_8,
 \;\;\;
 H_1(X^{(3)}_\epsilon;\ZZ)
 \cong H_2(X^{(3)}_\epsilon;\ZZ)
 \cong (\ZZ/2)^2,
 \;\;\; \operatorname{sign}(X^{(3)}_\epsilon)=0.
\end{equation}
The residual $C_3$-action on $X^{(3)}_\epsilon$ fixes exactly one
spin structure.
\end{proposition}

\begin{proof}
Lemma~\ref{nf:lem:calibration} gives
$
 \spn_{s_3}(\gamma_3,\lambda_{\epsilon,3})
   =(-1)^{3\eta_\epsilon}=(-1)^{\eta_\epsilon}.
$
The class $a_3$ evaluates to one on the circle $\gamma_3$ in the positive
base direction. Hence, \eqref{nf:eq:spin-rules} gives
$
 \spn_{\widehat s_{3,\epsilon}}
       (\gamma_3,\lambda_{\epsilon,3})
 =(-1)^{1-\eta_\epsilon}(-1)^{\eta_\epsilon}=-1.
$
By Lemma~\ref{lem:circle-spin}, this is precisely the sign for which
the spin structure extends across the circle-surgery cap. Thus, the
filled manifold is spin for either $\epsilon$.

The fundamental group and integral homology are given by the case
$m\equiv3\pmod6$ of Theorem~\ref{per:thm:tower}; since
$Q_8^{\rm ab}\cong(\ZZ/2)^2$, this gives
\eqref{quat:eq:X3-homology}. In particular the rational second homology
vanishes, so the signature is zero.

Spin structures form a torsor over
$H^1(X^{(3)}_\epsilon;\FF_2)\cong\FF_2^2$. The residual order-three
deck transformation acts on $Q_8^{\rm ab}$ by the automorphism induced
by $\alpha$; it cyclically permutes its three nonzero elements. Thus
there is no nonzero invariant element of $H^1$. An order-three action
on the four-element spin-structure torsor must have an orbit of size
one, and two fixed spin structures would have a nonzero invariant
difference. Hence, the fixed spin structure is unique.
\end{proof}

\subsection{Teichner's secondary invariant}

Let $X$ be a closed connected oriented spin four-manifold with finite
fundamental group $\Pi$. The Hurewicz isomorphism identifies
$\pi_2X$ with $H_2(\widetilde X;\ZZ)$, where $\widetilde X$ is the
oriented universal cover. Write $S_X$ for the integral intersection
pairing on this group. Deck transformations preserve $S_X$; equivalently,
\[
 S_X\in\Hom_{\ZZ\Pi}
       (\pi_2X\otimes_{\ZZ}\pi_2X,\ZZ),
\]
where $\Pi$ acts diagonally on the tensor product and trivially on $\ZZ$.
For a bilinear form $q$ in this group, put $q^*(x,y)=q(y,x)$.
Teichner's intrinsic parity invariant
\cite[Introduction, pp.~5--6]{Teichner1992} is defined by
\begin{equation}\label{quat:eq:teichner-sec}
 \mathrm{sec}(X)=0
 \;\;\Longleftrightarrow\;\;
 S_X=q+q^*\ \text{for some }q\in
 \Hom_{\ZZ\Pi}(\pi_2X\otimes_{\ZZ}\pi_2X,\ZZ).
\end{equation}
The value is $1$ otherwise. This condition requires equivariant symmetrization,
rather than merely evenness of the nonequivariant form. The pairing considered here is not the rational intersection form of $X$. A homotopy
equivalence preserves the equivariant pairing up to a relabelling of $\Pi$
and an overall sign, so it preserves $\mathrm{sec}$.
Changing the spin structure does not change the pairing, and reversing
orientation replaces $S_X$ by $-S_X$, which is a symmetrization precisely
when $S_X$ is. Thus the invariant is independent of the spin structure and
the orientation.

For comparison with the group-ring convention, put
\[
 \lambda_X(x,y)=\sum_{g\in\Pi}S_X(x,gy)g,
 \;\;\;
 b_q(x,y)=\sum_{g\in\Pi}q(x,gy)g.
\]
Since $\Pi$ is finite, $q\mapsto b_q$ is a bijection from the
$\Pi$-invariant integral bilinear forms to the $\ZZ\Pi$-sesquilinear
forms; its inverse takes the coefficient of the identity element.
It carries $q^*$ to
$b_q^*(x,y)=\overline{b_q(y,x)}$, where $\overline g=g^{-1}$.
Thus \eqref{quat:eq:teichner-sec} is equivalent to
$\lambda_X=b_q+b_q^*$, exactly the meaning of an even equivariant
intersection form in \cite[Section~1.3.2]{KPT2020}.

We use two results to identify this invariant with the bordism coordinate
considered below. For a quaternion $2$-group $Q$, Teichner computes
$\Omega^{\Spin}_4(BQ)\cong16\ZZ\oplus\ZZ/2$ and defines its secondary
coordinate through $E^\infty_{3,1}$ of the James spectral sequence
\cite[Theorem~4.2.2 and the following remark, p.~52]{Teichner1992}; in the
spin case this is the Atiyah--Hirzebruch spectral sequence.
The identification of this coordinate with the equivariant
symmetrization obstruction \eqref{quat:eq:teichner-sec} is established in
\cite[Theorem~6.4.1, p.~90]{Teichner1992}, which applies more generally when
the finite fundamental group has quaternion $2$-Sylow subgroups. In the terminology of \cite{KPT2020}, a closely related statement is the
Secondary Property \cite[Secondary Property~1.11 and Theorem~1.12]{KPT2020},
which the authors attribute, for finite groups with quaternion $2$-Sylow
subgroups, to \cite[Theorem~6.4.1]{Teichner1992}. That property is not
formulated for $\lambda_X$ itself. Choose a finite $2$-complex $K$ with
$\pi_1(K)\cong\Pi$ and a splitting $s=(s_1,s_2)$ of the primary extension
\[
 0\to\pi_2(K)\to(\ZZ\Pi)^r\oplus\pi_2X\to H^2(K;\ZZ\Pi)\to0
\]
of \cite[Theorem~1.6]{KPT2020}; for $\Pi=Q_8$ such a splitting exists,
because the extension class lies in a group isomorphic to
$H_4(Q_8;\ZZ)=0$. The Secondary Property concerns the class, modulo even
forms, of the pull-back $\lambda_{X,s}=\lambda_X\bigl(s_2(-),s_2(-)\bigr)$.
If $\lambda_X=b_q+b_q^*$, then $\lambda_{X,s}=c+c^*$ with
$c=b_q\bigl(s_2(-),s_2(-)\bigr)$, so $\lambda_{X,s}$ is even for every
splitting. Hence $\mathrm{sec}(X)=0$ in the sense of
\eqref{quat:eq:teichner-sec} already forces the vanishing of the bordism
coordinate by \cite[Theorem~1.12]{KPT2020}. The converse implication, and
hence the equality of the bordism coordinate with $\mathrm{sec}(X)$, is the
part of \cite[Theorem~6.4.1]{Teichner1992} that is used below.
The first differentials are treated in \cite{Teichner1993}; see also
\cite[Section~6.1.3]{KPT2020} for the modern intersection-form formulation.
These are precisely the two inputs needed for $Q_8$. Put
\[
             \delta_\epsilon:=\mathrm{sec}(X^{(3)}_\epsilon)\in\ZZ/2,
             \;\;\;\epsilon\in\{+,-\}.
\]

Here, stable homeomorphism and stable diffeomorphism mean equivalence
after taking connected sums with finitely many copies of $S^2\times S^2$.

\begin{corollary}[An unmarked obstruction on the third cover]
\label{quat:cor:unmarked-delta}
If $\delta_+\ne\delta_-$, then $X^{(3)}_+$ and $X^{(3)}_-$ are not
homotopy equivalent and in particular are not diffeomorphic. Moreover,
$X^{(3)}_+$ and $X^{(3)}_-$ are stably homeomorphic if and only if
$\delta_+=\delta_-$.
\end{corollary}

\begin{proof}
The equivariant intersection pairing is preserved by a homotopy
equivalence up to a group automorphism and an overall sign, so the
first assertion follows from \eqref{quat:eq:teichner-sec}.
Stabilization replaces the group-ring pairing by
$
 \lambda_X\ \perp\ \mathsf H(\ZZ Q_8)^{\perp r},
$
where $\mathsf H(\ZZ Q_8)$ is the hyperbolic form on $(\ZZ Q_8)^2$.
A hyperbolic form is a symmetrization. Conversely, restricting a
symmetrization of this orthogonal sum to $\pi_2X$ gives a
symmetrization of $\lambda_X$. Hence $\mathrm{sec}$ is unchanged by
stabilization, and a stable homeomorphism forces
$\delta_+=\delta_-$, even if it reverses orientation or relabels $Q_8$.

Conversely, suppose that $\delta_+=\delta_-$. Both manifolds are spin
and have signature zero by Proposition~\ref{quat:prop:X3-spin}, and
$H_4(Q_8;\ZZ)=0$, so the $\pi_1$-fundamental classes vanish. Teichner's
stable classification of closed spin four-manifolds with a common finite
fundamental group whose $2$-Sylow subgroups have periodic cohomology, by
signature, $\pi_1$-fundamental class and $\mathrm{sec}$
\cite[Introduction, Theorem~(4.4.10), p.~5]{Teichner1992}, therefore gives a
stable homeomorphism; its smooth form
\cite[Corollary~4.4.10(1), p.~62]{Teichner1992} gives a stable
diffeomorphism. Since both manifolds have
Euler characteristic two, the numbers of $S^2\times S^2$ summands on
the two sides must agree. No cancellation of these summands is asserted.
\end{proof}

\noindent
Theorem~\ref{eval:thm:bit} below shows that $\delta_+=1\neq0=\delta_-$, so the
hypothesis of Corollary~\ref{quat:cor:unmarked-delta} holds; see
Corollary~\ref{eval:cor:third-covers}.

\subsection{Spin-bordism reformulation}

For a quaternion $2$-group $Q$, Teichner's computation
\cite[Theorem~4.2.2, p.~52]{Teichner1992} gives
\begin{equation}\label{quat:eq:spin-bordism-Q8}
 (\operatorname{sign},\mathrm{sec}):
 \Omega^{\Spin}_4(BQ)\xrightarrow{\ \cong\ }
        16\ZZ\oplus\ZZ/2.
\end{equation}
When $Q=Q_8$ and the signature is zero, only the bordism coordinate in
the second summand can be nonzero.

\noindent\emph{Normalization of the second bordism coordinate.}
Let $b:\mathrm{pt}\to BQ_8$ be the basepoint inclusion and let
$\mathrm{forg}:\Omega^{\Spin}_4(BQ_8)\to\Omega^{\Spin}_4$ be the map that forgets the
reference map. The splitting is normalized by
\[
 z\longmapsto z-b_*\mathrm{forg}(z)
       \in\widetilde\Omega^{\Spin}_4(BQ_8).
\]
On the full bordism group, the notation $\mathrm{sec}_{BQ_8}$ means the
secondary edge coordinate of this reduced class. This is the
normalization in the remark following
\cite[Theorem~4.2.2]{Teichner1992}; the secondary edge coordinate is
initially defined on the signature-zero part. Since
$\Omega^{\Spin}_4\cong\ZZ$ is detected by the signature divided by $16$,
no subtraction is needed for any signature-zero manifold used below.
For the classifying maps of $X_{\epsilon,3}$ this coordinate is the
intrinsic invariant $\mathrm{sec}(X_{\epsilon,3})$ by
\cite[Theorem~6.4.1, p.~90]{Teichner1992}. For $M_3$, the notation
$\mathrm{sec}_{BQ_8}$ refers to the bordism class of the specified reference
map $M_3\to BQ_8$.

\begin{corollary}[The bit as a single spin-bordism coordinate]
\label{quat:cor:bordism-bit}
Let $u_\epsilon:X^{(3)}_\epsilon\to BQ_8$ be a classifying map, and choose
any spin structure on $X^{(3)}_\epsilon$, for example the unique one
fixed by the residual $C_3$-action. Then
\[
 \delta_\epsilon=0
 \;\;\Longleftrightarrow\;\;
 [X^{(3)}_\epsilon,u_\epsilon]=0
 \;\;\text{in}\;\;
 \widetilde\Omega^{\Spin}_4(BQ_8)\cong\ZZ/2.
\]
Thus, computing $\delta_\epsilon$ reduces to a single spin-bordism
computation over $BQ_8$.
\end{corollary}

\begin{proof}
By Proposition~\ref{quat:prop:X3-spin} the manifold is spin and has
signature zero, so its class lies in the reduced group, which by
\eqref{quat:eq:spin-bordism-Q8} is detected by $\mathrm{sec}$. The class is independent of the chosen spin structure. Indeed,
changing it by
$\beta\in H^1(X^{(3)}_\epsilon;\ZZ/2)$ changes the secondary
coordinate by the class of
\[
 u_{\epsilon*}\bigl(\beta\cap[X^{(3)}_\epsilon]_{\FF_2}\bigr)
 \;\;\text{in}\;\;
 H_3(Q_8;\FF_2)/
 \operatorname{im}\bigl(d_2:E^2_{5,0}\to E^2_{3,1}\bigr),
\]
by \cite[Proposition~3.2.4, p.~41]{Teichner1992}.
For $Q_8$ this incoming differential is zero, and
$E^\infty_{3,1}=H_3(Q_8;\FF_2)$ by
Lemma~\ref{eval:lem:e31}. Thus the quotient is canonically identified with the group
$H_3(Q_8;\FF_2)$ used to evaluate the secondary coordinate here. Since $u_\epsilon$ induces an
isomorphism on fundamental groups, $\beta=u_\epsilon^*b$ for some
$b\in H^1(Q_8;\FF_2)$. Moreover, naturality of coefficient reduction gives
\[
 u_{\epsilon*}[X^{(3)}_\epsilon]_{\FF_2}
 =\rho_2\bigl(u_{\epsilon*}[X^{(3)}_\epsilon]_{\ZZ}\bigr)=0,
\]
because $H_4(BQ_8;\ZZ)=0$. The projection formula therefore makes the
cap-product expression zero. The signature and the secondary coordinate are
unchanged, so \eqref{quat:eq:spin-bordism-Q8} gives the same bordism class.
An elementary proof of this independence, using point preimages, is given in
Remark~\ref{eval:rem:spin-independence}.
\end{proof}

\subsection{The surgery trace as a spin bordism}

The degree-three mapping torus has
\[
 \pi_1(M_3)=\Gamma\rtimes_{\theta^3}\langle t_3\rangle,
 \;\;\; \theta^3|_{Q_8}=\id,
 \;\;\; \theta^3(c)=c^{-1}.
\]
Hence, there is a natural epimorphism
\begin{equation}\label{quat:eq:M3-to-Q8}
 \chi:\pi_1(M_3)\to Q_8,
 \;\;\;
 \chi(q)=q,\;\; \chi(c)=1,\;\; \chi(t_3)=1.
\end{equation}
Let $\bar u:M_3\to BQ_8$ realize $\chi$.

Since $\N$ is a rational homology three-sphere, the Wang sequence
\cite[Example~2.48]{Hatcher} for
$\N\to M_3\to S^1$ gives $H_2(M_3;\QQ)=0$. Thus
$\operatorname{sign}(M_3)=0$, and every spin-bordism class
$[M_3,s,\bar u]$ considered below lies in
$\widetilde\Omega^{\Spin}_4(BQ_8)$.

\begin{proposition}[Mapping-torus spin bordism]
\label{quat:prop:M3-bordism}
For each $\epsilon$, the five-dimensional trace of the circle surgery
$(M_3,\gamma_3,\lambda_{\epsilon,3})$ is a spin bordism over $BQ_8$
from
$
 (M_3,\widehat s_{3,\epsilon},\bar u)
$
to $(X^{(3)}_\epsilon,s^{\rm out}_\epsilon,u_\epsilon)$, where
$\widehat s_{3,\epsilon}$ is defined in \eqref{quat:eq:extending-spin}, and
$s^{\rm out}_\epsilon$ and $u_\epsilon$ are obtained by restricting the
extended spin structure and map, respectively, to the outgoing end of the
trace. Consequently,
\begin{equation}\label{quat:eq:M3-bordism-equality}
 [X^{(3)}_\epsilon,u_\epsilon]
   =[M_3,\widehat s_{3,\epsilon},\bar u]
 \;\;\text{in }\Omega^{\Spin}_4(BQ_8),
\end{equation}
and $\delta_\epsilon$ is the $\ZZ/2$-coordinate of the right-hand side.
In particular,
\[
 \begin{aligned}
   \delta_+&=\mathrm{sec}_{BQ_8}
      [M_3,s_3+a_3,\bar u],\\
   \delta_-&=\mathrm{sec}_{BQ_8}
      [M_3,s_3,\bar u].
 \end{aligned}
\]
Here $\mathrm{sec}_{BQ_8}$ is the secondary $\ZZ/2$-coordinate of the
bordism class over $BQ_8$ in \eqref{quat:eq:spin-bordism-Q8}.
\end{proposition}

\begin{proof}
The surgery trace is obtained from $M_3\times I$ by attaching a
five-dimensional two-handle along the framed circle. By
Proposition~\ref{quat:prop:X3-spin}, the source spin structure
$\widehat s_{3,\epsilon}$ restricts to the bounding spin structure on that
attaching circle, so it extends across the handle. The homomorphism
\eqref{quat:eq:M3-to-Q8} sends $[\gamma_3]=t_3$ to the identity, and
therefore the map $\bar u$ extends across the same handle. On the
outgoing boundary, the surviving fundamental group is $Q_8$ and the
extended map induces the identity on this group, so it is a classifying
map. This proves \eqref{quat:eq:M3-bordism-equality}; the final two formulas
follow from Corollary~\ref{quat:cor:bordism-bit} and
\eqref{quat:eq:extending-spin}.
\end{proof}

\subsection{Evaluation of the secondary invariant}
\label{eval:sec:evaluation}

The bit is evaluated using the spin-bordism description of
Proposition~\ref{quat:prop:M3-bordism}. The result, proved in
Theorem~\ref{eval:thm:bit}, is
$
                 \delta_+=1,\;\; \delta_-=0.
$
Besides the calibrated spin structures of
Section~\ref{nf:sec:spin} and the surgery trace of
Proposition~\ref{quat:prop:M3-bordism}, the proof has three ingredients:
a point-preimage description of the secondary bordism coordinate, an
elementary transfer count for the dihedral group of order fourteen, and
an explicit product structure on two finite covers of $M_3$.

Throughout,
$
            \mathcal Q:=S^3/Q_8
$
denotes the quotient by the left action of $Q_8\subset\Gamma$ in
\eqref{eq:action}, and $\iota:\mathcal Q\to BQ_8$ is a classifying map.
We write $S^1_{\mathrm{Lie}}$ for the circle with its nonbounding (Lie)
spin structure; its class generates $\Omega^{\Spin}_1\cong\ZZ/2$. In the
closing-sign convention of Section~\ref{nf:sec:spin}, a framed circle has
sign $+1$ exactly when the spin structure it induces on the circle is the
Lie structure, and sign $-1$ exactly when that structure bounds
(Lemma~\ref{lem:circle-spin}).

\subsubsection*{Point preimages and the secondary coordinate}

\begin{lemma}[The line $p+q=4$ for $Q_8$]\label{eval:lem:e31}
Consider the Atiyah--Hirzebruch spectral sequence
$E^2_{p,q}=H_p(BQ_8;\Omega^{\Spin}_q)\Rightarrow
\Omega^{\Spin}_{p+q}(BQ_8)$.
\begin{enumerate}[label=(\roman*),leftmargin=2.2em]
\item $E^2_{4,0}=H_4(Q_8;\ZZ)=0$, $E^2_{1,3}=0$, and
      $E^2_{3,1}=H_3(Q_8;\ZZ/2)\cong\ZZ/2$.
\item $E^\infty_{3,1}=E^2_{3,1}$ and $E^\infty_{2,2}=0$. Hence, the
      reduced group $\widetilde\Omega^{\Spin}_4(BQ_8)$, which is the
      torsion subgroup of $\Omega^{\Spin}_4(BQ_8)$, maps isomorphically
      onto $E^\infty_{3,1}\cong\ZZ/2$.
\item $\iota_*:H_3(\mathcal Q;\ZZ/2)\to H_3(Q_8;\ZZ/2)$ is an
      isomorphism.
\end{enumerate}
\end{lemma}

\begin{proof}
The free action of $Q_8$ on $S^3$ gives a periodic resolution, so
$H_p(Q_8;\ZZ)$ is $\ZZ$, $(\ZZ/2)^2$, $0$, $\ZZ/8$ and $0$ for $p=0,\dots,4$, respectively; compare
\cite[Lemma~4.1.1(b)]{Teichner1992}. Together
with $\Omega^{\Spin}_3=0$ and the universal coefficient theorem, this
gives~(i).

In Teichner's spectral-sequence framework the secondary invariant in
\eqref{quat:eq:spin-bordism-Q8} is the component of a bordism class in
$E^\infty_{3,1}$ \cite{Teichner1992}. The isomorphism
\eqref{quat:eq:spin-bordism-Q8} therefore gives
$E^\infty_{3,1}\cong\ZZ/2=E^2_{3,1}$. Since $E^\infty_{4,0}$ and
$E^\infty_{1,3}$ vanish by~(i), the only remaining filtration quotient of
the reduced group, $E^\infty_{2,2}$, is zero. The reduced group is the
torsion subgroup because $\Omega^{\Spin}_4\cong\ZZ$ is detected by the
signature.

The same equality follows from the spectral-sequence differentials.
Differentials leaving $E^r_{3,1}$ land in $E^2_{1,2}$ or in
$E^3_{0,3}=0$. The differential with the first target is dual to $\operatorname{Sq}^2$ on
$H^1(Q_8;\FF_2)$; this operation vanishes for degree reasons. The only differential
entering $E^r_{3,1}$ is $d_2:E^2_{5,0}\to E^2_{3,1}$, which is reduction modulo two
followed by the dual of
$\operatorname{Sq}^2:H^3(Q_8;\FF_2)\to H^5(Q_8;\FF_2)$; this is the
standard description of the first differentials
\cite[Lemma~2.3.2 and Theorem~3.1.3]{Teichner1992}, \cite{Teichner1993}. In the classical presentation
\[
 H^*(Q_8;\FF_2)=\FF_2[x,y,e]/(x^2+xy+y^2,\ x^2y+xy^2),
 \;\;\; |x|=|y|=1,\;\; |e|=4
\]
\cite{AdemMilgram} (equivalently
$\FF_2[x,y,e]/(x^2+xy+y^2,\ x^3)$, the form used in
\cite[Lemma~4.1.1(b)]{Teichner1992}, since $x^3=x^2y+xy^2$ modulo the
first relation), the group $H^3$ is spanned by $x^2y$, and
$x^3=x(xy+y^2)=0$. The Cartan formula gives
$\operatorname{Sq}^2(x^2y)=x^4y=0$.
All products of degree-one classes vanish in degrees at least four
\cite[remark following Lemma~4.1.1, p.~50]{Teichner1992}.
The required Steenrod-square vanishings follow directly from
instability and the Cartan formula, as in the calculation above.

Finally, $\iota$ is $3$-connected, so $\iota_*$ is onto in degree three
for every coefficient group. Both groups in~(iii) have order two.
\end{proof}

Figure~\ref{eval:fig:ahss} displays the part of the $E^2$-page used here.

\begin{figure}[tbp]
\centering
\begin{tikzpicture}[x=1.75cm,y=0.95cm,
  ent/.style={font=\footnotesize,inner sep=1.5pt},
  dtwo/.style={-{Stealth[length=1.8mm,width=1.4mm]},densely dashed,line width=0.65pt,
    liftblue,shorten <=2pt,shorten >=2pt}]
\def\Zt{\ZZ/2}\def\Zte{(\ZZ/2)^2}\def\Ze{\ZZ/8}
\fill[diagramlight!70,rounded corners=3pt] (-0.45,-0.5) rectangle (5.45,4.5);
\foreach \p in {0,...,6} \draw[white,line width=1.2pt] (\p-0.5,-0.5) -- (\p-0.5,4.5);
\foreach \q in {0,...,5} \draw[white,line width=1.2pt] (-0.45,\q-0.5) -- (5.45,\q-0.5);
\begin{scope}
\clip[rounded corners=3pt] (-0.45,-0.5) rectangle (5.45,4.5);
\draw[line width=0.78cm,liftorange!20,line cap=round] (4.15,-0.15) -- (-0.15,4.15);
\end{scope}
\draw[-{Stealth[length=1.8mm]},line width=0.55pt] (-0.55,-0.6) -- (5.7,-0.6) node[right,slab]{$p$};
\draw[-{Stealth[length=1.8mm]},line width=0.55pt] (-0.55,-0.6) -- (-0.55,4.75) node[above,slab]{$q$};
\foreach \p in {0,...,5} \node[slab] at (\p,-0.88) {$\p$};
\foreach \q/\o in {0/\ZZ,1/{\ZZ/2},2/{\ZZ/2},3/0,4/\ZZ} {
  \node[slab] at (-0.75,\q) {$\q$};
  \node[tlab,text=inkgray,anchor=east] at (-0.9,\q) {$\Omega^{\Spin}_{\q}=\o$};}
\foreach \p/\e in {0/\ZZ,1/\Zte,2/0,3/\Ze,4/0,5/\Zte} {\node[ent] (e\p0) at (\p,0) {$\e$}; \node[ent] (e\p4) at (\p,4) {$\e$};}
\foreach \p/\e in {0/\Zt,1/\Zte,2/\Zte,4/\Zt,5/\Zte} {\node[ent] (e\p1) at (\p,1) {$\e$}; \node[ent] (e\p2) at (\p,2) {$\e$};}
\node[ent,draw=liftblue,line width=0.9pt,rounded corners=2pt,fill=white] (e31) at (3,1) {$\Zt$};
\node[ent] (e32) at (3,2) {$\Zt$};
\foreach \p in {0,...,5} \node[ent,text=gray] (e\p3) at (\p,3) {$0$};
\draw[dtwo] (e50) -- (e31);
\draw[dtwo] (e31) -- (e12);
\draw[dtwo] (e41) -- (e22);
\draw[-{Stealth[length=2mm,width=1.5mm]},line width=0.9pt,orbitred,rounded corners=7pt]
  (e50.east) -- (5.62,0) -- (5.62,2.64) -- (2,2.64) -- (2,2.31);
\node[slab,text=orbitred,fill=diagramlight!70,inner sep=1.5pt] at (3.75,2.64) {$d_3$ on $E^3$};
\draw[orbitred,densely dashed,line width=0.6pt] (e22) ellipse[x radius=0.38cm,y radius=0.26cm];
\node[slab,text=liftblue,anchor=north west] at ($(e31.south east)+(-0.02,0.02)$) {$\mathrm{sec}$};
\node[slab,anchor=south] at (0,4.5) {signature$/16$};

\begin{scope}[shift={(0,-1.6)}]
\fill[liftorange!20,rounded corners=2pt] (-0.3,-0.13) rectangle (0.0,0.13);
\node[slab,anchor=west] at (0.05,0) {$p+q=4$};
\draw[dtwo,shorten <=0pt,shorten >=0pt] (1.0,0) -- (1.4,0);
\node[slab,anchor=west] at (1.45,0) {$d_2=0$};
\draw[-{Stealth[length=2mm]},line width=0.9pt,orbitred] (2.3,0) -- (2.7,0);
\node[slab,anchor=west] at (2.75,0) {$d_3$, an isomorphism};
\draw[liftblue,line width=0.9pt,rounded corners=2pt] (4.35,-0.15) rectangle (4.65,0.15);
\node[slab,anchor=west] at (4.75,0) {survives: $\mathrm{sec}$};
\end{scope}
\node[slab] at (2.5,5.25) {$E^2_{p,q}=H_p(BQ_8;\Omega^{\Spin}_q)\Longrightarrow\Omega^{\Spin}_{p+q}(BQ_8)$};
\end{tikzpicture}
\caption{Part of the $E^2$-page of the Atiyah--Hirzebruch spectral sequence for
$\Omega^{\Spin}_*(BQ_8)$ (Lemma~\ref{eval:lem:e31} and Remark~\ref{eval:rem:e22}). On
the shaded line $p+q=4$ the groups $E_{4,0}$ and $E_{1,3}$ vanish. The
$d_2$-differentials drawn are zero because the relevant Steenrod squares
$\operatorname{Sq}^2$ vanish, so $E_{3,1}\cong\ZZ/2$ survives and carries the
secondary coordinate $\mathrm{sec}$, while $E^3_{2,2}\cong(\ZZ/2)^2$ is killed by
Teichner's isomorphism $d_3:E^3_{5,0}\to E^3_{2,2}$. The red arrow is an overlay of this
$E^3$ differential on the $E^2$ table, rather than a differential on $E^2$; the
displayed bidegrees are $(-2,1)$ for $d_2$ and $(-3,2)$ for $d_3$. The group
$E_{0,4}=\Omega^{\Spin}_4\cong\ZZ$ is detected by the signature divided by sixteen in
smooth spin bordism. The coefficient groups $\Omega^{\Spin}_q$ are listed on the left.}
\label{eval:fig:ahss}
\end{figure}

\begin{remark}[The term $E^\infty_{2,2}$]\label{eval:rem:e22}
The vanishing of $E^\infty_{2,2}$ in Lemma~\ref{eval:lem:e31}(ii) requires
the $d_3$ differential. Indeed,
$E^2_{2,2}=H_2(Q_8;\FF_2)\cong\FF_2^{\,2}$. The differential
$d_2:E^2_{4,1}\to E^2_{2,2}$ is dual to
$\operatorname{Sq}^2:H^2(Q_8;\FF_2)\to H^4(Q_8;\FF_2)$, which vanishes because
$\operatorname{Sq}^2(xy)=x^2y^2=0$ and $\operatorname{Sq}^2(y^2)=y^4=0$ in the
presentation used in the proof (where $y^3=0$). The differential leaving
$E^2_{2,2}$ has target $E^2_{0,3}=0$. Hence, $E^3_{2,2}\cong\FF_2^{\,2}$, and
$E^\infty_{2,2}=0$ is equivalent to
\[
   d_3:E^3_{5,0}=H_5(Q_8;\ZZ)\cong(\ZZ/2)^2\to E^3_{2,2}
\]
being an isomorphism. This is precisely Teichner's
Proposition~4.2.1 \cite[Proposition~4.2.1, p.~51]{Teichner1992}. The value
$\widetilde\Omega^{\Spin}_4(BQ_8)\cong\ZZ/2$ is also listed in
\cite[Appendix~A, Table~6]{DavighiGripaiosLohitsiri}.
The same value, and the vanishing of $E^\infty_{2,2}$ through this $d_3$,
are also derived in \cite[Appendices~A.3 and~B]{DavighiGripaiosLohitsiri}
from the computation of the connective real $K$-theory of $BQ_8$ by Bayen and
Bruner \cite{BayenBruner1996}, using the isomorphism
$\Omega^{\Spin}_n(X)\cong ko_n(X)$ for $n\le7$. Together with the survival of $E_{3,1}$, this gives the stated value of
the bordism group and forces the displayed $d_3$ to be an isomorphism.
\end{remark}

\begin{lemma}[Compression into the space form]\label{eval:lem:compress}
Let $W$ be a closed connected smooth oriented four-manifold and let
$u:W\to BQ_8$ be a map. Then $u$ is homotopic to $\iota\circ g$ for a
smooth map $g:W\to\mathcal Q$.
\end{lemma}

\begin{proof}
Replace the classifying map
$\iota:\mathcal Q=S^3/Q_8\to BQ_8$ by a fibration in its homotopy class.
Its homotopy fibre is homotopy equivalent to the universal cover $S^3$.
Thus, the fibre is $2$-connected, and for the four-dimensional CW complex $W$
the only possible obstruction to lifting $u:W\to BQ_8$ through $\iota$ lies,
by obstruction theory for sections of fibrations with local coefficients
\cite[Chapter~VI]{WhiteheadGW}, in
\[
     H^4\bigl(W;\pi_3(S^3)\bigr)=H^4(W;\ZZ).
\]
The $Q_8$-action on $\pi_3(S^3)\cong\ZZ$ is trivial because the deck
transformations preserve orientation. The universal obstruction is therefore
a class $\kappa\in H^4(BQ_8;\ZZ)$. Teichner's integral homology calculation
\cite[Lemma~4.1.1(b)]{Teichner1992}, together with the universal coefficient
theorem, gives
\[
                 H^4(BQ_8;\ZZ)\cong\ZZ/8.
\]
Hence, $\kappa$ is torsion. Since $W$ is closed, connected and oriented,
the group $H^4(W;\ZZ)\cong\ZZ$ is torsion-free, so $u^*\kappa=0$. There are no higher
obstructions because $\dim W=4$. Therefore, $u$ is homotopic to
$\iota\circ g$ for a continuous map $g:W\to\mathcal Q$. By smooth
approximation, we may choose $g$ to be smooth.
\end{proof}

\begin{definition}\label{eval:def:point}
Let $(W,s)$ be a closed spin four-manifold, let $g:W\to\mathcal Q$ be a smooth
map, and let $p\in\mathcal Q$ be a regular value. Frame the normal bundle of the
closed one-manifold $g^{-1}(p)$ by pulling back a positive frame of
$T_p\mathcal Q$, and use this framing to induce a spin structure on this
one-manifold from $s$. The \emph{point-preimage bit}
$\mathfrak p(W,s,g)\in\ZZ/2$ is the class of this framed spin
one-manifold in $\Omega^{\Spin}_1\cong\ZZ/2$, that is, the number of its
components with closing sign $+1$, modulo two. When the spin structure
and the map are clear, we write $\mathfrak p(W)$.
\end{definition}

\begin{lemma}[Point-preimage evaluation]\label{eval:lem:point}
In Definition~\ref{eval:def:point}, the image of $[W,s,\iota\circ g]$ in
$\widetilde\Omega^{\Spin}_4(BQ_8)\cong\ZZ/2$ is $\mathfrak p(W,s,g)$.
Consequently $\mathfrak p(W,s,g)$ depends only on $(W,s)$ and on the
homotopy class of $\iota\circ g$. If $W$ has signature zero, then
$[W,s,\iota\circ g]=0$ in $\Omega^{\Spin}_4(BQ_8)$ if and only if
$\mathfrak p(W,s,g)=0$.
\end{lemma}

\begin{proof}
Use a CW structure on $\mathcal Q$ with a single three-cell containing
$p$. Since $\dim\mathcal Q=3$, the Atiyah--Hirzebruch filtration satisfies
$F_3\Omega^{\Spin}_4(\mathcal Q)=\Omega^{\Spin}_4(\mathcal Q)$. Its top
quotient is
$
   F_3/F_2=E^\infty_{3,1}(\mathcal Q).
$
No differential affects $E_{3,1}(\mathcal Q)$: the possible $d_2$ source
$E^2_{5,0}$ is zero for dimensional reasons, while the possible $d_2$ differential with target
$E^2_{1,2}$ is dual to
$\operatorname{Sq}^2:H^1(\mathcal Q;\ZZ/2)\to H^3(\mathcal Q;\ZZ/2)$,
which is zero; the possible $d_3$ target has coefficient
$\Omega^{\Spin}_3=0$, and there are no higher possibilities. Hence
$
   E^\infty_{3,1}(\mathcal Q)=H_3(\mathcal Q;\ZZ/2)\cong\ZZ/2.
$
More explicitly, the quotient by the two-skeleton gives the natural
relative-bordism map
\[
 \Omega^{\Spin}_4(\mathcal Q)
 \to
 \Omega^{\Spin}_4(\mathcal Q,\mathcal Q^{(2)})
 \cong
 \widetilde\Omega^{\Spin}_4(S^3)
 \cong
 \Omega^{\Spin}_1.
\]
The relative group is the coefficient group on the unique top cell.
The kernel of this map is $F_2$ by the exact sequence of the pair,
and the preceding differential calculation shows that the induced
map on $F_3/F_2$ is an isomorphism onto $\Omega^{\Spin}_1\cong\ZZ/2$.
The geometric description of the Atiyah--Hirzebruch spectral sequence
identifies the coefficient on a cell with the bordism class of a
regular-value preimage, equipped with its induced structure
\cite[Theorem~1.1 and the preceding paragraph, p.~172]{Tene2017}.
In the form used here, the edge map
$\Omega^{\Spin}_n(Y^{(k)})\to H_k(Y^{(k)};\Omega^{\Spin}_{n-k})$ of a
$k$-dimensional CW complex $Y^{(k)}$ sends the class of a map to
$\sum_i[N_i]\,e^k_i$, where $N_i$ is the framed, hence spin, regular preimage
of the barycentre of the $k$-cell $e^k_i$ \cite[Lemma~2.5]{KPT2020}.
Here the suspension identification frames the normal directions by
a positive frame of $T_p\mathcal Q$, so the induced spin structure
is exactly the one in Definition~\ref{eval:def:point}.
Consequently this map sends $[W,s,g]$ to $\mathfrak p(W,s,g)$.
Changing the positive frame through positive frames gives a
framing homotopy, and transversality applied to a homotopy of maps
gives the corresponding spin bordism of regular preimages
\cite[Chapter~3]{Hirsch}.

The map $\iota$ induces a morphism of spectral sequences, which on
$E^\infty_{3,1}$ is the isomorphism supplied by
Lemma~\ref{eval:lem:e31}(ii) and~(iii).
The class $\iota_*[W,s,g]$ and its reduced part differ by the image of
$[W,s]\in\Omega^{\Spin}_4$, which has filtration zero. Hence, the reduced
part has image $\mathfrak p(W,s,g)$ in $E^\infty_{3,1}(BQ_8)$, and
Lemma~\ref{eval:lem:e31}(ii) identifies the reduced group with this
quotient. The final assertion holds because a signature-zero class is
torsion.
\end{proof}

\begin{remark}[Spin-structure independence via point preimages]
\label{eval:rem:spin-independence}
The independence of the spin structure in the proof of
Corollary~\ref{quat:cor:bordism-bit} can also be seen directly from
Definition~\ref{eval:def:point}, without
\cite[Proposition~3.2.4]{Teichner1992}. Let $(W,s)$ be a closed spin
four-manifold, let $g\colon W\to\mathcal Q$ be a smooth map with regular
value $p$, and let $\omega\in H^3(\mathcal Q;\FF_2)$ be the generator. The
mod-two class of the one-manifold $g^{-1}(p)$ is Poincar\'e dual to
$g^*\omega$. If $\beta=g^*b$ with $b\in H^1(\mathcal Q;\FF_2)$, then, by
\eqref{nf:eq:spin-rules}, replacing $s$ by $s+\beta$ multiplies the closing
sign of each component $C$ of $g^{-1}(p)$ by
$(-1)^{\langle\beta,[C]\rangle}$, while the framing is unchanged. Hence
\[
 \mathfrak p(W,s+\beta,g)-\mathfrak p(W,s,g)
 =\bigl\langle\beta,[g^{-1}(p)]\bigr\rangle
 =\bigl\langle g^*(b\cup\omega),[W]_{\FF_2}\bigr\rangle=0,
\]
because $H^4(\mathcal Q;\FF_2)=0$. If $u\colon W\to BQ_8$ induces an
isomorphism on fundamental groups and $u\simeq\iota\circ g$
(Lemma~\ref{eval:lem:compress}), then every $\beta\in H^1(W;\FF_2)$ has this
form, since $\beta=u^*b'=g^*(\iota^*b')$ for some
$b'\in H^1(BQ_8;\FF_2)$. This applies to $X^{(3)}_\epsilon$ and, more
generally, to $X_{\epsilon,3m}$ with $m$ odd, whose fundamental group is
again $Q_8$.
\end{remark}

\begin{lemma}[Products with a circle]\label{eval:lem:product}
Every spin structure on $\mathcal Q\times S^1$ is a product
$s_{\mathcal Q}\times\sigma$. For the projection
$\mathrm{pr}:\mathcal Q\times S^1\to\mathcal Q$,
\[
 \mathfrak p(\mathcal Q\times S^1,s_{\mathcal Q}\times\sigma,\mathrm{pr})
 =\begin{cases}1,&\sigma\ \text{is the Lie structure},\\
               0,&\sigma\ \text{bounds}.\end{cases}
\]
Equivalently, the bit is $1$ exactly when a circle $\{x\}\times S^1$ with
a constant normal framing has closing sign $+1$. In particular
$[S^1_{\mathrm{Lie}}]\times[\mathcal Q,s_{\mathcal Q},\iota]$ is the
nonzero element of $\widetilde\Omega^{\Spin}_4(BQ_8)$.
\end{lemma}

\begin{proof}
Spin structures on $\mathcal Q\times S^1$ form a torsor over
$H^1(\mathcal Q;\ZZ/2)\oplus H^1(S^1;\ZZ/2)$, and the eight product
structures are distinct. The preimage of $p$ under $\mathrm{pr}$ is
$\{p\}\times S^1$ with a constant framing. The final assertion follows from
Lemma~\ref{eval:lem:point}.
\end{proof}

\subsubsection*{An odd dihedral transfer}
Let $D_{14}=\langle a,\rho\mid a^7=\rho^2=1,\ \rho a\rho=a^{-1}\rangle$.

\begin{proposition}[Odd dihedral transfer]\label{eval:prop:transfer}
Let $(W,s)$ be a closed connected spin four-manifold, let $g:W\to\mathcal Q$
be a smooth map, and let $\psi:\pi_1(W)\to D_{14}$ be an epimorphism. Let
$W^{\mathrm{irr}}\to W$ and $W^{\mathrm{reg}}\to W$ be the connected
coverings corresponding to $\psi^{-1}\langle\rho\rangle$ and $\ker\psi$,
of degrees $7$ and $14$, respectively, with the pulled-back spin structures and maps to
$\mathcal Q$. Then
\begin{equation}\label{eval:eq:transfer}
 \mathfrak p(W)=\mathfrak p(W^{\mathrm{irr}})+\mathfrak p(W^{\mathrm{reg}})
 \in\ZZ/2.
\end{equation}
\end{proposition}

\begin{proof}
Fix a regular value $p$. Its preimage under the map from each covering space
is the lift of $g^{-1}(p)$, with the pulled-back framing and spin structure. Let $C$ be a
component of $g^{-1}(p)$, write $[C]\in\ZZ/2$ for its class, and let
$h\in D_{14}$ be the image of its homotopy class, defined up to
conjugacy. The components of its preimage in $W^{\mathrm{irr}}$ and
$W^{\mathrm{reg}}$ correspond, respectively, to the orbits of
$\langle h\rangle$ on $D_{14}/\langle\rho\rangle$ and on $D_{14}$. A
component over an orbit of length $d$ covers $C$ with
degree $d$. Its closing sign is the $d$th power of that of $C$, as in the
proof of Lemma~\ref{nf:lem:calibration}. Thus, a lift of odd degree has
class $[C]$, and a lift of even degree is a Lie circle.

If $h=1$, the two contributions are $7[C]$ and $14[C]$. If $h$ is a
reflection, it fixes exactly one coset of $\langle\rho\rangle$ and
permutes the other six in three transpositions, while it acts on
$D_{14}$ by seven transpositions; the contributions are $[C]+3$ and $7$.
If $h$ is a nontrivial rotation, it has one orbit of length seven on the
cosets and two such orbits on $D_{14}$; the contributions are $[C]$ and
$2[C]$. In each case the two contributions add up to $[C]$ modulo two.
Summing over the components of $g^{-1}(p)$ proves \eqref{eval:eq:transfer}.
\end{proof}

Figure~\ref{eval:fig:transfer} illustrates this count and the two coverings
used below.

\begin{figure}[tbp]
\centering
\begin{tikzpicture}[
  cov/.style={draw=diagramblue!75,fill=diagramlight,rounded corners=3pt,line width=0.55pt,
    align=flush center,inner sep=3.5pt,font=\footnotesize}]
\begin{scope}
\fill[diagramlight] (0,0) circle[radius=1.18];
\foreach \k in {0,...,6} \coordinate (a\k) at (90+360*\k/7:1.0);
\draw[gray!45] (a0)--(a1)--(a2)--(a3)--(a4)--(a5)--(a6)--cycle;
\draw[orbitred,densely dashed,line width=0.55pt] (0,1.15) -- (0,-1.25);
\foreach \i/\j in {1/6,2/5,3/4} \draw[{Stealth[length=1.5mm]}-{Stealth[length=1.5mm]},liftblue,line width=0.8pt,shorten <=3pt,shorten >=3pt] (a\i) -- (a\j);
\foreach \k in {0,...,6} {\fill[diagramblue] (a\k) circle[radius=0.055]; \node[slab] at (90+360*\k/7:1.32) {$\k$};}
\draw[liftblue,line width=0.8pt] (a0) circle[radius=0.14];
\node[panelcap] at (0,-1.75) {(a) a reflection};
\end{scope}
\begin{scope}[shift={(3.25,0)}]
\fill[diagramlight] (0,0) circle[radius=1.18];
\foreach \k in {0,...,6} \coordinate (b\k) at (90+360*\k/7:1.0);
\foreach \k in {0,...,6} {\pgfmathtruncatemacro{\kk}{mod(\k+1,7)}\draw[-{Stealth[length=1.5mm]},liftblue,line width=0.8pt,shorten <=3pt,shorten >=3pt] (b\k) -- (b\kk);}
\foreach \k in {0,...,6} {\fill[diagramblue] (b\k) circle[radius=0.055]; \node[slab] at (90+360*\k/7:1.32) {$\k$};}
\node[panelcap] at (0,-1.75) {(b) $a$: $k\mapsto k+1$};
\end{scope}
\node[anchor=west,font=\footnotesize] at (5.15,0.1) {%
\renewcommand{\arraystretch}{1.25}%
\begin{tabular}{@{}lccc@{}}
\toprule
$h$ & $W^{\mathrm{irr}}$ & $W^{\mathrm{reg}}$ & sum\\
\midrule
$1$ & $7\,[C]$ & $14\,[C]$ & $[C]$\\
reflection & $[C]+3$ & $7$ & $[C]$\\
rotation $\neq1$ & $[C]$ & $2\,[C]$ & $[C]$\\
\bottomrule
\end{tabular}};
\node[panelcap] at (8.0,-1.75) {(c) contributions modulo two};
\begin{scope}[shift={(0,-3.35)}]
\node[cov] (R) at (0.6,0) {$M_3^{\mathrm{reg}}\cong\mathcal Q\times\RR/2\ZZ$};
\node[cov] (I) at (5.2,0) {$M_3^{\mathrm{irr}}\cong\mathcal Q\times\RR/\ZZ$};
\node[cov] (M) at (9.4,0) {$M_3$};
\draw[darrow] (R) -- node[above=2pt,slab]{$2{:}1$} (I);
\draw[darrow] (I) -- node[above=2pt,slab]{$7{:}1$} (M);
\node[slab] at (0.6,-0.58) {$\pi_1=Q_8\times\langle t_3^2\rangle=\ker\varpi_3$};
\node[slab] at (5.2,-0.58) {$\pi_1=Q_8\times\langle t_3\rangle=\varpi_3^{-1}\langle\rho\rangle$};
\node[slab] at (9.4,-0.58) {$\varpi_3\colon\pi_1(M_3)\twoheadrightarrow D_{14}$};
\node[slab,fill=liftblue!10,rounded corners=2pt,inner sep=2pt] at (0.6,-1.0) {bit $1$};
\node[slab,fill=liftblue!10,rounded corners=2pt,inner sep=2pt] at (5.2,-1.0) {bit $\eta_\varepsilon$};
\node[slab,fill=highlight!80,rounded corners=2pt,inner sep=2pt] at (9.4,-1.0) {bit $1+\eta_\varepsilon$};
\node[panelcap] at (5.0,-1.65) {(d) the two product covers of $M_3$};
\end{scope}
\end{tikzpicture}
\caption{The odd dihedral transfer (Proposition~\ref{eval:prop:transfer}) and the two
product covers (Lemmas~\ref{eval:lem:trivialize} and~\ref{eval:lem:signs}). (a), (b)
The image $h\in D_{14}$ of a component $C$ of a point preimage acts on the seven cosets
$a^k\langle\rho\rangle$: a reflection fixes one coset and interchanges the others in
pairs, and a nontrivial rotation permutes them in one seven-cycle. Panel~(b) draws the
generator $a$, acting by $k\mapsto k+1$; its nonidentity powers have the same cycle
type. (c) A lift of degree $d$ has class $[C]$ if $d$ is odd and is a Lie circle if $d$
is even; in every case the contributions of $W^{\mathrm{irr}}$ and $W^{\mathrm{reg}}$
add up to $[C]$ modulo two. (d) For $\varpi_3$ the two covers are products of
$\mathcal Q=S^3/Q_8$ with a circle representing $t_3$ or $t_3^2$, and their
point-preimage bits add up to $\delta_\varepsilon=1+\eta_\varepsilon$
(Theorem~\ref{eval:thm:bit}).}
\label{eval:fig:transfer}
\end{figure}

\subsubsection*{Two covers of the degree-three mapping torus}
In the model \eqref{nf:eq:mm}, $\pi_1(M_3)=G^{(3)}=\Gamma\rtimes\langle t_3\rangle$
acts on $S^3\times\RR$ through the maps $\gamma_{q,a}$ and
\begin{equation}\label{eval:eq:D3}
 D_3(v,s)=(f^{-3}(v),s+1)=(vj,s+1),
\end{equation}
because $f^{-3}(v)=\omega^3vj^{-3}$, $\omega^3=1$ and $j^{-3}=j$.
Since $\theta^3$ is the identity on $Q_8$ and inverts $c$, the formula
\[
 \varpi_3:G^{(3)}\to D_{14},\;\;\;
 \varpi_3(qc^kt_3^{\,n})=a^k\rho^n
\]
defines an epimorphism; it kills $Q_8$ and sends $c$ to $a$ and $t_3$ to
$\rho$. Since $\theta^3|_{Q_8}=\id$, one has
$\varpi_3^{-1}\langle\rho\rangle=Q_8\times\langle t_3\rangle$ and
$\ker\varpi_3=Q_8\times\langle t_3^2\rangle$. Put
\[
 M_3^{\mathrm{irr}}=(S^3\times\RR)/\langle\gamma_{q,0},D_3\rangle,
 \;\;\;
 M_3^{\mathrm{reg}}=(S^3\times\RR)/\langle\gamma_{q,0},D_3^2\rangle,
 \;\;\; q\in Q_8.
\]
These are the coverings of $M_3$ associated with $\varpi_3$ in
Proposition~\ref{eval:prop:transfer}.

\begin{lemma}[Both covers are products]\label{eval:lem:trivialize}
Put $\jmath(s)=\exp(\tfrac{\pi}{2}js)\in S^3$ and
$\widetilde\Upsilon(v,s)=(v\,\jmath(s),s)$ on $S^3\times\RR$. Then
$\widetilde\Upsilon$ commutes with $\gamma_{q,0}$ for every $q\in Q_8$, and
$
 \widetilde\Upsilon(v,s+1)=D_3\,\widetilde\Upsilon(v,s).
$
Hence, $\widetilde\Upsilon$ induces diffeomorphisms
\[
 \Upsilon_1:\mathcal Q\times\RR/\ZZ\to M_3^{\mathrm{irr}},
 \;\;\;
 \Upsilon_2:\mathcal Q\times\RR/2\ZZ\to M_3^{\mathrm{reg}}.
\]
The circle factors in these two products represent $t_3$ and $t_3^2$, respectively,
and the induced map on $\pi_1(\mathcal Q)=Q_8$ is the identity. For each of the two covers,
the pullback of the map $\bar u$ of \eqref{quat:eq:M3-to-Q8}, composed with
the corresponding $\Upsilon_k$, is homotopic to $\iota\circ\mathrm{pr}$.
\end{lemma}

\begin{proof}
Right multiplication commutes with the left action of $Q_8$, and
$\jmath(s+1)=\jmath(s)j$ gives the displayed identity by
\eqref{eval:eq:D3}. The path $s\mapsto\widetilde\Upsilon(v_0,s)$
($0\le s\le k$) runs from $(v_0,0)$ to $D_3^k(v_0,0)$, so its projection represents
$t_3^k$. Since $\chi(q)=q$ and $\chi(t_3)=1$ in \eqref{quat:eq:M3-to-Q8},
the induced homomorphism $Q_8\times\ZZ\to Q_8$ is the projection. As
$BQ_8$ is aspherical, this determines the map up to homotopy.
\end{proof}

\begin{lemma}[Closing signs of the product circles]\label{eval:lem:signs}
Let $C_1\subset M_3^{\mathrm{irr}}$ and $C_2\subset M_3^{\mathrm{reg}}$ be
the images under $\Upsilon_1$ and $\Upsilon_2$ of
$\{x_0\}\times\RR/\ZZ$ and $\{x_0\}\times\RR/2\ZZ$, respectively, each framed by the
image of a constant normal frame. For the pulled-back spin structure $s_3$,
\[
 \spn_{s_3}(C_1)=\spn_{s_3}(C_2)=+1.
\]
Consequently, for the spin structures \eqref{quat:eq:extending-spin},
\begin{equation}\label{eval:eq:product-signs}
 \spn_{\widehat s_{3,\epsilon}}(C_1)=(-1)^{1-\eta_\epsilon},\;\;\;
 \spn_{\widehat s_{3,\epsilon}}(C_2)=+1.
\end{equation}
\end{lemma}

\begin{proof}
Work in the trivialization \eqref{nf:eq:trivialization}. At $(v_0,s)$, the
constant frame whose first vector is the base direction is
$x\mapsto v_0x$, with spin lift $(v_0,1)$. The differential of
$\widetilde\Upsilon$ is right multiplication by $\jmath(s)$ together with
a triangular shear coming from the $s$-derivative of $v_0\jmath(s)$; the
shear can be removed through invertible orientation-preserving fibrewise
maps and does not affect sign comparisons.
To see that the shear is removed through closed loops of frames, write
$x=w+a\,v_0$ with $w\in T_{v_0}S^3$ and $a=\langle x,v_0\rangle\in\RR$. In
\eqref{nf:eq:trivialization} the differential of $\widetilde\Upsilon$ at
$(v_0,s)$ is $x\mapsto x\,\jmath(s)+a\,v_0\jmath'(s)$, where
$\jmath'(s)=\jmath(s)\,\tfrac{\pi}{2}j=\tfrac{\pi}{2}j\,\jmath(s)$. Hence,
for $0\le\tau\le1$, the maps
\[
 x\longmapsto\bigl(x+\tau a\,v_0\tfrac{\pi}{2}j\bigr)\,\jmath(s)
\]
join this differential, at $\tau=1$, to right multiplication by
$\jmath(s)$, at $\tau=0$. Each of them is invertible and orientation
preserving: since $v_0\tfrac{\pi}{2}j$ is orthogonal to $v_0$, the linear map
$x\mapsto x+\tau a\,v_0\tfrac{\pi}{2}j$ is unipotent. This linear map does
not depend on $s$, one has $\jmath(s+1)=\jmath(s)j$, and the derivative of
$D_3$ is right multiplication by $j$ in \eqref{nf:eq:trivialization}.
Together with $\widetilde\Upsilon(v,s+1)=D_3\widetilde\Upsilon(v,s)$, this
shows that every stage of the homotopy is compatible with the deck
identification. Thus it is a homotopy of closed loops of oriented frames in
$M_3^{\mathrm{irr}}$ and in $M_3^{\mathrm{reg}}$, and the closing signs may
be computed at $\tau=0$.
The transported frame $x\mapsto v_0x\,\jmath(s)$ therefore
has the lift
\[
 G(s)=\bigl(v_0,\jmath(s)^{-1}\bigr),\;\;\;
 G(1)=(v_0,-j),\;\;\; G(2)=(v_0,-1).
\]
The lifted path representing $C_1$ ends at the $D_3$-translate of its initial point, and the
prescribed lift of $D_3$ in \eqref{nf:eq:deck-spin} is
$d_3=(\omega^3,j^3)=(1,-j)$. Since $d_3G(0)=(v_0,-j)=G(1)$, the lift
closes. For $C_2$, the deck element is $D_3^2$, with lift $d_3^2=(1,-1)$,
and again $d_3^2G(0)=G(2)$. This proves the first assertion. The second
follows from $\widehat s_{3,\epsilon}=s_3+(1-\eta_\epsilon)a_3$,
$\langle a_3,[C_k]\rangle=k$, and \eqref{nf:eq:spin-rules}.
\end{proof}

\subsubsection*{The value of the bit}
\begin{theorem}[The quaternionic secondary bit]\label{eval:thm:bit}
Let $\mathbf x_{\mathrm{sec}}$ denote the nonzero element of
$\widetilde\Omega^{\Spin}_4(BQ_8)\cong\ZZ/2$. For each specified framing,
\begin{equation}\label{eval:eq:bit}
 [X_{\epsilon,3},s^{\rm out}_\epsilon,u_\epsilon]
 =[M_3,\widehat s_{3,\epsilon},\bar u]
 =(1+\eta_\epsilon)\,\mathbf x_{\mathrm{sec}}.
\end{equation}
Equivalently,
\[
 {\ \delta_+=\mathrm{sec}(X_{+,3})=1,\;\;\;
          \delta_-=\mathrm{sec}(X_{-,3})=0.\ }
\]
More precisely, for a compression $\bar u\simeq\iota\circ g$,
\[
 \mathfrak p(M_3,\widehat s_{3,\epsilon},g)
 =\mathfrak p(M_3^{\mathrm{irr}})+\mathfrak p(M_3^{\mathrm{reg}})
 =\eta_\epsilon+1.
\]
\end{theorem}

\begin{proof}
The first equality in \eqref{eval:eq:bit} follows from
Proposition~\ref{quat:prop:M3-bordism}; both manifolds have signature zero,
so both classes lie in the reduced group. Compress $\bar u$ to
$\iota\circ g$ by Lemma~\ref{eval:lem:compress}, and apply
Proposition~\ref{eval:prop:transfer} with $\psi=\varpi_3$.

Let $p_1:M_3^{\mathrm{irr}}\to M_3$ and
$p_2:M_3^{\mathrm{reg}}\to M_3$ be the covering projections.
Lemma~\ref{eval:lem:trivialize} gives, for $k=1,2$,
\[
 \iota\circ g\circ p_k
 \simeq \bar u\circ p_k
 \simeq \iota\circ\mathrm{pr}\circ\Upsilon_k^{-1}.
\]
These are homotopies of maps to $BQ_8$. We do not need the two maps
$g\circ p_k$ and $\mathrm{pr}\circ\Upsilon_k^{-1}$ to be homotopic as
maps to $\mathcal Q$. Lemma~\ref{eval:lem:point} identifies their
point-preimage bits. Lemma~\ref{eval:lem:product} and
\eqref{eval:eq:product-signs} therefore give
\[
 \mathfrak p(M_3^{\mathrm{irr}},\widehat s_{3,\epsilon})=\eta_\epsilon,
 \;\;\;
 \mathfrak p(M_3^{\mathrm{reg}},\widehat s_{3,\epsilon})=1.
\]
Hence, $\mathfrak p(M_3,\widehat s_{3,\epsilon},g)=1+\eta_\epsilon$, and
Lemma~\ref{eval:lem:point} gives \eqref{eval:eq:bit}. The values of
$\delta_\epsilon$ follow from Corollary~\ref{quat:cor:bordism-bit}.
\end{proof}

\subsection{The Gluck twins are inequivalent}
\label{eval:sec:consequences}

\begin{corollary}[The two third cyclic covers]\label{eval:cor:third-covers}
The manifolds $X_{+,3}$ and $X_{-,3}$ are not homotopy equivalent and are
not stably homeomorphic. In particular, they are not diffeomorphic, either as
unmarked manifolds or as manifolds with their ramification spheres, regardless
of orientation.
\end{corollary}

\begin{proof}
This follows from Theorem~\ref{eval:thm:bit} and
Corollary~\ref{quat:cor:unmarked-delta}.
\end{proof}

\begin{corollary}[The Gluck twins are inequivalent]\label{eval:cor:pairs}
\begin{enumerate}[label=(\roman*),leftmargin=2.2em]
\item No diffeomorphism $\Sig\to\Sigma_-(7)$, whether orientation-preserving
      or orientation-reversing, carries $K_+$ onto $K_-$, regardless of the
      chosen orientations of the spheres. Equivalently, no diffeomorphism of $\E$ carries $\phi_+$ to
      $\phi_-$ up to a diffeomorphism of $S^2\times S^1$ that extends over
      the cap.
\item Every smooth oriented fibred two-knot with quandle $P_7$ in a smooth
      homotopy four-sphere is diffeomorphic as a pair, possibly after reversing
      the knot orientation, to exactly one of
      $(\Sig,K_+)$ and $(\Sigma_-(7),K_-)$. Thus
      the classification in Theorem~\ref{thm:rigidity} consists of exactly
      two types under equivalence allowing reversal of the knot orientation.
\item Neither $K_+$ nor $K_-$ is smoothly reflexive: neither is
      determined, as a pair, by its exterior.
\end{enumerate}
\end{corollary}

\begin{proof}
(i) Suppose that $\Phi$ is such a diffeomorphism. After an isotopy, it
preserves tubular neighbourhoods of the spheres and is fibre-linear and
orthogonal on their normal disc bundles: tubular-neighbourhood
straightening \cite[Chapter~4, Section~5]{Hirsch} gives the fibre-linear form,
and polar decomposition of the
normal derivative gives the orthogonal form. Its restriction to $\E$
induces an automorphism of $G_7$ preserving
$G_7'=\Gamma$, hence preserving $G^{(3)}$, the unique subgroup of index
three containing $\Gamma$. It therefore lifts to the threefold cyclic
covers of the exteriors. On each normal disc it is complex linear or
antilinear, with a coefficient given by a map $S^2\to S^1$; such a map has a
continuous cube root. Indeed, $[S^2,S^1]=0$, so every coefficient map
$S^2\to S^1$ is null-homotopic and therefore lifts through the degree-three
covering $S^1\to S^1$, $z\mapsto z^3$. The coefficient map is smooth,
and its continuous lift is automatically smooth because the covering is
a local diffeomorphism. Thus, a suitable lift extends over
the lifted caps through $z\mapsto z^3$. This gives a diffeomorphism
$X_{+,3}\to X_{-,3}$,
contradicting Corollary~\ref{eval:cor:third-covers}.

For the equivalent boundary formulation, an extendable diffeomorphism of
$S^2\times S^1$ admits an extension preserving the core
$S^2\times\{0\}$. Indeed, the mapping class group of $S^2\times S^1$ is
$(\ZZ/2)^3$, generated by the two product orientation reversals and the
Gluck twist \cite{Gluck1962}, \cite[Sections~1--2]{DingGeiges2010}.
The product reversals extend over the cap and preserve its core.
After composition with the necessary product reversals, an extendable boundary map preserves both
orientations and the spin structure induced by the unique spin structure
of $S^2\times D^2$. The Gluck twist interchanges the two boundary spin
structures (Figure~\ref{fig:gluck}(b)): along the product circle over a fixed point of the rotation on
$S^2$, its derivative rotates the normal
two-plane once, changing the closing sign. Hence the Gluck-twist component
of an extendable map is zero. Such a map is therefore isotopic to a product
map; extend the isotopy over a boundary collar, away from the core, to
obtain the required core-preserving extension. Thus the stated boundary
condition gives a diffeomorphism of pairs. Conversely, a diffeomorphism
of pairs restricts to the stated exterior map after the straightening above.

(ii) Combine (i) with Theorem~\ref{thm:rigidity}.

(iii) By \eqref{eq:gluck}, the Gluck reconstruction of $(\Sig,K_+)$ is
$(\Sigma_-(7),K_-)$, and conversely; apply (i).
\end{proof}

Figure~\ref{fig:lift} illustrates the lifting argument used to prove part~(i).
\begin{figure}[tbp]
\centering
\begin{tikzpicture}[
  mb/.style={draw=#1!75,fill=#1!7,rounded corners=3pt,line width=0.55pt,
    align=flush center,font=\footnotesize,inner sep=4pt,minimum width=2.55cm},
  mb/.default=diagramblue,
  vflow/.style={-{Stealth[length=2mm,width=1.5mm]},draw=diagramblue,line width=0.6pt},
  hyp/.style={-{Stealth[length=2mm,width=1.5mm]},draw=orbitred,line width=0.7pt,densely dashed}]
\node[mb=liftblue] (Xp) at (0,2.3) {$X_{+,3}\supset R_{+,3}$};
\node[mb=liftorange] (Xm) at (5.6,2.3) {$X_{-,3}\supset R_{-,3}$};
\node[mb] (Sp) at (0,0) {$(\Sigma_+(7),K_+)$};
\node[mb] (Sm) at (5.6,0) {$(\Sigma_-(7),K_-)$};
\draw[vflow] (Xp) -- node[left=1pt,slab,align=right]{$3{:}1$, branched\\ along $K_+$} (Sp);
\draw[vflow] (Xm) -- node[right=1pt,slab,align=left]{$3{:}1$, branched\\ along $K_-$} (Sm);
\draw[hyp] (Sp) -- node[above=2pt,slab,text=orbitred]{$\Phi$ (hypothetical)}
  node[below=3.4mm,slab,align=center,text=inkgray]{$\Phi_*$ preserves $\Gamma=G_7'$,\\ hence $G^{(3)}$} (Sm);
\draw[hyp] (Xp) -- node[above=4.5pt,slab,text=orbitred]{the lift $\widetilde\Phi$ would exist} (Xm);
\node[font=\small,text=orbitred,fill=white,inner sep=0.5pt] at (2.8,2.3) {$\boldsymbol\times$};
\node[slab,fill=liftblue!10,rounded corners=2pt,inner sep=2pt,anchor=south] at ($(Xp.north)+(0,0.08)$) {$\mathrm{sec}=\delta_+=1$};
\node[slab,fill=liftorange!12,rounded corners=2pt,inner sep=2pt,anchor=south] at ($(Xm.north)+(0,0.08)$) {$\mathrm{sec}=\delta_-=0$};
\node[panelcap] at (2.8,-1.55) {(a) a diffeomorphism of pairs would lift};
\begin{scope}[shift={(9.85,1.15)}]
\def\r{0.66}
\foreach \y/\lab/\cc/\ang in {1.15/{z\mapsto b(x)\,z}/liftblue/40,-1.15/{w\mapsto a(x)\,w}/diagramblue/120} {
  \begin{scope}[shift={(0,\y)}]
  \shade[inner color=white,outer color=diagramblue!18] (0,0) circle[radius=\r];
  \draw[diagramblue!70,line width=0.5pt] (0,0) circle[radius=\r];
  \fill[belt] (0,0) circle[radius=0.05];
  \draw[-{Stealth[length=1.5mm,width=1.2mm]},\cc,line width=0.8pt] (15:0.42) arc[start angle=15,end angle=15+\ang,radius=0.42];
  \node[slab,text=\cc,anchor=west] at (0.78,0.05) {$\lab$};
  \end{scope}}
\foreach \k in {0,1,2} \draw[gray!50,line width=0.4pt] (0,1.15) -- ++(120*\k:\r);
\draw[vflow] (0,0.38) -- node[left=1pt,slab]{$z\mapsto w=z^3$} (0,-0.38);
\node[slab,text=diagramblue,anchor=west] at (0.78,-1.5) {$a(x)=b(x)^3$};
\node[slab,anchor=east,text=inkgray] at (-0.8,1.15) {upstairs};
\node[slab,anchor=east,text=inkgray] at (-0.8,-1.15) {downstairs};
\node[panelcap] at (0.85,-2.7) {(b) the cube root on a normal disc};
\end{scope}
\end{tikzpicture}
\caption{The lifting argument in the proof of Corollary~\ref{eval:cor:pairs}(i). (a) A
diffeomorphism of pairs $\Phi$ induces on the exteriors an automorphism of $G_7$ that
preserves the commutator subgroup $\Gamma$, hence the unique index-three subgroup
$G^{(3)}$ containing $\Gamma$; so it lifts to the threefold cyclic covers of the
exteriors. The lift would extend to a diffeomorphism $X_{+,3}\to X_{-,3}$, which is
impossible by Corollary~\ref{eval:cor:third-covers}, since the secondary invariants
differ (Theorem~\ref{eval:thm:bit}). With
homeomorphisms in place of diffeomorphisms, and Fox completions in place of~(b), the
same diagram gives Remark~\ref{eval:rem:topological}. (b) After
straightening, $\Phi$ acts on each normal disc of $K_+$ as a complex linear or
antilinear map with coefficient $a(x)\in S^1$, $x\in S^2$ (the linear case is drawn).
Since $[S^2,S^1]=0$, the map $a$ has a cube root $b$, and $z\mapsto b(x)z$ is a lift
of $w\mapsto a(x)w$ through the branched covering $z\mapsto w=z^3$; it extends over
the lifted cap.}
\label{fig:lift}
\end{figure}

\begin{remark}[What the diffeomorphism statements use]
\label{eval:rem:what-is-used}
The smooth inequivalence argument requires only the top
Atiyah--Hirzebruch filtration quotient, not the complete
computation of $\Omega^{\Spin}_4(BQ_8)$. Since
$E^\infty_{4,0}(BQ_8)=0$, the filtration satisfies
$F_3=\Omega^{\Spin}_4(BQ_8)$, so there is a natural quotient map
\[
 q:\Omega^{\Spin}_4(BQ_8)\to
 F_3/F_2=E^\infty_{3,1}(BQ_8)\cong\ZZ/2.
\]
The Steenrod-square calculation in Lemma~\ref{eval:lem:e31} proves that
no differential affects $E_{3,1}$; it does not require the $d_3$
computation of Remark~\ref{eval:rem:e22}.
The point-preimage argument of Lemma~\ref{eval:lem:point} identifies
$q[W,s,\iota\circ g]$ with $\mathfrak p(W,s,g)$ without using
$E^\infty_{2,2}=0$. The transfer and product-cover calculation proving
Theorem~\ref{eval:thm:bit} therefore also gives
\[
 q[X_{+,3},s,u_+]=1,\;\;\; q[X_{-,3},s,u_-]=0.
\]
For classifying maps that induce isomorphisms on fundamental groups,
Remark~\ref{eval:rem:spin-independence} proves independence of $s$.
Changing a classifying map acts by an automorphism of $Q_8$, whose
induced automorphism of $E^\infty_{3,1}\cong\ZZ/2$ is the identity;
reversing orientation multiplies the value by $-1=1$. Thus any
diffeomorphism, with either orientation, would preserve these values.

This proves the diffeomorphism assertions of
Corollaries~\ref{eval:cor:third-covers} and~\ref{eval:cor:pairs}, and thus
the smooth inequivalence and nonreflexivity stated in
Corollary~\ref{std:cor:knots}(ii), without the complete reduced-bordism
calculation or the identification with the intrinsic invariant
$\mathrm{sec}$. The same argument preserves the framing label in
Remark~\ref{std:rem:label-preserved}. The identification
\cite[Theorem~6.4.1]{Teichner1992} is used for the stronger assertions
about homotopy equivalence and stable homeomorphism.
\end{remark}

\begin{remark}[Topological inequivalence]\label{eval:rem:topological}
The knots $K_+$ and $K_-$ are not even topologically equivalent: no
homeomorphism $\Sig\to\Sigma_-(7)$ carries $K_+$ onto $K_-$, whatever the
orientations. Indeed, such a homeomorphism restricts to a homeomorphism of the
knot complements. As in the proof of Corollary~\ref{eval:cor:pairs}(i), it
preserves the index-three subgroup of the knot group that contains the
commutator subgroup $\Gamma$, so it lifts to the threefold cyclic coverings of
the complements. A homeomorphism of the bases that lifts to these unbranched
coverings extends to their completions in the sense of Fox \cite{Fox1957},
and these completions are the branched covers $X_{+,3}$ and $X_{-,3}$. The
result would be a homeomorphism $X_{+,3}\to X_{-,3}$, contradicting
Corollary~\ref{eval:cor:third-covers}. In particular, neither knot is
reflexive in the topological category. Unlike the smooth statements, this
argument uses the homotopy invariance of $\mathrm{sec}$, and hence
\cite[Theorem~6.4.1]{Teichner1992}.
\end{remark}

\begin{remark}[Relation with known reflexivity results]
\label{eval:rem:reflexivity}
Plotnick's odd-order nonreflexivity theorem
\cite[Theorem~6.2]{Plotnick1986} provides background for the use of cyclic
branched covers to distinguish Gluck twins. The corresponding theorem for fibred two-knots
whose closed fibre is a spherical space form states that such knots are not
reflexive when the monodromy has nontrivial odd order
\cite[Theorem~1.1, p.~514]{PlotnickSuciu}. Branched $r$-twist
spins of prime simple knots with geometric cyclic branched covers are not
reflexive for $r>2$ \cite{HillmanPlotnick1990}; see also
\cite[Section~18.7]{Hillman}. The broader expectation of nonreflexivity for
monodromy of any finite order greater than two is recorded as an open problem in
\cite[Chapter~18, p.~341]{Hillman}. For fibred two-knots with a spherical closed
fibre, reflexivity had been decided for ten of the seventeen possible families
\cite[Section~5, Table~1, p.~525]{PlotnickSuciu}, as summarized in \cite[Introduction]{Hillman2011}, where
the remaining knots were expected to be nonreflexive. 

None of these earlier results directly proves our assertion. Here the displayed representative of the
unpointed closed monodromy class has order six, not odd order. There is no
periodic punctured-fibre monodromy \cite[pp.~5--6]{Hillman2023}, and finite order in that unpointed closed class
must not be confused with periodic monodromy of the knot exterior
\cite[Section~16.3, p.~316]{Hillman}. Corollary~\ref{eval:cor:pairs}(iii) proves, in the smooth category, the
corresponding statement for
$\pi'\cong Q(8)\times(\ZZ/7\ZZ)$. The table in \cite[Section~5, Table~1, p.~525]{PlotnickSuciu} lists the
family $\pi'\cong Q(8)\times(\ZZ/n\ZZ)$, with $n\neq1$ and
order-six monodromy, as realized only in a homotopy four-sphere; the table
leaves reflexivity undecided. Corollary~\ref{eval:cor:pairs}(iii) gives the $n=7$ smooth standard-sphere case.

In particular, nonperiodicity of a representative that fixes a
boundary collar pointwise would not by itself exclude the classical periodic
examples. Indeed, a finite-order diffeomorphism that is the identity on
a nonempty open collar is the identity everywhere: average a metric over
the finite cyclic group and use the fact that an isometry of a connected manifold
is determined by its restriction to a nonempty open set. In the
periodic-monodromy literature the boundary is therefore allowed to move.
The stronger fact relevant to the present prism case is the absence of
a finite-order geometric punctured-fibre monodromy, as recalled in
Remark~\ref{rem:closed-vs-punctured-monodromy} and
\cite[Proposition~15.21]{QextC}; compare
\cite[Section~16.3]{Hillman}.

For the approach to reflexivity via mapping tori and framings, see
\cite{HillmanPlotnick1990} and \cite[Sections~18.1--18.2]{Hillman}.
Corollary~\ref{eval:cor:pairs} separates the present twins by the intrinsic
secondary invariants of their third cyclic branched covers
(Theorem~\ref{eval:thm:bit}).
\end{remark}

\section{Consequences}\label{std:sec:consequences}

\begin{corollary}[Fibred realizations of $P_7$ in $S^4$]
\label{std:cor:knots}
\begin{enumerate}[label=(\roman*),leftmargin=2.2em]
\item $P_7$ has smooth fibred realizations in the standard $S^4$; both
      alternatives in \eqref{eq:fibred} hold.
\item Under identifications of $\Sigma_\pm(7)$ with $S^4$, the knots $K_+$ and $K_-$
      are smooth fibred two-knots in $S^4$ with group $G_7$ and fibre
      $\N\setminus\Int D^3$ (Proposition~\ref{prop:exists}). Their unpointed
      closed monodromy class is represented by the fixed-point-free order-six
      isometry $\bar f$, whereas neither knot admits a finite-order geometric
      punctured-fibre monodromy
      (Remark~\ref{rem:closed-vs-punctured-monodromy}). Their exteriors are
      diffeomorphic, but the knots are not equivalent for any choice of
      orientations (Corollary~\ref{eval:cor:pairs}(i)), even in the topological
      category (Remark~\ref{eval:rem:topological}). In particular, each is
      a nonreflexive knot in $S^4$ (Corollary~\ref{eval:cor:pairs}(iii)).
      Their third cyclic branched covers are not homotopy equivalent
      (Corollary~\ref{eval:cor:third-covers}).
\item For every smooth oriented fibred two-knot with quandle $P_7$ in a smooth
      homotopy four-sphere, the ambient manifold is diffeomorphic to the standard
      $S^4$, and the resulting pair is diffeomorphic, possibly after reversing
      the knot orientation, to one of the two pairs $(\Sigma_+(7),K_+)$ and
      $(\Sigma_-(7),K_-)$
      (Theorems~\ref{thm:rigidity} and~\ref{std:thm:main}).
\item For every $m\geq1$ with $m\equiv\pm1\pmod 6$ and for either framing, the $m$-fold cyclic
      branched cover $X_{\eps,m}=\Sigma^{(m)}_\eps$ of
      Theorem~\ref{per:thm:tower} is diffeomorphic to $S^4$. Moreover, the
      pair $(X_{\eps,m},R_{\eps,m})$ is diffeomorphic, possibly after
      reversing orientations, to $(\Sigma_\eps(7),K_\eps)$.
\end{enumerate}
\end{corollary}

\begin{proof}
Part~(i) follows from Theorem~\ref{std:thm:main} and the equivalence
\eqref{eq:fibred}. For part~(ii), Proposition~\ref{prop:exists} gives the
common fibred exterior, fibre, quandle and associated knot group, while
Theorem~\ref{std:thm:main} identifies both ambient homotopy spheres with the
standard $S^4$. The closed-monodromy assertions follow from
Theorem~\ref{per:thm:periodic}, Proposition~\ref{per:prop:fixed}, and
Remark~\ref{rem:closed-vs-punctured-monodromy}. Inequivalence and
nonreflexivity follow from Corollary~\ref{eval:cor:pairs}(i) and~(iii), and
the assertion about the third cyclic branched covers follows from
Corollary~\ref{eval:cor:third-covers}. Part~(iii) follows by combining
the pair classification of Theorem~\ref{thm:rigidity}, including its knot-orientation
qualification, with Theorem~\ref{std:thm:main}.

For part~(iv), Theorem~\ref{per:thm:periodic}(ii) gives $\bar f^{\,6}=\id$,
so for $m\equiv\pm1\pmod 6$ the covering mapping torus $M_m$ of \eqref{nf:eq:mm}
is the mapping torus of $\bar f^{\,\pm1}$. The mapping torus of
$\bar f^{\,-1}$ is diffeomorphic to $\M$ by $[x,s]\mapsto[x,1-s]$. In both
cases there is therefore a diffeomorphism $M_m\to\M$, and it carries the
section $\gamma_m$ to an embedded loop of degree $\pm1$. By
\eqref{nf:eq:lifted-surgery}, $(X_{\eps,m},R_{\eps,m})$ is a framed surgery
pair on that loop, so Lemma~\ref{lem:weight-loops} identifies it, possibly after
reversing orientations, with $(\Sigma_+(7),K_+)$ or $(\Sigma_-(7),K_-)$.
Theorem~\ref{std:thm:main} now gives $X_{\eps,m}\diff S^4$.
Remark~\ref{std:rem:label-preserved} shows that the label is $\eps$.
\end{proof}

\begin{remark}[The label in Corollary~\ref{std:cor:knots}(iv)]
\label{std:rem:label-preserved}
In Corollary~\ref{std:cor:knots}(iv) the label is preserved:
$(X_{\eps,m},R_{\eps,m})$ is diffeomorphic, possibly after reversing
orientations, to $(\Sigma_\eps(7),K_\eps)$. First, for every odd $m\geq1$,
\[
 \mathrm{sec}(X_{\eps,3m})=1+\eta_\eps=\delta_\eps.
\]
Indeed, $3m\equiv3\pmod 6$, so $\pi_1(X_{\eps,3m})\cong Q_8$ by
Theorem~\ref{per:thm:tower}, and $\theta^{3m}=\theta^3$. The spin-extension
and homology arguments in Proposition~\ref{quat:prop:X3-spin}, together
with the proofs of Proposition~\ref{quat:prop:M3-bordism},
Corollary~\ref{quat:cor:bordism-bit}, Lemmas~\ref{eval:lem:trivialize}
and~\ref{eval:lem:signs}, and Theorem~\ref{eval:thm:bit}, apply to $M_{3m}$ and
$X_{\eps,3m}$ after replacing $D_3$, $d_3$ and $\jmath$ by
$D_{3m}(v,s)=(vj^{-3m},s+1)$, $d_{3m}=(1,j^{3m})$ and
$\jmath_{3m}(s)=\exp(-\tfrac{3m\pi}{2}js)$, which satisfies
$\jmath_{3m}(s+1)=\jmath_{3m}(s)j^{-3m}$. The chosen extending spin structure is
$\widehat s_{3m,\eps}=s_{3m}+(1+\eta_\eps)a_{3m}$, because
$3m\eta_\eps\equiv\eta_\eps\pmod 2$. The product circles again have closing
sign $+1$ for $s_{3m}$, since the lifted frame
$G(s)=(v_0,\jmath_{3m}(s)^{-1})$ satisfies $G(1)=(v_0,j^{3m})=d_{3m}G(0)$ and
$G(2)=d_{3m}^2G(0)$. Hence the bit is again $1+\eta_\eps$.

Now let $m\geq1$ satisfy $m\equiv\pm1\pmod 6$. The knot group of $R_{\eps,m}$ is
$\Gamma\rtimes_{\theta^{\pm1}}\langle t^m\rangle$, with commutator subgroup
$\Gamma$, and the threefold cyclic branched cover of $(X_{\eps,m},R_{\eps,m})$
is $X_{\eps,3m}$, corresponding to the subgroup
$\Gamma\rtimes\langle t^{3m}\rangle$.
This identification is compatible with the smooth fillings. On the
complements, the successive degree-$m$ and degree-three cyclic covers
correspond to the subgroup chain
\[
 \Gamma\rtimes\langle t^{3m}\rangle
 \ \leq\ \Gamma\rtimes\langle t^m\rangle
 \ \leq\ G_7.
\]
In a normal-disc coordinate, the composite covering is
$z\mapsto(z^3)^m=z^{3m}$. The lifted framing is likewise the
$3m$-fold pullback of the original framing. Thus the iterated cover
is $X_{\eps,3m}$ with the same calibrated framing label $\eps$. If $(X_{\eps,m},R_{\eps,m})$ were
diffeomorphic to $(\Sigma_{\eps'}(7),K_{\eps'})$ with $\eps'\neq\eps$, the
lifting argument in the proof of Corollary~\ref{eval:cor:pairs}(i) would give
a diffeomorphism $X_{\eps,3m}\to X_{\eps',3}$. This is impossible: the point-preimage calculation above gives
$q[X_{\eps,3m},s,u]=\delta_\eps\neq\delta_{\eps'}
=q[X_{\eps',3},s',u']$, whereas a diffeomorphism preserves these quotient
values by Remark~\ref{eval:rem:what-is-used}. Thus this label argument
requires only the diffeomorphism-invariant quotient, not the homotopy
invariance of $\mathrm{sec}$.
\end{remark}


\subsection*{Acknowledgements}

The author used OpenAI's ChatGPT and Anthropic's Claude for editorial
assistance, mathematical discussion, figure generation, and verification of
computer code.

\end{document}